\documentclass[10pt,a4paper,reqno]{article}

\usepackage[latin,english]{babel}
\usepackage[utf8]{inputenc}
\usepackage[T1]{fontenc}
\usepackage{eufrak, color, mathrsfs,dsfont}
\usepackage{geometry}
\usepackage{amsmath}
\usepackage{amssymb}
\usepackage{amsthm}
\usepackage{mathtools}
\usepackage{mathabx}
\usepackage{cases}
\usepackage{braket}
\usepackage[toc,page]{appendix}
\usepackage{pdfsync}
\usepackage{hyperref}
\usepackage{psfrag}
\usepackage[mathscr]{euscript}
\usepackage{math}
\numberwithin{equation}{section}

\theoremstyle{plain}
\newtheorem{thm}{Theorem}[section]
\newtheorem{lem}[thm]{Lemma}
\newtheorem{prop}[thm]{Proposition}

\theoremstyle{definition}
\newtheorem{defn}[thm]{Definition}

\theoremstyle{remark}
\newtheorem{rem}[thm]{Remark}

\newcommand{\LieTr}[2]{e^{-#1} #2 e^{#1}}

\DeclareMathOperator{\diag}{diag}

\DeclareMathOperator{\dist}{dist}

\newcommand{\Op}{{\rm Op}}
\newcommand{\Ops}{{\rm OP}S}

\renewcommand{\bar}{\overline}

\newcommand{\even}{{\rm even}}
\newcommand{\odd}{{\rm odd}}
\newcommand{\ora}[1]{\vec{#1}}
\newcommand{\pa}{\partial}
\newcommand{\vs}{\varsigma}
\newcommand{\vphi}{\varphi}
\newcommand{\IK}{\mathfrak K}
\newcommand{\sign}{{\rm sign}\,}
\newcommand{\normk}[2]{\| #1 \|_{#2}^{k_0,\upsilon}}

\renewcommand{\whi}{\widehat \imath}
\renewcommand{\wti}{\widetilde \imath}
\newcommand{\acca}{\fH}

\title{\bf Traveling quasi-periodic  
water waves \\
 with constant vorticity}

\begin{document}

\date{}

\author{M. Berti, 
 L. Franzoi,
 A. Maspero\footnote{
International School for Advanced Studies (SISSA), Via Bonomea 265, 34136, Trieste, Italy. 
 \textit{Emails: } \texttt{berti@sissa.it}, \texttt{luca.franzoi@sissa.it}, \texttt{alberto.maspero@sissa.it} 
 }}

\maketitle

\noindent
{\bf Abstract.}
We prove the first bifurcation result of time quasi-periodic {\it traveling} waves 
solutions for space periodic water waves with vorticity. In particular 
we prove existence of small amplitude 
time quasi-periodic  solutions of the gravity-capillary water waves equations 
with  {\it constant vorticity},  for a bidimensional fluid over a flat bottom 
delimited by a space-periodic free interface. 
These quasi-periodic solutions exist for all the values of depth, gravity and vorticity,  and restricting 
the surface tension  to a Borel set of asymptotically full Lebesgue measure. 

\smallskip


\smallskip


\tableofcontents

\section{Introduction and main result}

The search for traveling surface waves in  inviscid  fluids is a very important problem in fluid mechanics, widely studied  
since the pioneering work of Stokes \cite{stokes} in 1847. 
The existence of  steady traveling  waves,
namely solutions which look stationary in a moving frame,  
either periodic or localized in space, 
  is nowadays  well understood in many different situations, 
mainly  for bidimensional fluids. 

On the other hand, 
the natural question regarding the existence of  {\it time quasi-periodic traveling 
waves} --which can not be reduced to steady solutions in a moving frame-- has been
not answered so far. This is the goal of the present paper. We  consider  space periodic 
waves. Major difficulties in this project concern the presence of  `small divisors"  
and the quasi-linear nature of the equations. 
Related difficulties appear in the search of time periodic standing waves
which have been constructed  in the last  years
in a series of papers 
by Iooss, Plotnikov, Toland \cite{PlTo, IPT, IP-SW1,IP-Mem-2009} for pure gravity waves,  by 
Alazard-Baldi \cite{AB} in presence of surface tension and subsequently extended
to time quasi-periodic standing waves solutions
 by Berti-Montalto  \cite{BM} and Baldi-Berti-Haus-Montalto \cite{BBHM}.  Standing
 waves are not traveling 
 as they are even in the space variable. We  also mention that all these last results concern  irrotational fluids.   

In this paper we prove the first existence result of {\em time quasi-periodic traveling} wave solutions for 
the gravity-capillary water waves equations 
with {\em constant vorticity} for bidimensional fluids. 
The small amplitude solutions that we construct  exist for any value of the vorticity (so also for irrotational fluids), any value of the gravity and 
 depth of the fluid, and provided the surface tension is restricted to a Borel set of asymptotically full measure, see Theorem \ref{thm:main0}. 
 For irrotational fluids the traveling wave solutions that we construct  
do not clearly reduce to  the standing wave solutions in \cite{BM}. 
We remark that, in case of 
non zero vorticity, one can not expect the bifurcation of standing waves since they are not allowed by the linear theory.

Before presenting in detail our main result, we  introduce the water waves equations. 

\paragraph{The water waves equations.}

We consider the Euler equations of hydrodynamics for a 2-dimensional perfect, incompressible, inviscid fluid with {constant vorticity $\gamma$}, under the action of gravity and capillary forces at the free surface.
The fluid fills  an ocean with  depth $\tth > 0 $ (eventually infinite) and with space periodic boundary conditions,   namely it  occupies the region
\begin{equation}
\label{domain}
	\cD_{\eta, \tth} := \big\{ (x,y)\in \T\times \R \ : \ -\tth\leq y<\eta(t,x) \big\} \, , 
	\quad \T := \T_x :=\R/ (2\pi\Z) \,.
\end{equation} 
The unknowns of the problem are the divergence free  velocity field
$\begin{pmatrix} u(t,x,y) \\ v(t,x,y) \end{pmatrix} $ 
which solves the Euler equation and the free surface  $ y = \eta (t, x)$
of the time dependent domain $\cD_{\eta,\tth} $. 
In case of a fluid with constant vorticity 
$$ v_x - u_y = \gamma \, ,  $$
 the velocity field   is the sum of the Couette flow $\begin{pmatrix} - \gamma y \\ 0 \end{pmatrix}$, which carries
all the  vorticity $ \gamma $ of the fluid,  and an irrotational field, 
expressed as the gradient of a harmonic function $\Phi $, called the generalized velocity potential. 

Denoting by $\psi(t,x)$  the evaluation of the generalized velocity potential at the free interface
$ \psi (t,x) := \Phi (t,x, \eta(t,x)) $, 
one recovers $ \Phi $ by solving the elliptic problem
\begin{equation}
\label{dir}
\Delta \Phi = 0  \ \mbox{ in } \cD_{\eta, \tth} \, , \quad
\Phi = \psi \  \mbox{ at } y = \eta(t,x) \, , \quad
\Phi_y \to  0  \  \mbox{ as } y \to  - \tth \, .
\end{equation}
The third condition in \eqref{dir}  means the impermeability  property of the bottom
$$
\Phi_y ( t, x, - \tth) = 0 \, ,  \ {\rm if} \ \tth < \infty  \, , \qquad
\lim\limits_{y \to - \infty } \Phi_y ( t, x, y) = 0 \, , \  {\rm if} \ \tth = + \infty  \, . 
$$
Imposing  
that the fluid particles at the free surface remain on it along the evolution
(kinematic boundary condition), and that
the pressure of the fluid 
plus the capillary forces at the free surface  is equal to the constant 
atmospheric pressure (dynamic boundary condition),  the 
time evolution of the fluid is determined by the following 
system of equations 
(see \cite{CIP, Wh})
\begin{equation}
\label{ww}
\begin{cases}
\eta_t = G(\eta)\psi + \gamma \eta \eta_x \\
\displaystyle{\psi_t = - g\eta  - \frac{\psi_x^2}{2} + 
\frac{(  \eta_x \psi_x + G(\eta)\psi)^2}{2(1+\eta_x^2)} +\kappa \Big(\frac{\eta_x}{\sqrt{1+\eta_x^2}} \Big)_{\!\!x} + \gamma \eta \psi_x  + \gamma \partial_x^{-1} G(\eta) \psi} \, .
\end{cases}
\end{equation}
Here $ g $ is the gravity, $ \kappa $ is the surface tension coefficient, which we assume to belong to an interval $ [\kappa_1, \kappa_2] $ with $ \kappa_1 > 0 $, and 
$G(\eta)$ is the  Dirichlet-Neumann operator  
\begin{equation}
\label{DN}
G(\eta)\psi := G(\eta,\tth)\psi := \sqrt{1+\eta_x^2} \, (\partial_{\vec n} \Phi )\vert_{y = \eta(x)} = (- \Phi_x \eta_x + \Phi_y)\vert_{y = \eta(x)} \, . 
\end{equation}
The water waves equations \eqref{ww} are a Hamiltonian system that we describe
in Section \ref{ham.s},  and 
 enjoy two important symmetries. 
 First, they are time reversible: we say that a solution of \eqref{ww} is \emph{reversible} if  
\begin{equation}\label{time-rev}
\eta(-t,-x) = \eta (t,x) \, , \quad  \psi (-t,-x) = - \psi (t,x) \, . 
\end{equation}
Second, since the bottom of the fluid domain is flat, the equations \eqref{ww} 
 are invariant by space translations. 
We refer to Section \ref{ham.s} for more details.

Let us comment shortly about the phase space of \eqref{ww}. As 
$ G(\eta)\psi $ is a function with zero average, 
the quantity $\int_\T \eta(x) \, \di x$ is a prime integral of \eqref{ww}.
Thus, with no loss of generality,  we restrict to interfaces  with zero spatial average
$ \int_\T \eta(x) \, \di x = 0 $.  
Moreover,  since $ G(\eta ) [1] = 0 $, the vector field on the right hand side of \eqref{DN} depends only on $ \eta $ and $  \psi - \frac{1}{2 \pi}\int_\T \psi \, \di x  $. As a consequence, 
the variables $ (\eta, \psi) $ of system \eqref{ww} belong to some Sobolev space
$ H^s_0(\T) \times \dot H^s (\T) $ for some $ s $ large.  
Here  $H^s_0(\T)$, $s \in \R$, 
 denotes the Sobolev space of functions with zero average
$$
H^s_0(\T) := \Big\{
u \in H^s(\T) \ \colon 
\  
\int_\T u(x) \di x = 0 \Big\} 
$$
and  $\dot H^s(\T)$, $s \in \R$, the corresponding 
homogeneous Sobolev space, namely the quotient space obtained by identifying all 
the $H^s(\T)$ functions which differ only by a constant. For simplicity of notation
we shall denote the equivalent class $ [\psi] = \{ \psi + c,   c \in \R \} $, just by 
$ \psi $.  

\paragraph{Linear water waves.}

When looking to small amplitude solutions of \eqref{ww}, 
a fundamental role is played by  
the  system obtained linearizing \eqref{ww} at the equilibrium $(\eta, \psi) = (0,0)$, namely
\begin{equation}
\label{lin.ww1}
\begin{cases}
\partial_t \eta & =G(0) \psi \\
\partial_t \psi & = -(g-\kappa\partial_x^2) \eta + \gamma\partial_x^{-1} G(0) \psi \, .
\end{cases} 
\end{equation}
The  Dirichlet-Neumann operator at the flat surface $\eta = 0$  is  the  Fourier 
multiplier 
\begin{equation}\label{G(0)}
G(0) := G(0,\tth) = 
\begin{cases}
D \, \tanh( \tth D) & {\rm if } \ \tth < \infty \\
|D| & {\rm if } \ \tth = + \infty \, , 
\end{cases} \qquad {\rm where} \qquad D := \frac{1}{\im} \partial_x  \, , 
\end{equation}
with symbol
\begin{equation}\label{def:Gj0}
G_j(0):= 	G_j(0,\tth)= \begin{cases}
	j\tanh(\tth j) & \text{ if }\tth<\infty \\
	\abs j & \text{ if } \tth=+\infty \, . 
	\end{cases} 
\end{equation}
As we will show in Section \ref{lin:op}, 
all reversible solutions (see \eqref{time-rev}) of \eqref{lin.ww1}
are
{\begin{equation}
\label{linz200}
\begin{aligned}
\begin{pmatrix}
\eta(t,x) \\ \psi(t,x)
\end{pmatrix} & =  \sum_{n \in \N}  \begin{pmatrix}
M_n \rho_n \cos ( n x - \Omega_n (\kappa) t) \\ 
P_n \rho_n \sin ( n x - \Omega_n (\kappa) t) 
\end{pmatrix} \\ 
& + \sum_{n \in \N}
\begin{pmatrix}
M_n \rho_{-n} \cos ( n x + \Omega_{-n}(\kappa) t) \\ 
P_{-n} \rho_{-n} \sin ( n x + \Omega_{-n} (\kappa) t) 
\end{pmatrix} 
 \,,
 \end{aligned}
\end{equation}}
where $\rho_n\geq 0$ are arbitrary amplitudes and   $M_n$ and $P_{\pm n}$ 
are the real coefficients
 \begin{equation}\label{def:Mn}
 M_j :=\left( \frac{G_j(0)}{\kappa j^2 + g +
\frac{\gamma^2}{4} \frac{G_j(0)}{j^2}} \right)^{\frac14}, \ 
j \in \Z \setminus \{0\} \,  , 
 \quad 
P_{\pm n} := \frac{\gamma}{2} \frac{M_n}{n} \pm M_n^{-1}, \ 
n \in \N \, .
\end{equation}
Note  that the map $ j \mapsto M_j$  is even. 
The frequencies $\Omega_{\pm n}(\kappa)$ in \eqref{linz200} are 
\begin{equation}
\label{def:Omegajk}
\Omega_j(\kappa) := 
\sqrt{ \Big( \kappa  j^2 +g  + \frac{\gamma^2}{4}\frac{G_j(0)}{j^2}   \Big) 
G_j(0) } + \frac{\gamma}{2}\frac{G_j(0)}{j} , \qquad
j \in \Z \setminus\{0 \} \, . 
\end{equation}
Note that the map $ j \mapsto \Omega_j (\kappa )$
is not even due to the vorticity term $ \gamma G_j (0) / j $, which is odd in $j $. 
Note that  $\Omega_j(\kappa) $ 
actually depends also on the depth $ \tth $, the gravity $ g $ and the vorticity $ \gamma $, 
but we highlight in \eqref{def:Omegajk} only its dependence with respect to the surface tension coefficient $ \kappa $, 
since in this paper we shall move just $ \kappa $ 
as a parameter to impose suitable non-resonance conditions, see Theorem \ref{thm:main0}. 
Other choices are possible. 

All the linear solutions \eqref{linz200}, depending on the irrationality properties of the frequencies 
$ \Omega_{\pm n} (\kappa) $ and the number of non zero 
amplitudes $\rho_{\pm n}  > 0 $, 
are either time periodic, quasi-periodic or almost-periodic.
Note that  the 
functions \eqref{linz200} are the   linear superposition  of plane waves traveling 
either to the right 
or to the left. 

\begin{rem}
Actually,  \eqref{linz200} contains also standing waves, for example when
the vorticity $\gamma = 0$ (which implies $ \Omega_{-n}(\kappa) = \Omega_n(\kappa) $, $ P_{-n} = - P_n$) and $ \rho_{-n} = \rho_n $,  giving  solutions  even in $x$. 
This is the well known  superposition effect of 
waves with the same amplitude, frequency and wavelength  
 traveling in opposite directions. 
\end{rem}

\paragraph{Main result.}
We first provide  the notion of  quasi-periodic traveling wave.

\begin{defn} {\bf (Quasi-periodic traveling wave)} \label{def:TV}
We say that $ (\eta (t,x), \psi(t,x)) $ is a 
time quasi-periodic {\it traveling} wave  
with irrational frequency vector  $ \omega = ( \omega_1, \ldots, \omega_\nu)  \in \R^\nu $, $ \nu \in \N $,  i.e.
$ \omega \cdot 	\ell \neq 0 $, $ \forall \ell \in \Z^\nu \setminus \{0 \} $, 
and ``wave vectors'' $ ( j_1, \ldots, j_\nu)  \in \Z^\nu $,  if there exist 
functions
$ ( \breve\eta, \breve\psi ) : \T^\nu \to \R^2 $ such that 
\begin{equation} \label{trav-etazeta}
\begin{pmatrix}
\eta ( t, x) \\ \psi ( t, x) 
\end{pmatrix} = 
\begin{pmatrix}
\breve\eta( \omega_1  t- j_1 x ,\ldots, \omega_\nu t- j_\nu x )  \\
\breve\psi( \omega_1  t- j_1 x ,\ldots, \omega_\nu t- j_\nu x )
\end{pmatrix} \, . 
\end{equation}
\end{defn}
\begin{rem}
If $ \nu = 1 $, such functions are  time periodic and indeed stationary in a moving frame
with speed $ \omega_1 / j_1 $.
On the other hand, if the number of frequencies  $ \nu $ is $ \geq 2 $,  
the waves \eqref{trav-etazeta} cannot be reduced to steady waves by any appropriate choice of the moving frame. 
\end{rem}
In this  paper we  shall construct traveling quasi-periodic  
solutions of \eqref{ww} with a diophantine frequency vector $ \omega $ belonging to an open bounded 
subset $  {\mathtt \Omega} $ in $ \R^\nu $, 
namely, for some 
$ \upsilon \in (0,1) $, $ \tau >  \nu - 1 $, 
\begin{equation}\label{def:DCgt}
\tD\tC (\upsilon, \tau) := \Big\{ \omega\in {\mathtt \Omega} \subset\R^\nu \ : \ 
\abs{\omega\cdot \ell}\geq \upsilon \braket{\ell}^{-\tau} \ , 
\ \forall\,\ell\in\Z^\nu \setminus \{0\} \, , \ \langle \ell \rangle :=  \max\{1, |\ell| \} 
\Big\} \, . 
\end{equation}
Regarding regularity, we will prove the existence of 
quasi-periodic traveling waves $ (\breve\eta,  \breve\psi ) $
 belonging to some Sobolev space 
\begin{equation} \label{unified norm}
 H^s(\T^{\nu}, \R^2)
= \Big\{ \breve f (\vf) = 
\sum_{\ell \in \Z^{\nu}} f_{\ell} \, e^{\im \ell \cdot \vf } \ , \ \ \ f_\ell \in \R^2 \ \ : \, 
\| \breve f \|_s^2 := \sum_{\ell \in \Z^{\nu}} | f_{\ell}|^2 \langle \ell \rangle^{2s} < \infty 
\Big\} \, . 
\end{equation}
Fixed   finitely  many  arbitrary  
{\em  distinct} natural numbers 
\begin{equation}
\label{Splus}
\S^+ := \{ \bar n_1, \ldots, \bar n_\nu \}\subset \N  \ , \quad
1 \leq \bar n_1 < \ldots < \bar n_\nu \, , 
\end{equation}
and signs
\begin{equation}
\label{signs}
\Sigma := \{ \sigma_1 , \ldots , \sigma_\nu \} , \quad \sigma_a \in \{ -1, 1 \} \, ,
\quad a = 1, \ldots, \nu \, ,  
\end{equation}
consider the reversible quasi-periodic traveling wave 
solutions   of the linear system \eqref{lin.ww1} given by 
 \begin{equation}
\label{sel.sol}
\begin{aligned}
\begin{pmatrix}
\eta(t,x) \\ \psi(t,x)
\end{pmatrix} & =  \sum_{a \in \{1, \ldots, \nu \colon   \sigma_a = + 1\}}  
\begin{pmatrix}
M_{\bar n_a} \sqrt{\xi_{\bar n_a}} \cos ( \bar n_a x - \Omega_{\bar n_a} (\kappa) t) \\ 
P_{\bar n_a} \sqrt{\xi_{\bar n_a}} \sin ( \bar n_a x - \Omega_{\bar n_a} (\kappa) t) 
\end{pmatrix} \\ 
& +   \sum_{a \in \{1, \ldots, \nu \colon   \sigma_a = - 1\}}
\begin{pmatrix}
M_{\bar n_a} \sqrt{\xi_{- \bar n_a}} \cos ( \bar n_a x + \Omega_{- \bar n_a}(\kappa) t) \\ 
 P_{-\bar n_a} \sqrt{\xi_{- \bar n_a}} \sin ( \bar n_a x + \Omega_{- \bar n_a} (\kappa) t) \ 
\end{pmatrix}  
 \end{aligned}
\end{equation}
where $ \xi_{\pm \bar n_a} >  0 $, $ a  = 1, \ldots, \nu $. The frequency vector 
of \eqref{sel.sol}  is 
\begin{equation}\label{Omega-kappa}
\vec \Omega (\kappa) := (\Omega_{\sigma_a \bar n_a} (\kappa ))_{a=1, \ldots, \nu}  
\in \R^\nu \, . 
\end{equation}
\begin{rem}\label{rem:st}
If $\sigma_a = +1$, we select in \eqref{sel.sol}  a right traveling wave, whereas, if $\sigma_a = -1$, a left traveling one. By \eqref{Splus}, the linear solutions \eqref{sel.sol} are genuinely  traveling waves: superposition of identical waves traveling in opposite direction, generating standing waves, does not happen.
\end{rem}

The main result of this paper proves that the  linear 
solutions \eqref{sel.sol}
can be continued to  quasi-periodic traveling wave  solutions of the nonlinear 
water waves equations
\eqref{ww}, 
for most values of the surface tension $ \kappa \in [\kappa_1, \kappa_2 ]$,
 with a  frequency vector 
$ \widetilde  \Omega  := ( \widetilde  \Omega_{\sigma_a \bar n_a})_{a=1, \ldots, \nu}  $,   
 close to 
$  \vec \Omega (\kappa) := (\Omega_{\sigma_a \bar n_a} (\kappa))_{a =1, \ldots, \nu} $. 
Here is the precise statement.

\begin{thm} \label{thm:main0}  {\bf (KAM for traveling gravity-capillary 
water waves with constant vorticity)}
Consider finitely many tangential sites $ \S^+ \subset \N  $
as in \eqref{Splus} and signs $ \Sigma $ as in \eqref{signs}. Then 
there exist $ \bar s >  0 $,  
$ \varepsilon_0 \in (0,1) $ such that,  
for every $ |\xi |   \leq \varepsilon_0^2  $, 
$  \xi := (\xi_{ \sigma_a {\bar n}_a} )_{a = 1, \ldots, \nu} \in \R_+^\nu $, the following hold:
\begin{enumerate}
\item
there exists 
a Cantor-like set  $ {\cal G}_\xi \subset [\kappa_1, \kappa_2] $ 
with asymptotically full measure as $ \xi \to 0 $, i.e. 
$ \lim_{\xi \to 0} | {\cal G}_\xi |  = {\kappa}_2- {\kappa}_1 $;
\item
for any  $ \kappa \in {\cal G}_\xi $,  the 
gravity-capillary water waves equations \eqref{ww}
have a  reversible 
quasi-periodic traveling wave solution (according to 
Definition \ref{def:TV}) of the form 
\begin{equation}
\label{QP:soluz}
\begin{aligned}
\begin{pmatrix}
\eta( t ,x) \\ \psi( t ,x)
\end{pmatrix} & =  \sum_{a \in \{1, \ldots, \nu\} \colon   \sigma_a = + 1}  
\begin{pmatrix}
M_{\bar n_a} \sqrt{\xi_{\bar n_a}} \cos ( \bar n_a  x - \widetilde{\Omega}_{ \bar n_a}  t) \\ 
P_{\bar n_a} \sqrt{\xi_{\bar n_a}} \sin ( \bar n_a x - \widetilde{\Omega}_{\bar n_a} t) 
\end{pmatrix}   \\ 
& +   \sum_{a \in \{1, \ldots, \nu\} \colon   \sigma_a = - 1}
\begin{pmatrix}
M_{\bar n_a} \sqrt{\xi_{- \bar n_a}} \cos ( \bar n_a x + \widetilde{\Omega}_{-\bar n_a}  t) \\ 
P_{-\bar n_a} \sqrt{\xi_{- \bar n_a}} \sin ( \bar n_a x + \widetilde{\Omega}_{-\bar n_a}   t) 
\end{pmatrix} + r ( t, x ) 
 \end{aligned}
\end{equation}
where
$$
 r ( t, x )
 = 
\breve r(\wt\Omega_{\sigma_1 \bar n_1}t-\sigma_1\bar n_1 x,\ldots, \wt\Omega_{\sigma_\nu \bar n_\nu}t-\sigma_\nu\bar n_\nu x)  \, , \quad \breve r \in H^{\bar s} ( \T^\nu , \R^2) \, ,
\quad  \lim_{\xi \to 0} \frac{\| \breve r \|_{\bar s}}{\sqrt{|\xi|}} = 0 \, , 
$$
 with a Diophantine 
frequency vector $ \widetilde  \Omega  := ( \widetilde  \Omega_{\sigma_a \bar n_a})_{a=1, \ldots, \nu} \in \R^\nu $, depending on $\kappa, \xi$, and  satisfying 
$ \lim_{\xi \to 0}{\widetilde \Omega} = \vec \Omega (\kappa) $. 
In addition these quasi-periodic solutions are  linearly stable.  
\end{enumerate}
\end{thm}

Let us make some comments.
\\[1mm]
\indent  1) 
 Theorem \ref{thm:main0} holds for any value of the vorticity $\gamma$, so in particular  it guarantees existence of quasi-periodic traveling waves also for irrotational fluids, i.e. $\gamma =0$.
In this case 
the solutions \eqref{QP:soluz} 
do not reduce to those in \cite{BM}, 
which are standing, i.e. even in $ x $.
If  the vorticity $ \gamma \neq 0 $,  
one does not expect the existence of standing wave solutions
since the water waves vector field \eqref{ww}
does not leave invariant the subspace of functions even in $ x $. 
\\[1mm]
\indent  2) Theorem \ref{thm:main0} produces time quasi-periodic solutions of the Euler equation with a velocity field which is a small perturbation of the Couette flow $\begin{pmatrix}
-\gamma y \\ 0
\end{pmatrix}$. Indeed, from the solution $ (\eta(t, x), \psi(t, x))$  in \eqref{QP:soluz}, one  recovers the generalized velocity potential $\Phi(t, x,y)$ by solving the elliptic problem \eqref{dir} and finally constructs the velocity field 
$
\begin{pmatrix}
u(t,x,y) \\
v(t,x,y)
\end{pmatrix} = 
\begin{pmatrix}
-\gamma y\\
0
\end{pmatrix} + \grad \Phi(t,x,y)$. 
The time quasi-periodic potential  $\Phi(t, x,y)$ has size  $O(\sqrt{|\xi|})$, 
as $\eta(t,x)$ and $\psi(t,x)$. 
\\[1mm]
\indent  3)  In the case $ \nu = 1 $ the solutions constructed in Theorem \ref{thm:main0} reduce to steady  
periodic traveling waves, 
which can be obtained 
by an application of the Crandall-Rabinowitz theorem, see e.g. \cite{Martin, Wh0, Wh1}.  
\\[1mm] 
\indent
4)  Theorem \ref{thm:main0} selects initial data giving raise to global in time solutions
\eqref{QP:soluz} of the water waves equations \eqref{ww}. So far,  
 no results about global existence for \eqref{ww} with periodic boundary conditions 
 are known. The available results concern  local well posedness with a general vorticity, see e.g. \cite{CouSh},  and a $ \varepsilon^{-2} $ existence for
  initial data of size $  \varepsilon $ 
 in the case of constant vorticity \cite{IT}.
\\[1mm]
\indent  5) 
With the choice \eqref{Splus}-\eqref{signs} the unperturbed frequency vector
$ \vec \Omega (\kappa) = (\Omega_{\sigma_a \bar n_a} (\kappa ) )_{a = 1, \ldots, \nu} $ 
is diophantine for most values  of the surface tension $ \kappa $ and for all values of  vorticity, gravity and depth. It follows by the more general 
results of Sections \ref{sec:deg_kam} and \ref{subsec:measest}.
This may  not be  
true for an arbitrary choice of the linear frequencies 
$ \Omega_{j}(\kappa) $, $ {j \in \Z\setminus\{0\}}$. 
For example, in the case $ \tth = + \infty $, the vector
$$ 
\vec \Omega(\kappa) = \big( \Omega_{-n_3}(\kappa), \Omega_{-n_2}(\kappa), \Omega_{-n_1}(\kappa), \Omega_{n_1}(\kappa), \Omega_{n_2}(\kappa), 
\Omega_{n_3}(\kappa) \big)   
$$
is resonant, for all the values of $ 	\kappa $,  also 
taking into account the restrictions on the indexes for the search of traveling waves, see
Section \ref{subsec:momentum}.  
Indeed, recalling \eqref{def:Omegajk} and that, for $ \tth = + \infty $,  $ G_j (0, \tth) = |j| $, 
we have, 
for $ \ell = 
\big( - \ell_{n_3}, - \ell_{n_2}, - \ell_{n_1},  \ell_{n_1},  \ell_{n_2},  \ell_{n_3} \big) $
 that the system 
$$
\vec \Omega(\kappa) \cdot \vec \ell = \gamma  ( \ell_{n_1} + \ell_{n_2} +   \ell_{n_3} )  = 0 
\, ,  \quad 
  n_1 \ell_{n_1} + n_2 \ell_{n_2} + n_3 \ell_{n_3} = 0 \, , 
$$
has integer solutions. 
In this case the possible existence of quasi-periodic solutions of the water waves system 
\eqref{ww} depends on the frequency modulation induced by the nonlinear terms. 
\\[1mm]
{6) {\sc Comparison with \cite{BM}.}  
There are significant differences with respect to \cite{BM}, which proves the existence
of quasi-periodic {\it standing} waves for {\it irrotational} fluids, 
not only in the result --the solutions of Theorem \ref{thm:main0} 
are {\it traveling} waves of fluids with {\it constant vorticity}-- but also in the techniques. 

({\it 1})  
The first difference --which is a novelty of this paper-- is a new formulation  of degenerate KAM theory exploiting ``momentum conservation'', namely the space invariance
of the Hamilton equations. The degenerate  KAM theory approach for PDEs has been 
developed in \cite{BaBM}, and then \cite{BM}, \cite{BBHM}, 
in order to prove the
non-trivial dependence of the linear frequencies with respect to a 
parameter --in our case the surface tension 
 $ \kappa $--, see the ``Transversality" 
Proposition \ref{prop:trans_un}.  A key assumption  used in 
\cite{BaBM}, \cite{BM}, \cite{BBHM} is that the linear frequencies are simple
(because of Dirichlet boundary conditions in \cite{BaBM} and  
Neumann boundary conditions  in 
\cite{BM}, \cite{BBHM}). 
This is not true for traveling waves (e.g. in case of zero vorticity one has $\Omega_j(\kappa) = \Omega_{-j}(\kappa)$ identically in $\kappa$). In order to deal with these resonances
we strongly exploit the invariance of the 
equations \eqref{ww} under space translations, which ultimately imply   
the restrictions to the indexes \eqref{eq:1_meln}-\eqref{eq:2_meln+}.
In this way,  assuming that the moduli of the tangential sites are all different
as in \eqref{Splus}, cfr. with item 5),  we can remove  some otherwise possibly  degenerate case.  
This requires to keep trace along 
all the proof of the ``momentum conservation property'' that we 
characterize in different ways in Section \ref{subsec:momentum}.
The momentum conservation law 
has been used in several KAM results for semilinear PDEs 
since the works \cite{geng1, geng2}, \cite{kl, PP}, 
see also  \cite{MP18, G, FGP} and references therein. 
The present paper gives a new application  in the context of degenerate KAM theory
(with additional difficulties arising by the quasi-linear nature of the water waves equations).

({\it 2}) Other significant differences with respect to \cite{BM} arise 
in the reduction in orders (Section \ref{sec:linnorm})
of the quasi-periodic linear operators obtained 
along the Nash-Moser iteration. 
In particular  we mention  that we have to preserve the Hamiltonian nature of these 
operators (at least until Section \ref{sec:block_dec}). 
Otherwise it would appear a  time dependent
operator at the order $ |D|^{1/2} $,  of the form $ \im a(\vf) {\cal H} |D|^{\frac12} $, 
with $ a(\vf) \in \R $ independent of $ x $, 
compatible with the reversible structure,  
which can not be eliminated.   Note that the operator 
$ \im a(\vf) {\cal H} |D|^{\frac12} $ is not Hamiltonian (unless $ a(\vf) = 0 $). 
Note also that the above difficulty was not present in \cite{BM} 
dealing with standing waves, because an operator of the form 
$ \im a(\vf) {\cal H} |D|^{\frac12} $ 
does not map even functions into even functions. 
In order to overcome this difficulty 
we have to perform always symplectic changes of variables (at least until Section \ref{sec:block_dec}), and not just reversible  ones as in \cite{BM, BBHM}. 
We finally mention that we perform 
as a first step in Section  \ref{sec:repa}  
a quasi-periodic time reparametrization 
to avoid otherwise a technical difficulty in the conjugation 
of the remainders obtained by the Egorov theorem in Section \ref{sec:order32}.
This difficulty was not present in \cite{BM}, since
it arises conjugating the additional pseudodifferential term 
due to vorticity, see Remark \ref{rem:REPA}. 
\\[1mm]
{7)} Another novelty  of our result is to exploit the momentum conservation  also to prove  that the obtained quasi-periodic  solutions  are indeed 
quasi-periodic traveling waves, according to Definition \ref{def:TV}. 
This requires to check that the approximate solutions 
constructed along the Nash-Moser iteration of Section \ref{sec:NaM} 
(and  Section \ref{sec:approx_inv}) are indeed traveling waves.
Actually this approach  shows that the preservation of the 
momentum condition along  the Nash-Moser-KAM 
iteration is equivalent to the construction of 
embedded invariant tori 
which support quasi-periodic  traveling waves, namely 
of the form 
$ u(\vf,x) = U(\vf-\ora{\jmath}x)  $ (see Definition \ref{QPTW}), or  equivalently,   
in action-angle-normal variables,  which satisfy 
\eqref{mompres_aa}. We expect  this method  can be used to obtain quasi-periodic traveling waves for other PDE's which are translation invariant.

\paragraph{Literature.}\label{sec:lit}

We now shortly describe the literature 
regarding the existence of  time periodic or quasi-periodic solutions of the water waves equations, focusing 
on the results more related to Theorem \ref{thm:main0}.
We describes only results concerning space periodic waves, 
that we divide in   three  distinct
  groups:   
\begin{description}
\item $(i)$ steady traveling solutions,  
\item $(ii)$ time periodic standing waves,
\item  $(iii)$  time quasi-periodic standing waves.
\end{description}
This distinction takes into account 
not only  the different shapes of 
the waves,  but also  the  techniques for their construction.
\\[1mm]
($i$) {\it Time and space periodic traveling waves which are steady in a moving frame}.
The literature concerning  steady traveling wave solutions is huge, and we 
refer to \cite{const_book} for an extended presentation. Here we only mention that,
after the pioneering work of Stokes  \cite{stokes}, the first rigorous construction 
of small amplitude  space periodic steady traveling waves goes back to the 1920's with the papers of 
Nekrasov \cite{Nek}, Levi-Civita \cite{LC} and Struik \cite{Struik}, in case of  irrotational bidimensional flows  under the action of pure gravity.
Later  Zeidler \cite{Zei} considered the effect of capillarity.
In the presence of vorticity, the first result is due to  Gerstner \cite{gerstner} in 1802, who gave an explicit example of periodic traveling wave, in infinite depth, and with a particular non-zero vorticity.
One has to wait  the work of Dubreil-Jacotin \cite{dubreil} in 1934 for the first existence  results of small amplitude, periodic traveling waves with general  (H\"older continuous, small) vorticity, and, later,  the works of  Goyon \cite{goyon} and Zeidler \cite{Zei2} in the case of large vorticity.
More recently we point out the works of Wahl\'en \cite{Wh0}  for capillary-gravity waves and non-constant vorticity, and of Martin  \cite{Martin} and Walh\'en \cite{Wh} for constant vorticity. 
All these results deal with 2d water waves, and  can ultimately 
be  deduced by the Crandall-Rabinowitz 
bifurcation theorem from a simple eigenvalue. 

We also mention that these local bifurcation results can be extended to  global branches 
of steady traveling waves by applying the methods of global bifurcation theory. We refer to 
Keady-Norbury \cite{KN}, Toland \cite{To}, McLeod \cite{ML}
for irrotational flows and  Constantin-Strauss \cite{CS}
for fluids with non-constant vorticity. 

In the case of three dimensional irrotational fluids, bifurcation of small amplitude traveling waves periodic in space  has been proved  in Reeder-Shinbrot \cite{RS}, Craig-Nicholls \cite{CN,CN2} 
for both gravity-capillary  waves (by  variational bifurcation arguments \a la Weinstein-Moser)  
and  by Iooss-Plotnikov \cite{IP-Mem-2009,IP2} for gravity waves (this is a small divisor problem).   
These solutions,
in a moving frame,
look steady bi-periodic waves.  
\\[1mm]
($ii$) {\it Time  periodic standing waves}. 
Bifurcation of 
time periodic standing water waves were obtained in a series of  pioneering 
papers
by Iooss, Plotnikov and Toland   \cite{PlTo, IPT, IP-SW1,IP-Mem-2009} for  pure gravity waves, and by  Alazard-Baldi \cite{AB}  for  gravity-capillary fluids. 
Standing waves are even in the space variable and so they do not travel in space. 
 There is a huge difference with the results of the first group: 
 the construction of time periodic standing waves involves small divisors. 
Thus the proof  
is based on  Nash-Moser implicit function techniques and not only on the
 classical implicit function theorem. 
\\[1mm]
($iii$) {\em Time quasi-periodic standing waves}. The first results in this direction 
were obtained very recently by Berti-Montalto \cite{BM} for the gravity-capillary system and by Baldi-Berti-Haus-Montalto \cite{BBHM} for the gravity water waves. 
Both papers deal with  irrotational fluids. 

\section{Hamiltonian structure and linearization at the origin }

In this section we describe the Hamiltonian structure of the water waves equations \eqref{ww}, their symmetries and the solutions of the linearized system \eqref{lin.ww1} at the equilibrium.

\subsection{Hamiltonian structure}\label{ham.s}

The Hamiltonian formulation of the water waves equations \eqref{ww} 
with non-zero constant 
vorticity was obtained by Constantin-Ivanov-Prodanov  \cite{CIP} and Wahl\'en 
 \cite{Wh} in the case of finite depth. 
For irrotational flows it reduces to the classical Craig-Sulem-Zakharov formulation
in \cite{Zak1}, \cite{CrSu}. 

On the phase space $H^1_0(\T) \times \dot H^1(\T)$, endowed with the non canonical Poisson tensor
 \begin{equation}
 \label{eq:magn_sympl}
 J_M(\gamma) := \begin{pmatrix}
0 & {\rm Id} \\
- {\rm Id}  & 
\gamma \partial_x^{-1} 
 \end{pmatrix} \, ,
\end{equation}
we
consider the Hamiltonian 
\begin{equation}
\label{ham1}
H(\eta, \psi) = 
\frac12 \int_{\T} \left( \psi G(\eta ) \psi + g \eta^2 \right) \wrt x 
+ \kappa  \int_{\T}  \sqrt{1 + \eta_x^2} \, \wrt x + \frac{\gamma}{2} 
\int_{\T} 
\left( -   \psi_x \eta^2 + \frac{\gamma}{3} \eta^3 \right) \wrt x \, . 
\end{equation}
Such Hamiltonian  is well defined on $ H^1_0(\T) \times \dot H^1(\T) $ since $ G(\eta ) [1] = 0 $ 
 and $ \int_{\T} G(\eta) \psi \, \di x = 0 $.

It turns out \cite{CIP, Wh} 
that equations \eqref{ww} are the Hamiltonian system generated by
$H(\eta, \psi)$ with respect to the Poisson tensor $J_M(\gamma)$, namely 
\begin{equation}
\label{ham.eq1}
\partial_t 
\begin{pmatrix}
\eta \\
\psi  
 \end{pmatrix} = J_M (\gamma) 
 \begin{pmatrix}
\nabla_\eta H  \\
 \nabla_\psi H
 \end{pmatrix}
\end{equation}
 where  $ (\nabla_\eta H, \nabla_\psi H) \in \dot L^2(\T) \times L^2_0(\T) $  denote the $ L^2 $-gradients.  

\begin{rem}\label{rem:dual}
	The non canonical Poisson tensor $J_M(\gamma)$ in \eqref{eq:magn_sympl} has to be regarded as an operator from (subspaces of) $(L_0^2\times \dot L^2)^* = \dot{L}^2\times L_0^2$ to $L_0^2\times \dot{L}^2$, that is
	\begin{equation*}
		J_M(\gamma) = \begin{pmatrix}
		0 & {\rm Id}_{L_0^2\to L_0^2} \\ -{\rm Id}_{\dot{L}^2\to \dot{L}^2} & \gamma\pa_x^{-1}
		\end{pmatrix}\,.
	\end{equation*}
	The operator $ \pa_x^{-1} $ maps a dense subspace of $ L^2_0 $ in $ \dot L^2 $. 
	For sake of simplicity, throughout the paper we may omit this detail. 
 Above the dual space $(L_0^2\times \dot L^2)^* $ with respect to the 
	scalar product in $L^2$ is identified with 
	$ \dot L^2\times L_0^2 $. 
\end{rem}

The Hamiltonian \eqref{ham1} enjoys several symmetries which we  now describe.

\paragraph{Reversible structure.}
Defining on the phase space $H_0^1(\T) \times \dot{H}^1(\T)$ the involution
\begin{equation}\label{rev_invo}
\cS\left( \begin{matrix}
\eta \\ \psi 
\end{matrix} \right) := \left( \begin{matrix}
 \eta^\vee \\ -  \psi^\vee 
\end{matrix} \right) \, , \quad \eta^\vee (x) :=  \eta (-  x) \, ,  
\end{equation}
the Hamiltonian \eqref{ham1} is invariant under $\cS$, that is 
$$ 
H \circ \cS = H \, , 
$$ 
or, equivalently, the water waves vector field 
$ X $  defined in the right hand side on \eqref{ww} satisfies 
\begin{equation}\label{revNL}
X\circ \cS = - \cS \circ X \, .
\end{equation}
This property follows noting that the Dirichlet-Neumann operator satisfies 
\begin{equation}\label{DN-rev}
G(  \eta^\vee ) [ \psi^\vee ] = \left( G(\eta) [\psi ] \right)^\vee  \, . 
\end{equation}
{\bf Translation invariance.}
Since the bottom of the fluid domain \eqref{domain} is flat (or in case of infinite depth there is no bottom), the water waves 
equations   \eqref{ww}  are  invariant under space translations. Specifically, 
defining the translation operator  
\begin{equation}\label{trans}
\tau_\vs \colon u(x) \mapsto u(x+\vs) \, ,  \qquad \varsigma \in \R \, , 
\end{equation} 
the Hamiltonian \eqref{ham1} satisfies 
$ H \circ \tau_\vs = H $ for any $\vs \in \R $, 
or, equivalently, the water waves vector field 
$ X $  defined in the right hand side on \eqref{ww} satisfies 
\begin{equation}\label{eq:mom_pres}
X\circ \tau_\vs = \tau_\vs\circ X  
\, , \quad \forall \vs \in \R \,  .
\end{equation}
In order to verify this property, note that 
the Dirichlet-Neumann operator satisfies
\begin{equation}\label{DN:trans}
\tau_\vs \circ G( \eta ) = G(\tau_\vs  \eta ) \circ \tau_\vs \, , \quad 
\forall \vs \in \R \, .
\end{equation}
\paragraph{Wahl\'en coordinates.} 
The variables $(\eta, \psi)$ are not Darboux coordinates, in the sense that the Poisson tensor \eqref{eq:magn_sympl}  is not the canonical one for values of the vorticity $\gamma \neq 0$.
 Wahl\'en \cite{Wh} noted that  in the variables $(\eta, \zeta)$, where $\zeta$ is defined by 
\begin{equation}\label{Whalen-c}
  \zeta:= \psi -   \frac{\gamma}{2} \partial_x^{-1} \eta \, , 
\end{equation}
the  symplectic form induced by $J_M(\gamma)$ becomes  the canonical one. 
Indeed, under the linear transformation of the phase space 
$H^1_0 \times \dot H^1  $ into itself defined by 
\begin{equation}\label{eq:gauge_wahlen}
\left( \begin{matrix}
	\eta \\ \psi
	\end{matrix} \right) =  W \left(\begin{matrix}
	\eta \\ \zeta
	\end{matrix}\right) \, , \quad
	W :=  \left(\begin{matrix}
	{\rm Id} & 0 \\  \frac{\gamma}{2}\partial_x^{-1} & {\rm Id}
	\end{matrix}\right)  \, , \quad
	W^{-1}  := 
	 \left(\begin{matrix}
	{\rm Id} & 0 \\ - \frac{\gamma}{2}\partial_x^{-1} & {\rm Id}
	\end{matrix}\right) \, , 
\end{equation}
the Poisson tensor 
$ J_M(\gamma)$ is transformed into the canonical one,
\begin{equation}\label{Jtensor}
	W^{-1} J_M(\gamma) (W^{-1})^{*} =  J \, , \quad 
	J := \begin{pmatrix}
	0 & {\rm Id} \\ - {\rm Id}  & 0
	\end{pmatrix}   \, .
\end{equation}
Here $ W^* $ and $ (W^{-1})^*$ are the adjoints 
maps from (a dense subspace of) $ \dot L^2 \times L^2_0 $ into itself,
and the Poisson tensor $ J $ acts 
from (subspaces of) $ \dot{L}^2\times L_0^2$ to $L_0^2\times \dot{L}^2$.
Then the Hamiltonian
\eqref{ham1} becomes 
\begin{equation}\label{Ham-Wal}
\cH := H \circ W\,, \quad \text{ i.e. } \quad  {\cal H}(\eta,\zeta):=H\Big(\eta,\zeta + \frac{\gamma}{2}\partial_x^{-1}\eta\Big) \, , 
\end{equation}
and the Hamiltonian equations \eqref{ham.eq1} (i.e. \eqref{ww})  are transformed into
\begin{equation}\label{eq:Ham_eq_zeta}
	\partial_t\left(\begin{matrix}
	\eta \\ \zeta 
	\end{matrix}\right) = X_{\cal H} (\eta, \zeta) \, , \quad 
	X_{\cal H} (\eta, \zeta) := J
	 \begin{pmatrix} \nabla_\eta {\cal H} \\ \nabla_\zeta {\cal H} \end{pmatrix}  ( \eta, \zeta )\,. 
\end{equation}
By \eqref{Jtensor}, the symplectic form of  \eqref{eq:Ham_eq_zeta} is  the standard one,
\begin{align} \label{sympl-form-st}
{\cal W}  
\left( \begin{pmatrix}
\eta_1 \\
\zeta_1
\end{pmatrix}, 
\begin{pmatrix}
\eta_2 \\
\zeta_2
\end{pmatrix}
\right) = 
\left( J^{-1} \left(\begin{matrix}
\eta_1 \\ \zeta_1
\end{matrix}\right), 
\left(\begin{matrix}
\eta_2 \\ \zeta_2
\end{matrix}\right) \right)_{L^2}  
 = (  -  \zeta_1 , \eta_2 )_{L^2} + (\eta_1  , \zeta_2 )_{L^2}  \, , 
\end{align}
where  $ J^{-1} $ is the symplectic operator 
\begin{equation}\label{def:J-1}
J^{-1} = \begin{pmatrix}
	0 & - {\rm Id} \\  {\rm Id}  & 0
	\end{pmatrix} 
\end{equation}
regarded as a map 
from $L_0^2\times \dot{L}^2$ into $ \dot{L}^2\times L_0^2 $. 
Note that $ J J^{-1} = {\rm Id}_{L_0^2\times \dot{L}^2}$ and $ J^{-1}J  = {\rm Id}_{\dot{L}^2\times L_0^2}$. 
The  Hamiltonian vector field 
$ X_{\cal H} (\eta, \zeta) $	 in  \eqref{eq:Ham_eq_zeta} is characterized by the identity
$$
d {\cal H} (\eta, \zeta) [ \widehat u ]= 
{\cal W} \big( X_{\cal H} (\eta, \zeta),  \widehat u \big) \, , \quad  \forall \widehat u :=  \begin{pmatrix}
\widehat \eta \\
\widehat \zeta
\end{pmatrix} \, . 
 $$
The transformation $W$ defined in \eqref{eq:gauge_wahlen} is reversibility preserving, 
namely it commutes with the involution
 $ {\cal S} $ in \eqref{rev_invo} (see Definition \ref{rev_defn} below), and thus also the Hamiltonian $\cH$ in \eqref{Ham-Wal} is invariant under the involution $\cS$, as well as 
 $H$ in \eqref{ham1}. 
For this reason  we look for solutions $(\eta(t,x),\zeta(t,x))$
of \eqref{eq:Ham_eq_zeta} which are reversible, i.e. see  \eqref{time-rev}, 
\begin{equation}\label{rev:soluz}
\left(\begin{matrix}
\eta \\ \zeta
\end{matrix}\right)(-t)= \cS\left(\begin{matrix}
\eta \\ \zeta
\end{matrix}\right)(t)\, . 
\end{equation}
The corresponding solutions $(\eta(t,x), \psi (t,x))$ of \eqref{ww} induced by 
\eqref{eq:gauge_wahlen} are reversible as well. 

We finally note that  the transformation $W$ defined in \eqref{eq:gauge_wahlen} 
 commutes with  the translation operator $ \tau_\vs$, 
 therefore the Hamiltonian $\cH$ in \eqref{Ham-Wal} is invariant under
$ \tau_\vs $, 
as well as $ H $ in \eqref{ham1}. 
By  Noether theorem, the horizontal momentum 
$ \int_\T \zeta \eta_x  \wrt x  $
is a prime integral of \eqref{eq:Ham_eq_zeta}.

\subsection{Linearization at the equilibrium}\label{lin:op}

In this section we study the linear system \eqref{lin.ww1} and prove that its reversible solutions have the form \eqref{linz200}.

In view of the Hamiltonian  \eqref{ham1} of the water waves equations \eqref{ww}, 
also  the linear system \eqref{lin.ww1} is Hamiltonian and it is  generated by the quadratic Hamiltonian
$$
H_L(\eta,\psi) := \frac{1}{2} \int_\T\left( \psi G(0)\psi + g \eta^2 + \kappa \eta_x^2 \right) \wrt x =
\frac12 
\left( \b\Omega_L  \left(\begin{matrix}
\eta \\ \psi
\end{matrix}\right), 
\left(\begin{matrix}
\eta \\ \psi
\end{matrix}\right) \right)_{L^2} 
\, .
$$
Thus, recalling \eqref{ham.eq1}, the linear system \eqref{lin.ww1} is 
\begin{equation}\label{Lin:HS}
\partial_t \left(\begin{matrix}
\eta \\ \psi
\end{matrix}\right) = J_M(\gamma) \b\Omega_L \left(\begin{matrix}
\eta \\ \psi
\end{matrix}\right) \ , \qquad 
\b\Omega_L:=\left(\begin{matrix}
-\kappa\partial_x^2 + g & 0 \\ 0 & G(0)
\end{matrix}\right) \, .
\end{equation} 
The linear operator $ \b\Omega_L $ acts from (a dense subspace) of 
$ L^2_0 \times \dot L^2 $  to $ \dot L^2 \times L^2_0 $. 
In the Wahl\'en  coordinates  \eqref{eq:gauge_wahlen}, 
the linear Hamiltonian system \eqref{lin.ww1}, i.e. 
\eqref{Lin:HS}, 
transforms into the  linear Hamiltonian system 
\begin{equation}\label{eq:lin00_wahlen} 
\begin{aligned}
&	\partial_t\left(\begin{matrix}
\eta \\ \zeta
\end{matrix}\right) = J \b\Omega_W \left( \begin{matrix}
\eta \\ \zeta
\end{matrix} \right) \ ,  \\
& \b\Omega_W := W^* \b  \Omega_L W =\left(\begin{matrix}
-\kappa \partial_x^2 + g -  \left( \frac{\gamma}{2}\right)^2 \partial_x^{-1}G(0)\partial_x^{-1} & - \frac{\gamma}{2}\partial_x^{-1}G(0) \\ \frac{\gamma}{2}G(0)\partial_x^{-1} & G(0)
\end{matrix}\right) 	
\end{aligned}
\end{equation}
generated by the  quadratic Hamiltonian  
\begin{equation}\label{lin_real}
\cH_L (\eta, \zeta) := (H_L \circ W) (\eta, \zeta)  = \frac12 
\left( \b\Omega_W  \left(\begin{matrix}
\eta \\ \zeta
\end{matrix}\right), 
\left(\begin{matrix}
\eta \\ \zeta
\end{matrix}\right) \right)_{L^2} \, .
\end{equation}
The linear operator $ \b\Omega_W $ acts from (a dense subspace) of 
$ L^2_0 \times \dot L^2 $  to $ \dot L^2 \times L^2_0 $. 
The linear  system
\eqref{eq:lin00_wahlen}  is the  
Hamiltonian system obtained by linearizing \eqref{eq:Ham_eq_zeta} at the equilibrium
$ (\eta, \zeta) = (0, 0) $. 
We want to transform \eqref{eq:lin00_wahlen}
in diagonal form by using a symmetrizer  and then introducing complex coordinates. 
We first conjugate \eqref{eq:lin00_wahlen} under  the  symplectic transformation
(with respect to the standard symplectic form $ {\cal W}   $ in \eqref{sympl-form-st}) of the phase space
$$
\begin{pmatrix}
\eta \\
\zeta
\end{pmatrix} = \cM \begin{pmatrix}
u \\
v
\end{pmatrix} 
$$
where $ \cM $ is the diagonal matrix of self-adjoint Fourier multipliers 
\begin{equation}\label{eq:T_sym}
	\cM:= \left(\begin{matrix}
	 M(D) & 0 \\ 0 & M(D)^{-1}
	\end{matrix}  \right) \ , \quad 
	M(D) :=\left( \frac{G(0)}{\kappa D^2 + g - 
	\frac{\gamma^2}{4} \partial_x^{-1}G(0)\partial_x^{-1}} \right)^{1/4} \, , 
\end{equation}
with the real valued symbol $ M_j $ defined in \eqref{def:Mn}.
The map $ \cal {M} $ is reversibility preserving. 

\begin{rem}
In \eqref{eq:T_sym} the  Fourier multiplier $ M(D) $  
acts in $ H^1_0 $. On the other hand, with a slight abuse of notation,  
$M(D)^{-1} $  denotes  the Fourier multiplier operator in  $ \dot H^1 $ defined as 
$$ 
M(D)^{-1} [\zeta] :=  \big[ \sum_{j \neq 0} M_j^{-1} \zeta_j e^{\im j x } \big] \, , \quad 
\quad  \zeta (x) =  \sum_{j \in \Z} \zeta_j e^{\im j x }  \, .  
$$
where $[\zeta] $ is the element in $ \dot H^1 $ with representant 
$ \zeta (x)  $. 
\end{rem}
By a direct computation, 
the Hamiltonian system  \eqref{eq:lin00_wahlen} assumes the symmetric form
\begin{eqnarray}
\label{lin.ww3}
	\partial_t \begin{pmatrix}
u \\
v
\end{pmatrix} = J \b\Omega_S \begin{pmatrix}
u \\
v
\end{pmatrix} \, , \ \  \b\Omega_S:=\cM^*\b\Omega_W\cM = \left( \begin{matrix}
	\omega (\kappa, D) & -\frac{\gamma}{2}\partial_x^{-1}G(0) \\ \frac{\gamma}{2}G(0)\partial_x^{-1} & \omega (\kappa, D)
	\end{matrix}\right)  \, ,
	\end{eqnarray}
	where
\begin{equation}\label{eq:omega0}
	\omega(\kappa,D):= \sqrt{\kappa D^2 \,G(0) + g\,G(0) - \left( \frac{\gamma}{2}\pa_x^{-1} G(0) \right)^2 } \, . 
\end{equation}
\begin{rem} To be  precise, the Fourier multiplier 
operator $ \omega (\kappa, D) $ in the top left position in \eqref{lin.ww3} maps 
$ H^1_0 $ into $ \dot H^1 $ and the one  in the bottom right 
position maps $ \dot H^1 $ into $ H^1_0 $. The operator
$\partial_x^{-1}G(0) $ acts on $ \dot H^1  $ and 
$ G(0) \partial_x^{-1} $ on $ H^1_0  $. 
\end{rem}
Now we introduce   complex coordinates  by  the transformation 
\begin{equation}\label{C_transform}
\begin{pmatrix}
u \\
v
\end{pmatrix} = \cC 
\begin{pmatrix}
z \\
\bar z
\end{pmatrix} \, , \qquad 
\cC := \frac{1}{\sqrt 2} \left(\begin{matrix}
{\rm Id} & {\rm Id} \\ -\im & \im
\end{matrix}\right) \ ,\quad \cC^{-1}:= \frac{1}{\sqrt 2}\left(\begin{matrix}
{\rm Id} & \im \\ {\rm Id} & -\im
\end{matrix}\right) \, .
\end{equation}
In these variables, the Hamiltonian system  \eqref{lin.ww3} becomes the diagonal system 
\begin{eqnarray}\label{eq:lin00_ww_C}
	\partial_t\left( \begin{matrix}
	z \\ \bar z
	\end{matrix} \right)=
	\begin{pmatrix}
	- \im & 0 \\ 0 & \im
	\end{pmatrix} 
	\b\Omega_D\left( \begin{matrix}
	z \\ \bar z
	\end{matrix} \right) \, , \quad \b\Omega_D := \cC^*\b\Omega_S\cC = \begin{pmatrix}
	\Omega (\kappa, D) & 0 \\ 0 & \bar \Omega (\kappa, D)
	\end{pmatrix} \,,
	\end{eqnarray}
	where 
\begin{equation}
\label{Omega}
\Omega (\kappa, D) := \omega (\kappa, D) + \im \,\frac{\gamma}{2}\partial_x^{-1} G(0)
\end{equation}
	is the Fourier multiplier with symbol  $  \Omega_j(\kappa)   $ defined in \eqref{def:Omegajk} and
$\bar \Omega(\kappa, D)$ is  defined by 
$$
\bar \Omega(\kappa, D) z:= \bar{\Omega(\kappa, D) \bar z} \, , 
\quad \bar \Omega(\kappa, D)  = 
\omega (\kappa, D) - \im \,\frac{\gamma}{2}\partial_x^{-1} G(0)  \, . 
$$
Note that $\bar \Omega(\kappa, D)$ is the Fourier multiplier with symbol $\{\Omega_{-j}(\kappa)\}_{j \in \Z\setminus\{0\}}$.

\begin{rem}
We regard the system  \eqref{eq:lin00_ww_C} in $ \dot H^1 \times \dot H^1 $. 
\end{rem}
The diagonal system 
\eqref{eq:lin00_ww_C} amounts to the scalar equation
\begin{equation}\label{zjF}
\partial_t z = - \im \Omega (\kappa, D) z \, , \quad 
z(x) =  \sum_{j \in \Z\setminus\{0\}} z_j e^{\im j x } \, , 
\end{equation}
and, writing \eqref{zjF}
 in the exponential Fourier basis,  
to the infinitely many decoupled harmonic oscillators 
\begin{equation}\label{eq:zj}
\dot z_j = - \im \Omega_j (\kappa) z_j \, , \quad j \in \Z \setminus \{0\} \, .
\end{equation}
Note that, in these complex coordinates,   
the involution $\cS$ defined in \eqref{rev_invo} reads as the map
\begin{equation}\label{inv-complex}
	\begin{pmatrix}
	z(x) \\ \bar{z(x)}
	\end{pmatrix} \mapsto \begin{pmatrix}
	\, \bar{z(-x)} \\ z(-x) \, 
	\end{pmatrix}
\end{equation}
that we may read just  as the scalar map $z(x)\mapsto \bar{z(-x)}$. Moreover,
in the Fourier coordinates introduced in \eqref{zjF}, it amounts to 
\begin{equation}\label{inv-zj}
z_j \mapsto \overline{z_j} \, , \quad \forall j \in \Z \setminus \{ 0 \} \,  .
\end{equation}
In view of \eqref{eq:zj} and \eqref{inv-zj}
every {\em reversible} solution (which is characterized as in \eqref{rev:soluz}) of  
\eqref{zjF} has  the form 
\begin{equation}
\label{linz}
	z(t,x):= \frac{1}{\sqrt{2}}\sum_{j\in \Z\setminus\{0\}} \rho_{j} \,  e^{-\im\,\left( \Omega_j(\kappa)t- j\,x \right)} 
	 \quad {\rm with} \quad  \rho_j \in \R \, .
\end{equation}
Let us see the form of these solutions 
back 
in  the original variables $ (\eta, \psi)$.
First, by \eqref{eq:T_sym}, \eqref{C_transform},
\begin{equation}
\label{linz2}
\begin{pmatrix}
\eta \\
\zeta
\end{pmatrix} = \cM \, \cC 
\begin{pmatrix}
z \\
\bar z
\end{pmatrix} = 
\frac{1}{\sqrt{2}}\begin{pmatrix}
M(D) & M(D)  \\
-\im M(D)^{-1} & \im M(D)^{-1} 
\end{pmatrix}\begin{pmatrix}
z \\
\bar z
\end{pmatrix} =
\frac{1}{\sqrt{2}}\begin{pmatrix}
M(D) (z + \bar z)  \\
-\im M(D)^{-1} (z - \bar z )
\end{pmatrix} \, , 
\end{equation} 
and the solutions \eqref{linz}  assume the form
$$
\begin{aligned}
\begin{pmatrix}
\eta(t,x) \\ \zeta(t,x)
\end{pmatrix} & =  \sum_{n \in \N}  \begin{pmatrix}
M_n \rho_n \cos ( n x - \Omega_n (\kappa) t) \\ 
M_n^{-1} \rho_n \sin ( n x - \Omega_n (\kappa) t) 
\end{pmatrix} \\ 
& +  \sum_{n \in \N}
\begin{pmatrix}
M_n \rho_{-n} \cos ( n x + \Omega_{-n}(\kappa) t) \\ 
- M_n^{-1} \rho_{-n} \sin ( n x + \Omega_{-n} (\kappa) t) 
\end{pmatrix} 
 \,.
 \end{aligned}
$$
Back to the variables $(\eta, \psi)$ with the change of coordinates  \eqref{eq:gauge_wahlen} one obtains formula \eqref{linz200}.

\paragraph{Decomposition of the phase space in Lagrangian  subspaces invariant 
under
\eqref{eq:lin00_wahlen}.}
We express the Fourier coefficients $ z_j \in \C $ in  \eqref{zjF} as 
$$
z_j = \frac{\alpha_j + \im \beta_j}{\sqrt{2}}, \quad (\alpha_j, \beta_j) \in \R^2 \, , \quad j \in \Z\setminus\{0\} \, .
$$
In the new coordinates $ (\alpha_j, \beta_j)_{ j \in \Z\setminus\{0\}} $,
  we  write  \eqref{linz2} as (recall that $ M_j = M_{-j} $)
\begin{equation}\label{deco-real}
\begin{pmatrix}
\eta (x) \\
\zeta (x) 
\end{pmatrix}  = \sum_{j \in \Z\setminus\{0\}}  
\begin{pmatrix}
M_j ( \alpha_j \cos (jx) - \beta_j  \sin (jx) )  \\
M_j^{-1} (  \beta_j  \cos (jx) +  \alpha_j \sin (jx) ) 
\end{pmatrix} 
\end{equation}
with
\begin{equation}\label{proj-Vj}
\begin{aligned}
& \alpha_j = \frac{1}{2\pi} 
\Big( M_j^{-1} (\eta, \cos (jx))_{L^2} + M_j (\zeta, \sin (jx))_{L^2}  \Big) \, , \\
& \beta_j = 
\frac{1}{2\pi} 
\Big( M_j (\zeta, \cos (jx))_{L^2} - M_j^{-1} (\eta, \sin (jx))_{L^2}  \Big)\, .
\end{aligned}
\end{equation}
The symplectic form \eqref{sympl-form-st}  then becomes
$$
 2\pi\sum_{j \in \Z\setminus\{0\}} \di \alpha_j \wedge \di \beta_j\,.
$$
Each $ 2$-dimensional 
subspace in the sum \eqref{deco-real}, spanned by 
$ (\alpha_j, \beta_j ) \in \R^2 $ is therefore a symplectic subspace. 
The quadratic Hamiltonian $ \cH_L $ in \eqref{lin_real} reads 
\begin{equation}\label{QFH}
2 \pi \sum_{j \in \Z\setminus\{0\}} \frac{\Omega_j(\kappa)}{2} (\alpha_j^2 + \beta_j^2 )\, . 
\end{equation}
In view of \eqref{deco-real}, the involution $ \cS $ defined in \eqref{rev_invo} reads
\begin{equation}\label{ab-rev}
(\alpha_j, \beta_j) \mapsto (\alpha_j, - \beta_j) \, ,  \quad \forall j \in \Z \setminus \{ 0 \} \, ,  
\end{equation}
and the translation operator $ \tau_\vs $ defined in \eqref{trans} as 
\begin{equation}\label{ab-tras}
\begin{pmatrix}
\alpha_j  \\
\beta_j
\end{pmatrix}   \mapsto \begin{pmatrix}
\cos (j \vs) & - \sin (j \vs) \\
\sin (j \vs) & \cos (j \vs) 
\end{pmatrix}
\begin{pmatrix}
\alpha_j  \\
\beta_j
\end{pmatrix} \, ,  \quad \forall j \in \Z \setminus \{ 0 \}  \, .
\end{equation}
We may also  enumerate the independent variables $ (\alpha_j, \beta_j )_{j \in \Z\setminus\{0\}} $ as
$  \big( \alpha_{-n}, \beta_{-n}, \alpha_{n}, \beta_{n} \big) $, $  n \in \N $. 
Thus  the phase space $\acca := L^2_0 \times \dot L^2 $ of 
\eqref{eq:Ham_eq_zeta} decomposes as the direct sum 
$$
\acca = \sum_{n \in \N} V_{n,+}  \oplus V_{n,-} 
$$
of $ 2 $-dimensional Lagrangian symplectic subspaces 
\begin{align} 
V_{n,+} & :=
\left\{ 
\begin{pmatrix}
	\eta \\ \zeta
	\end{pmatrix} 
 = 
 \begin{pmatrix}
M_{n} ( \alpha_{n} \cos (nx) - \beta_n \sin (nx) ) \\ 
M_{n}^{-1} ( \beta_{n} \cos (nx) + \alpha_n \sin (nx) ) 
\end{pmatrix} \, , (\alpha_n , \beta_n) \in \R^2  
\right\} \, , \label{Vn+} \\
V_{n,-} & :=
\left\{ 
\begin{pmatrix}
	\eta \\ \zeta
	\end{pmatrix}  = 
 \begin{pmatrix}
M_{n} ( \alpha_{-n} \cos (nx) + \beta_{-n} \sin (nx) ) \\ 
M_{n}^{-1} ( \beta_{-n} \cos (nx) - \alpha_{-n} \sin (nx) ) 
\end{pmatrix} \, , (\alpha_{-n} , \beta_{-n}) \in \R^2  
\right\} \, , \label{Vn-}
 \end{align}  
which are invariant for the linear Hamiltonian system \eqref{eq:lin00_wahlen}, namely
$ J  \b\Omega_W : V_{n,\sigma} \mapsto V_{n,\sigma} $ (for a proof see e.g. remark \ref{rem:Vn+}).  
The symplectic projectors $ \Pi_{V_{n,\sigma}}$, $ \sigma \in \{ \pm \} $, 
on the symplectic subspaces $ V_{n, \sigma} $ are explicitly provided by 
\eqref{deco-real} and \eqref{proj-Vj} with $ j = n \sigma $. 

Note that the involution $ \cS $ defined in \eqref{rev_invo} and the translation operator 
$ \tau_\vs $ in \eqref{trans} leave the subspaces $ V_{n,\sigma} $, $ \sigma \in \{ \pm  \} $,  invariant.

\subsection{Tangential and normal subspaces of the phase space}
\label{sec:decomp}

We  decompose the phase space $ \acca $ of \eqref{eq:Ham_eq_zeta}  
into a direct sum of {\em tangential} and {\em normal }  Lagrangian subspaces 
$ \acca_{\S^+,\Sigma}^\intercal $ and 
$ \acca_{\S^+,\Sigma}^\angle  $. Note that the 
main part of the solutions \eqref{QP:soluz} that we shall obtain in Theorem 
\ref{thm:main0} is the component  in the tangential subspace 
$ \acca_{\S^+,\Sigma}^\intercal $,  whereas the component in  the normal subspace $ \acca_{\S^+,\Sigma}^\angle  $ is much smaller. 

Recalling  the definition of the sets $\S^+$ and 
$\Sigma$ defined in \eqref{Splus} respectively \eqref{signs}, we split 
\begin{equation}\label{H-split}
\acca =\acca_{\S^+,\Sigma}^\intercal\oplus\acca_{\S^+,\Sigma}^\angle 
\end{equation}
where $\acca^{\intercal}_{\S^+, \Sigma} $ is the  finite dimensional {\em tangential subspace}  
\begin{equation}
\label{cHStang}
\acca^{\intercal}_{\S^+, \Sigma} := 
\sum_{a = 1}^\nu V_{{\bar n}_a, \sigma_a} 
\end{equation}
and $\acca^{\angle}_{\S^+, \Sigma} $ is the {\em normal subspace} defined as
 its symplectic orthogonal 
 \begin{equation}
\label{cHSnorm}
\acca^{\angle}_{\S^+, \Sigma} := \sum_{a = 1}^\nu V_{{\bar n}_a, - \sigma_a} 
\oplus  
 \sum_{n \in \N \setminus {\mathbb S}^+} \big( V_{n,+}  \oplus V_{n,-}\big) \, . 
\end{equation}
Both the subspaces 
$\acca_{\S^+,\Sigma}^\intercal$ and $\acca_{\S^+,\Sigma}^\angle$ are Lagrangian. 
We denote  by
$ \Pi_{\S^+,\Sigma}^\intercal $ and
$  \Pi_{\S^+,\Sigma}^\angle  $ 
the symplectic projections on the subspaces $\acca_{\S^+,\Sigma}^\intercal$ and $\acca_{\S^+,\Sigma}^\angle$, respectively. 
Since $\acca_{\S^+,\Sigma}^\intercal$ and $\acca_{\S^+,\Sigma}^\angle$ are symplectic orthogonal, the symplectic form $ {\cal W} $ in  \eqref{sympl-form-st} decomposes as
$$
{\cal W} ( v_1+ w_1, v_2 + w_2) = {\cal W} ( v_1, v_2)+ {\cal W} ( w_1, w_2) \, ,
\quad \forall v_1, v_2 \in \acca_{\S^+,\Sigma}^\intercal \, ,  \
w_1, w_2 \in \acca_{\S^+,\Sigma}^\angle \, .
$$
The symplectic 
projections $ \Pi_{\S^+,\Sigma}^\intercal $ and
$  \Pi_{\S^+,\Sigma}^\angle  $  satisfy the following properties:

 \begin{lem}
\label{Pi_adj_1}
We have that
\begin{align}
\label{pi.t.J}
&\Pi^\intercal_{\S^+, \Sigma}\, J  = J  \big( \Pi^\intercal_{\S^+, \Sigma} \big)^* \ , \qquad 
\big(\Pi^\intercal_{\S^+, \Sigma} \big)^* \, J^{-1}  = J^{-1} \,\Pi^\intercal_{\S^+, \Sigma} \,  , \\
\label{pi.a.J}
&\Pi^\angle_{\S^+, \Sigma}\, J  = J \, \big(\Pi^\angle_{\S^+, \Sigma} \big)^* \ , \qquad 
\big(\Pi^\angle_{\S^+, \Sigma} \big)^* \,  J^{-1} = J^{-1} \Pi^\angle_{\S^+, \Sigma}  \, . 
\end{align}
\end{lem}
\begin{proof}
Since  the subspaces  $\acca^\intercal := \acca_{\S^+,\Sigma}^\intercal $ and 
$\acca^\angle := \acca_{\S^+,\Sigma}^\angle $ are symplectic orthogonal, we have, recalling 
 \eqref{sympl-form-st}, that
$$
( J^{-1} v , w )_{L^2} = (J^{-1} w , v )_{L^2} = 0 , \qquad \forall v \in \acca^\intercal \, , \ 
\forall w \in \acca^\angle \, .
$$
Thus, using the projectors 
$ \Pi^\intercal := \Pi^\intercal_{\S^+, \Sigma} $, 
$\Pi^\angle := \Pi^\angle_{\S^+, \Sigma} $,  we have that 
$$
( J^{-1} \Pi^\intercal v , \Pi^\angle w )_{L^2} = ( J^{-1} \Pi^\angle w , \Pi^\intercal v )_{L^2} = 0 \, , \quad \forall v, w \in \acca \, , 
$$
and, taking adjoints,
$ ( (\Pi^\angle)^* J^{-1} \Pi^\intercal v ,  w )_{L^2} = 
( (\Pi^\intercal)^* J^{-1} \Pi^\angle w ,  v )_{L^2} = 0 $ for any 
$ v, w \in \acca $, 
so that 
\begin{equation}\label{rel:sopra}
(\Pi^\angle)^* J^{-1} \Pi^\intercal  = 0 = (\Pi^\intercal)^* J^{-1} \Pi^\angle \,  .
\end{equation}
Now inserting the identity  $\Pi^\angle = {\rm Id} - \Pi^\intercal$  in  \eqref{rel:sopra}, 
we get
$$
 J^{-1} \Pi^\intercal = (\Pi^\intercal)^*  J^{-1} \Pi^\intercal  = (\Pi^\intercal)^* J^{-1} 
$$
proving  the second identity of \eqref{pi.t.J}. 
The first identity  of \eqref{pi.t.J} follows applying  $J$ to the left and to the right of the second identity.
The identity \eqref{pi.a.J} follows in the same way. 
\end{proof}

Note that the restricted symplectic form $\cW\vert_{\acca_{\S^+, \Sigma}^\angle} $ is represented by the symplectic structure  
\begin{equation}\label{simpl-ristretto}
J_\angle^{-1} : \acca_{\S^+, \Sigma}^\angle \to \acca_{\S^+, \Sigma}^\angle \, , \
\quad 
J_\angle^{-1} := \Pi^{L^2}_\angle \,   J^{-1}_{| \acca_{\S^+,\Sigma}^\angle }    \, ,  
\end{equation}
where $ \Pi^{L^2}_\angle $ is the $ L^2 $-projector on the subspace 
$ \acca_{\S^+,\Sigma}^\angle $.  
Indeed 
\begin{align*}
\cW\vert_{\acca_{\S^+, \Sigma}^\angle}(w, \hat w ) & = (  J_\angle^{-1} w, \wh{w} )_{L^2}  
= (  J^{-1} w, \wh{w} )_{L^2}  , \quad \forall  w, \hat w \in  \acca_{\S^+, \Sigma}^\angle \, . 
\end{align*}
We also  denote 
the associated (restricted) Poisson tensor
\begin{equation}
\label{Jangle}
J_\angle :  \acca_{\S^+, \Sigma}^\angle \to \acca_{\S^+, \Sigma}^\angle \, , \quad
J_\angle := \Pi^\angle_{\S^+, \Sigma}  \, J_{| \acca_{\S^+,\Sigma}^\angle } \, . 
\end{equation}
In the next lemma we prove that $J^{-1}_\angle$ and $J_\angle$ are  each other inverses.
\begin{lem}\label{lem:invPS}
$J^{-1}_\angle  \, J_\angle  = J_\angle \, J^{-1}_\angle =  {\rm Id}_{ \acca_{\S^+, \Sigma}^\angle} $. 
\end{lem}
\begin{proof}
Let $v \in \acca_{\S^+, \Sigma}^\angle $. By \eqref{simpl-ristretto} and \eqref{Jangle},  
for any  $h \in  \acca_{\S^+, \Sigma}^\angle $ one has 
\begin{align*}
 (J^{-1}_\angle  \, J_\angle \, v , h )_{L^2} & =  (J^{-1} \Pi^\angle_{\S^+, \Sigma}  \ J v,  \Pi^{L^2}_\angle h)_{L^2}  
 = -( \Pi^\angle_{\S^+, \Sigma}  \, J v, \, J^{-1}  h)_{L^2} \\
 & = - ( J v, (\Pi^\angle_{\S^+, \Sigma} )^* J^{-1}  h)_{L^2} 
 \stackrel{\eqref{pi.a.J}} =- ( J v,  J^{-1} \, \Pi^\angle_{\S^+, \Sigma}  h)_{L^2}  = ( v, h)_{L^2} \, .
\end{align*}
The proof that 
$ J_\angle J^{-1}_\angle  \,   =  {\rm Id}_{ \acca_{\S^+, \Sigma}^\angle}$ is similar.
\end{proof}

\begin{lem}\label{ident-simp}
$ \Pi^\angle_{\S^+, \Sigma} J \Pi^{L^2}_\angle = \Pi^\angle_{\S^+, \Sigma} J $.
\end{lem}

\begin{proof}
For any $ u, h \in \acca $ we have, using Lemma \ref{Pi_adj_1},  
\begin{align*}
(  \Pi^\angle_{\S^+, \Sigma} J \Pi^{L^2}_\angle u, h )_{L^2} & = 
- (   \Pi^{L^2}_\angle u, J (\Pi^\angle_{\S^+, \Sigma})^* h )_{L^2}  
= - (   \Pi^{L^2}_\angle u, \Pi^\angle_{\S^+, \Sigma} J  h )_{L^2} \\
& = - (  u, \Pi^\angle_{\S^+, \Sigma} J  h )_{L^2} =
( J (\Pi^\angle_{\S^+, \Sigma})^*  u,  h )_{L^2}   = 
( \Pi^\angle_{\S^+, \Sigma} J  u,  h )_{L^2} 
\end{align*}
implying the lemma. 
\end{proof}

\paragraph{Action-angle coordinates.}
Finally we introduce action-angle coordinates on the tangential subspace 
$ \acca^{ \intercal}_{\S^+, \Sigma}$ defined in \eqref{cHStang}. 
Given the sets $\S^+$ and $\Sigma$ defined respectively in \eqref{Splus} and \eqref{signs},  
we define the set 
\begin{equation}
\label{def.S}
\S := \{ \bar \jmath_1 , \ldots, \bar \jmath_\nu \} \subset \Z \,\setminus\{0\}\, , \quad \bar 
\jmath_a := \sigma_a \bar n_a \, ,
\quad a = 1, \ldots, \nu \, , 
\end{equation} 
and  the action-angle coordinates $ (\theta_j, I_j)_{j \in {\mathbb S}} $, by the relations  
\begin{equation}\label{ajbjAA}
\alpha_j = \sqrt{\frac{1}{\pi}(I_j + \xi_j)}\cos(\theta_j) \,, \ 
\beta_j = -\sqrt{\frac{1}{\pi}(I_j + \xi_j)}\sin(\theta_j) \, , \quad \xi_j >0 \, , \ | I_j | <  \xi_j \, ,
\ \forall j \in {\mathbb S} \, . 
\end{equation}
In view of \eqref{H-split}-\eqref{cHSnorm}, we represent any function
 of the phase space $ \acca $ as
\begin{align}
\label{aacoordinates}
A(\theta,I,w)& := v^\intercal(\theta,I)+ w \notag \, ,  \\
& := \frac{1}{\sqrt{\pi}}\sum_{j\in\S}
\left[ \begin{pmatrix}
M_j\sqrt{I_j+ \xi_j}\cos(\theta_j) \\ -M_j^{-1}\sqrt{I_j+\xi_j}\sin(\theta_j)
\end{pmatrix}\cos(j x)+\begin{pmatrix}
M_j\sqrt{I_j+ \xi_j}\sin(\theta_j) \\ M_j^{-1}\sqrt{I_j+\xi_j}\cos(\theta_j)
\end{pmatrix}\sin(j x) \right] + w  \nonumber \\
& = \frac{1}{\sqrt{\pi}}\sum_{j\in\S}
\left[ \begin{pmatrix}
M_j\sqrt{I_j+ \xi_j} \cos(\theta_j - j x ) \\ 
- M_j^{-1}\sqrt{I_j+\xi_j}\sin(\theta_j - j x )
\end{pmatrix} \right] + w 
\end{align}
where $ \theta := (\theta_j)_{j \in \S} \in \T^\nu $, $ I := (I_j)_{j \in \S } \in \R^\nu $
and $ w \in  \acca^{ \angle}_{\S^+, \Sigma} $. 

\begin{rem} In  these coordinates the  solutions \eqref{sel.sol} 
of the linear system \eqref{lin.ww1}  simply read as
$ W  v^\intercal(\ora{\Omega}(\kappa)t, 0) $,  where $
\ora{\Omega}(\kappa) := (\Omega_j(\kappa))_{j \in \S} $ is given in 
\eqref{Omega-kappa}. 
\end{rem}

In view of  \eqref{aacoordinates}, 
the involution $\cS$ in \eqref{rev_invo} reads
\begin{equation}\label{rev_aa}
	\vec \cS: (\theta,I,w)\mapsto \left( -\theta,I,\cS w \right)\,, 
\end{equation}
the  translation operator $\tau_\varsigma$ in  \eqref{trans} reads
\begin{equation}
\label{vec.tau}
\vec \tau_\vs : 
(\theta, \, I, \, w) \mapsto 
(\theta - \ora{\jmath} \vs, \, I, \, \tau_\vs w), \quad \forall \vs \in \R \, , 
\end{equation}
where 
\begin{equation}\label{def:vecj}
\ora{\jmath}:= (j)_{j \in \S} = ( \bar{\jmath}_1,\ldots,\bar{\jmath}_\nu) \in \Z^\nu\setminus\{0\} \, , 
\end{equation}
and the symplectic 2-form \eqref{sympl-form-st} becomes
\begin{equation}\label{sympl_form}
{\cal W} = 
\sum_{j\in\S} (\di \theta_j \wedge  \di I_j)  \, \oplus \, 
{\cal W}|_{\acca_{\S^+, \Sigma}^\angle} \,  . 
\end{equation}
We also note that  ${\cal W} $ is exact, namely 
\begin{equation}\label{liouville}
{\cal W} = d \Lambda\,,   \qquad {\rm where} \qquad 
\Lambda_{(\theta,I,w)} [ \wh\theta,\whI ,\wh{w} ] := - \sum_{j\in\S} I_j \wh{\theta}_j + \tfrac12 \left( J_\angle^{-1} w, \wh{w} \right)_{L^2} 
\end{equation}
is the  associated  Liouville 1-form (the operator $J_\angle^{-1} $ is defined in \eqref{simpl-ristretto}).

Finally, given a Hamiltonian $ K \colon \T^\nu \times \R^\nu \times \acca_{\S^+, \Sigma}^\angle \to \R$, 
the associated Hamiltonian vector field 
(with respect to the symplectic form \eqref{sympl_form}) is
\begin{equation}\label{HS3var}
X_K  := 
 \big( \pa_I K, -\pa_\theta K,  J_\angle \nabla_{w} K \big)  =
 \big( \pa_I K,  -\pa_\theta K,  \Pi_{\S^+,\Sigma}^\angle J \nabla_{w} K \big) 	\, , 
\end{equation}
where $\grad_w K $ denotes the $L^2$ gradient of $K$ with respect to $ w
\in \acca_{\S^+, \Sigma}^\angle$. 
Indeed, the only nontrivial component of the vector field $X_K$ is the last one, which we denote by $[X_K]_w \in \acca_{\S^+, \Sigma}^\angle $. It fulfills 
\begin{equation}\label{L2-grad}
( J^{-1}_\angle [X_K]_w, \wh w )_{L^2} = \di_w K [\wh w] =
 ( \grad_w K, \wh w)_{L^2}\,
 , \quad \forall\, \wh w \in \acca_{\S^+, \Sigma}^\angle \, ,
\end{equation}
and \eqref{HS3var} follows by Lemma \ref{lem:invPS}. 
We remark that along the paper we only consider Hamiltonians such that the $ L^2$-gradient $ \nabla_w K $ defined by \eqref{L2-grad}, as well as the Hamiltonian 
vector field $ \Pi_{\S^+,\Sigma}^\angle J \nabla_{w} K $, 
 maps spaces of Sobolev functions into Sobolev functions (not just distributions), with possible loss of derivatives. 

\paragraph{Tangential and normal subspaces in complex variables.}

Each $ 2 $-dimensional symplectic subspace $ V_{n,\sigma} $, 
$ n \in \N $, $ \sigma = \pm 1 $,  defined in \eqref{Vn+}-\eqref{Vn-}
is isomorphic,   
through the linear map $ {\cal M } {\cal C} $ defined in \eqref{linz2}, 
to the complex subspace 
$$
	\bH_j  
	:= \Big\{  \begin{pmatrix}
	z_j e^{\im j x }  \\ \overline{z_j} e^{- \im j x } 
	\end{pmatrix} 	\, , \  z_j \in \C   \Big\}  \qquad {\rm with } \qquad j = n \sigma \in \Z  \, .
$$
Denoting by $ \Pi_j $  the $ L^2 $-projection on $	\bH_j   $, we have that 
$\Pi_{V_{n, \sigma}}  = {\cal M } {\cal C} \,  \Pi_j \, ({\cal M } {\cal C})^{-1} $.  
Thus $ {\cal M } {\cal C} $
is an isomorphism between the tangential subspace 
$ \acca^{\intercal}_{\S^+, \Sigma} $ defined in \eqref{cHStang} and 
$$
	\bH_{\mathbb S} 
	:= \Big\{ \begin{pmatrix}
	z \\ \bar z
	\end{pmatrix}  \, : \, 
	z (x) = \sum_{j \in {\mathbb S} } z_j e^{\im j x }  \Big\}
$$
and between  the  normal subspace $ \acca^{\angle}_{\S^+, \Sigma} $ defined in \eqref{cHSnorm} and 
\begin{equation}\label{def:HS0bot}
	 \bH_{{\mathbb S}_0}^\bot := \Big\{ \begin{pmatrix}
	z \\ \bar z
	\end{pmatrix} \, : \, 
	z (x) = \sum_{j \in \S_0^c } z_j e^{\im j x } \in L^2   \Big\}  \, ,
	\quad \S_0^c := \Z \setminus (\S \cup \{0\}) \, . 
\end{equation}
Denoting by 
$ \Pi_{\S}^\intercal $, $\Pi_{\S_0}^\perp  $, the $ L^2 $-orthogonal 
projections on the subspaces $ \bH_\S $ and 
$ \bH_{\S_0}^\perp $, we have that
\begin{equation}\label{proiez}
\Pi_{\S^+,\Sigma}^\intercal = {\cal M } {\cal C} \,  \Pi_{\S}^\intercal \, ({\cal M } {\cal C})^{-1} \, , 
\quad
\Pi_{\S^+,\Sigma}^\angle = {\cal M } {\cal C} \, \Pi_{\S_0}^\perp \, ({\cal M } {\cal C})^{-1} \, . 
\end{equation} 
The following lemma, used in Section \ref{sec:NM},  is an
easy corollary of the previous analysis. 
\begin{lem}\label{zero_term}
	We have that $\left( v^\intercal , \b\Omega_W w \right)_{L^2} = 0$, for 
	any  
	$ v^\intercal  \in \acca_{\S^+,\Sigma}^\intercal $ and $ w \in \acca_{\S^+,\Sigma}^\angle $.
\end{lem}
\begin{proof}
Write $ v^\intercal  =  \cM\cC z^\intercal  $ and $ \cM\cC z^\perp  $
with $z^\intercal  \in\bH_\S$ 
and $z^\perp \in \bH_{\S_0}^\perp $. 
Then, by  \eqref{lin.ww3} and \eqref{eq:lin00_ww_C}, 
$$
	\left( v^\intercal,\b\Omega_W w \right)_{L^2} 
	= \left( \cM\cC z^\intercal , \b\Omega_W \cM\cC z^\perp \right)_{L^2} 
	= \left( z^\intercal , \b\Omega_D z^\perp \right)_{L^2} = 0 \,,
$$
	since $\b\Omega_D$ preserves the subspace $\bH_{\S_0}^\perp$.
\end{proof}

\begin{rem}\label{rem:Vn+}
The same proof of Lemma \ref{zero_term} actually shows that 
$ ( v_{n,-\sigma} , \b\Omega_W v_{n,\s} )_{L^2} = 0 $ for any 
$ v_{n,\pm \sigma} \in V_{n, \pm \sigma}$,  
for any 
$ n \in \N $, $ \sigma = \pm 1 $. Thus $ {\cal W}( v_{n,-\sigma} , J \b\Omega_W v_{n,\s} ) =
( v_{n,-\sigma} ,  J^{-1} J \b\Omega_W v_{n,\s} )_{L^2} = 0 $ which shows that   $ 
J \b\Omega_W $ maps $ V_{n,\sigma } $ in itself. 
\end{rem}

\noindent
{\bf Notation.} For $ a \lesssim_s b $ means that $ a \leq C(s) b $ for some positive constant 
$ C (s) $. We denote $ \N := \{1, 2, \ldots \} $ and $ \N_0 := \{0\} \cup \N $.

\section{Functional setting}

Along the paper we consider functions $u(\vf,x)\in L^2\left(\T^{\nu+1},\C\right)$ depending on the space variable $x\in\T=\T_x$ and the angles $\vf\in\T^\nu=\T_\vf^\nu$ (so that $\T^{\nu+1}= \T_\vf^\nu\times \T_x$) which we expand in Fourier series as
\begin{equation}\label{u_fourier}
u(\vf,x) =  \sum_{j\in\Z} u_j(\vf)e^{\im\,jx} = 
\sum_{\ell\in\Z^\nu,j\in\Z}u_{\ell,j}e^{\im (\ell\cdot \vf +jx )}\, . 
\end{equation}
We also consider real valued functions $u(\vf,x)\in\R$, as well as vector valued functions $u(\vf,x)\in\C^2$ (or $u(\vf,x)\in\R^2$). 
When no confusion appears, we denote simply by $L^2$, $L^2(\T^{\nu+1})$, $L_x^2:=L^2(\T_x)$, $L_\vf^2:= L^2(\T^\nu)$ either the spaces of real/complex valued, scalar/vector valued, $L^2$-functions.

In this paper a crucial role is played by the following subspace of
functions  of $ (\vf,x) $. 
 
\begin{defn} \label{QPTW}
 {\bf (Quasi-periodic  traveling waves)} 
 Let  $ \ora{\jmath} := (\bar{\jmath}_1, \ldots ,\bar{\jmath}_\nu) \in \Z^\nu $ be the vector defined in 
\eqref{def:vecj}. 
A function 
 $ u (\vf, x) $  is called a {\em quasi-periodic  traveling} wave if it has 
 the form $ u(\vf,x) = U(\vf-\ora{\jmath}x)  $ 
 where $ U :  \T^\nu \to \C^K $, $ K \in \N $, is 
 a $ (2 \pi)^\nu $-periodic function. 
\end{defn}

Comparing with Definition \ref{def:TV}, we find convenient to call 
{\em quasi-periodic  traveling} wave both the function $ u(\vf,x) = U(\vf-\ora{\jmath}x)  $ and the function of time  $ u(\omega t,x) = U(\omega t-\ora{\jmath}x)  $. 

Quasi-periodic traveling waves are characterized by the relation 	
\begin{equation}
\label{chartj1}
 u(\vf - \ora{\jmath} \vs, \cdot) = \tau_\vs u \,   
 \   \ 
 \forall \vs \in \R  \, ,  
\end{equation}
where $\tau_\vs $ is the translation operator in \eqref{trans}. 
Product and composition of quasi-periodic traveling waves  is a quasi-periodic traveling wave. 
Expanded in Fourier series as in \eqref{u_fourier}, 
a quasi-periodic traveling wave has the form 
\begin{equation}\label{u-trav-Fourier}
u(\vf, x) = 
\sum_{\ell\in\Z^\nu,j\in\Z, j + \vec \jmath \cdot \ell = 0 }u_{\ell,j}e^{\im (\ell\cdot \vf +jx )}
 \, , 
\end{equation}
namely, comparing with  Definition \ref{QPTW}, 
\begin{equation}\label{uphiU}
u(\vf, x) = U(\vf - \vec \jmath x) \, , \quad U (\psi) = \sum_{\ell \in \Z^\nu}
U_\ell e^{\im \ell \cdot \psi} \, , \quad U_\ell = u_{\ell, - \vec \jmath \cdot \ell} \, . 
\end{equation}
The traveling waves $ u (\vf, x) = U (\vf - \vec \jmath x ) $ where 
$ U (\cdot ) $ belongs to the Sobolev space $ H^s (\T^\nu, \C^K ) $ 
in \eqref{unified norm} (with values in $ \C^K $, $ K \in \N $), 
form a subspace of the 
Sobolev space 
\begin{equation} \label{Sobonorm}
 H^s(\T^{\nu+1})
= \Big\{ u = \sum_{(\ell,j) \in \Z^{\nu+1}} u_{\ell,j} \, e^{\im (\ell \cdot \vf + jx)} \, : \, 
\| u \|_s^2 := \sum_{(\ell,j) \in \Z^{\nu+1}} | u_{\ell, j}|^2 \langle \ell,j \rangle^{2s} < \infty 
\Big\}
\end{equation}
where  $\langle \ell,j \rangle := \max \{ 1, |\ell|, |j| \} $. 
Note the equivalence of the norms (use \eqref{uphiU})
$$
\| u \|_{H^s (\T^{\nu}_\vf \times \T_x )} \simeq_s \| U \|_{H^s (\T^{\nu})} \, .
$$ 
For 
$ s \geq s_0 := \big[ \frac{\nu +1}{2} \big] +1 \in \N  $
one has $ H^s ( \T^{\nu+1}) \subset C ( \T^{\nu+1})$, and $H^s(\T^{\nu+1})$ is an algebra. 
Along the paper we  denote by $ \| \ \|_s $ both the Sobolev norms in 
\eqref{unified norm} and \eqref{Sobonorm}.  

For $K\geq 1$ we define the smoothing operator  $\Pi_{K}$  on the traveling waves
\begin{equation}\label{pro:N}
\Pi_K :  u = \sum_{\ell \in \Z^\nu,\, j \in \S_0^c, \,j + \vec \jmath \cdot \ell = 0}
u_{\ell,j}e^{\im(\ell\cdot \vf+jx)} \mapsto
\Pi_K u = \sum_{\braket{\ell}\leq K,\, j \in \S_0^c,\, j + \vec \jmath \cdot \ell = 0}
u_{\ell,j}e^{\im(\ell\cdot \vf+jx)}   \, ,  
\end{equation}
and $ \Pi_K^\perp := {\rm Id}-\Pi_K  $. 
Note that, writing a traveling wave as in  
\eqref{uphiU}, the projector $ \Pi_K $ in \eqref{pro:N} is equal to 
$$ 
(\Pi_K u)(\vf, x) =  U_K (\vf - \vec \jmath x ) \, , \quad U_K (\psi) := 
\sum_{\ell \in \Z^\nu,\, \langle \ell \rangle \leq K} U_\ell e^{\im \ell \cdot \psi} \, . 
$$
{\bf Whitney-Sobolev functions.} 
Along the paper we consider  families of Sobolev functions $\lambda\mapsto u(\lambda)\in H^s (\T^{\nu+1}) $ and  
$\lambda\mapsto U(\lambda)\in H^s (\T^{\nu}) $ which are $k_0$-times differentiable in the sense of Whitney with respect to the parameter
$ \lambda :=(\omega,\kappa) \in F \subset \R^\nu\times [\kappa_1,\kappa_2] $
where $F\subset \R^{\nu+1}$ is a closed set.
The case that we encounter is when $ \omega $ belongs to  the closed set of Diophantine vectors  $ \tD\tC (\upsilon, \tau)  $  defined in \eqref{def:DCgt}.
We refer to Definition 2.1 in \cite{BBHM}, for the 
definition of a Whitney-Sobolev function $ u : F \to H^s $ where 
$ H^s $ may be either the Hilbert space $ H^s (\T^\nu \times \T) $ or $ H^s (\T^\nu) $.
Here we mention that, given $ \upsilon \in (0,1) $,  
we can identify a  Whitney-Sobolev function $ u : F \to H^s $
with $ k_0 $ derivatives 
with the equivalence class of functions $ f \in W^{k_0,\infty,\upsilon}(\R^{\nu+1},H^s)/\sim$ with respect to the equivalence relation $f\sim g$ when $\pa_\lambda^j f(\lambda) = \pa_\lambda^j g(\lambda)$ for all $\lambda\in F$, $\abs j \leq k_0-1$, 
with equivalence of the norms 
\begin{equation*}
	\normk{u}{s,F}  \sim_{\nu,k_0} \norm{u}_{W^{k_0,\infty,\upsilon}(\R^{\nu+1},H^s)}:= \sum_{\abs\alpha\leq k_0} \upsilon^{\abs\alpha} \| \pa_\lambda^\alpha u \|_{L^\infty(\R^{\nu+1},H^s)} \,.
\end{equation*}
The key result is the Whitney extension theorem,  
 which associates to a Whitney-Sobolev function 
$u : F \to H^s $ with $ k_0 $-derivatives a 
 function $\wtu : \R^{\nu+1} \to H^s $, $\wtu $ in $ W^{k_0,\infty}(\R^{\nu+1},H^s) $
 (independently of the target Sobolev space $H^s$) with an equivalent norm.  
For sake of simplicity in the notation we  
often denote  $ \| \ \|_{s,F}^{k_0,\upsilon} =  \| \ \|_{s}^{k_0,\upsilon} $.

Thanks to this equivalence, all the tame 
estimates which hold for Sobolev spaces carry over for Whitney-Sobolev functions. 
For example the following classical tame estimate for the product holds: 
(see e.g.  Lemma 2.4 in \cite{BBHM}): 
for all $s\geq s_0 > (\nu+1)/2$, 
	\begin{equation}
	\label{prod}
	\normk{u v}{s} \leq C(s,k_0) \normk{u}{s}\normk{v}{s_0} + C(s_0,k_0)\normk{u}{s_0}\normk{v}{s}\,.
	\end{equation}
Moreover the following  estimates hold for the smoothing operators defined 
in \eqref{pro:N}: for any traveling wave $ u $
\begin{equation}\label{SM12}
	\normk{\Pi_K u}{s}  \leq K^\alpha \normk{u}{s-\alpha}\,, \  0\leq \alpha\leq s \,, \quad 
		\normk{\Pi_K^\perp u}{s}  \leq K^{-\alpha} \normk{u}{s+\alpha}\,, \ \alpha\geq 0 \,.
\end{equation}
We also state a standard Moser tame estimate for the nonlinear composition operator, see e.g. Lemma 2.6 in \cite{BBHM}, 
$$
u(\vf,x)\mapsto \tf(u)(\vf,x) :=f(\vf,x,u(\vf,x))\,.
$$
Since the variables $(\vf,x)=:y$ have the same role, we state it for a generic Sobolev space $H^s(\T^d)$.
\begin{lem}{\bf (Composition operator)}\label{compo_moser}
	Let $f\in\cC^\infty(\T^d\times\R,\R)$. If $u(\lambda)\in H^s(\T^d)$ is a family of Sobolev functions satisfying $\normk{u}{s_0}\leq 1$, then, for all $s\geq s_0:=(d+1)/2$,
	$$
	  \normk{\tf(u)}{s}\leq C(s,k_0,f)\big( 1+\normk{u}{s} \big) \,.
	$$
	If $ f(\varphi, x, 0) = 0 $ then 
	$	  \normk{\tf(u)}{s}\leq C(s,k_0,f) \normk{u}{s} $. 
\end{lem}

\paragraph{Diophantine equation.} 
If $\omega $ is a Diophantine vector in $ \tD\tC(\upsilon,\tau)$, see \eqref{def:DCgt}, then the equation $\omega\cdot \pa_\vf v = u$, where $u(\vf,x)$ has zero average with respect to $\vf$, has the periodic solution
$$
	(\omega\cdot\pa_\vf)^{-1} u := \sum_{\ell\in\Z^\nu\setminus\{0\},j\in\Z} \frac{u_{\ell,j}}{\im\,\omega\cdot \ell} e^{\im(\ell\cdot\vf+jx)} \,.
$$
For all $\omega\in\R^\nu$, we define its extension
\begin{equation}\label{paext}
	(\omega\cdot\pa_\vf)_{\rm ext}^{-1} u(\vf,x) := \sum_{(\ell,j)\in\Z^{\nu+1}}\frac{\chi(\omega \cdot \ell \upsilon^{-1}\braket{\ell}^\tau)}{\im \omega\cdot\ell} u_{\ell,j}e^{\im(\ell\cdot\vf+jx)}\,,
\end{equation}
where $\chi\in\cC^\infty(\R,\R)$ is an even positive $\cC^\infty$ cut-off function such that
\begin{equation}\label{cutoff}
	\chi(\xi) = \begin{cases}
	0 & \text{ if } \ \abs\xi\leq \frac13 \\
	1 & \text{ if } \ \abs\xi \geq \frac23 
	\end{cases}\,, \qquad \pa_\xi \chi(\xi) >0   \quad \forall\,\xi \in (\tfrac13,\tfrac23) \,.
\end{equation}
Note that $(\omega\cdot\pa_\vf)_{\rm ext}^{-1} u = (\omega\cdot\pa_\vf)^{-1}u$ for all $\omega\in\tD\tC(\upsilon,\tau)$. 
Moreover, if $ u (\vf, x) $ is a quasi-periodic traveling wave 
with zero average with respect to $ \vf $, then, by \eqref{u-trav-Fourier}, we see that 
$ (\omega\cdot\pa_\vf)_{\rm ext}^{-1} u(\vf,x)  $ 
is a  quasi-periodic traveling wave. The following estimate holds
\begin{equation}\label{lem:diopha.eq}
		\normk{(\omega\cdot\pa_\vf)_{\rm ext}^{-1}u}{s,\R^{\nu+1}} \leq C(k_0)\upsilon^{-1}\normk{u}{s+\mu,\R^{\nu+1}}\,, \quad  \mu:=k_0+\tau(k_0+1) \,.
\end{equation}
and, for $F\subseteq \tD\tC(\upsilon,\tau)\times \R_+$, one has
$\normk{(\omega\cdot\pa_\vf)^{-1}u}{s,F} \leq C(k_0)\upsilon^{-1}\normk{u}{s+\mu,F} $.

\paragraph{Linear operators.}
Along the paper we consider $\vf$-dependent families of linear operators $A:\T^\nu\mapsto \cL(L^2(\T_x))$, $\vf\mapsto A(\vf)$, acting on subspaces of $L^2(\T_x)$, either real or complex valued. We also regard $A$ as an operator (which for simplicity we denote by $A$ as well) that acts on functions $u(\vf,x)$ of space and time, that is
\begin{equation}\label{linear_A}
(Au)(\vf,x) := \left( A(\vf)u(\vf,\,\cdot \,) \right)(x) \,.
\end{equation}
The action of an operator $A$ as in \eqref{linear_A} on a scalar function $u(\vf,x)\in L^2$ expanded as in \eqref{u_fourier} is
\begin{align}\label{A_expans}
Au(\vf,x) & = \sum_{j,j'\in\Z} A_j^{j'}(\vf) u_{j'}(\vf)e^{\im\, jx}  
 = \sum_{j,j'\in\Z}\sum_{\ell,\ell'\in\Z^\nu} A_j^{j'}(\ell-\ell') u_{\ell',j'} e^{\im\left(\ell\cdot \vf + jx\right)} \,.
\end{align}
We identify an operator $A$ with its matrix $ \big( A_j^{j'}(\ell-\ell') \big)_{j,j'\in\Z,\ell,\ell'\in\Z^\nu}$, which is T\"oplitz with respect to the index $\ell$. In this paper we always consider T\"oplitz operators as in \eqref{linear_A}, \eqref{A_expans}.

\paragraph{Real operators.} 
A linear  operator $ A $ is {\it real} if 
$ A = \bar A  $,
where $ \bar A $ is defined by  
$ \bar{A}(u):= \bar{A(\bar u)} $. 
Equivalently $A$ is \emph{real} if it maps real valued functions into real valued functions.
We represent a real operator acting on $(\eta,\zeta) $ belonging to 
(a subspace of) $  L^2(\T_x,\R^2) $   by a matrix
\begin{equation}\label{real_matrix}
\cR = \begin{pmatrix}
A & B \\ C & D
\end{pmatrix}
\end{equation}
where $A,B,C,D$ are real operators acting on the scalar valued components 
$\eta,\zeta \in L^2(\T_x,\R) $. 

The  change of coordinates \eqref{C_transform}
transforms the real operator $\cR$ into a complex one acting on the variables
$(z,\bar z)  $, given by the matrix
\begin{equation}\label{C_transformed}
\begin{split}
\bR & := \cC^{-1} \cR \cC = \left( \begin{matrix}
\cR_1 & \cR_2 \\ \bar\cR_2 & \bar\cR_1
\end{matrix} \right) \ , \\
\cR_1 & := \frac{1}{2} \left\{ (A+D) -\im(B-C) \right\} \ , \quad \cR_2:= \frac{1}{2}\left\{ (A-D) + \im (B+C)  \right\} \, . 
\end{split}
\end{equation}
A matrix operator acting on the complex variables $(z,\bar z)$ of the form
\eqref{C_transformed}, we call it \emph{real}.
We shall also consider  real operators $\bR $ of the form 
\eqref{C_transformed} acting on subspaces of $ L^2 $. 

\paragraph{Lie expansion.} Let $X(\vf)$ be a linear operator
with  associated  flow $\Phi^\tau ( \vf)$ defined by
$$
	\begin{cases}
		\pa_\tau \Phi^\tau (\vf) = X(\vf) \Phi^\tau (\vf) \\
		\Phi^0 (\vf) = {\rm Id} \, , 
	\end{cases} \quad  \tau \in[0,1] \, . 
$$
Let $ \Phi(\vf) := \Phi^\tau (\vf)_{|\tau = 1} $ denote the time-$1$ flow.
Given a linear operator $ A ( \vf ) $,
the conjugated operator 
$$ 
A^+(\vf):=\Phi(\vf)^{-1}A(\vf) \Phi(\vf) 
$$
admits the  Lie expansion, for any  $ M \in \N_0 $, 
\begin{equation}\label{lie_abstract}
	\begin{aligned}
	 & A^+(\vf) = 
	 \sum_{m=0}^M \frac{(-1)^m}{m!} \ad_{X(\vf)}^m(A(\vf)) +R_M(\vf)\,, \\ 
	R_M(\vf)& = \frac{(-1)^{M+1}}{M!}\int_0^1 (1- \tau)^M \, (\Phi^\tau (\vf))^{-1} \ad_{X(\vf)}^{M+1}(A(\vf)) \Phi^\tau(\vf) \, \wrt \tau \,,
	\end{aligned}
\end{equation}
where $\ad_{X(\vf)}(A(\vf)) := [X(\vf), A(\vf)] = X(\vf) A(\vf) - A(\vf) X(\vf) $ 
and  $\ad_{X(\vf)}^0 := {\rm Id} $. 

In particular, for $A=\omega\cdot\pa_\vf$, since $[X(\vf), \omega\cdot \pa_\vf] = 
- ( \omega\cdot \pa_\vf X)(\vf) $, we obtain
\begin{equation}\label{lie_omega_devf}
	\begin{aligned}
		\Phi(\vf)^{-1} \circ  \omega\cdot\pa_\vf \circ \Phi(\vf) = &\, \omega\cdot\pa_\vf 
		+ 
		\sum_{m=1}^{M} \frac{(-1)^{m+1}}{m!} \ad_{X(\vf)}^{m-1}(\omega\cdot \pa_\vf X(\vf)) \\
		& + \frac{(-1)^M}{M!}\int_0^1(1- \tau)^{M} (\Phi^\tau(\vf))^{-1}\ad_{X(\vf)}^M(\omega\cdot\pa_\vf X(\vf))\Phi^\tau(\vf) \wrt \tau \,.
	\end{aligned}
\end{equation}
For matrices of operators $\bX(\vf)$ and $\bA(\vf)$ as in \eqref{C_transformed}, the same formula \eqref{lie_abstract} holds.

\subsection{Pseudodifferential calculus}\label{subsec:pseudo_calc}

In this section we report fundamental notions of pseudodifferential calculus,
following \cite{BM}. 
\begin{defn}{\bf ($\Psi$DO)} 
A  \emph{pseudodifferential} symbol $ a (x,j) $
	of order $m$ is 
	the restriction to $ \R \times \Z $ of a function $ a (x, \xi ) $ which is $ \cC^\infty $-smooth on $ \R \times \R $,
	$ 2 \pi $-periodic in $ x $, and satisfies 
$$
	| \pa_x^\alpha \pa_\xi^\beta a (x,\xi ) | \leq C_{\alpha,\beta} \langle \xi \rangle^{m - \beta} \, , \quad \forall \alpha, \beta \in \N_0 \, .
$$
	We denote by $ S^m $ 
	the class of  symbols  of order $ m $ and 
	$ S^{-\infty} := \cap_{m \geq 0} S^m $. 
	To a  symbol $ a(x, \xi ) $ in $S^m$ we associate its quantization acting on a $ 2 \pi $-periodic function 
$ u(x) = \sum_{j \in \Z} u_j \, e^{\im j x} $
as 
$$
[\Op(a)u](x) : = \sum_{j \in \Z}  a(x, j )  u_j \, e^{\im j x} \, . 
$$
We denote by  $ \Ops^m $ 
the  set of pseudodifferential  operators of order $ m $ and
 $ \Ops^{-\infty} := \bigcap_{m \in \R} \Ops^{m} $.
For a matrix of pseudodifferential operators 
\begin{equation}\label{operatori matriciali sezione pseudo diff}
\bA = \begin{pmatrix}
A_1  & A_2 \\
A_3 & A_4
\end{pmatrix}, \quad A_i \in \Ops^m, \quad i =1, \ldots , 4
\end{equation}
we say that  $\bA \in \Ops^m$.
\end{defn}

When the symbol $ a (x) $ is independent of $ \xi $, the operator $ \Op (a) $ is 
the multiplication operator by the function $ a(x)$, i.e.\  $ \Op (a) : u (x) \mapsto a ( x) u(x )$. 
In such a case we also denote $ \Op (a)  = a (x) $.

We shall use 
the following notation, used also in \cite{AB,BM,BBHM}.
For any $m \in \R \setminus \{ 0\}$, we set
$$
|D|^m := \Op \big( \chi(\xi) |\xi|^m \big)\,,
$$
where $\chi$ is an  even, positive $\cC^\infty$ cut-off satisfying  \eqref{cutoff}.
We also identify the Hilbert transform $\cH$, acting on the $2 \pi$-periodic functions, defined by
\begin{equation}\label{Hilbert-transf}
\cH ( e^{\im j x} ) := - \im \, \sign(j) e^{\im jx} \,  \quad  \forall j \neq 0 \, ,  \quad \cH (1) := 0\,, 
\end{equation}
with the Fourier multiplier $\Op (- \im\, \sign(\xi) \chi(\xi) )$. Similarly we regard 
the  operator 
\begin{equation}\label{pax-1}
	\pa_x^{-1}\left[ e^{\im jx}\right] := -\,\im \,j^{-1} \,e^{\im jx} \, \quad \forall\, j\neq 0\,,  \quad \pa_x^{-1}[1] := 0\,,
\end{equation}
as the  Fourier  multiplier $\pa_x^{-1} =  \Op \left( - \im\,\chi(\xi) \xi^{-1} \right)$ and the projector $\pi_0 $, defined on the $ 2 \pi $-periodic functions  as
\begin{equation}\label{defpi0}
\pi_0 u := \frac{1}{2\pi} \int_\T u(x)\, d x\, , 
\end{equation}
with the Fourier multiplier $ {\rm Op}\big( 1 - \chi(\xi) \big)$.
Finally we define, for  any $m \in \R \setminus \{ 0 \}$, 
$$
\langle D \rangle^m := \pi_0 + |D|^m
:= \Op \big( ( 1 - \chi(\xi)) + \chi(\xi) |\xi|^m \big) \, .
$$
Along the paper we consider families of pseudodifferential operators 
with a symbol $  a(\lambda;\vf,x,\xi)  $
which is $k_0$-times differentiable with respect to a parameter
$ \lambda:=(\omega,\kappa) $ in an open subset $ \Lambda_0 \subset \R^\nu\times [\kappa_1,\kappa_2] $. 
Note that $\partial_\lambda^k A = \Op\left( \partial_\lambda^k a \right) $ for any $k\in \N_0^{\nu+1}$. 

We  recall the  pseudodifferential norm introduced in Definition 2.11 in \cite{BM}. 

\begin{defn}
	{\bf (Weighted $\Psi DO$ norm)}
	Let $ A(\lambda) := a(\lambda; \vf, x, D) \in \Ops^m $ 
	be a family of pseudodifferential operators with symbol $ a(\lambda; \vf, x, \xi) \in S^m $, $ m \in \R $, which are 
	$k_0$-times differentiable with respect to $ \lambda \in \Lambda_0 \subset \R^{\nu + 1} $. 
	For $ \upsilon \in (0,1) $, $ \alpha \in \N_0 $, $ s \geq 0 $, we define  
	$$
	\norm{A}_{m, s, \alpha}^{k_0, \upsilon} := \sum_{|k| \leq k_0} \upsilon^{|k|} 
	\sup_{\lambda \in {\Lambda}_0}\norm{\partial_\lambda^k A(\lambda)}_{m, s, \alpha}  
	$$
	where 
	$
	\norm{A(\lambda)}_{m, s, \alpha} := 
	\max_{0 \leq \beta  \leq \alpha} \, \sup_{\xi \in \R} \|  \partial_\xi^\beta 
	a(\lambda, \cdot, \cdot, \xi )  \|_{s} \  \langle \xi \rangle^{-m + \beta} $. 
	For a matrix of pseudodifferential operators $\bA \in \Ops^m$ as in \eqref{operatori matriciali sezione pseudo diff}, we define 
	$
	\norm{\bA}_{m, s, \alpha}^{k_0, \upsilon} 
	:= \max_{i = 1, \ldots, 4} \norm{A_i}_{m, s, \alpha}^{k_0, \upsilon}\,. 
	$
\end{defn}

Given a function $a(\lambda; \vf, x) \in \cC^\infty$ which is $k_0$-times differentiable with respect to $\lambda$, the weighted norm of the corresponding multiplication operator is
\begin{equation}
\label{norm.mult}
\normk{\Op(a)}{0,s,\alpha} = \normk{a}{s}  \, , \quad \forall \alpha \in \N_0 \, .  
\end{equation}

\paragraph{Composition of pseudodifferential operators.}
If $ {\rm Op}(a) $, ${\rm Op}(b) 
$ are pseudodifferential operators with symbols $a\in S^m$, $b\in S^{m'}$, 
$m,m'\in\R$,  
then the composition operator 
$ {\rm Op}(a) {\rm Op}(b) $ 
is a pseudodifferential operator  $ {\rm Op}(a\# b) $ with symbol $a\# b\in S^{m+m'}$. 
It 
 admits the  asymptotic expansion: for any $N\geq 1$
\begin{align}\label{compo_symb}
	(a\# b)(\lambda;\vf,x,\xi) & = \sum_{\beta= 0}^{N-1} \frac{1}{\im^\beta \beta!} \pa_\xi^\beta a(\lambda;\vf,x,\xi) \pa_x^\beta b(\lambda;\vf,x,\xi) + (r_N(a,b))(\lambda;\vf,x,\xi) 
\end{align}
where  $ r_N(a,b) \in S^{m+m'-N} $. 
The following result is proved in Lemma 2.13 in \cite{BM}.

\begin{lem}{\bf (Composition)} \label{pseudo_compo}
	Let $ A = a(\lambda; \vf, x, D) $, $ B = b(\lambda; \vf, x, D) $ be pseudodifferential operators
	with symbols $ a (\lambda;\vf, x, \xi) \in S^m $, $ b (\lambda; \vf, x, \xi ) \in S^{m'} $, $ m , m' \in \R $. Then $ A \circ B \in \Ops^{m + m'} $
	satisfies,   for any $ \alpha \in \N_0 $, $ s \geq s_0 $, 
	\begin{equation}\label{eq:est_tame_comp}
	\begin{split}
	\norm{A B}_{m + m', s, \alpha}^{k_0, \upsilon} &
	\lesssim_{m,  \alpha, k_0} C(s) \norm{ A }_{m, s, \alpha}^{k_0, \upsilon} 
	\norm{ B}_{m', s_0 + |m|+\alpha, \alpha}^{k_0, \upsilon}  \\
	& \ \quad \qquad + C(s_0) \norm{A }_{m, s_0, \alpha}^{k_0, \upsilon}  
	\norm{ B }_{m', s + |m|+\alpha, \alpha}^{k_0, \upsilon} \, .
	\end{split}
	\end{equation}
	Moreover, for any integer $ N \geq 1  $,  
	the remainder $ R_N := {\rm Op}(r_N) $ in \eqref{compo_symb} satisfies
	\begin{equation}
	\begin{aligned}
	\norm{\Op(r_N(a,b))}_{m+ m'- N, s, \alpha}^{k_0, \upsilon}
	&\lesssim_{m, N,  \alpha, k_0} 
	C(s) \norm{ A}_{m, s, N + \alpha}^{k_0, \upsilon} 
	\norm{ B }_{m', s_0 + \abs{m} +  2N  + \alpha,N+\alpha }^{k_0, \upsilon}  
	\\
	& \ \qquad \qquad  +  C(s_0)\norm{A}_{m, s_0   , N + \alpha}^{k_0, \upsilon}
	\norm{ B}_{m', s +|m| + 2 N  + \alpha, N+ \alpha }^{k_0, \upsilon}.
	\label{eq:rem_comp_tame} 
	\end{aligned}
	\end{equation}
	Both  \eqref{eq:est_tame_comp}-\eqref{eq:rem_comp_tame} hold  
	with the constant $ C(s_0) $ 
	interchanged with $ C(s) $. \\
	Analogous estimates hold if $\bA$ and $\bB$ are matrix operators of the form \eqref{operatori matriciali sezione pseudo diff}. 
\end{lem}

The commutator between two pseudodifferential operators $ \Op(a)\in \Ops^m$ and $\Op(b)\in\Ops^{m'}$ is a pseudodifferential operator
in $ \Ops^{m+m'-1}$ with symbol $a\star b\in S^{m+m'-1}$, 
namely $ \left[ \Op(a), \Op(b)\right] = \Op\left( a\star b \right)$, 
that admits, by \eqref{compo_symb},  the expansion
\begin{equation}\label{eq:moyal_exp}
\begin{aligned}
& a\star b= -\im\left\{ a,b \right\} + \wt{r_2}(a,b) \,, \quad \wt{r_2}(a,b):=r_2(a,b)-r_2(b,a)\in S^{m+m'-2} \,, \\
& 
{\rm where} \quad  \{ a,b \}:= \pa_\xi a \pa_x b - \pa_x a \pa_\xi b \, , 
\end{aligned}
\end{equation}
is the Poisson bracket between $a(x,\xi)$ and $b(x,\xi)$. 
As a corollary of  Lemma \ref{pseudo_compo} we have: 
\begin{lem}{\bf (Commutator)} \label{pseudo_commu}
	Let $A = {\rm Op}(a) $ and $B = {\rm Op} (b) $ be pseudodifferential operators with symbols $a(\lambda;\vf,x,\xi)\in S^{m}$, $b(\lambda;\vf,x,\xi)\in S^{m'}$, $m,m'\in \R$. Then the commutator $[A,B]:=AB-BA\in \Ops^{m+m'-1}$ satisfies
	\begin{equation}\label{eq:comm_tame_AB}
	\begin{aligned}
	\norm{[A,B]}_{m+m'-1,s,\alpha}^{k_0,\upsilon} & \lesssim_{m, m', \alpha, k_0} C(s)\norm{A }_{m,s+|m'|+\alpha+2,\alpha+1}^{k_0,\upsilon}\norm{ B }_{m',s_0+|m|+\alpha+2,\alpha+1}^{k_0,\upsilon}\\
	& \qquad \quad \ + C(s_0)\norm{A }_{m,s_0+|m'|+\alpha+2,\alpha+1}^{k_0,\upsilon}\norm{ B }_{m',s+|m|+\alpha+2,\alpha+1}^{k_0,\upsilon} \,.
	\end{aligned}
	\end{equation}
\end{lem}

Finally we consider the exponential of a pseudodifferential operator of order $0$.
The following lemma follows as in Lemma 2.12 of \cite{BKM} (or Lemma 2.17 in \cite{BM}).

\begin{lem} {\bf (Exponential map)}
 \label{Neumann pseudo diff}
If $ A := {\rm Op}(a(\lambda; \vf, x,  \xi ))$ 
is in $ OPS^{0} $,   
then $e^A$ is in $ OPS^{0} $ and 
for any $s \geq s_0$, $\alpha \in \N_0 $,
there is a constant 
$C(s, \alpha) > 0$ so that 
$$
\normk{e^A - {\rm Id} }{0, s, \alpha} \leq  \normk{A}{0, s + \alpha, \alpha} \, {\rm exp} \big( C(s, \alpha) 
\normk{A}{0, s_0 + \alpha, \alpha}\big)\, .
$$
The same holds for a matrix $\bA$ of the form \eqref{operatori matriciali sezione pseudo diff}  in $\Ops^0$.
\end{lem}

\paragraph{Egorov Theorem.}

Consider the family of $ \vphi $-dependent diffeomorphisms of $ \T_x $ defined by 
\begin{equation}
\label{diffeo_inverso}
y= x + \beta(\vf, x) \qquad \Longleftrightarrow \qquad 
x= y + \breve\beta(\vf, y) 	\, , 
\end{equation}
where $\beta(\vf, x)$ is a small smooth function, 
and the induced operators 
\begin{equation}\label{defB-diffeo}
(\cB u)(\vf, x) :=  u(\vf, x + \beta(\vf, x)) \, , \quad 
(\cB^{-1}u)(\vf, y) :=  u(\vf, y + \breve\beta(\vf, y)) \,.
\end{equation}
\begin{lem}{\bf (Composition)}\label{product+diffeo}
	Let $\normk{\beta}{2s_0+k_0+2}\leq \delta(s_0,k_0)$ small enough. 
Then the composition operator
	$ 	\cB $
	satisfies the tame estimates, for any $s\geq s_0$,
	$$
	\normk{\cB u}{s} \lesssim_{s,k_0} \normk{u}{s+k_0} + \normk{\beta}{s} \normk{u}{s_0+k_0+1} \, , 
	$$
and the function $\breve{\beta}$ defined in \eqref{diffeo_inverso} by the inverse diffeomorphism
	satisfies 
	$	\normk{\breve{\beta}}{s} \lesssim_{s,k_0} \normk{\beta}{s+k_0} $. 
\end{lem}

The following result is a  small variation of Proposition 2.28 of \cite{BKM}.
\begin{prop}\label{egorov} {\bf (Egorov)}
Let $N \in \N$, $\tq_0 \in \N
_0 $,  $S > s_0$ and assume that 
$\pa_\lambda^k \beta(\lambda; \cdot, \cdot)$ 
are $ \cC^\infty$ for all $|k| \leq k_0$. There exist constants  $\sigma_N, \sigma_N(\tq_0) >0$,  $\delta = \delta(S, N, \tq_0, k_0) \in (0,1)$ such that, if
$ \normk{\beta}{s_0 + \sigma_N(\tq_0)} \leq \delta $, 
then the conjugated operator
$  \cB^{-1} \circ \pa_x^{m}\circ \cB $, $ m \in \Z $,   
is a pseudodifferential operator of order $ m $ with an expansion of the form
\begin{equation*}
\cB^{-1} \circ \pa_x^{m} \circ \cB = \sum_{i=0}^N p_{m - i}(\lambda; \vf, y) \pa_y^{m - i}  + \cR_N(\vf)
\end{equation*}
with the following properties:
\\[1mm]
1. The principal symbol of $p_{m}$  is 
$$
p_{m}(\lambda; \vf, y) = \Big( [1+\beta_x(\lambda;\vf,x)]^{m} \Big)\vert_{x=y+
\breve \beta(\lambda;\vf,y)}
$$
where $\breve{\beta}(\lambda;\vf,y)$ has been introduced in \eqref{diffeo_inverso}. For any $s \geq s_0$ and $i=1, \ldots, N$,
\begin{equation}\label{norm-pk}
 \normk{p_m - 1}{s} \, , \ 
\normk{p_{m-i}}{s} \lesssim_{s, N} \normk{\beta}{s+\sigma_N}\, . 
\end{equation}
2. For any $ \tq \in \N^\nu_0 $ with $ |\tq| \leq \tq_0$, 
 $n_1, n_2 \in \N_0 $  with $ n_1 + n_2 + \tq_0 \leq N + 1 - k_0 - m  $,   the  
 operator $\langle D \rangle^{n_1}\partial_{\vphi}^\tq {\cal R}_N(\vphi) \langle D \rangle^{n_2}$ is 
$\cD^{k_0} $-tame with a tame constant satisfying, for any $s_0 \leq s \leq S $,   
\begin{equation}\label{stima resto Egorov teo astratto}
{\mathfrak M}_{\langle D \rangle^{n_1}\partial_{\vphi}^\tq {\cal R}_N(\vphi) \langle D \rangle^{n_2}}(s) \lesssim_{S, N, \tq_0} 
\| \beta\|_{s + \sigma_N(\tq_0)}^{k_0,\upsilon} 
\,. 
\end{equation}
3.  Let $s_0 < s_1 $ 
and assume that  $\| \beta_j \|_{s_1 + \sigma_N(\tq_0)} \leq \delta,$ 
 $j = 1,2$. Then 
$ \| \Delta_{12} p_{m - i} \|_{s_1} \lesssim_{s_1, N} 
 \| \Delta_{12} \beta\|_{s_1 + \sigma_N} $, $ i = 0, \ldots, N $, 
and, for any $ |\tq| \leq \tq_0$,  $n_1, n_2 \in \N_0 $ with $n_1 + n_2 + \tq_0 \leq N  - m$, 
$$
\| \langle D \rangle^{n_1}\partial_{\vphi}^\tq \Delta_{12} {\cal R}_N(\vphi) \langle D \rangle^{n_2} \|_{{\cal B}(H^{s_1})} \lesssim_{s_1, N, n_1, n_2} 
 \| \Delta_{12} \beta\|_{s_1 + \sigma_N(\tq_0)} \, . 
$$
Finally, if $ \beta (\vf, x ) $ is a quasi-periodic traveling wave,  
then $ \cB $ is momentum preserving (we refer to 
Definition \ref{def:mom.pres} and Lemma \ref{lem:dMP}), 
as well as the conjugated operator $ \cB^{-1} \circ\pa_x^m \circ\cB $,
and each function $ p_{m-i} $, $ i = 0, \ldots, N $, is a quasi-periodic traveling wave. 
\end{prop}

\paragraph{Dirichlet-Neumann operator.} 

We finally remind the following decomposition of the Dirichlet-Neumann operator proved in
\cite{BM}, in the case of infinite depth, and in \cite{BBHM},  for finite depth.

\begin{lem}\label{DN_pseudo_est}
{\bf (Dirichlet-Neumann)}
Assume that $\pa_\lambda^k \eta(\lambda, \cdot, \cdot) $ is $\cC^\infty (\T^\nu \times \T_x)$ for all $|k| \leq k_0$.
 There exists $ \delta(s_0, k_0) >0$ such that, if
$ \normk{\eta}{2s_0 +2k_0  +1} \leq \delta (s_0, k_0) $, 
then the Dirichlet-Neumann operator $ G(\eta ) = G(\eta, \tth)$ may be written as 
\begin{equation}
\label{DN.dec}
G(\eta, \tth) = G(0, \tth) + \cR_G(\eta)
\end{equation}
where $ \cR_G(\eta) := \cR_G(\eta, \tth)  \in \Ops^{-\infty}$  
satisfies, for all $m, s, \alpha \in \N_0$, the estimate
\begin{align}
\label{est:RG}
\normk{\cR_G(\eta)}{-m, s, \alpha} 
\leq  C(s, m, \alpha, k_0) \normk{\eta}{s+s_0 +2k_0 +m+\alpha + 3} \, . 
\end{align}
\end{lem}

\subsection{$\cD^{k_0}$-tame and modulo-tame operators}\label{subsec:semiphi}

We present the notion of tame and modulo tame operators introduced in \cite{BM}. 
Let $ A := A(\lambda) $ be a linear operator as in \eqref{linear_A}, 
$ k_0 $-times differentiable with respect to
the parameter $ \lambda $ in the open set $ \Lambda_0 \subset \R^{\nu+1}$. 

\begin{defn}{\bf ($\cD^{k_0}$-$\sigma$-tame)} \label{Dk_0sigma}
	Let $\sigma\geq 0$. A linear operator $A:=A(\lambda)$   is $\cD^{k_0}$-$\sigma$-tame if there exists a non-decreasing function $[s_0,S]\rightarrow[0,+\infty)$, $s\mapsto\fM_A(s)$, with possibly $S=+\infty$, 
	such that, for all $s_0\leq s\leq S$ and $u\in H^{s+\sigma} $, 
	\begin{equation}\label{tame_constants}
	\sup_{\abs k \leq k_0}\sup_{\lambda\in \Lambda_0} \upsilon^{\abs k} \norm{(\pa_\lambda^k A(\lambda))u }_s \leq \fM_A(s_0) \norm u_{s+\sigma} + \fM_A(s)\norm u_{s_0+\sigma} \,.
	\end{equation}
	We say that $\fM_A(s)$ is a \emph{tame constant} of the operator $A$. The constant $\fM_A(s)=\fM_A(k_0,\sigma,s)$ may also depend on $k_0,\sigma$ but
	 we shall often omit to write them.
	When the "loss of derivatives" $\sigma$ is zero, we simply write $\cD^{k_0}$-tame instead of $\cD^{k_0}$-$0$-tame. 
	For a matrix operator as in \eqref{C_transformed}, we denote the tame constant $\fM_{\bR}(s):=\max\left\{ \fM_{\cR_1}(s),\fM_{\cR_2}(s) \right\}  $.
\end{defn}

Note that the tame constants $\fM_A(s)$ are not uniquely determined. 
An immediate consequence of \eqref{tame_constants} is that $\norm A_{\cL\left(H^{s_0+\sigma},H^{s_0}\right)}\leq 2 \fM_{A}(s_0)$. 
Also note that, representing the operator $A$ by its matrix elements 
$ (A_j^{j'}(\ell-\ell') )_{\ell,\ell'\in\Z^\nu,j,j'\in\Z}$ as in \eqref{A_expans}, we have for all $\abs k \leq k_0$, $j'\in\Z$, $\ell'\in\Z^\nu$,
\begin{equation}\label{salva}
\upsilon^{2\abs k} \sum_{\ell,j}\braket{\ell,j}^{2s} \big| \pa_\lambda^k A_j^{j'}(\ell-\ell') \big|^2 \leq 2 \big( \fM_A(s_0)\big)^2 
\braket{\ell',j'}^{2(s+\sigma)} + 2 (\fM_A(s))^2\braket{\ell',j'}^{2(s_0+\sigma)} \,.
\end{equation}
The class of $\cD^{k_0}$-$\sigma$-tame operators is closed under composition.
\begin{lem}{\bf (Composition, Lemma 2.20 in \cite{BM})}\label{tame_compo}
	Let $A,B$ be respectively $\cD^{k_0}$-$\sigma_A$-tame and $\cD^{k_0}$-$\sigma_B$-tame operators with tame constants respectively $\fM_A(s)$ and $\fM_B(s)$. Then the composed operator $A\circ B$ is $\cD^{k_0}$-$(\sigma_A+\sigma_B)$-tame with tame constant
$$
	\fM_{AB}(s) \leq C(k_0) \left( \fM_A(s) \fM_B(s_0+\sigma_A) + \fM_A(s_0)\fM_B(s+\sigma_A) \right)\,.
$$
\end{lem}
It is proved in Lemma 2.22 in \cite{BM} that 
the action of a $\cD^{k_0}$-$\sigma$-tame operator $A(\lambda)$ on a 
Sobolev function $  u = u(\lambda)\in H^{s+\sigma}$
is bounded by 
\begin{equation}\label{action_Hs_tame}
\normk{ Au }{s} \lesssim_{k_0} \fM_A(s_0) \normk{u}{s+\sigma} + \fM_A(s)\normk{u}{s_0+\sigma}\,.
\end{equation}

Pseudodifferential operators are tame operators. We  use in particular the following lemma: 
\begin{lem}\label{tame_pesudodiff} {\bf (Lemma 2.21 in \cite{BM})}
	Let $A=a(\lambda;\vf,x,D)\in\Ops^0$ be a family of pseudodifferential operators 
	satisfying $\normk{A}{0,s,0}<\infty$ for $s\geq s_0$. Then 
	$A$ is $\cD^{k_0}$-tame with a tame constant $ \fM_A(s) $ satisfying, for any
	$s\geq s_0$,  
	\begin{equation}\label{TameA_PS}
	\fM_A(s) \leq C(s) \normk{A}{0,s,0}\,.
	\end{equation}  
	The same statement holds for a matrix operator $\bR$ as in \eqref{C_transformed}.
\end{lem}

In view of the KAM reducibility scheme of Section \ref{sec:KAM} we also consider the stronger notion of $\cD^{k_0}$-modulo-tame operator, that we need only for operators with loss of derivative $\sigma=0$. We first recall the notion of \emph{majorant operator}:
	given a linear operator $A$ acting as  in \eqref{A_expans}, we define the 
	majorant operator $| A |$ by its matrix elements  
		$ (|A_j^{j'}(\ell-\ell') |)_{\ell,\ell'\in\Z^\nu, j,j'\in\Z} $.

\begin{defn}{\bf ($\cD^{k_0}$-modulo-tame)} \label{Dk0-modulo}
	A linear operator $A=A(\lambda)$ 
	is $\cD^{k_0}$\emph{-modulo-tame} if there exists a non-decreasing function $[s_0,S]\rightarrow[0,+\infty]$, $s\mapsto \fM_A^\sharp(s)$, such that for all $k\in\N_0^{\nu+1}$, $\abs k\leq k_0$, the majorant operator $\abs{\pa_\lambda^k A}$ satisfies,  
	for all $s_0\leq s\leq S$ and $u\in H^s$,
\begin{equation}\label{def:mod-tame}
	\sup_{\abs k \leq k_0}\sup_{\lambda \in {\Lambda}_0}\upsilon^{\abs k} \norm{ \abs{\pa_\lambda^k A}u }_s \leq \fM_A^\sharp(s_0) \norm u_s + \fM_A^\sharp(s) \norm u_{s_0} \,.
\end{equation}
	The constant $\fM_A^\sharp (s)$ is called a \emph{modulo-tame constant} for the operator $A$. For a matrix of operators as in \eqref{C_transformed}, we denote the modulo-tame constant $\fM_{\bR}^\sharp(s):= \max\{ \fM_{\cR_1}^\sharp(s),\fM_{\cR_2}^\sharp(s) \}$.
\end{defn}
If $A$, $B$ are $\cD^{k_0}$-modulo-tame operators with 
$ | A_j^{j'}(\ell)  | \leq | B_j^{j'}(\ell) | $, then $\fM_A^\sharp(s)\leq  \fM_B^\sharp(s)$. A $\cD^{k_0}$-modulo-tame operator is also $\cD^{k_0}$-tame and $\fM_A(s)\leq \fM_A^\sharp(s)$.

In view of the next lemma, given a linear operator $A$ acting as  in \eqref{A_expans}, we define the operator
 $\braket{\pa_{\vf}}^\tb A$, $ \tb\in\R$, whose matrix elements are $\braket{\ell-\ell'}^\tb A_j^{j'}(\ell-\ell')$.

\begin{lem}{\bf (Sum and composition, Lemma 2.25 in \cite{BM})} 
\label{modulo_sumcomp}
Let $A$, $B$, $\braket{\pa_{\vf}}^\tb A$, $\braket{\pa_{\vf}}^\tb B$ 
 be $\cD^{k_0}$-modulo-tame operators. Then 
$A+B$, $A\circ B$ and $\braket{\pa_{\vf}}^\tb(AB)$  are $\cD^{k_0}$-modulo-tame with
\begin{align*}
& \fM_{A+B}^\sharp(s)\leq \fM_A^\sharp (s)+ \fM_B^\sharp(s) \\
&
	\fM_{AB}^\sharp(s) \leq C(k_0) \big(  \fM_A^\sharp(s)\fM_B^\sharp(s_0) + \fM_A^\sharp(s_0)\fM_B^ \sharp(s) \big) \\
&	\fM_{\braket{\pa_{\vf}}^\tb(AB)}^\sharp(s) \leq C(\tb)C(k_0) \big(  \fM_{\braket{\pa_{\vf}}^\tb A}^\sharp(s) \fM_B^\sharp(s_0) +  \fM_{\braket{\pa_{\vf}}^\tb A}^\sharp(s_0) \fM_B^\sharp(s)  \notag\\
	&  \qquad \qquad \qquad \qquad \qquad \quad 
	+ \fM_{A}^\sharp(s) \fM_{\braket{\pa_{\vf}}^\tb B}^\sharp(s_0)+\fM_{A}^\sharp(s_0) \fM_{\braket{\pa_{\vf}}^\tb B}^\sharp(s) \big)  \, . 
	\end{align*}
	The same statement holds for  matrix operators $\bA$, $\bB$ as in \eqref{C_transformed}.
\end{lem}

By Lemma \ref{modulo_sumcomp} we deduce the following result, 
cfr. Lemma 2.20 in \cite{BKM}.

\begin{lem}{\bf (Exponential)} \label{modulo_expo}
	Let $A$ and $\braket{\pa_\vf}^\tb A$ be $\cD^{k_0}$-modulo-tame and assume 
	that $\fM_A^\sharp(s_0) \leq 1  $. Then 
	 the operators $e^{\pm A}-{\rm Id}$ and $\braket{\pa_\vf}^\tb e^{\pm A}- {\rm Id}$ are $\cD^{k_0}$-modulo-tame  with modulo-tame constants satisfying
$$
	\fM_{e^{\pm A} -{\rm Id}}^\sharp (s) \lesssim_{k_0} \fM_A^\sharp (s) \,, \quad
	\fM_{\braket{\pa_\vf}^\tb e^{\pm A}-{\rm Id}}^\sharp(s) \lesssim_{k_0,\tb} \fM_{\braket{\pa_\vf}^\tb A}^\sharp (s)  + \fM_A^\sharp(s)\fM_{\braket{\pa_\vf}^\tb A}^\sharp (s_0) \,.
	$$
\end{lem}

Given a linear operator $A$ acting as  in \eqref{A_expans}, we define the 
\emph{smoothed operator}
 $\Pi_N A$, $N\in\N $ whose matrix elements are
		\begin{equation}\label{PiNA}
		(\Pi_N A)_j^{j'}(\ell-\ell') := \begin{cases}
		A_j^{j'}(\ell-\ell') & \text{if } \braket{\ell-\ell'} \leq N \\
		0 & \text{otherwise} \, . 
		\end{cases}
		\end{equation}
		We also denote $\Pi_N^\perp:= {\rm Id}-\Pi_N$. 
		It is proved in Lemma 2.27 in \cite{BM} that 	
\begin{equation}\label{modulo_smooth}
	\fM_{\Pi_N^\perp A}^\sharp(s) \leq N^{-\tb} \fM_{\braket{\pa_{\vf}}^\tb A}^\sharp(s)\,, \quad \fM_{\Pi_N^\perp A}^\sharp(s)\leq \fM_{A}^\sharp(s) \,.
\end{equation}
The same estimate holds with a matrix operator $\bR$ as in \eqref{C_transformed}.

\subsection{Hamiltonian and Reversible operators}

In this paper we  shall exploit both the Hamiltonian and reversible structure 
along the reduction of the  linearized operators, that we now present. 
\\[1mm]
{\bf Hamiltonian operators.} 
A matrix operator $\cR $ as in \eqref{real_matrix} is Hamiltonian if 
the matrix
$$
J^{-1} \cR = 
\begin{pmatrix}
0 & - {\rm Id}
 \\  {\rm Id}& 0 
\end{pmatrix}\begin{pmatrix}
A & B \\ C & D
\end{pmatrix} = \begin{pmatrix}
- C & - D \\ A &  B
\end{pmatrix}
$$
is self-adjoint, namely   
$ B^* = B $, $ C^*=C $, $  A^* = - D $
and $ A, B, C, D $ are real.

Correspondingly, a matrix operator as in \eqref{C_transformed} is Hamiltonian if 
\begin{equation} \label{RHamC}
\cR_1^* = - \cR_1 \, , \quad  \cR_2^* = \bar{\cR_2} \, . 
\end{equation}
{\bf Symplectic operators. }
A $\vphi $-dependent family of linear operators 
$ \cR (\vphi ) $, 
$ \vphi \in \T^\nu $, as in \eqref{real_matrix} is {\it symplectic} if    
\begin{equation}\label{symp-op}
{\cal W} ( \cR (\vphi) u, \cR (\vphi)  v) =
{\cal W} (u, v) \,  \quad \forall u,v \in L^2 (\T_x, \R^2)\, , 
\end{equation}
where the symplectic 2-form $ {\cal W} $ is defined in \eqref{sympl-form-st}. 

\paragraph{Reversible and reversibility preserving operators.}

Let $\cS$ be an involution  as in \eqref{rev_invo} acting on the real variables
$ (\eta, \zeta) \in \R^2 $, 
or 
 as in \eqref{rev_aa} acting on the action-angle-normal variables $ (\theta, I, w ) $, or as in
\eqref{inv-complex} acting in the $ (z, \bar z )$ complex variables introduced in \eqref{C_transform}.  
\begin{defn}{\bf (Reversibility)}\label{rev_defn}
	A $ \vf $-dependent family of operators $\cR (\vf) $, $ \vf \in \T^\nu $,
	 is
	\begin{itemize}
		\item \emph{reversible} if $\cR(-\vf) \circ\cS = -\cS \circ \cR(\vf)$ for all $\vf\in\T^\nu$;
		\item \emph{reversibility preserving} if $\cR(-\vf)\circ \cS = \cS \circ \cR(\vf)$ for all $\vf\in\T^\nu$. 
	\end{itemize}
\end{defn}

Since in the complex coordinates $(z,\bar z) $  
the involution $\cS$ defined in \eqref{rev_invo} reads as in 
\eqref{inv-complex}, 
an operator $\bR (\vf)$  as in \eqref{C_transformed} is 
reversible, respectively anti-reversible,  if, for any $i=1,2$,
\begin{equation}\label{Ri-RAR}
\cR_{i} (- \vf) \circ \cS = - \cS \circ \cR_{i} (\vf)  \, , \quad {\rm resp.} \  \
\cR_{i} (- \vf) \circ \cS =  \cS \circ \cR_{i} (\vf)  \, , 
\end{equation}
 where, with a small abuse of notation, we still denote $ (\cS u)(x) = \overline{u(-x)}$. 
Moreover, recalling that in the Fourier coordinates such involution reads as in \eqref{inv-zj}, 
we obtain the following lemma.

\begin{lem}\label{rev_defn_C}
A $ \vf $-dependent family of operators $\bR (\vf)$, $ \vf \in \T^\nu $,  as in \eqref{C_transformed} is 
	\begin{itemize}
		\item  reversible if, for any $ i = 1, 2 $, 
		\begin{equation}\label{rev:Fourier}
			\left( \cR_{i} \right)_j^{j'}(-\vf) = - \bar{ \left( \cR_{i} \right)_{j}^{j'}(\vf) } \quad \forall\,\vf\in\T^\nu \, , \ \ i.e. \ \left( \cR_{i} \right)_j^{j'}(\ell) = - 
			\bar{ \left( \cR_{i} \right)_{j}^{j'}(\ell) } \,  \quad \forall\,\ell\in\Z^\nu \,;
		\end{equation}
		\item reversibility preserving if, for any $ i = 1, 2 $, 
		\begin{equation}
		\left( \cR_{i} \right)_j^{j'}(-\vf) = \bar{ \left( \cR_{i} \right)_{j}^{j'}(\vf) } \ \  \forall\,\vf\in\T^\nu   \, , \ \ i.e.   \ \left( \cR_{i} \right)_j^{j'}(\ell) = \bar{ \left( \cR_{i} \right)_{j}^{j'}(\ell) } 
		\,  \ \  \forall\,\ell\in\Z^\nu \,.
		\end{equation}
	\end{itemize}
\end{lem}

Note that the composition of a reversible operator with a reversibility preserving operator is reversible.
The flow 
generated by a reversibility preserving operator is reversibility preserving. 
If $ \cR (\vf) $ is reversibility preserving, then $ (\omega \cdot \pa_\vf \cR) (\vf) $
is reversible.

We shall say that a linear operator of the form $  \omega\cdot \pa_\vf + A(\vf)$ is reversible if $A(\vf)$ is reversible.  Conjugating the linear operator $
 \omega\cdot\pa_\vf+A(\vf)$ by a family of invertible linear maps $\Phi(\vf)$, we get the transformed operator
\begin{equation}\label{trasf-op}
\begin{aligned}
	&  \Phi^{-1}(\vf) \circ \big(  \omega\cdot \pa_\vf + A(\vf) \big) \circ  \Phi(\vf) = \omega\cdot\pa_\vf + A_+(\vf)\,, \\
	&   A_+(\vf)  := \Phi^{-1}(\vf)\left( \omega\cdot \pa_\vf\Phi(\vf) \right) + \Phi^{-1}(\vf) A(\vf) \Phi(\vf)\,.
\end{aligned}
\end{equation}
The conjugation of a reversible operator with a reversibility preserving operator  is reversible. 

\begin{lem}\label{lem:rev-pse}
	A pseudodifferential operator 
	$  \Op(a(\vf, x, \xi))$ is reversible, respectively reversibility 
	preserving, if and only if its symbol
	satisfies
\begin{equation}\label{pseudo-rev}
	a(- \vf, - x, \xi) = - \overline{a(\vf, x, \xi)} \, , 
	\quad \text{resp.}  \quad 
	a(- \vf, - x, \xi) =  \overline{a(\vf, x, \xi)} \, . 
\end{equation}
\end{lem}
\begin{proof}
If the symbols $ a $ satisfies \eqref{pseudo-rev}, then, 
recalling the complex form of the involution $ \cS $ in \eqref{inv-complex}-\eqref{inv-zj},
we deduce that $  \Op(a(\vf, x, \xi)) $ is reversible, respectively anti-reversible. 
 The vice versa follows using that $ a(\vf, x, j) = e^{- \im j x } \Op(a(\vf, x, \xi)) [ e^{ \im j x } ]$. 
\end{proof}

\begin{rem}\label{rem:REV}
Let $ A(\vf) = R(\vf) + T(\vf)$ be a reversible operator. Then
$ A(\vf) = R_+(\vf) + T_+(\vf)  $ 
where both operators 
$$
R_+(\vf)  := \tfrac{1}{2} ( R(\vf) - \cS R(-\vf) \cS) \, , \quad
T_+(\vf)  := \tfrac{1}{2} ( T(\vf) - \cS T(- \vf) \cS) \, , 
$$
are reversible. If $ R(\vphi) = {\rm Op}(r (\vf,x, \xi))$ is pseudodifferential, then
$$
R_+ (\vphi) = {\rm Op}( r_+ (\vf,x, \xi)) \, , \quad 
r_+ (\vf,x, \xi) := \tfrac12 ( r (\vf,x, \xi) - \overline{r (-\vf,-x, \xi)} )
$$
and the pseudodifferential norms of $ {\rm Op}(r)$ and $ {\rm Op}(r_+) $
are equivalent. If  $ T(\vf)$ is a tame operator with a tame constant $ \fM_T (s) $,
then $ T_+ (\vf) $   is a  tame operator as well with an equivalent 
tame constant. 
\end{rem}

\begin{defn} \label{defRAR} {\bf (Reversible and anti-reversible function)}
A function $ u (\vf, \cdot) $ is called 
$$
{\it Reversible} \ \ {\rm if} \ \ \cS u (\vf, \cdot ) = u(-\vf, \cdot ) \  ({\rm cfr.} 
\eqref{rev:soluz});  \quad
{\it Anti-reversible} \  \ {\rm if} 	\  - \cS u (\vf, \cdot ) = u(-\vf, \cdot) \, . 
$$
The same definition holds in the action-angle-normal variables $ (\theta, I, w ) $
 with the involution  $ \vec \cS $ defined in \eqref{rev_aa} and in the $ (z, \bar z )$ complex variables with the involution in \eqref{inv-complex}.  
\end{defn}

A reversibility preserving operator maps reversible, respectively anti-reversible, functions into 
reversible, respectively anti-reversible, functions. 

\begin{lem}\label{lem:REV}
Let $ X $ be a reversible vector field, according to \eqref{revNL},
and $  u(\vf, x)  $ be a  reversible quasi-periodic function.
	Then the linearized operator  $ \di_u X( u(\vf, \cdot) ) $  is reversible, according to
	Definition \ref{rev_defn}.
\end{lem}

\begin{proof}
Differentiating  \eqref{revNL}  we get  
$ (\di_u X)( \cS u) \circ \cS = -  \cS (\di_u X)(u) $ and use $ \cS u ( \vf, \cdot) = u (- \vf, \cdot ) $. 
\end{proof}

Finally we note the following lemma.  

\begin{lem}\label{proj_rev}
The projections $\Pi^\intercal_{\S^+, \Sigma}$, $\Pi^\angle_{\S^+, \Sigma}$ 
defined in Section \ref{sec:decomp} 
commute with the involution 
$ \cS $ defined in \eqref{rev_invo}, i.e. 
 are reversibility  preserving.
 The orthogonal projectors $\Pi_{\S}$ and $\Pi_{\S_0}^\bot $ 
 commute with the involution in \eqref{inv-complex}, i.e.  
 are reversibility  preserving. 
\end{lem}

\begin{proof}
The involution
$ \cS $ defined in \eqref{rev_invo}  maps $ V_{n,\pm}$ into itself, acting as in \eqref{ab-rev}. 
Then, by the decomposition \eqref{deco-real},  
each projector $ \Pi_{V_{n,\sigma}} $ commutes with $ \cS $. 
\end{proof}

\subsection{Momentum preserving operators}\label{subsec:momentum}

The following definition is crucial in the construction of traveling waves. 

\begin{defn}	\label{def:mom.pres}
{\bf (Momentum preserving)}	
	A  $ \vf $-dependent family of linear operators 
	$A(\vf) $, $ \vf \in \T^\nu $,  is  {\em momentum preserving} if
	\begin{equation}\label{eq:mp_A_tw}
	A(\vf - \ora \jmath \vs )  \circ \tau_\vs = \tau_\vs \circ A(\vf ) \,  , \quad 
	\forall \,\vf \in \T^\nu \, , \ \vs \in \R \,  ,
	\end{equation}
	where the translation operator $\tau_\vs$ is defined in \eqref{trans}.
	A linear matrix operator $\bA(\vf ) $ of the form \eqref{real_matrix} or \eqref{C_transformed} is  
	 {\em momentum preserving} if each of its components is momentum preserving.
\end{defn}

Momentum preserving operators are closed under several operations.
\begin{lem}\label{lem:mom_prop}
	Let $A(\vf), B(\vf)$ be  momentum preserving operators. Then:
	\begin{itemize}
		\item[(i)] {\rm (Composition)}: $A (\vf) \circ B (\vf) $ is a momentum preserving operator. 
		\item[(ii)] {\rm (Adjoint)}: the adjoint $ (A(\vf))^*$ is momentum preserving.
		\item[(iii)] {\rm (Inversion)}: If $A(\vf)$ is invertible then $A(\vf)^{-1}$ is momentum preserving.
		\item[(iv)] {\rm (Flow)}: Assume that 
		\begin{equation}\label{same-CP}
		\partial_{t} \Phi^t (\vf) = A (\vf) \Phi^t (\vf) \, , \quad 
		\Phi^0 (\vf) = {\rm Id} \, , 
	\end{equation}
		has a unique  propagator $\Phi^t (\vf) $ for any $ t\in[0,1] $. 
		Then $\Phi^t ( \vf ) $ is  momentum preserving.
	\end{itemize}
\end{lem}
\begin{proof}
 Item $(i)$ follows directly by \eqref{eq:mp_A_tw}. Item $(ii)$, respectively
 $ (iii)$, follows by taking the adjoint, respectively the inverse,
 of \eqref{eq:mp_A_tw} and using that
 $ \tau_\vs^* = \tau_{-\vs} = \tau_\vs^{-1} $.  
	Finally,  item $(iv)$ holds because  
	$ \tau_\vs^{-1} \Phi^t ( \vf - \vec \jmath \vs ) \tau_\vs $ solves the same Cauchy  in 
	\eqref{same-CP}. 
\end{proof}

We shall say that a linear operator of the form $  \omega\cdot \pa_\vf + A(\vf)$ is momentum preserving if $A(\vf)$ is momentum preserving.  
In particular, conjugating a momentum preserving operator 
$ \omega\cdot\pa_\vf+A(\vf) $ 
by a family of invertible linear momentum preserving maps $\Phi(\vf)$, 
we obtain  the transformed operator 
$ \omega\cdot\pa_\vf + A_+(\vf) $ in
\eqref{trasf-op} which  is momentum preserving.

\begin{lem}
\label{A.mom.cons}
Let $A(\vf)$ be a momentum preserving linear operator
 and  $u$ a quasi-periodic traveling wave, according to Definition \ref{QPTW}.
 Then $A(\vf) u $ is a quasi-periodic traveling wave.
\end{lem}

\begin{proof}
It follows by Definition \ref{def:mom.pres} and  by the characterization of traveling waves in \eqref{chartj1}.
\end{proof}

\begin{lem}\label{lem:MP}
Let $ X $ be a vector field translation invariant, according to \eqref{eq:mom_pres}. 
Let $ u   $ be a quasi-periodic traveling wave. 
	Then the linearized operator  $ \di_u X( u(\vf, \cdot) ) $  is momentum preserving.
\end{lem}

\begin{proof}
Differentiating  \eqref{eq:mom_pres}  we get  
$ (\di_u X)(\tau_\vs u) \circ \tau_\vs = \tau_\vs (\di_u X)(u) $, 
$ \vs \in \R $. 
Then, apply \eqref{chartj1}. 
\end{proof}

We now provide a characterization of 
the momentum preserving property in Fourier space.

\begin{lem}
	\label{lem:mom_pres}
	Let $ \vf $-dependent family of operators $ A(\vf) $, $ \vf \in \T^\nu $, 
	is momentum preserving  if and only if 
	the matrix elements of $A(\vf)$, defined by \eqref{A_expans},  fulfill
		\begin{equation}\label{momentum}
		A_j^{j'}(\ell ) \neq 0 \quad \Rightarrow  \quad \ora{\jmath}\cdot \ell + j-j' = 0 \, , \quad
		\forall\, \ell \in\Z^\nu , \ \ j,j'\in\Z \, . 
		\end{equation}
\end{lem}
\begin{proof}
By \eqref{A_expans} we have, for any function $ u (x)$, 
$$
\tau_\vs ( A (\vf) u )
 = \sum_{j,j'\in\Z}\sum_{\ell\in\Z^\nu} 
 A_j^{j'}(\ell) e^{\im j \vs} u_{j'} e^{\im (\ell\cdot \vf + jx )} 
$$
and 
$$
A (\vf - \vec \jmath \vs)  [\tau_\vs u] 
 = \sum_{j,j'\in\Z}\sum_{\ell \in\Z^\nu} 
 A_j^{j'}(\ell) e^{- \im \ell \cdot \vec \jmath \vs } e^{\im j' \vs} u_{j'} 
 e^{\im (\ell\cdot \vf + jx )}  \, .
$$
Therefore \eqref{eq:mp_A_tw} is equivalent to \eqref{momentum}. 
\end{proof}

We  characterize the symbol of a pseudodifferential operator 
which is  momentum preserving:
\begin{lem}\label{lem:mom_pseudo}
	A pseudodifferential operator 
	$A(\vf, x,D)= \Op(a(\vf, x, \xi))$ is momentum preserving if and only if its symbol
	satisfies
\begin{equation}\label{pseudo-MM}
	a(\vf - \ora{\jmath}\vs, x, \xi) =  a(\vf, x+\vs, \xi) \, , \quad \forall\, \vs \in \R \, . 
\end{equation}
\end{lem}
\begin{proof}
If the symbol $ a $ satisfies \eqref{pseudo-MM}, then, for all $ \vs \in \R  $, 
$$
\tau_\vs \circ \Op(a(\vf, x, \xi)) = \Op( a(\vf, x+\vs, \xi)  )\circ  \tau_\vs = 
 \Op( a(\vf - \ora{\jmath}\vs, x, \xi ) )\circ  \tau_\vs \, ,
 $$
 proving that $ \tau_\vs \circ A(\vf, x,D) = A(\vf - \ora{\jmath}\vs , x,D) \circ \tau_\vs $. 
 The vice versa follows using that $ a(\vf, x, \xi) = e^{- \im \xi x } A(\vf, x,D) [ e^{ \im \xi x } ]$. 
\end{proof}

Note that, if a symbol $ a(\vf, x, \xi) $ satisfies \eqref{pseudo-MM}, then 
$ (\omega \cdot \pa_\vf a)(\vf, x, \xi) $
satisfies \eqref{pseudo-MM} as well. 

\begin{lem}\label{lem:dMP}
If $ \beta (\vf, x ) $ is a quasi-periodic traveling wave,  then 
the operator $ \cB (\vf) $ defined in \eqref{defB-diffeo} is momentum preserving.
\end{lem}

\begin{proof}
We have 
$
\cB (\vf - \vec \jmath \vs) [\tau_\vs  u] = 
u(x+ \beta (\vf - \vec \jmath \vs,x) + \vs) = u(x+ \vs + \beta (\vf,x+ \vs)) =
\tau_\vs \big( \cB (\vf )   u\big)$. 
\end{proof}

 We also note the following lemma. 

\begin{lem}\label{lem:proj.momentum}
The symplectic projections $ \Pi^\intercal_{\S^+, \Sigma}$, $ \Pi^\angle_{\S^+, \Sigma}$,
 the $ L^2 $-projections  $ \Pi^{L^2}_\angle  $
and  $\Pi_{\S}$, $\Pi_{\S_0}^\bot $ 
defined in Section \ref{sec:decomp} 
commute with the translation operators 
$ \tau_\vs $ defined in \eqref{trans}, i.e. are momentum preserving.
\end{lem}

\begin{proof}
Recall  that the translation 
$ \tau_\vs $  maps $ V_{n,\pm}$ into itself, acting as in \eqref{ab-tras}. 
Consider the $ L^2 $-orthogonal decomposition 
$ \acca = \acca_\angle  \oplus \acca_\angle^\bot  $, setting
$  \acca_\angle  := \acca_{\S^+,\Sigma}^\angle $ for brevity: 
$$
u = \Pi_{\acca_\angle}^{L^2} u + \Pi_{\acca_\angle^\bot}^{L^2} u \, , \quad 
\Pi_{\acca_\angle}^{L^2} u \in  \acca_\angle \, , \quad \Pi_{\acca_\angle^\bot}^{L^2} u 
\in  \acca_\angle^\bot  \, .
$$
Applying $\tau_\vs $ we get 
$ \tau_\vs u = \tau_\vs  \Pi_{\acca_\angle}^{L^2} u + \tau_\vs  \Pi_{\acca_\angle^\bot}^{L^2} u $. 
As shown above, 
$ \tau_\vs $  maps $  \acca_\angle $ into itself for all $ \vs $. Thus also the $ L^2 $-orthogonal subspace $  \acca_\angle^\bot  $ is  invariant under the action of $ \tau_\vs $ 
and we conclude, by the uniqueness of the orthogonal decomposition, 
 that 
$ \tau_\vs  \Pi_{\acca_\angle}^{L^2} u = \Pi_{\acca_\angle}^{L^2}  \tau_\vs  u $, 
$ \tau_\vs  \Pi_{\acca_\angle^\bot}^{L^2} u = \Pi_{\acca_\angle^\bot}^{L^2}  \tau_\vs  u $.  
\end{proof}

The next lemma concerns 
the Dirichlet-Neumann operator. 
\begin{lem}\label{mom_dirichlet}
 The Dirichlet-Neumann operator  $G(\bar\eta,\tth)$, evaluated at 
 a quasi-periodic  traveling wave  $\bar \eta  (\vf, x)  $,  is momentum preserving.
\end{lem}
\begin{proof}
It follows by \eqref{DN:trans} and the characterization in \eqref{chartj1}
of the quasi-periodic traveling wave $\bar \eta(\vf,x) $.
\end{proof}

\paragraph{Quasi-periodic traveling waves in action-angle-normal coordinates.}

We now discuss  how the momentum preserving condition reads in the 
coordinates $(\theta, I, w)$ introduced in \eqref{aacoordinates}. 
Recalling \eqref{vec.tau},
if $u(\vf,x)$ is a quasi-periodic traveling wave with action-angle-normal components $(\theta(\vf), I(\vf), w(\vf, x))$, the condition $\tau_\vs u =  u(\vf - \ora{\jmath} \vs, \cdot)$ becomes 
\begin{equation}
\label{mompres_aa}
\begin{pmatrix}
\theta(\vf) - \ora{\jmath} \vs \\
I(\vf) \\
\tau_\vs w(\vf, \cdot)
\end{pmatrix}  = 
\begin{pmatrix}
\theta(\vf -  \ora{\jmath} \vs) \\
I(\vf- \ora{\jmath} \vs) \\
 w(\vf - \ora{\jmath} \vs, \cdot)
\end{pmatrix} \, , \quad \forall\, \vs \in \R \, . 
\end{equation}
As we look for $\theta(\vf)$ of the form $\theta(\vf) = \vf + \Theta(\vf)$, with 
a $ (2 \pi)^\nu $-periodic function $ \Theta : \R^\nu \mapsto \R^\nu $, 
$ \vf \mapsto \Theta (\vf ) $, the traveling wave condition becomes
\begin{equation}
\label{mompres_aa1}
\begin{pmatrix}
\Theta(\vf)  \\
I(\vf) \\
\tau_\vs w(\vf, \cdot)
\end{pmatrix}  = 
\begin{pmatrix}
\Theta(\vf -  \ora{\jmath} \vs) \\
I(\vf- \ora{\jmath} \vs) \\
 w(\vf - \ora{\jmath} \vs, \cdot)
\end{pmatrix} \, , \quad \forall\, \vs \in \R \, . 
\end{equation}

\begin{defn} {\bf(Traveling wave variation)}\label{trav-vari}
We call a traveling wave variation 
$ g(\vf) = (g_1(\vf), g_2(\vf), g_3(\vf, \cdot)) \in 
\R^\nu \times \R^\nu \times \acca_{\S^+,\Sigma}^\angle $ a function satisfying \eqref{mompres_aa1}, i.e.
$$
g_1(\vf) = g_1(\vf - \ora{\jmath}\vs) , 
	\quad
		g_2(\vf) = g_2(\vf - \ora{\jmath}\vs) , 
		\quad
			\tau_\vs g_3(\vf) = g_3(\vf - \ora{\jmath}\vs) , 
			\  \forall\, \vs \in \R \, , 
$$
or equivalently  $D \vec \tau_\vs g(\vf) = g(\vf - \ora{\jmath}\vs) $ for any  $\vs \in \R $, 
where $D \vec \tau_\vs$ is the differential of $ \vec \tau_\vs$, namely
$$
D\vec \tau_\vs
\begin{pmatrix}
\Theta \\
I \\
w
\end{pmatrix} 
=
\begin{pmatrix}
\Theta  \\
I \\
\tau_\vs w
\end{pmatrix} \, , \quad \forall\, \vs \in \R \, . 
$$
\end{defn}

According to  Definition \ref{def:mom.pres}, 
a linear operator acting in $ \R^\nu \times \R^\nu \times \acca_{\S^+,\Sigma}^\angle$ momentum preserving if 
\begin{equation}\label{eq:mp_A_tw'}
	A(\vf - \ora \jmath \vs)\circ D\vec \tau_\vs  = D\vec \tau_\vs \circ A(\vf) 
	\, , \quad \forall\, \vs \in \R \, . 
	\end{equation}
Similarly to Lemma \ref{A.mom.cons}, one proves the following result:
\begin{lem}
\label{A.mom.cons'}
Let $A(\vf)$ be a momentum preserving linear operator acting on $\R^\nu \times \R^\nu \times \acca_{\S^+,\Sigma}^\angle$ and 
$g \in \R^\nu \times \R^\nu \times \acca_{\S^+,\Sigma}^\angle$ be  a traveling wave variation. Then  $ A(\vf) g(\vf)$ is a traveling wave variation.
\end{lem}

\section{Transversality of linear frequencies}\label{sec:deg_kam}

In this section we extend the  KAM theory approach of \cite{BM}, \cite{BaBM}
in order to deal with the linear frequencies $ \Omega_j(\kappa) $ defined in \eqref{def:Omegajk}. 
The main novelty  is the use of the momentum condition in the proof of 
Proposition \ref{prop:trans_un}.
We shall also exploit that the  tangential sites 
$ \S:=\Set{ \bar{\jmath}_1, \ldots ,\bar{\jmath}_\nu } \subset \Z\setminus\{0\}$  
 defined  in \eqref{def.S}, have  all distinct 
 modulus $ | \bar{\jmath}_a | = \bar n_a $, see assumption \eqref{Splus}.

We first introduce the following definition.

\begin{defn}\label{def:non-deg}
	A function $f=(f_1,\dots,f_N):[\kappa_1,\kappa_2]\rightarrow\R^N$ is 
	 \emph{non-degenerate} if, for any $c\in \R^N\setminus\{0\}$, the scalar function $f\cdot c$ is not identically zero on the whole interval $[\kappa_1,\kappa_2]$.
\end{defn}

From a geometric point of view, if $ f $ is non-degenerate it means that the image of the curve $ f([\kappa_1,\kappa_2]) \subset \R^N $ is not contained in any  hyperplane of $ \R^N $. 

\smallskip

We shall use in the sequel 
that the maps $  \kappa \mapsto \Omega_j (\kappa) $ are analytic in $ [\kappa_1, \kappa_2] $. 
We decompose  
\begin{equation}\label{Om-om}
	\Omega_j(\kappa) = \omega_j(\kappa) + \frac{\gamma}{2}\frac{G_j(0)}{j}\,, \quad \omega_j ( \kappa) := \sqrt{ \kappa \, G_j(0)j^2 + g \,G_j(0) + \left( \frac{\gamma}{2}\frac{G_j(0)}{j} \right)^2 } \, . 
\end{equation}
Note  that the dependence on $ \kappa $ of $ \Omega_j (\kappa) $ enters only through 
$ \omega_j(\kappa) $, because $ \frac{G_j(0)}{j}$ is independent of $ \kappa $. 
 Note also that $ j \mapsto 
\omega_j (\kappa) $ is even in $ j $, whereas the component   due to the vorticity 
$ j \mapsto \gamma \frac{G_j(0)}{j} $ is odd.  
Moreover this term is, in view of \eqref{def:Gj0}, uniformly bounded in $ j $.

\begin{lem}\label{lem:non_deg_vectors}
{\bf (Non-degeneracy-I)}
	The following frequency vectors are non-degenerate:
	\begin{enumerate}
		\item $\ora{\Omega}(\kappa) := ( \Omega_j (\kappa) )_{j \in \S} \in\R^\nu$;
		\item $\big( \ora{\Omega}(\kappa),\sqrt\kappa \big)\in\R^{\nu+1}$;
		\item $\big( \ora{\Omega}(\kappa),\Omega_j(\kappa) \big) \in\R^{\nu+1}$, for any  $j\in\Z\setminus\left(\{ 0 \}\cup\S\cup (-\S)\right)$;
		\item $\big( \ora{\Omega}(\kappa),\Omega_j(\kappa) ,\Omega_{j'}(\kappa) \big)\in\R^{\nu+2}$, for any  $j, j'\in\Z\setminus\left( \{0\}\cup\S\cup(-\S) \right)$ and  $|j| \neq |j'|$.
	\end{enumerate}
\end{lem}
\begin{proof}
Let
\begin{equation}\label{Omega-tilde}
		\wt\Omega_j(\kappa) := \begin{cases}
		\Omega_j(\kappa) & \text{ for }j \neq 0 \\
		\sqrt{\kappa} & \text{ for } j=0 \, , 
		\end{cases} \qquad \qquad \wt\omega_j(\kappa) := \begin{cases}
		\omega_j(\kappa) & \text{ for }j \neq 0 \\
		\sqrt{\kappa} & \text{ for } j=0 \, . 
		\end{cases} 
\end{equation}
Recalling  \eqref{Om-om}, we have that, for any $ j \in \Z $, 
\begin{equation}\label{def:la}
	\partial_\kappa\wt\omega_j(\kappa) = \lambda_j(\kappa) \wt\omega_j(\kappa) \,, \quad \lambda_j(\kappa) := \begin{cases}
	\frac{G_j(0)j^2}{2\left( \kappa \,G_j(0) j^2 + g\, G_j(0)+ \left(\frac{\gamma}{2}\frac{G_j(0)}{j}\right)^2 \right)} & \text{ for } j \neq 0 \\
	\frac{1}{2\kappa} & \text{ for } j = 0 \, . 
	\end{cases} 
\end{equation}
	Moreover $\partial_\kappa\lambda_j(\kappa)= -2 \lambda_j(\kappa)^2 $, for any $j\in\Z$, and therefore, for any $n\in\N $, 
	\begin{equation}\label{eq:it_der_omega}
	\partial_\kappa^n \wt\omega_j(\kappa) = \wtc_n \lambda_j(\kappa)^n \wt\omega_j(\kappa) \,, \quad \wtc_n:= c_1 \cdot \ldots \cdot c_n \,, \quad c_n:=3-2n \,.
	\end{equation}
	We  now prove  items 2 and 3, i.e. 
	the non-degeneracy of the vector $\big( \ora{\Omega}(\kappa),\wt\Omega_j(\kappa) \big) \in\R^{\nu+1}$ for any $j \in \Z \setminus ( \S \cup (-\S)) $, where $\wt\Omega_j(\kappa) $ is defined in \eqref{Omega-tilde}. Items 1 and 4 follow similarly. 
 For this purpose, by  analyticity, it is  sufficient to find one value of $\kappa \in [\kappa_1, \kappa_2]$ so that  the determinant of the $(\nu+1)\times(\nu+1)$ matrix 
	$$
	\cA(\kappa):= \begin{pmatrix}
	\partial_\kappa \Omega_{\bar{\jmath}_1}(\kappa) & \cdots & \partial_\kappa\Omega_{\bar{\jmath}_\nu}(\kappa) & \partial_\kappa \wt\Omega_j(\kappa) \\
	\vdots & \ddots & \vdots & \vdots \\
	\partial_\kappa^{\nu+1} \Omega_{\bar{\jmath}_1}(\kappa) & \cdots & \partial_\kappa^{\nu+1}\Omega_{\bar{\jmath}_\nu}(\kappa) & \partial_\kappa^{\nu+1}\wt\Omega_j(\kappa)
	\end{pmatrix}
	$$
	is not zero. We actually show that $\det \cA(\kappa) \neq 0$ for any $\kappa \in [\kappa_1, \kappa_2]$.
	By \eqref{Omega-tilde}-\eqref{eq:it_der_omega} and  the multilinearity of the determinant function, we get
	$$
	\det \cA(\kappa) = C(\kappa) \det\begin{pmatrix}
	1 & \cdot & 1 & 1\\
	\lambda_{\bar{\jmath}_1}(\kappa) & \cdot & \lambda_{\bar{\jmath}_\nu}(\kappa) & \lambda_j(\kappa) \\
	\vdots & \ddots & \vdots & \vdots \\
	\lambda_{\bar{\jmath}_1}(\kappa)^{\nu} & \cdot & \lambda_{\bar{\jmath}_\nu}(\kappa)^{\nu} & \lambda_j(\kappa)^{\nu}
	\end{pmatrix}=: C(\kappa) \det\cB(\kappa)
	$$
	where 
	$$
	C(\kappa) := \prod_{q=1}^{\nu+1}\wtc_q \cdot\prod_{p\in\{\bar{\jmath}_1, \ldots ,\bar{\jmath}_\nu,j\}} \lambda_p(\kappa)\wt\omega_p(\kappa) \neq 0 	\, , \quad \forall\,\kappa\in[\kappa_1,\kappa_2] \,.
	$$
Since  $\cB(\kappa)$ is a Vandermorde matrix,  we conclude that
	$$
	\det\cA(\kappa) = C(\kappa) \prod_{p,p'\in\{\bar{\jmath}_1, \ldots ,\bar{\jmath}_\nu,j\},  p< p'}
	\! \! \! \!\! \! \big( \lambda_p(\kappa) -\lambda_{p'}(\kappa)\big) \,.
	$$
	Now, the fact that $\det\cA(\kappa)\neq 0$ for any $\kappa\in[\kappa_1,\kappa_2]$
	is a consequence from the following
	\\[1mm]
	{\bf Claim}: {\it For any $p,p'\in\{\bar{\jmath}_1,\ldots,\bar{\jmath}_\nu,j\}$, 
	$p \neq  p'$, one has  $\lambda_p(\kappa)\neq \lambda_{p'}(\kappa)$ for any $\kappa\in[\kappa_1,\kappa_2]$}.
		\\[1mm]
	{\sc Proof of the Claim}: If  $p' = 0$ and $p \neq 0$, 
	the claim follows because, by \eqref{def:la}, 
$$
\lambda_p(\kappa) = \frac{1}{2\Big( \kappa + \frac{g}{p^2} + 
\frac{\gamma^2}{4} \frac{G_p(0)}{p^4} \Big)}  <
 \frac{1}{2\kappa} = \lambda_0(\kappa) \, .
$$
Consider now the case $p, p' \neq 0 $. 
We now prove that the map $ p \mapsto \lambda_p(\kappa) $
 is strictly monotone on $(0, + \infty)$. 
 In case of finite depth,   $G_p(0) =p\tanh(\tth p) $,  and 
\begin{align*}
		\pa_p\lambda_p(\kappa) &= 
		 \frac{1}{2\left( \kappa + \frac{g}{p^2}+ 
		\frac{\gamma^2}{4} \frac{\tanh(\tth p)}{p^3}\right)^2}\left\{ \frac{2 g}{p^3} + 
	\frac{\gamma^2}{4} \, \frac{3\tanh(\tth p)-(1-\tanh^2(\tth p))\tth p}{p^4} \right\}\,.\label{der_lambda}
	\end{align*}
	The function $f(y):= 3 \tanh(y)-(1-\tanh^2(y))y $ is positive for any $y>0$. Indeed
	$f(y)\rightarrow 0$ as $y\rightarrow 0$, 
	and it is strictly monotone increasing for $ y > 0 $, since	
	$ f'(y) = 2(1-\tanh^2(y)) ( 1+ y \tanh(y) ) > 0 $.
We deduce that $\pa_p\lambda_p(\kappa) > 0 $,  also 
if the depth $ \tth = + \infty $. 
Since the function  $ p \mapsto \lambda_p(\kappa)$ is  even we have proved that 
that  it  is strictly monotone decreasing on  $(-\infty,0)$ and increasing in $(0, +\infty)$. Thus, if $\lambda_p(\kappa) = \lambda_{p'}(\kappa)$ then $p = -p'$. 
	But this case is  excluded by the assumption 
    	\eqref{Splus} and the condition  $ j \not\in \S \cup (-\S) $, which together imply 
$|p| \neq |p'|$.
\end{proof}

Note that in items 3 and 4 of Lemma \ref{lem:non_deg_vectors} we require
that $ j $ and $ j' $ do not belong to 
$   \{0\} \cup {\mathbb S} \cup (-{\mathbb S}) $. In order to
deal  in Proposition \ref{prop:trans_un} when $ j $ and $  j' $  are 
in $  {\mathbb S} \cup (-{\mathbb S}) $, 
we need also the following lemma. It is actually a direct 
consequence of the proof of
Lemma \ref{lem:non_deg_vectors},  noting that
 $\Omega_j (\kappa ) - \omega_j (\kappa)$ is independent of 
$ \kappa $. 

\begin{lem}
\label{rem:omega-nondeg}
{\bf (Non-degeneracy-II)}
Let $\ora{\omega}(\kappa):= \left( \omega_{\bar{\jmath}_1}(\kappa),\ldots,\omega_{\bar{\jmath}_\nu}(\kappa) \right)$. The following  vectors are non-degenerate:
\begin{enumerate}
		\item $(\ora{\omega}(\kappa),1)\in\R^{\nu+1}$; 
		\item $\left( \ora{\omega}(\kappa),\omega_j(\kappa),1 \right)\in\R^{\nu+2}$, for any  $j\in\Z\setminus\left(\{ 0 \}\cup\S\cup (-\S)\right)$.
	\end{enumerate}
\end{lem}

\noindent
For later use, we provide the following asymptotic  estimate of the linear frequencies. 

\begin{lem}\label{rem:exp_omegaj_k}
{\bf (Asymptotics)}
For any $j\in\Z\setminus\{0\}$, we have 
	\begin{equation}\label{eq:rem1}
	\omega_j(\kappa)= \sqrt\kappa \abs j^\frac32 + \frac{c_j(\kappa)}{\sqrt\kappa\abs j^\frac12} \,,  
	\end{equation}
	where,  
	for any $ n \in \N_0 $,  	
	there exists a constant $C_{n,\tth} >0$ such that 
	\begin{equation}\label{eq:rem2}
\sup_{j\in \Z\setminus\{0\} \atop \kappa\in[\kappa_1,\kappa_2]}	\Big| \partial_\kappa^n \frac{c_j(\kappa)}{\sqrt\kappa} \Big| \leq C_{n,\tth} \, .
	\end{equation}	
	\end{lem}
	\begin{proof}
By \eqref{Om-om} we deduce \eqref{eq:rem1} with 
$$
c_j(\kappa):=
\tfrac{\kappa \abs j\left( G_j(0)-\abs j \right) + \frac{g\,G_j(0)}{\abs j}\left( 1+ \left(\frac{\gamma}{2}\right)^2\frac{G_j(0)}{g\abs j^2} \right)}{1 + \sqrt{ 1+ \frac{G_j(0)-\abs j}{\abs j} + \frac{g\,G_j(0)}{\kappa\abs j^3}\left( 1+ \left(\frac{\gamma}{2}\right)^2\frac{G_j(0)}{g\abs j^2} \right)  }} \,.
$$
Then  \eqref{eq:rem2} follows	exploiting that  (both for finite and infinite depth)
the quantities $|j| (G_j(0) - |j|)$ and $G_j(0)/|j|$ are uniformly bounded in $j$, see 
\eqref{def:Gj0}. 
\end{proof}

The next proposition is the key of the argument. 
We remind that 
$\ora{\jmath} = ( \bar{\jmath}_1, \ldots ,\bar{\jmath}_\nu ) $ denotes the vector in $  \Z^\nu $ 
of tangential sites  introduced in \eqref{def:vecj}.

\begin{prop} {\bf (Transversality)} \label{prop:trans_un}
There exist $m_0\in\N$ and $\rho_0>0$ such that, for any $\kappa\in[\kappa_1,\kappa_2]$, the following hold:
	\begin{align}
	&\max_{0\leq n \leq m_0}
	| \partial_\kappa^n \ora{\Omega}(\kappa)\cdot \ell | \geq \rho_0\braket{\ell} \,, \quad \forall\,\ell\in\Z^\nu\setminus\{0\} \,; \label{eq:0_meln}\\
	&\begin{cases}
	\max\limits_{0\leq n \leq m_0}|\partial_\kappa^n\,( \ora{\Omega}(\kappa)\cdot \ell + \Omega_j(\kappa) ) | \geq \rho_0 	\braket{\ell} \\
	\ora{\jmath}\cdot \ell + j = 0 \,, \quad \ell\in\Z^\nu\,, \ j\in\S_0^c \, ; 
	\end{cases} \label{eq:1_meln}\\
	&\begin{cases}
	\max\limits_{0\leq n \leq m_0}| \partial_\kappa^n\,( \ora{\Omega}(\kappa)\cdot \ell + \Omega_j(\kappa)-\Omega_{j'}(\kappa) ) | \geq \rho_0 	\braket{\ell} \\
	\ora{\jmath}\cdot \ell + j -j'= 0 \,, \quad \ell\in\Z^\nu\,, \ j,j'\in\S_0^c \,, \ (\ell,j,j')\neq (0,j,j) 
	\, ; 
	\end{cases}  \label{eq:2_meln-}\\
	&\begin{cases}
	\max\limits_{0\leq n \leq m_0}| \partial_\kappa^n\,( \ora{\Omega}(\kappa)\cdot \ell + \Omega_j(\kappa)+\Omega_{j'}(\kappa) ) | \geq \rho_0 	\braket{\ell} \\
	\ora{\jmath}\cdot \ell + j + j' = 0 \,, \ \ell\in\Z^\nu\,, \ j, j' \in\S_0^c  \, . 
	\end{cases} \label{eq:2_meln+}
	\end{align}
	We call $\rho_0$ the amount of non-degeneracy and $m_0$ the index of non-degeneracy.
\end{prop}
\begin{proof}
	We prove separately  \eqref{eq:0_meln}-\eqref{eq:2_meln+}. 
	In this proof we set for brevity
	$ \IK := [ \kappa_1, \kappa_2]$. 
	\\
	{\bf Proof of \eqref{eq:0_meln}}. By contradiction, assume that
	for any $m\in\N$ there exist $\kappa_m\in \IK $ and $\ell_m\in\Z^\nu\setminus\{0\}$ such that
	\begin{equation}\label{eq:0_abs_m}
	\Big| \partial_\kappa^n \ora{\Omega}(\kappa_m) \cdot \frac{\ell_m}{\braket{\ell_m}}  \Big| < \frac{1}{\braket{m}} \,, \quad \forall\,0\leq n\leq m \, .
	\end{equation}
	The sequences $(\kappa_m)_{m\in\N}\subset\IK$ and $(\ell_m/\braket{\ell_m})_{m\in\N}\subset \R^\nu\setminus\{0\}$ are both bounded. By compactness, up to subsequences
	 $\kappa_m\to \bar\kappa\in\IK$ and $\ell_m/\braket{\ell_m}\rightarrow\bar c\neq 0$. 
	Therefore, in the limit for $m\rightarrow + \infty$, by  \eqref{eq:0_abs_m} we get $\partial_\kappa^n \ora{\Omega}(\bar\kappa)\cdot \bar c = 0$ for any $n\in\N_0$.
	 By the analyticity of $ \ora{\Omega}(\kappa)$, we deduce 
	 that the function $ \kappa \mapsto \ora{\Omega}(\kappa)\cdot \bar c$ is identically zero on $\IK$, which contradicts Lemma \ref{lem:non_deg_vectors}-1.
	 \\[1mm]
	{ \bf Proof of \eqref{eq:1_meln}}. We divide the proof in 4 steps.\\
	{\sc Step 1. } Recalling \eqref{Om-om} and 
	Lemma \ref{rem:exp_omegaj_k},  
	we have that, for any $\kappa\in\IK $, 
	$$
	| \ora{\Omega}(\kappa)\cdot \ell + \Omega_j(\kappa)|
	 \geq |\Omega_j(\kappa)| - |\ora{\Omega}(\kappa)\cdot \ell |  \geq  \sqrt{\kappa_1}\abs j^\frac32 - C \langle \ell \rangle \geq \langle \ell \rangle 
	$$
	whenever $\abs j^\frac32 \geq C_0 \langle \ell \rangle $, 
	for some $C_0>0$.  In this cases \eqref{eq:1_meln} is already fulfilled with $n=0$.  Hence we  restrict  in the sequel to indexes $\ell\in\Z^{\nu}$ and $j \in \S_0^c$ satisfying
	\begin{equation}\label{eq:1_restr}
	\abs j ^\frac32 < C_0 \langle \ell \rangle \,.
	\end{equation}
	{\sc Step 2.} By contradiction, we assume that, for any $m\in\N$, there exist $\kappa_m\in\IK$, $\ell_m\in\Z^\nu$ and $j_m\in\S_0^c$, with $\abs{j_m}^\frac32<C_0 \langle \ell_m \rangle $, such that, for any $n\in\N_0$ with $n\leq m$,
	\begin{equation}\label{eq:1_abs_m}
	\begin{cases}
	\big| \partial_\kappa^n \big( \ora{\Omega}(\kappa)\cdot\frac{\ell_m}{\braket{\ell_m}}+\frac{1}{\braket{\ell_m}}\Omega_{j_m}(\kappa) \big)_{|\kappa= \kappa_m} \big|  < \frac{1}{\braket{m}} \\
	\ora{\jmath}\cdot \ell_m + j_m = 0 \, . 
	\end{cases} 
	\end{equation}
Up to subsequences 
$\kappa_m\rightarrow\bar\kappa\in\IK$ and $\ell_m/\braket{\ell_m}\rightarrow \bar c\in\R^\nu$. \\
	{\sc Step 3.} We consider first the case when the sequence $(\ell_m)_{m\in\N}\subset \Z^\nu$ is bounded. Up to subsequences, we have definitively that 
	$\ell_m=\bar \ell\in\Z^\nu $. Moreover, since $j_m$ and $\ell_m$ satisfy \eqref{eq:1_restr}, also the sequence $(j_m)_{m\in\N}$ is bounded and, up to subsequences,  definitively  $ j_m = \bar \jmath \in\S_0^c $. Therefore, in the limit $m\rightarrow \infty$, from \eqref{eq:1_abs_m} we obtain
	\begin{equation*}
	\partial_\kappa^n\big( \ora{\Omega}(\kappa)\cdot \bar \ell + \Omega_{\bar \jmath}(\kappa) \big)_{|\kappa = \bar \kappa} = 0 \ , \ \forall\, n\in\N_0  \, , \quad 
	\ora{\jmath}\cdot \bar \ell + \bar \jmath = 0  \, . 
	\end{equation*}
By analyticity this implies 
	\begin{equation}\label{eq:1_limit_b}
	\ora{\Omega}(\kappa)\cdot \bar \ell + \Omega_{\bar \jmath}(\kappa) = 0 \, , 
	\ \forall \, \kappa \in \IK \, ,  \quad 
	\ora{\jmath}\cdot \bar \ell + \bar \jmath = 0  \, . 
	\end{equation}
	We distinguish two cases:
	\begin{itemize}
		\item  Let $\bar \jmath \notin -\S$. By \eqref{eq:1_limit_b} the vector  
		$\big( \ora{\Omega}(\kappa),\Omega_{\bar \jmath}(\kappa) \big) $ 
		is degenerate according to Definition \ref{def:non-deg} 
		with $c:=(\bar \ell, 1)\neq 0$. This contradicts Lemma \ref{lem:non_deg_vectors}-3.
		\item Let $\bar \jmath \in -\S$. With no loss of generality suppose $\bar \jmath = - \bar{\jmath}_1$. Then, denoting $\bar \ell = (\bar{\ell_1}, \ldots ,\bar{\ell_\nu})$, system \eqref{eq:1_limit_b} reads, for any $ \kappa \in \IK $,  
		\begin{equation}\label{eq:1_inter}
		\begin{cases}
		(\bar{\ell_1}+1)\omega_{\bar{\jmath}_1}(\kappa) + \sum_{a=2}^{\nu}\bar{\ell_a}\omega_{\bar{\jmath}_a}(\kappa) + \frac{\gamma}{2}\left( (\bar{\ell_1}-1)\frac{G_{\bar{\jmath}_1}(0)}{\bar{\jmath}_1} + \sum_{a=2}^\nu \bar{\ell_a} \frac{G_{\bar{\jmath}_a}(0)}{\bar{\jmath}_a} \right) = 0  \\
		(\bar{\ell_1}-1)\bar{\jmath}_1 + \sum_{a=2}^\nu \bar{\ell_a}\, \bar{\jmath}_a = 0 \, . 
		\end{cases}
		\end{equation}
By Lemma \ref{rem:omega-nondeg} the   vector   $(\ora{\omega}(\kappa),1) $ 
is non-degenerate, which is a contradiction for $\gamma \neq 0 $. If $ \gamma = 0 $ we 
only deduce $\bar{\ell_1}= - 1$ and $\bar{\ell_2}=\ldots=\bar{\ell_\nu}=0$. Inserting these values in the momentum  condition in \eqref{eq:1_inter}, we get $2\bar{\jmath}_1=0$. This is a contradiction with $ \bar{\jmath}_1 \neq 0$.
	\end{itemize}
	{\sc Step 4.} We consider now the case when the sequence $(\ell_m)_{m\in\N}$ is unbounded. Up to subsequences
	$\abs{\ell_m}\rightarrow \infty$ as $m\rightarrow\infty$ and  $  \lim_{m\rightarrow\infty}\ell_m/\braket{\ell_m} =: \bar c  \neq 0 $. By  \eqref{Om-om} and  \eqref{eq:rem1}, 
	for any $ n \in \N_0 $,
	\begin{equation*}\label{eq:1_expand}
	\begin{split}
	\partial_\kappa^n \frac{1}{\braket{\ell_m}}\Omega_{j_m}(\kappa_m) & 
	=  \partial_\kappa^n \Big( \frac{1}{\braket{\ell_m}}\sqrt{\kappa}\abs{j_m}^\frac32 + \frac{c_{j_m}(\kappa)}{\braket{\ell_m}\sqrt{\kappa}\abs{j_m}^\frac12} + \frac{\gamma}{2\braket{\ell_m}}\frac{G_{j_m}(0)}{j_m} \Big)_{|\kappa= \kappa_m} \\
&   
\stackrel{\eqref{eq:rem2}} \to 
\bar d (\partial_\kappa^n\sqrt{\kappa})_{|\kappa = \bar \kappa } \, , \ {\rm for}\ m \rightarrow \infty \, , 
	\end{split}
	\end{equation*}
with $\bar d:= \lim_{m\rightarrow\infty}\abs{j_m}^\frac32/\braket{\ell_m} \in \R $.  Note that $ \bar d $ is finite  because 
	$j_m$ and $\ell_m$ satisfy \eqref{eq:1_restr}. Therefore
	 \eqref{eq:1_abs_m}  becomes, in the limit $m\rightarrow\infty $,   
	$$
	\partial_\kappa^n \big( \ora{\Omega}(\kappa) \cdot \bar c + \bar d \sqrt{\kappa} \, 
	\big)_{|\kappa= \bar \kappa } = 0 \, , \quad \forall \, n\in\N_0 \, . 
$$
By analyticity, this implies that $\ora{\Omega}(\kappa)\cdot \bar{c} + \bar{d}\sqrt{\kappa} = 0 $ for any $\kappa\in\IK$. This contradicts the non-degeneracy of the vector $(\ora{\Omega}(\kappa), \sqrt{\kappa})$ in Lemma \ref{lem:non_deg_vectors}-2, since $(\bar c, \bar d) \neq 0$.
\\[1mm]
	{\bf Proof of \eqref{eq:2_meln-}}. We split again the proof into 4 steps.\\
	{\sc Step 1.} By Lemma \ref{rem:exp_omegaj_k}, for any $\kappa\in\IK $, 
	\begin{equation*}
	| \ora{\Omega}(\kappa)\cdot \ell  + \Omega_j(\kappa)-\Omega_{j'}(\kappa)|  \geq 
	|\Omega_j(\kappa)-\Omega_{j'}(\kappa)| - | \ora{\Omega}(\kappa)\cdot \ell|  
	\geq \sqrt{\kappa_1}
	\big| \abs j^\frac32 - |j'|^\frac32 \big| - C  \langle \ell \rangle \geq \langle \ell \rangle 
	\end{equation*}
	whenever $ | \abs j^\frac32 - |j'|^\frac32 | \geq C_1 \langle \ell \rangle $ 
	for some $C_1>0$. In this case \eqref{eq:2_meln-} is already fulfilled with $ n = 0 $.	
	Thus we restrict to 
	indexes $\ell\in\Z^\nu$ and $j,j'\in\S_0^c$, 
	such that
	\begin{equation}\label{eq:2_restr}
	\big| |j|^\frac32 - |j'|^\frac32 \big|  < C_1 \langle \ell \rangle \,.
	\end{equation}
	Furthermore we may assume $ j_m \neq j_m' $ because the case $ j_m = j_m' $ is included in \eqref{eq:0_meln}.
	\\[1mm]
	{\sc Step 2.} By contradiction, we assume that, for any $m\in\N$,  there exist $\kappa_m\in\IK$, $\ell_m\in\Z^\nu$ and $j_m,j_m'\in\S_0^c$, 
	satisfying \eqref{eq:2_restr}, such that, for any $0\leq n\leq m$,
	\begin{equation}\label{eq:2_abs_m}
	\begin{cases}
	\big| \partial_\kappa^n\big( \ora{\Omega}(\kappa) \cdot \frac{\ell_m}{\braket{\ell_m}}+ \frac{1}{\braket{\ell_m}}\big( \Omega_{j_m}(\kappa)-\Omega_{j_m'}(\kappa) \big) \big)_{|
	\kappa= \kappa_m} \big| < \frac{1}{\braket{m}} \\
	\ora{\jmath} \cdot \ell_m + j_m -j_m' =0 \, . 
	\end{cases} 
	\end{equation}
	Up to subsequences $\kappa_m\rightarrow\bar\kappa\in\IK$ and $\ell_m/\braket{\ell_m}\rightarrow \bar c\in\R^\nu$. \\
	{\sc Step 3.} We start with the case when $(\ell_m)_{m\in\N}\subset\Z^\nu$ is bounded. Up to subsequences, we have definitively that  $\ell_m = \bar \ell\in\Z^\nu $. Moreover, if
	$ |j_m| \neq |j_m'| $, there is $ c > 0 $ such that
	\begin{equation*}
	c \big( |j_m|^\frac12 + |j_m'|^\frac12 \big) 
	\leq \big| |j_m|^{\frac32} -|j_m'|^\frac32 \big| < C_1 \langle \ell_m \rangle 
	\leq C \, ,
	 \qquad \forall m \in \N \, ,
	\end{equation*}
	If $ j_m = - j_m'  $ we deduce by the momentum relation that 
	$ |j_m| = |j_m'| \leq C \langle \ell_m \rangle \leq C $, and  
	we conclude that in any case 
	the sequences $(j_m)_{m\in\N}$ and $(j_m')_{m\in\N}$ are bounded.
	Up to subsequences, we have definitively that  $j_m= \bar \jmath$ and 
	$ j_m'=\bar{\jmath}' $, 
	with $\bar \jmath, \bar{\jmath}'\in\S_0^c$  and such that
\begin{equation}
\label{elljjneq0}
\jmath \neq  \bar{\jmath}' \, .  
\end{equation}
 Therefore \eqref{eq:2_abs_m} becomes,  in the limit $m\rightarrow\infty $,  
	\begin{equation*}
 \partial_\kappa^n\big( \ora{\Omega}(\kappa) \cdot \bar \ell + \Omega_{\bar \jmath}(\kappa)-\Omega_{\bar{\jmath}'}(\kappa)  \big)_{| \kappa= \bar \kappa }  = 0 \, , \ \forall\,n\in\N_0\, , 
  \quad 
	\ora{\jmath}\cdot \bar \ell + \bar \jmath - \bar{\jmath}' = 0 \, . 
	\end{equation*}
By analyticity, we obtain that
	\begin{equation}\label{eq:2-_limit_b}
	\ora{\Omega}(\kappa) \cdot \bar \ell + \Omega_{\bar \jmath}(\kappa) - \Omega_{\bar{\jmath}'}(\kappa) = 0 \, ,  \  \forall\,\kappa\in\IK \, , \quad 
	\ora{\jmath}\cdot \bar \ell + \bar \jmath-\bar{\jmath}' = 0 \, . 
	\end{equation}
	We distinguish several cases:
	\begin{itemize}
		\item Let $\bar \jmath,\bar{\jmath}' \notin -\S$ and $|\bar{\jmath}|\neq |\bar{\jmath}'|$. By \eqref{eq:2-_limit_b} the vector $(\ora{\Omega}(\kappa),\Omega_{\bar \jmath}(\kappa),\Omega_{\bar{\jmath}'}(\kappa))$ is degenerate with $c:= (\bar \ell,1,-1)\neq 0$, contradicting  Lemma \ref{lem:non_deg_vectors}-4.
		\item  Let $\bar \jmath,\bar{\jmath}' \notin-\S$ and $\bar{\jmath}'=- \bar{\jmath}$. 
		In view of  \eqref{Om-om},  system \eqref{eq:2-_limit_b} becomes
		\begin{equation}\label{eq:2_case1}
		\begin{cases}
		\ora{\omega}(\kappa)\cdot \bar \ell + \frac{\gamma}{2} \left( \sum_{a=1}^{\nu}\bar{\ell_a}\frac{G_{\bar{\jmath}_a }(0)}{\bar{\jmath}_a} +2 \frac{G_{\bar \jmath}(0)}{\bar \jmath} \right) = 0  \, , \qquad \forall \kappa \in \IK \, , \\
		\ora{\jmath}\cdot \bar \ell +2 \bar \jmath = 0 \, . 
		\end{cases}
		\end{equation}
By Lemma \ref{rem:omega-nondeg}  the vector $(\ora{\omega}(\kappa),1)$ is non-degenerate, which is a contradiction for $ \gamma \neq 0 $. 
If $ \gamma = 0 $ the first equation in \eqref{eq:2_case1} implies  $\bar \ell = 0$. 
Then the momentum condition implies  $ 2 \bar \jmath =0$, which  is a 
contradiction with  $\bar \jmath \neq 0$.
		\item Let $\bar{\jmath}'\notin -\S$ and $\bar{\jmath}\in-\S$. With no loss of generality 
		suppose  $\bar \jmath= -\bar{\jmath}_1 $. In view of  \eqref{Om-om}, 
		the first equation in \eqref{eq:2-_limit_b} implies that, for any $\kappa\in \IK $
$$
		(\bar{\ell_1}+1)\omega_{\bar{\jmath}_1}(\kappa)  + \sum_{a=2}^{\nu}\bar{\ell_a}\omega_{\bar{\jmath}_a}(\kappa)-\omega_{\bar{\jmath}'}(\kappa) 
	 +  \frac{\gamma}{2}\Big( (\bar{\ell_1}-1)\frac{G_{\bar{\jmath}_1}(0)}{\bar{\jmath}_1} + \sum_{a=2}^\nu \bar{\ell_a} \frac{G_{\bar{\jmath}_a}(0)}{\bar{\jmath}_a} -\frac{G_{\bar{\jmath}'}(0)}{\bar{\jmath}'} \Big) = 0 \, .
$$
		By Lemma \ref{rem:omega-nondeg}
		the vector $\big( \ora{\omega}(\kappa),\omega_{\bar{\jmath}'}(\kappa),1 \big)$ is non-degenerate, which is a contradiction.  
		\item Last, let $\bar \jmath,\bar{\jmath}'\in -\S$ and $\bar{\jmath}\neq \bar{\jmath}'$, by \eqref{elljjneq0}.
		With no loss of generality  suppose  $\bar{\jmath} =- \bar{\jmath}_1$ and $\bar{\jmath}' = -\bar{\jmath}_2$.
		Then \eqref{eq:2-_limit_b} reads 
		\begin{equation}\label{eq:2_case3}
		\begin{cases}
		(\bar{\ell_1}+1)\omega_{\bar{\jmath}_1}(\kappa) + \left( \bar{\ell_2} - 1 \right)\omega_{\bar{\jmath}_2} +\sum_{a=3}^{\nu}\bar{\ell_a}\omega_{\bar{\jmath}_a}(\kappa)\\
		  +  \frac{\gamma}{2}\left( (\bar{\ell_1}-1)\frac{G_{\bar{\jmath}_1}(0)}{\bar{\jmath}_1}+ (\bar{\ell_2}+1 )\frac{G_{\bar{\jmath}_2}(0)}{\bar{\jmath}_2} + \sum_{a=3}^\nu \bar{\ell_a} \frac{G_{\bar{\jmath}_a}(0)}{\bar{\jmath}_a} ) \right) = 0 \, , \ 
		\forall \kappa \in \IK \, ,  \\
		(\bar{\ell_1}-1)\bar{\jmath}_1+ (\bar{\ell_2}+1)\bar{\jmath}_2 +\sum_{a=3}^\nu \bar{\ell_a}\, \bar{\jmath}_a  = 0 \, . 
		\end{cases} 
		\end{equation}
By Lemma \ref{rem:omega-nondeg}
the vector $ (\ora{\omega}(\kappa),1) $ is non-degenerate, which is a contradiction for
$ \gamma \neq 0 $. If $ \gamma = 0 $
the first equation in \eqref{eq:2_case3} implies that $\bar{\ell_1}=-1$, $\bar{\ell_2} = 1$, $\bar{\ell_3}=\ldots=\bar{\ell_\nu}=0$. Inserting these values in  the momentum condition we obtain $-2\bar{\jmath}_1+ 2 \bar{\jmath}_2 = 0$. This contradicts $\bar{\jmath}\neq \bar{\jmath}'$.
	\end{itemize}
	{\sc Step 4.} We finally consider the case when $(\ell_m)_{m\in\N}$ is unbounded. Up to subsequences 
	$\abs{\ell_m}\rightarrow \infty$ as $m\rightarrow\infty$ and 
	$	 \lim_{m\to\infty}\ell_m/\braket{\ell_m} =: \bar c  \neq 0 $.  
In addition, by \eqref{eq:2_restr}, up to subsequences
\begin{equation}
\label{eq:d1}
\lim_{m \to \infty} \frac{|j_m|^\frac32-|j_m'|^\frac32}{\braket{\ell_m}}=  \bar d_1 \in \R \, . 
\end{equation}
By \eqref{Om-om} and  \eqref{eq:rem1} we have, for any $ n  $, 
	\begin{align*}
	& \partial_\kappa^n\frac{1}{\braket{\ell_m}} 
	\Big( \Omega_{j_m}(\kappa)-\Omega_{j_m'}(\kappa) \Big)_{|\kappa= \kappa_m}  =
	\partial_\kappa^n \Big( \frac{\sqrt{\kappa}}{\braket{\ell_m}}\big( |j_m|^\frac32-|j_m'|^\frac32 \big)  \\
	&  + \frac{1}{\braket{\ell_m}\sqrt{\kappa}}\Big( \frac{c_{j_m}(\kappa)}{|j_m|^\frac12}- \frac{c_{j_m'}(\kappa)}{|j_m'|^\frac12} \Big) + \frac{\gamma}{2\braket{\ell_m}}\Big( \frac{G_{j_m}(0)}{j_m}-\frac{G_{j_m'}(0)}{j_m'} \Big)_{|\kappa= \kappa_m}   \Big) 
\to \bar d_1 \partial_\kappa^n( \sqrt{\kappa})_{|\kappa = \bar \kappa}	  \nonumber 
	\end{align*}
using \eqref{eq:d1} and $\la \ell_m \ra \to \infty$. 
 Therefore \eqref{eq:2_abs_m} becomes, in the limit $m\rightarrow\infty $,  
	$$
	\partial_\kappa^n\big( \ora{\Omega}(\kappa)\cdot \bar c + \bar d_1 \sqrt{\kappa} \big)
	_{|\kappa = \bar \kappa} = 0 \, , \quad  \forall n\in\N_0 \, .
	$$
By analyticity this implies 
 $\ora{\Omega}(\kappa)\cdot\bar c + \bar d_1 \sqrt\kappa =  0$, for all  $ \kappa \in  \IK $. Thus $(\ora{\Omega}(\kappa), \sqrt{\kappa})$ is degenerate with $c = (\bar c, \bar d_1) \neq 0$,  contradicting Lemma \ref{lem:non_deg_vectors}-2. 
\\[1mm]
	{\bf Proof of \eqref{eq:2_meln+}}. The proof is similar to that 
	for  \eqref{eq:2_meln-} and we omit it.
\end{proof}

\section{Nash-Moser theorem and measure estimates}\label{sec:NM}

Under the 
rescaling $ (\eta,\zeta)\mapsto (\varepsilon \eta, \varepsilon\zeta) $, the Hamiltonian  system \eqref{eq:Ham_eq_zeta} transforms into the Hamiltonian system generated by 
\begin{equation}\label{cHepsilon}
\cH_\varepsilon(\eta,\zeta) := \varepsilon^{-2} \cH (\varepsilon\eta,\varepsilon\zeta) = \cH_L(\eta,\zeta) + \varepsilon P_\varepsilon(\eta,\zeta) \,,
\end{equation}
where $ \cH $ is the water waves Hamiltonian \eqref{Ham-Wal} expressed in the
Wahl\'en coordinates \eqref{eq:gauge_wahlen}, 
$\cH_L $ is defined in \eqref{lin_real} and 
\begin{align*}
P_\varepsilon(\eta,\zeta):= & \ \frac{1}{2\varepsilon}\int_\T 
\left( \zeta+\frac{\gamma}{2}\pa_x^{-1} \eta\right) 
( G(\varepsilon\eta)-G(0) ) \left( \zeta+\frac{\gamma}{2}\pa_x^{-1} \eta\right) \wrt x \\
& + \frac{\kappa}{\varepsilon^3}\int_\T\left( \sqrt{1+\varepsilon^2 \eta_x^2}-1-\frac{\varepsilon^2 \eta_x^2}{2} \right) \wrt x   + \frac{\gamma}{2}\int_\T\left( 
- \left( \zeta+\frac{\gamma}{2}\pa_x^{-1} \eta\right)_{\! x} \eta^2 +\frac{\gamma}{3}\eta^3 \right)\wrt x \,. \notag
\end{align*}
We now study the Hamiltonian system  generated by the Hamiltonian $\cH_\varepsilon(\eta,\zeta) $, 
in  the action-angle and normal 
coordinates $ (\theta, I, w) $ defined in Section \ref{sec:decomp}. 
Thus we consider the Hamiltonian $H_\varepsilon (\theta, I, w )$ defined by 
\begin{equation}\label{Hepsilon}
H_\varepsilon := \cH_\varepsilon \circ A = \varepsilon^{-2} \cH \circ \varepsilon A
\end{equation}
where  
$A$ is the map defined in \eqref{aacoordinates}. 
The associated symplectic form is given in \eqref{sympl_form}. 

By  Lemma \ref{zero_term} (see also \eqref{QFH}, \eqref{ajbjAA}), 
in the variables $ (\theta, I, w) $
the quadratic Hamiltonian $\cH_L $ 
defined in \eqref{lin_real}  simply reads, up to a constant,
$$
\cN:= \cH_L\circ A = \ora{\Omega}(\kappa)\cdot I + \tfrac12 \left( \b\Omega_W w,w \right)_{L^2} 
$$
where $\ora{\Omega}(\kappa) \in \R^\nu $ is defined in \eqref{Omega-kappa} and $\b\Omega_W$ in \eqref{eq:lin00_wahlen}.
Thus the Hamiltonian $H_\varepsilon$ in \eqref{Hepsilon} is
\begin{equation}\label{cNP}
 H_\varepsilon =\cN + \varepsilon P  \qquad {\rm with}
\qquad P:= P_\varepsilon \circ A \, . 
\end{equation}
We look for an embedded invariant torus
$$
i :\T^\nu \rightarrow \R^\nu \times \R^\nu \times \acca_{\S^+,\Sigma}^\angle \,, \quad \vf \mapsto i(\vf):= ( \theta(\vf), I(\vf), w(\vf)) \, , 
$$
of the Hamiltonian vector field 
$ X_{H_\varepsilon} := ( \pa_I H_\varepsilon , -\pa_\theta H_\varepsilon, \Pi_{\S^+,\Sigma}^\angle J \nabla_{w} H_\varepsilon ) $ 
filled by quasi-periodic solutions with Diophantine frequency 
vector $\omega\in\R^\nu$ (which satisfies also first and second order Melnikov non-resonance conditions, 
see \eqref{0meln}-\eqref{2meln-}).

\subsection{Nash-Moser theorem of hypothetical conjugation}

For $\alpha\in \R^\nu$, we consider the family of modified Hamiltonians
\begin{equation}\label{Halpha}
H_\alpha := \cN_\alpha + \varepsilon P \,, \quad \cN_\alpha:= \alpha \cdot I + \tfrac12 \left( w, \b\Omega_W w \right)_{L^2} \, , 
\end{equation}
and  the nonlinear operator
\begin{align}\label{F_op}
\cF(i,\alpha) & := \cF(\omega,\kappa,\varepsilon;i,\alpha) := \omega\cdot\pa_\vf i(\vf) - X_{H_\alpha}(i(\vf)) \notag \\
& = \begin{pmatrix}
\omega\cdot \pa_\vf \theta(\vf) & - \alpha -\varepsilon \pa_I P(i(\vf)) \\
\omega\cdot \pa_\vf I(\vf) &+ \varepsilon \pa_\theta P(i(\vf)) \\
\omega\cdot \pa_\vf w(\vf) &  - \, \Pi_{\S^+,\Sigma}^\angle J ( \b\Omega_W w(\vf) +\varepsilon \nabla_{w} P(i(\vf)) )  
\end{pmatrix} \, . 
\end{align}
If $\cF(i,\alpha)=0$, then 
the embedding $\vf\mapsto i(\vf)$ is an invariant torus for the Hamiltonian vector field $X_{H_\alpha}$, filled with quasi-periodic solutions with frequency $\omega$.

Each Hamiltonian $H_\alpha$ in \eqref{Halpha} is invariant under the
 involution $\vec \cS$   
 and the translations $\vec \tau_\vs$, $\vs \in \R$, defined respectively in \eqref{rev_aa} and in  \eqref{vec.tau}:
\begin{equation}
\label{Halpha.symm}
H_\alpha\circ \vec \cS = H_\alpha \, , \qquad
H_\alpha\circ \vec\tau_\vs = H_\alpha \, ,  \quad \forall\, \vs \in \R \, . 
\end{equation}
 We look for a reversible   traveling torus embedding 
 $ \vphi \mapsto  i (\vphi) = $ $ ( \theta(\vf), I(\vf), w(\vf)) $, namely 
 satisfying  
\begin{equation}\label{RTTT}
\vec \cS i(\vf)=  i(-\vf) \, , \qquad 
\vec \tau_\vs i(\vf) = i(\vf - \ora{\jmath} \vs) \, , \quad \forall \,\vs \in \R \, . 
\end{equation}
\begin{lem}\label{lem:RT}
The operator $ \cF ( \cdot , \alpha) $ maps a reversible, respectively traveling, wave 
into an anti-reversible,  respectively traveling, wave variation, according to Definition \ref{trav-vari}.
\end{lem}

\begin{proof}
It follows directly by \eqref{F_op} and \eqref{Halpha.symm}. 
\end{proof}

The norm of the periodic components of the embedded torus
\begin{equation}\label{ICal}
\fI (\vf):= i(\vf)-(\vf,0,0) := \left( \Theta(\vf), I(\vf), w(\vf) \right)\,, \quad \Theta(\vf):= \theta(\vf)-\vf\,,
\end{equation}
is $
\norm{ \fI }_s^{k_0,\upsilon} := \norm{\Theta}_{H_\vf^s}^{k_0,\upsilon} + \norm{I}_{H_\vf^s}^{k_0,\upsilon} + \norm{w}_s^{k_0,\upsilon} $, 
where
\begin{equation}\label{k0_def}
k_0:= m_0 + 2
\end{equation}
and $m_0 \in \N $ is the index of non-degeneracy provided by Proposition \ref{prop:trans_un}, which only depends on the linear unperturbed frequencies. Thus, $k_0$ is considered as an absolute constant and we will often omit to write the dependence of the various constants with respect to $k_0$. We look for quasi-periodic solutions of frequency $\omega$ belonging to a $\delta$-neighbourhood (independent of $\varepsilon$)
$$
\t\Omega := \big\{ \omega \in \R^\nu \ : \
 \dist \big( \omega, \ora{\Omega}[\kappa_1,\kappa_2] \big) < \delta \big\} \,, \quad \delta >0 \,,
$$
of the curve
 $\ora{\Omega}[\kappa_1,\kappa_2]$ defined by \eqref{Omega-kappa}.
\begin{thm} {\bf (Nash-Moser)}\label{NMT}
There exist positive constants ${\rm a_0},\varepsilon_0, C$ depending on $\S$, $k_0 $ and $\tau\geq 1$ such that, for all $\upsilon = \varepsilon^{\rm a}$, ${\rm a}\in (0,{\rm a}_0)$ and for all $\varepsilon\in (0,\varepsilon_0)$, there exist 
	\begin{enumerate}
	\item 
	a $k_0$-times differentiable function
	\begin{align}\label{alpha_infty}
	\alpha_\infty : & \, \t\Omega \times [\kappa_1,\kappa_2] \mapsto \R^\nu \,,\notag\\
	& \alpha_\infty(\omega,\kappa) := \omega + r_\varepsilon(\omega,\kappa)  \quad \text{ with } \quad  |r_\varepsilon|^{k_0,\upsilon} \leq C \varepsilon \upsilon^{-1} \, ;  
	\end{align}
	\item
	a family of embedded reversible traveling tori $i_\infty (\vf) $ (cfr. \eqref{RTTT}), 
	defined for all $(\omega,\kappa)\in\t\Omega \times[\kappa_1,\kappa_2] $, satisfying
	\begin{equation}\label{i.infty.est}
	\| i_\infty (\vf) -(\vf,0,0) \|_{s_0}^{k_0,\upsilon} \leq C \varepsilon\upsilon^{-1} \, ; 
	\end{equation}
	\item 
	a sequence of $k_0$-times differentiable functions $\mu_j^\infty : \R^\nu \times [\kappa_1,\kappa_2] \rightarrow \R$, $j\in \S_0^c  = \Z\,\setminus\,(\S\cup\{0\})$, of the form
	\begin{equation}\label{def:FE}
	\mu_j^\infty(\omega,\kappa) = \tm_{\frac32}^\infty(\omega,\kappa) \Omega_j(\kappa) + \tm_1^\infty(\omega,\kappa)j + \tm_{\frac12}^\infty(\omega,\kappa)\abs j^\frac12 + \fr_j^\infty(\omega,\kappa)\,,
\end{equation}
	with $\Omega_j(\kappa) $ defined in \eqref{def:Omegajk}, satisfying
	\begin{equation}\label{coeff_fin_small}
	| \tm_{\frac32}^\infty-1 |^{k_0,\upsilon}  ,\,  |\tm_1^\infty |^{k_0,\upsilon}, \, 
	|\tm_{\frac12}^\infty|^{k_0,\upsilon} \leq C\varepsilon \, , \quad   \sup_{j\in\S_0^c} | \fr_j^\infty |^{k_0,\upsilon} \leq C \varepsilon \upsilon^{-1} \, , 
	\end{equation}
\end{enumerate}
such that, for all $(\omega,\kappa)$ in the Cantor-like set
	\begin{align}
	\cC_\infty^\upsilon := & \Big\{ (\omega,\kappa) \in \t\Omega\times[\kappa_1,\kappa_2] \ : \ \abs{ \omega\cdot \ell } \geq 8 \upsilon \braket{\ell}^{-\tau} \,, \ \forall\,\ell\in \Z^\nu\setminus\{ 0\}\,;  \label{0meln}\\
	&  \ \abs{ \omega\cdot \ell + \mu_j^\infty(\omega,\kappa) } \geq 4 \upsilon \abs j^{\frac32}\braket{\ell}^{-\tau} \,,  \label{1meln} 
	  \forall\, \ell \in\Z^\nu , \, j\in \S_0^c \text{ with } \ora{\jmath}\cdot\ell + j =0  \,;   \\
	& 
	 \ \abs{ \omega\cdot \ell + \mu_j^\infty(\omega,\kappa)-\mu_{j'}^\infty(\omega,\kappa) }\geq 4 \upsilon\,
	\langle | j|^\frac32 - | j'|^\frac32 \rangle 
	\braket{\ell}^{-\tau} \,,  \label{2meln-}\\
	&\  \quad \quad \forall \ell\in\Z^\nu, \, j,j'\in\S_0^c,\, (\ell,j,j')\neq (0,j,j) \text{ with } \ora{\jmath}\cdot \ell + j-j'=0 \, ,   \nonumber \\
	& \  \abs{ \omega\cdot \ell + \mu_j^\infty(\omega,\kappa) +\mu_{j'}^\infty(\omega,\kappa)  } \geq 4\upsilon \,\big(\abs j^\frac32 + |j'|^\frac32 \big) \braket{\ell}^{-\tau}  \,, \label{2meln+}\\
	& \ \quad \quad \forall\,\ell\in \Z^\nu , \, j ,   j'\in\S_0^c \, , 
	 \text{ with } \ora{\jmath}\cdot\ell+j+j'=0\, \Big\}  \,,\nonumber
	\end{align}
	the function $i_\infty(\vf):= i_\infty(\omega,\kappa,\varepsilon;\vf)$ is a 
	solution of $\cF(\omega,\kappa,\varepsilon; i_\infty,\alpha_\infty(\omega,\kappa))=0$. As a consequence, the embedded torus $\vf\mapsto i_\infty(\vf)$ is invariant for the Hamiltonian vector field $X_{H_{\alpha_\infty(\omega,\kappa)}}$ as it is filled by  quasi-periodic reversible traveling wave solutions  with frequency $\omega$.	
\end{thm}

We remind that the conditions on the indexes in \eqref{1meln}-\eqref{2meln-}
(where $ \vec \jmath \in \Z^\nu $ is the vector in \eqref{def:vecj}) are due to the fact that
we look for traveling wave solutions. 
These restrictions are essential to prove the measure estimates of the next section. 

\subsection{Measure estimates}\label{subsec:measest}

By \eqref{alpha_infty}, the function $\alpha_\infty(\,\cdot\,,\kappa)$ from $\t\Omega$ into its image $\alpha_\infty(\t\Omega,\kappa)$ is invertible and 
\begin{equation}\label{inv_alpha}
\begin{aligned}
& \beta = \alpha_\infty(\omega,\kappa) = \omega+r_\varepsilon(\omega,\kappa) \ \Leftrightarrow \\
&  \omega = \alpha_\infty^{-1}(\beta,\kappa) = \beta+\breve{r}_\varepsilon(\beta,\kappa) , \quad \abs{ \breve{r}_\varepsilon }^{k_0,\upsilon} \leq C\varepsilon\upsilon^{-1}\,.
\end{aligned}
\end{equation}
Then, for any $\beta\in\alpha_\infty(\cC_\infty^\upsilon)$, Theorem \ref{NMT} proves the existence of an embedded invariant torus filled by quasi-periodic solutions with Diophantine frequency $\omega=\alpha_\infty^{-1}(\beta,\kappa)$ for the Hamiltonian
\begin{equation*}
H_\beta = \beta \cdot I+  \tfrac12(w,\b\Omega_W w)_{L^2} + \varepsilon P \,.
\end{equation*}
Consider the curve of the unperturbed tangential frequency vector 
$ \ora{\Omega}(\kappa) $
 in \eqref{Omega-kappa}. In Theorem \ref{MEASEST} below we prove that for "most" values of $\kappa\in[\kappa_1,\kappa_2]$ the vector $(\alpha_\infty^{-1}(\ora{\Omega}(\kappa),\kappa) ,\kappa )$ is in $\cC_\infty^\upsilon$,
 obtaining
 an embedded torus for the Hamiltonian $H_\varepsilon$ in  \eqref{Hepsilon}, filled by quasi-periodic solutions with Diophantine frequency vector 
 $\omega = \alpha_\infty^{-1}(\ora{\Omega}(\kappa),\kappa) $, denoted  $ \wt \Omega $ in Theorem \ref{thm:main0}.
Thus $\varepsilon A(i_\infty(\wt \Omega t))$,  where $A$ is defined in \eqref{aacoordinates},
is a quasi-periodic traveling wave solution of the water waves equations 
\eqref{eq:Ham_eq_zeta} written in the Wahl\'en  variables. 
Finally, going back to the original Zakharov variables via \eqref{Whalen-c} 
we obtain solutions of \eqref{ww}. 
  This proves Theorem \ref{thm:main0} together with the following measure estimate.
  
\begin{thm} {\bf (Measure estimates)}\label{MEASEST}
	Let
	\begin{equation}\label{param_small_meas}
	\upsilon = \varepsilon^{\rm a} \,, \quad  0 <{\rm a}<\min\{ {\rm a}_0,1/(1+k_0) \}\,, \quad \tau > m_0 (\nu+4) \,,
	\end{equation}
	where $m_0$ is the index of non-degeneracy given in Proposition \ref{prop:trans_un} and $k_0:= m_0+2$. Then, for $ \varepsilon \in (0, \varepsilon_0) $ small enough,  the measure of the set
	\begin{equation}\label{Gvare}
	\cG_\varepsilon := \big\{ \kappa \in [\kappa_1,\kappa_2] 
	\ : \  \big( \alpha_\infty^{-1}( \ora{\Omega}(\kappa),\kappa ),\kappa \big) 
	\in \cC_\infty^\upsilon \big\}
	\end{equation}
	satisfies $ | \cG_\varepsilon | \rightarrow\kappa_2-\kappa_1$ as 
	$\varepsilon\rightarrow 0$.
\end{thm}

The rest of this section is devoted to prove Theorem \ref{MEASEST}. By \eqref{inv_alpha}
we have
\begin{equation}\label{Om-per}
\ora{ \Omega}_\varepsilon (\kappa):= \alpha_\infty^{-1}(\ora{\Omega}(\kappa),\kappa) =
\ora{\Omega}(\kappa) +\ora{r}_\varepsilon \,,
\end{equation}
where $\ora{r}_\varepsilon(\kappa) 
:= \breve{r}_\varepsilon(\ora{\Omega}(\kappa),\kappa) $
satisfies 
\begin{equation}\label{eq:tang_res_est}
\abs{\partial_\kappa^k {\vec r}_\varepsilon (\kappa)} \leq C \varepsilon\upsilon^{-(1+k)} \,, \quad \forall\,\abs k \leq k_0 \,, \ \text{uniformly on } [\kappa_1,\kappa_2] \,.
\end{equation}
We also denote, with a small abuse of notation, for all $j\in\S_0^c$, 
\begin{equation}\label{eq:final_eig_kappa}
\mu_j^\infty(\kappa):= \mu_j^\infty \big( \ora{\Omega}_\varepsilon(\kappa),\kappa \big) := \tm_{\frac32}^\infty(\kappa)\Omega_j(\kappa)+\tm_1^\infty(\kappa) j + \tm_{\frac12}^\infty(\kappa) \abs j^\frac12 + \fr_j^\infty(\kappa)\,,
\end{equation}
where 
$\tm_{\frac32}^\infty (\kappa) :=\tm_{\frac32}^\infty (\ora{\Omega}_\varepsilon(\kappa),\kappa) $, $\tm_1^\infty (\kappa) :=\tm_1^\infty (\ora{\Omega}_\varepsilon(\kappa),\kappa) $, $\tm_{\frac12}^\infty (\kappa) :=\tm_{\frac12}^\infty (\ora{\Omega}_\varepsilon(\kappa),\kappa) $ and $\fr_j^\infty (\kappa) :=\fr_j^\infty (\ora{\Omega}_\varepsilon(\kappa),\kappa) $.

By \eqref{coeff_fin_small} and \eqref{eq:tang_res_est} 
we have 
\begin{align}
& \big| \pa_\kappa^k\big( \tm_{\frac32}^\infty(\kappa)-1 \big) \big| , \, 
|\pa_\kappa^k\tm_1^\infty(\kappa)|,\, 
| \pa_\kappa^k\tm_{\frac12}^\infty(\kappa) |  \leq C \varepsilon\upsilon^{-k}, \label{small_coeff_k}\\
&\sup_{j\in\S_0^c }\abs{ \pa_\kappa^k \fr_j^\infty(\kappa) } \leq C\varepsilon\upsilon^{- 1-k}\,, \quad \forall\, 0\leq k\leq k_0 \,. \label{small_rem_k}
\end{align}
Recalling \eqref{0meln}-\eqref{2meln-}, the Cantor set in \eqref{Gvare} becomes
\begin{align*}
\cG_\varepsilon  := & \Big\{ 
\kappa \in [\kappa_1,\kappa_2] \ :  
\ | \ora{\Omega}_\varepsilon(\kappa)\cdot \ell | \geq 8 \upsilon \braket{\ell}^{-\tau} \,, 
\ \forall\,\ell\in \Z^\nu\setminus\{ 0\}\,;  \\
& \  \  | \ora{\Omega}_\varepsilon(\kappa)\cdot \ell + \mu_j^\infty(\kappa) | \geq 4 \upsilon 
|j|^{\frac32}\braket{\ell}^{-\tau} \,, \  \forall\, \ell \in\Z^\nu \, ,  \, j\in\S_0^c \, , \text{ with } \ora{\jmath}\cdot\ell + j =0  \,;  \notag \\
& \ \ | \ora{\Omega}_\varepsilon(\kappa)\cdot \ell + \mu_j^\infty(\kappa)-\mu_{j'}^\infty(\kappa) | \geq 4 \upsilon \, \langle | j|^\frac32 - | j'|^\frac32 \rangle\braket{\ell}^{-\tau} \,,  \\
& \ \  \forall \ell\in\Z^\nu, \, j,j'\in\S_0^c,\, (\ell,j,j')\neq (0,j,j) \text{ with } \ora{\jmath}\cdot \ell + j-j'=0  \, ; \notag \\
& \ \ | \ora{\Omega}_\varepsilon(\kappa)\cdot \ell + \mu_j^\infty(\kappa) +\mu_{j'}^\infty(\kappa)  | \geq 4\upsilon \,\big(|j|^\frac32 + |j'|^\frac32 \big) \braket{\ell}^{-\tau}  \,, \\
&  \ \  \forall\,\ell\in \Z^\nu , \, j,  j'\in\S_0^c \text{ with } \ora{\jmath}\cdot\ell+j+j'=0 \Big\} \, .  \notag 
\end{align*}
We estimate the measure of the complementary set
\begin{align}\label{union_gec}
\cG_\varepsilon^c & := [\kappa_1,\kappa_2] \setminus\cG_\varepsilon \notag \\
& = \left( \bigcup_{\ell\neq 0} R_{\ell}^{(0)} \right) \cup \left( \bigcup_{\ell\in\Z^\nu, \, j\in\S_0^c\atop \ora{\jmath}\cdot\ell+j=0} R_{\ell,j}^{(I)} \right) 
\cup \left( \bigcup_{(\ell,j,j')\neq(0,j,j), j \neq j' 
\atop \ora{\jmath}\cdot\ell+j-j'=0 } R_{\ell,j,j'}^{(II)} \right)
\cup \left( \bigcup_{\ell\in\Z^\nu, j,  j'\in\S_0^c \, , 
\atop \ora{\jmath}\cdot\ell+j+j'=0} Q_{\ell,j,j'}^{(II)} \right) \,,
\end{align}
where the ``nearly-resonant sets"
\begin{align}
R_{\ell}^{(0)} := & \big\{ \kappa\in[\kappa_1,\kappa_2] \ : \ | \ora{\Omega}_\varepsilon(\kappa)\cdot \ell | < 8 \upsilon \braket{\ell}^{-\tau}  \big\} \,, \label{R00}\\
R_{\ell,j}^{(I)} := & \big\{ \kappa\in[\kappa_1,\kappa_2]  \ : \ | \ora{\Omega}_\varepsilon(\kappa)\cdot \ell + \mu_j^\infty(\kappa) | < 4 \upsilon | j|^\frac32 \braket{\ell}^{-\tau}  \big\} \,, \label{RI0}\\
R_{\ell,j,j'}^{(II)} := & \big\{ \kappa\in[\kappa_1,\kappa_2] \ : \ | \ora{\Omega}_\varepsilon(\kappa)\cdot \ell +\mu_j^\infty(\kappa)-\mu_{j'}^\infty(\kappa) | < 4 \upsilon \,\langle | j|^\frac32 - | j'|^\frac32 \rangle\braket{\ell}^{-\tau}  \big\} \, ,  \label{RII0} \\
Q_{\ell,j,j'}^{(II)} := & \big\{ \kappa\in[\kappa_1,\kappa_2] \  : \ | \ora{\Omega}_\varepsilon(\kappa)\cdot \ell + \mu_j^\infty(\kappa)+\mu_{j'}^\infty(\kappa) | < 4 \upsilon \big(|j|^\frac32 +|j'|^\frac32 \big) \braket{\ell}^{-\tau}  \big\} \, .  \label{QII0}
\end{align}
Note that in the third union in \eqref{union_gec} we may require $ j \neq j' $ because
	$R_{\ell,j,j}^{(II)} \subset  R_\ell^{(0)}$. 
	In the sequel we shall always suppose the momentum conditions on the 
	indexes $ \ell, j, j' $  written in \eqref{union_gec}. 
Some of the above sets are empty.
\begin{lem}\label{lem:emptysets}
Consider the sets in \eqref{union_gec}-\eqref{QII0}. 
For $\varepsilon\in(0,\varepsilon_0)$ small enough,	we have that  
	\begin{enumerate}
		\item If $R_{\ell,j}^{(I)}\neq \emptyset$ then $ |j|^\frac32 \leq C \braket{\ell} $;
		\item If $R_{\ell,j,j'}^{(II)}\neq \emptyset$ then $\big| | j|^\frac32 -|j'|^\frac32 \big|\leq C \braket{\ell}$; 
		\item If $Q_{\ell,j,j'}^{(II)}\neq \emptyset$ then $ |j|^\frac32 + |j'|^\frac32\leq C \braket{\ell}$.
	\end{enumerate}
\end{lem}
\begin{proof}
We provide the proof for $R_{\ell,j,j'}^{(II)}$. 
If $R_{\ell,j,j'}^{(II)}\neq \emptyset$
then there exists $\kappa\in[\kappa_1,\kappa_2]$ such that
\begin{equation}\label{RII1}
\abs{ \mu_j^\infty(\kappa)-\mu_{j'}^\infty(\kappa)} < \frac{4\upsilon\, 
\langle | j|^\frac32 - |j'|^\frac32 \rangle}{\braket{\ell}^\tau} + 
| \ora{\Omega}_\varepsilon(\kappa)\cdot\ell | \leq 
4\upsilon \,\big| |j|^\frac32-|j'|^\frac32\big|+ C\braket{\ell} \,.
\end{equation}
By \eqref{eq:final_eig_kappa} we have 
$$
\mu_{j}^\infty(\kappa)-\mu_{j'}^\infty(\kappa) = 
\tm_{\frac32}^\infty(\kappa) ( \Omega_{j}(\kappa) - \Omega_{j'}(\kappa) )
+ \tm_1^\infty(\kappa)(j-j') +  \tm_{\frac12}^\infty(\kappa)( |j|^\frac12-|j'|^\frac12 ) 
+ \fr_j^\infty(\kappa)- \fr_{j'}^\infty(\kappa) \, .
$$
Then, by \eqref{small_coeff_k}-\eqref{small_rem_k} with $ k = 0 $, 
 \eqref{eq:rem1}-\eqref{eq:rem2}, 
 the momentum condition $j-j'=-\ora{\jmath}\cdot\ell $, and
the elementary inequality 
$ |  |j|^\frac32 - |j'|^\frac32 | \geq | |j|^\frac12 - |j'|^\frac12 | $, 
 we deduce the lower bound 
\begin{align}
|\mu_{j}^\infty(\kappa)-\mu_{j'}^\infty(\kappa)|  & \geq  
(1- C \varepsilon) \sqrt\kappa \big( 
\big| |j|^\frac32 - |j'|^\frac32 \big| - C \big)
- C \varepsilon  | \ora{\jmath}\cdot\ell | - C \varepsilon 
\big| |j|^\frac12 - |j'|^\frac12 \big| - C \varepsilon \upsilon^{-1} \nonumber \\
&  \geq \tfrac{\sqrt\kappa}{2}\,\big| | j|^\frac32 - |j'|^\frac32 \big| - C\varepsilon 
| \ell| - C' -C\varepsilon\upsilon^{- 1} \, .  \label{RII2}
	\end{align}
	Combining \eqref{RII1} and \eqref{RII2}, 
	we deduce 
	$ |  |j|^{\frac32}-|j'|^\frac32 | \leq C\braket{\ell}$, for  $ \varepsilon $ small enough. 
\end{proof}

In order to estimate the measure of the sets \eqref{R00}-\eqref{QII0} that are nonempty, the key point is to prove that the perturbed frequencies satisfy estimates similar to \eqref{eq:0_meln}-\eqref{eq:2_meln+}. 

\begin{lem} {\bf (Perturbed transversality)} \label{lem:pert_trans}
For $\varepsilon\in(0,\varepsilon_0)$ small enough and for all $\kappa\in[\kappa_1,\kappa_2]$, 
	\begin{align}
	&\max_{0\leq n \leq m_0} | \partial_\kappa^n \ora{\Omega}_\varepsilon(\kappa)\cdot \ell | \geq \frac{\rho_0}{2}\braket{\ell} \,, \quad \forall\,\ell\in\Z^\nu\setminus\{0\} \,; \label{eq:0_meln_pert}\\
	&\begin{cases}
	\max_{0\leq n \leq m_0}| \partial_\kappa^n( \ora{\Omega}_\varepsilon(\kappa)\cdot \ell + \mu_j^\infty(\kappa) ) | \geq \frac{\rho_0}{2}	\braket{\ell} \\
	\ora{\jmath}\cdot \ell + j = 0 \,, \quad \ell\in\Z^\nu\,, \ j\in\S_0^c \,; 
	\end{cases} \label{eq:1_meln_pert}\\
	&\begin{cases}
	\max_{0\leq n \leq m_0}| \partial_\kappa^n( \ora{\Omega}_\varepsilon(\kappa)\cdot \ell + \mu_j^\infty(\kappa)-\mu_{j'}^\infty(\kappa) ) | \geq \frac{\rho_0}{2}	\braket{\ell} \\
	\ora{\jmath}\cdot \ell + j -j'= 0 \,, \quad \ell\in\Z^\nu\,, \ j,j'\in\S_0^c \,, \ (\ell,j,j')\neq (0,j,j) \, ; 
	\end{cases} \label{eq:2_meln-_pert}\\
	&\begin{cases}
	\max_{0\leq n \leq m_0}| \partial_\kappa^n ( \ora{\Omega}_\varepsilon(\kappa)\cdot \ell + \mu_j^\infty(\kappa)+\mu_{j'}^\infty(\kappa) ) | \geq \frac{\rho_0}{2} 	\braket{\ell} \\
	\ora{\jmath}\cdot \ell + j + j'= 0 \,, \quad \ell\in\Z^\nu\,, \ j,   j'\in\S_0^c  \, . 
	\end{cases} \label{eq:2_meln+_pert}
	\end{align}
	We recall that $\rho_0$ is the amount of non-degeneracy that has been defined in Proposition \ref{prop:trans_un}.
\end{lem}
\begin{proof}
	We prove  \eqref{eq:2_meln-_pert}. 
	The proofs of \eqref{eq:0_meln_pert}, \eqref{eq:1_meln_pert} and 
	\eqref{eq:2_meln+_pert} are similar. 
	By \eqref{eq:final_eig_kappa} we have
	\begin{align}
&	\ora{\Omega}_\varepsilon(\kappa)  \cdot \ell  + \mu_j^\infty (\kappa) -\mu_{j'}^\infty(\kappa) = \ora{\Omega}(\kappa) \cdot \ell + \ora{r}_\varepsilon(\kappa)\cdot \ell + \Omega_j(\kappa)-\Omega_{j'}(\kappa) \label{lobb} \\
	& +( \tm_{\frac32}^\infty(\kappa)-1 ) \left( \Omega_j(\kappa)-\Omega_{j'}(\kappa) \right) + \tm_1^\infty(\kappa)(j-j') + \tm_{\frac12}^\infty(\kappa) \,( | j|^\frac12 - |j'|^\frac12 ) + \fr_j^\infty(\kappa) - \fr_{j'}^\infty(\kappa) \, . \nonumber 
	\end{align}
By Lemma \ref{rem:exp_omegaj_k} we get that, for any 
	$n\in \{ 0,\ldots,m_0 \} $,  
	\begin{equation}\label{eq:restr_2-}
	\abs{ \pa_\kappa^n (\Omega_j(\kappa)-\Omega_{j'}(\kappa))} 
	\leq C(\kappa)\big| |j|^\frac32 -|j'|^\frac32 \big| + C \leq  C'(\kappa) \braket{\ell} \, , 
	\end{equation}
because, by Lemma \ref{lem:emptysets}-(2), we can restrict to  
indexes $\ell, j,j' $  such that
$ | |j|^\frac32 - |j'|^\frac32 | \leq C \braket{\ell} $. 	
Furthermore
\begin{equation}\label{12j12j'}
	\big| |j|^\frac12 -|j'|^\frac12 \big|  \leq \big| 
	|j|^\frac32 -|j'|^\frac32 \big| \leq  C \braket{\ell} \,.
	\end{equation}
Therefore, by \eqref{lobb}, 	
 \eqref{small_coeff_k}, \eqref{small_rem_k}, \eqref{eq:tang_res_est}, 
 \eqref{eq:restr_2-}, \eqref{12j12j'},
and the momentum condition $j-j'= -\ora{\jmath}\cdot \ell$, we have
that,  for any $n\in \{ 0,\ldots,m_0 \} $, 
	\begin{align*}
	| \partial_\kappa^n\,( \ora{\Omega}_\varepsilon(\kappa)\cdot \ell  + \mu_j^\infty(\kappa)-\mu_{j'}^\infty(\kappa)  )  | 
	&\geq | \partial_\kappa^n\,(\ora{\Omega}(\kappa) \cdot \ell+ \Omega_j(\kappa)-\Omega_{j'}(\kappa))| - C\varepsilon\upsilon^{-(1 +m_0)}\braket{\ell} \, . 
	\end{align*}
	Since $\ora{\Omega}(\kappa)\cdot \ell +\Omega_j(\kappa)-\Omega_{j'}(\kappa)$ satisfies \eqref{eq:2_meln-}, we deduce that 
	\begin{equation*}
	\max_{0\leq n \leq m_0}| \partial_\kappa^n\,( \ora{\Omega}_\varepsilon(\kappa)\cdot \ell  + \mu_j^\infty(\kappa)-\mu_{j'}^\infty(\kappa) )  | \geq \rho_0\braket{\ell}- C\varepsilon\upsilon^{-(1+m_0)}\braket{\ell} \geq \tfrac{\rho_0}{2}\braket{\ell}
	\end{equation*}
	for $\varepsilon>0$ small enough.
\end{proof}

As an application of R\"ussmann Theorem 17.1 in \cite{Russ}, we deduce the following result:

\begin{lem} {\bf (Estimates of the resonant sets)} \label{lem:meas_res}
	The measure of the sets \eqref{union_gec}-
	\eqref{QII0} satisfy
	\begin{align*}
	| R_\ell^{(0)} |\lesssim ( \upsilon\braket{\ell}^{-(\tau+1)} )^{\frac{1}{m_0}}  \,, & \quad | R_{\ell,j}^{(I)} |\lesssim \big( \upsilon |j|^{\frac32}\braket{\ell}^{-(\tau+1)} \big)^{\frac{1}{m_0}} \,,\\
	| R_{\ell,j,j'}^{(II)} |\lesssim \big( \upsilon \,
	\langle | j|^\frac32 - | j'|^\frac32 \rangle\braket{\ell}^{-(\tau+1)} \big)^{\frac{1}{m_0}}\,, & \quad | Q_{\ell,j,j'}^{(II)} |\lesssim \big( \upsilon\,\big( |j|^\frac32 + |j'|^{\frac32} \big)\braket{\ell}^{-(\tau+1)} \big)^{\frac{1}{m_0}} \, ,
	\end{align*}
	and, recalling   Lemma \ref{lem:emptysets}, 
	\begin{equation*}
	| R_{\ell,j}^{(I)} | \, , \ | R_{\ell,j,j'}^{(II)} | \, , \ | Q_{\ell,j,j'}^{(II)} | \lesssim 
	( \upsilon\braket{\ell}^{-\tau})^{\frac{1}{m_0}} \,.
	\end{equation*}
\end{lem}
\begin{proof}
	We  estimate  $R_{\ell,j,j'}^{(II)}$ defined in \eqref{RII0}. 
	The other cases follow similarly.
	Defining
	$ f_{\ell,j,j'}(\kappa):= ( \ora{\Omega}_\varepsilon(\kappa)\cdot \ell + \mu_j^\infty(\kappa)-\mu_{j'}^\infty(\kappa)  )\braket{\ell}^{-1} $, 
	we  write
$$
	R_{\ell,j,j'}^{(II)} = \big\{ \kappa\in[\kappa_1,\kappa_2]  \,:\, \abs{ f_{\ell,j,j'}(\kappa) } < 4\upsilon\,\langle | j|^\frac32 - | j'|^\frac32 \rangle\braket{\ell}^{-\tau-1} \big\} \,.
$$
By Lemma \ref{lem:emptysets} we restrict to indexes satisfying 
$ \big| | j|^\frac32 -|j'|^\frac32 \big|\leq C \braket{\ell} $.
	By \eqref{eq:2_meln-_pert},
	\begin{equation*}
	\max_{0\leq n \leq m_0}\abs{ \pa_\kappa^n f_{\ell,j,j'}(\kappa) } \geq \rho_0/2 \,, \quad \forall\,\kappa\in[\kappa_1,\kappa_2] \,.
	\end{equation*}
	In addition, by \eqref{Om-per}-\eqref{small_rem_k}, Lemma \ref{rem:exp_omegaj_k}, 
	  the momentum condition $ |j-j'| = | \vec \jmath \cdot \ell |$, and  \eqref{12j12j'}, we deduce that 
	 $\max_{0\leq n \leq k_0}\abs{ \pa_\kappa^n f_{\ell,j,j'}(\kappa) }\leq C$ for all $\kappa\in[\kappa_1,\kappa_2]$, provided $\varepsilon\upsilon^{-(1+k_0)}$ is small enough, namely, by \eqref{param_small_meas} and $\varepsilon$ small enough. In particular, $f_{\ell,j,j'}$ is of class $\cC^{k_0-1}=\cC^{m_0+1}$. Thus Theorem 17.1 in \cite{Russ} applies.
\end{proof}

\begin{proof}[Proof of Theorem \ref{MEASEST} completed.]
We estimate the measure of all the sets in  \eqref{union_gec}. 
By Lemma \ref{lem:emptysets} and Lemma \ref{lem:meas_res} we have  that
\begin{align}
\label{R0est}
&\Big| \bigcup_{\ell \neq 0} R^{(0)}_\ell \Big| \leq \sum_{\ell \neq 0} | R^{(0)}_\ell| \lesssim \sum_{\ell\neq 0} \Big( \frac{\upsilon}{\braket{\ell}^{\tau +1}} \Big)^{\frac{1}{m_0}} \,,\\
\label{R1est}
&\abs{\bigcup_{\ell,  \, j\in\S_0^c \atop \ora{\jmath}\cdot\ell+j=0}R_{\ell,j}^{(I)}} \leq 
\sum_{\abs{j} \leq C \braket{\ell}^\frac23 \atop \ora{\jmath}\cdot\ell+j=0} | R_{\ell,j}^{(I)}| 
\lesssim  \sum_{\abs j \leq C \braket{\ell}^\frac23 } \Big( \frac{\upsilon}{\braket{\ell}^{\tau}} \Big)^{\frac{1}{m_0}} \lesssim \sum_{\ell \in \Z^\nu} \frac{\upsilon^{\frac{1}{m_0}} }{\la \ell \ra^{\frac{\tau}{m_0} - \frac23}}\,, \\
\label{Q2est}
&\abs{\bigcup_{\ell, \, j,  j'\in\S_0^c\atop \ora{\jmath}\cdot \ell +j+j'=0} Q_{\ell,j,j'}^{(II)}} \leq \sum_{\abs j, \abs{j'}\leq C \braket{\ell}^\frac23 } |Q_{\ell,j,j'}^{(II)}|
\lesssim
\sum_{\abs j, \abs{j'} \leq C \braket{\ell}^\frac23 } \left( \frac{\upsilon}{\braket{\ell}^{\tau}} \right)^{\frac{1}{m_0}}   \lesssim \sum_{\ell\in\Z^\nu} \frac{ \upsilon^{\frac{1}{m_0}}}{\braket{\ell}^{\frac{\tau}{m_0}-\frac43}} \, . 
\end{align}
We are left with estimating the measure of 
\begin{equation}\label{union_gec_RII}
	\bigcup_{(\ell,j,j')\neq(0,j,j), j \neq j' \atop \ora{\jmath}\cdot\ell+j-j'=0 } R_{\ell,j,j'}^{(II)} = 
	  \left( \bigcup_{\ell,j\in\S_0^c \atop \ora{\jmath}\cdot\ell+2j=0 } R_{\ell,j,-j}^{(II)} \right)
	   \cup 
	   \left(\bigcup_{\ell,j,j'\,, \ \abs j\neq |j'|\atop \ora{\jmath}\cdot\ell+j-j'=0 } R_{\ell,j,j'}^{(II)}\right)\,.
\end{equation}
By the momentum condition $\ora{\jmath}\cdot \ell + 2j =0$ we get 
$|j| \leq C \la \ell \ra$, and,  by Lemma \ref{lem:meas_res}, 
\begin{equation}
\label{R2est2}
	\Big|   \bigcup_{\ell, j\in\S_0^c, \ora{\jmath}\cdot\ell+2j=0 } R_{\ell,j,-j}^{(II)} \Big| \leq
	\sum_{  |j| \leq C \la \ell \ra}
\big| R_{\ell,j,-j}^{(II)} \big| 
\lesssim 
\sum_{ \abs j \leq C \braket{\ell} } \left( \frac{\upsilon}{\braket{\ell}^{\tau}} \right)^{\frac{1}{m_0}}   \lesssim \sum_{\ell\in\Z^\nu} \frac{ \upsilon^{\frac{1}{m_0}}}{\braket{\ell}^{\frac{\tau}{m_0}- 1}} \, . 
\end{equation}
Finally we estimate the measure of the second union in \eqref{union_gec_RII}. 
By Lemma \ref{lem:emptysets} we can restrict  to  indexes satisfying 
$
| |j|^{3/2} - |j'|^{3/2} | \leq C \la \ell \ra \, . 
$
Now, for any  $|j| \neq |j'|$, we have 
\begin{align*}
\big| |j|^{\frac32} - |j'|^{\frac32} \big| &= \big| |j|^{\frac12} - |j'|^{\frac12} \big|\, 
\big(|j| + |j'|  + |j|^{\frac12} |j'|^{\frac12} \big)
 \geq \frac{|j| + |j'|  + |j|^{\frac12} |j'|^{\frac12}}{|j|^{\frac12} + |j'|^{\frac12}} \geq
\frac{|j|^{\frac12} + |j'|^{\frac12}}{2} \, ,  
 \end{align*} 
implying the  upper bounds $|j|, |j'| \leq C \la \ell \ra^2$.
Hence 
\begin{align}
\label{R2est3}
\abs{ \bigcup_{\ell,j,j'\,, \, \abs j\neq |j'|\atop \ora{\jmath}\cdot\ell+j-j'=0 } R_{\ell,j,j'}^{(II)}}
 \leq
\sum_{ |j| , |j'| \leq C \la \ell \ra^2  } |R_{\ell,j,j'}^{(II)}|
\lesssim 
\sum_{\abs{j}, |j'| \leq C \braket{\ell}^2 } \left( \frac{\upsilon}{\braket{\ell}^{\tau}} \right)^{\frac{1}{m_0}}   \lesssim \sum_{\ell\in\Z^\nu} \frac{ \upsilon^{\frac{1}{m_0}}}{\braket{\ell}^{\frac{\tau}{m_0}- 4}} \, . 
\end{align}
As  $\frac{\tau}{m_0}- 4> \nu$ by \eqref{param_small_meas}, all the series in \eqref{R0est}, \eqref{R1est}, \eqref{Q2est}, \eqref{R2est2}, \eqref{R2est3} are convergent, and we deduce 
	\begin{equation*}
	\abs{\cG_\varepsilon^c} \leq C \upsilon^{\frac{1}{m_0}} \,.
	\end{equation*}
For $\upsilon= \varepsilon^\ta$ as in \eqref{param_small_meas}, we get 
$| \cG_\varepsilon | \geq \kappa_2-\kappa_1 - C \varepsilon^{\ta/m_0} $. The proof of Theorem \ref{MEASEST} is concluded.
\end{proof}

\section{Approximate inverse}\label{sec:approx_inv}
In order to implement a convergent Nash-Moser scheme that leads to a solution of $\cF(i,\alpha)=0$, where $ \cF (i, \alpha) $ is the nonlinear operator  defined in \eqref{F_op},  
we construct an \emph{almost approximate right inverse} of the linearized operator
\begin{equation*}
\di_{i,\alpha}\cF(i_0,\alpha_0)[\whi,\wh\alpha] = \omega\cdot \pa_\vf \whi - \di_i X_{H_\alpha}\left( i_0(\vf) \right)[\whi] - \left(\wh\alpha,0,0\right) \,.
\end{equation*}
Note that $\di_{i,\alpha}\cF(i_0,\alpha_0)=\di_{i,\alpha}\cF(i_0)$ is independent of $\alpha_0$.
We assume that the torus $ i_0 (\vf) = ( \theta_0 (\vf), I_0 (\vf), w_0 (\vf)) $ 
is  reversible and traveling,  according to  \eqref{RTTT}.

In the sequel we shall assume  the smallness condition,  
for some $\tk := \tk (\tau,\nu)>0$, 
$$
\varepsilon\upsilon^{-\tk} \ll 1 \, . 
$$
We closely follow the  strategy presented in \cite{BB} and implemented for the water waves 
equations in \cite{BM,BBHM}. The main novelty is to check that this construction preserves
the momentum preserving properties needed for the search of traveling waves. Therefore, 
along this section we shall focus on this verification. 
The estimates are very similar to those in \cite{BM,BBHM}. 

\smallskip

First of all, we state  tame estimates for the composition operator induced by the Hamiltonian vector field $X_{P}= ( \pa_I P , - \pa_\theta P, \Pi_{\S^+,\Sigma}^\angle J \nabla_{w} P )$ in \eqref{F_op}.
\begin{lem}{\bf (Estimates of the perturbation $P$)} \label{XP_est}
	Let $\fI(\vf)$ in \eqref{ICal} satisfy $\norm{ \fI }_{3 s_0 + 2 k_0 + 5}^{k_0,\upsilon}\leq 1$. Then, for any $ s \geq s_0 $, 
	$	\norm{ X_{P}(i) }_{s}^{k_0,\upsilon} \lesssim_s 1 + \norm{ \fI }_{s+2 s_0 + 2 k_0 + 3}^{k_0,\upsilon} $, 
	and, for all $\whi:= (\wh\theta,\whI,\whw)$,
	\begin{align*}
	\norm{ \di_i X_{P}(i)[\whi] }_{s}^{k_0,\upsilon} &\lesssim_s \norm{\whi}_{s+1}^{k_0,\upsilon} + \norm{ \fI }_{s+2 s_0 + 2 k_0 + 4}^{k_0,\upsilon}\norm{ \whi}_{s_0+1}^{k_0,\upsilon} \,, \\
	\norm{ \di_i^2 X_{P}(i)[\whi,\whi] }_{s}^{k_0,\upsilon} &\lesssim_s \norm{\whi}_{s+1}^{k_0,\upsilon}\norm{\whi}_{s_0+1}^{k_0,\upsilon} + \norm{ \fI }_{s+2 s_0 + 2 k_0 + 5}^{k_0,\upsilon} ( \norm{\whi}_{s_0+1}^{k_0,\upsilon} )^2 \,.
	\end{align*}
\end{lem}
\begin{proof}
	The proof goes as in Lemma 5.1 of \cite{BM}, using also the 
	estimates of the Dirichlet-Neumann operator in Lemma \ref{DN_pseudo_est}.
\end{proof}
Along this section, we assume the following hypothesis, which is verified by the approximate solutions obtained at each step of the Nash-Moser Theorem \ref{NASH}.
\begin{itemize}
	\item {\sc ANSATZ.} The map $(\omega,\kappa)\mapsto \fI_0(\omega,\kappa) = i_0(\vf;\omega,\kappa)- (\vf,0,0)$ is $k_0$-times differentiable with respect to the parameters $(\omega,\kappa)\in \R^\nu\times [\kappa_1,\kappa_2]$ and, for some $\mu:=\mu(\tau,\nu)>0$, $\upsilon\in (0,1)$,
	\begin{equation}\label{ansatz}
	\norm{\fI_0}_{s_0+\mu}^{k_0,\upsilon} + \abs{ \alpha_0-\omega }^{k_0,\upsilon} \leq C \varepsilon \upsilon^{-1} \,.
	\end{equation} 
\end{itemize}

As in \cite{BB,BM,BBHM}, we first modify the approximate torus $i_0 (\vf) $ to obtain a nearby isotropic torus $i_\delta (\vf) $, namely such that the pull-back 1-form  $i_\delta^*\Lambda $  is closed, 
where $\Lambda$ is the Liouville 1-form defined in 
\eqref{liouville}.  
We first consider the pull-back $ 1$-form 
\begin{align}\label{ak}
i_0^*\Lambda & = \sum_{k=1}^{\nu} a_k(\vf) \di \vf_k \, , \quad
a_k(\vf) := -\big( [ \pa_\vf \theta_0(\vf) ]^\top I_0(\vf) \big)_k +\tfrac12 
\big( J_\angle^{-1} w_0(\vf), \pa_{\vf_k} w_0(\vf) \big)_{L^2} \, , 
\end{align} 
and its exterior differential 
\begin{align*}
i_0^*\cW & = \di i_0^*\Lambda = \sum_{1\leq k < j \leq \nu} A_{kj} \di \vf_k \wedge \di \vf_j \,, 
\quad 
A_{kj}(\vf)  := \pa_{\vf_k} a_j(\vf) - \pa_{\vf_j}a_k(\vf) \, .
\end{align*}
By the formula given in Lemma 5 in \cite{BB}, we deduce, 
if $\omega$ belongs to ${\tt DC}(\upsilon,\tau) $, the estimate 
\begin{equation*}\label{stimaAjk}
\norm{ A_{kj} }_s^{k_0,\upsilon} \lesssim_s  \upsilon^{-1}\big( \norm{ Z }_{s+\tau(k_0+1)+k_0+1}^{k_0, \upsilon} + \norm{ Z }_{s_0+1}^{k_0,\upsilon} \norm{ \fI_0 }_{s+\tau(k_0+1)+k_0+1}^{k_0,\upsilon} \big) \, , 
\end{equation*}
where $Z(\vf)   $  is the ``error function''
\begin{equation*}\label{ZError}
Z(\vf)   
:= \cF(i_0,\alpha_0)(\vf) = \omega\cdot \pa_\varphi i_0(\vf) - X_{H_{\alpha_0}}(i_0(\vf))\,.
\end{equation*}
Note that if $ Z (\vf) = 0 $, the torus $ i_0 (\vf) $ is invariant for $ X_{H_{\alpha_0}} $ and
 the  1-form $ i_0^* \Lambda $ is closed, namely the torus $ i_0 (\vf) $ is isotropic. 
 We denote below the Laplacian $\Delta_\vf:= \sum_{k=1}^{\nu}\pa_{\vf_k}^2$.
\begin{lem} {\bf (Isotropic torus)} \label{torus_iso}
	The torus $i_\delta(\vf):= ( \theta_0(\vf),I_\delta(\vf),w_0(\vf) )$, defined by
	\begin{equation}\label{Idelta}
	I_\delta(\vf):= I_0(\vf) + [ \pa_\vf \theta_0(\vf) ]^{-\top}\rho(\vf) \,, 
	\quad \rho = (\rho_j)_{j=1, \ldots,\nu} \, , \quad \rho_j(\vf)
	:= \Delta_\vf^{-1} \sum_{k=1}^{\nu}\pa_{\vf_k}A_{kj}(\vf)\,,
	\end{equation}
	is isotropic. Moreover, there is $\sigma:= \sigma(\nu,\tau)$ such that, for all 
	$ s \geq s_0 $, 
	\begin{align}
	\norm{ I_\delta-I_0 }_s^{k_0,\upsilon} &\lesssim_s \norm{\fI_0}_{s+1}^{k_0,\upsilon} \, ,    \label{ebb1} \\
	 \norm{ I_\delta-I_0 }_s^{k_0,\upsilon} & \lesssim_s \upsilon^{-1}
	\big( \norm{Z}_{s+\sigma}^{k_0,\upsilon} +\norm{Z}_{s_0+\sigma}^{k_0,\upsilon} \norm{ \fI_0 }_{s+\sigma}^{k_0,\upsilon} \big)  \\
	\norm{ \cF(i_\delta,\alpha_0) }_s^{k_0,\upsilon} &\lesssim_s  \norm{Z}_{s+\sigma}^{k_0,\upsilon} +\norm{Z}_{s_0+\sigma}^{k_0,\upsilon} \norm{ \fI_0 }_{s+\sigma}^{k_0,\upsilon} \\
	\norm{ \di_i(i_\delta)[\whi] }_{s_1} &
	\lesssim_{s_1} \norm{ \whi }_{s_1+1}  \, ,   \label{ebb3} 
	\end{align}
	for $  s_1 \leq s_0 + \mu $ (cfr. \eqref{ansatz}).
	Furthermore  $i_\delta(\vf)$
	is  a reversible and traveling torus,  cfr.  \eqref{RTTT}.
\end{lem}
\begin{proof}
Since  $i_0(\vf)$ is a  traveling torus  (see \eqref{mompres_aa}),
in order to prove that  $i_\delta(\vf)$ is a traveling torus it is sufficient to 
prove that $ I_\delta (\vf- \ora{\jmath} \vs)  =  I_\delta (\vf ) $, for any $\vs \in \R $. 
In view of \eqref{Idelta}, this  follows by checking that 
$	\partial_\vf \theta_0(\vf - \ora{\jmath}\vs) = \partial_\vf \theta_0(\vf)$
and $\rho(\vf - \ora{\jmath}\vs) = \rho(\vf) $ for any $\vs \in \R$. 
The first identity is a trivial consequence of the fact that 
$ \theta_0(\vf - \ora{\jmath}\vs ) = \theta_0 (\vf) - \ora{\jmath} \vs $ for any $ \vs \in \R $, 
whereas the second one follows once we prove that the functions $ a_k (\vf)$ defined in 
\eqref{ak} satisfy 
	\begin{equation}
	\label{ak.trav}
	a_k(\vf - \ora{\jmath}\vs) = a_k(\vf) \quad \forall\, \vs \in \R \, , \ \ \forall k  = 1 , \ldots, \nu \, . 
	\end{equation}
Using that 
$i_0(\vf)$ is a traveling torus, we get, for any $\vs \in \R $,  
 $$
 \left( \pa_{\vf_k} w_0(\vf- \ora{\jmath}\vs ), J_\angle^{-1}  w_0(\vf- \ora{\jmath}\vs) \right)_{L^2} = 
  \left( \pa_{\vf_k} \tau_\vs w_0(\vf ), J_\angle^{-1} \tau_\vs w_0(\vf) \right)_{L^2} =
    \left( \pa_{\vf_k} w_0(\vf ), J_\angle^{-1}  w_0(\vf) \right)_{L^2} 
 $$
 and, recalling \eqref{ak},  we deduce \eqref{ak.trav}. 
Moreover, since  $ i_0 (\vf) $ is reversible, in order to prove that 
$ i_\delta (\vf) $ is reversible as well, it is sufficient to show that 
$ I_\delta (\vf ) $ is even.
 This follows by \eqref{ak}, 
Lemma \ref{proj_rev} and $ \cS J^{-1} = - J^{-1} \cS $.
Finally, the estimates \eqref{ebb1}-\eqref{ebb3} follow e.g. as in Lemma 5.3 in \cite{BBHM}.
\end{proof}

In the sequel we denote by $\sigma = \sigma(\nu,\tau) $ constants, 
which may increase from lemma to lemma, which represent "loss of derivatives".

In order to find an approximate inverse of the linearized operator $\di_{i,\alpha}\cF(i_\delta)$, we introduce the symplectic diffeomorphism $G_\delta:(\phi,y,\tw) \rightarrow (\theta,I,w)$ of the phase space $ \T^\nu\times \R^\nu \times \acca_{\S^+,\Sigma}^\angle$, 
\begin{equation}\label{Gdelta}
\begin{pmatrix}
\theta \\ I \\ w
\end{pmatrix} := G_\delta \begin{pmatrix}
\phi \\ y \\ \tw
\end{pmatrix} := \begin{pmatrix}
\theta_0(\phi) \\ 
I_\delta(\phi) + \left[ \pa_\phi \theta_0(\phi) \right]^{-\top}y + \left[(\pa_\theta\wtw_0)(\theta_0(\phi))  \right]^\top J_\angle^{-1} \tw \\
w_0(\phi) + \tw
\end{pmatrix}\,,
\end{equation}
where $\wtw_0(\theta):= w_0(\theta_0^{-1}(\theta))$.
It is proved in Lemma 2 of \cite{BB} that $G_\delta$ is symplectic, because the torus $i_\delta$ is isotropic (Lemma \ref{torus_iso}). In the new coordinates, $i_\delta$ is the trivial embedded torus $(\phi,y,\tw)=(\phi,0,0)$.
\begin{lem}
The diffeomorphism $G_\delta$ in \eqref{Gdelta} is reversibility and momentum preserving, 
in the sense that 
\begin{equation}
\label{Gdelta.inv}
\vec \cS \circ G_\delta = G_\delta \circ \vec  \cS  \, , \quad  
\vec \tau_\vs \circ G_\delta = G_\delta \circ \vec \tau_\vs \, , \quad \forall \,\vs \in \R \, , 
\end{equation}
where $\vec \cS $ and $ \vec \tau_\vs $ are defined respectively in \eqref{rev_aa}, \eqref{vec.tau}. 
\end{lem}
\begin{proof}
We prove the second identity in  \eqref{Gdelta.inv}, which, in view of
\eqref{Gdelta}, \eqref{vec.tau} amounts to   
\begin{align}
\label{Gdelta.inv1}
&\theta_0(\phi) - \ora{\jmath}\vs = \theta_0(\phi - \ora{\jmath}\vs) \, , \ \forall \vs \in \R\,, \\
\label{Gdelta.inv2}
&I_\delta(\phi) + \left[ \pa_\phi \theta_0(\phi) \right]^{-\top}y + \left[(\pa_\theta\wtw_0)(\theta_0(\phi))  \right]^\top J_\angle^{-1}\tw 
\\
\notag
& \qquad =
I_\delta(\phi- \ora{\jmath}\vs) + \left[ \pa_\phi \theta_0(\phi- \ora{\jmath}\vs) \right]^{-\top}y + \left[(\pa_\theta\wtw_0)(\theta_0(\phi- \ora{\jmath}\vs))  \right]^\top J_\angle^{-1}\tau_\vs \tw\,,
\\
\label{Gdelta.inv3}
&  \tau_\vs w_0(\phi) + \tau_\vs \tw  = w_0(\phi- \ora{\jmath}\vs) + \tau_\vs \tw  \, . 
\end{align}
Identities \eqref{Gdelta.inv1} and \eqref{Gdelta.inv3} follow because 
$i_\delta(\vf)$ is a traveling torus (Lemma \ref{torus_iso}). 
For the same reason 
$I_\delta(\phi) = I_\delta(\phi- \ora{\jmath}\vs) $ and
$ \pa_\phi \theta_0(\phi) =  \pa_\phi \theta_0(\phi- \ora{\jmath}\vs)$ for any $\vs \in \R$. Hence, for verifying \eqref{Gdelta.inv2} it is sufficient to check that
$ [(\pa_\theta\wtw_0)(\theta_0(\phi)) ]^\top   = [(\pa_\theta\wtw_0)(\theta_0(\phi- \ora{\jmath}\vs))  ]^\top \tau_\vs $
(we have used that $J_\angle^{-1}$ and $\tau_\vs $ commute by Lemma \ref{lem:proj.momentum}), which in turn follows by
\begin{equation}
\label{Gdelta.inv4}
\tau_\vs \circ (\pa_\theta\wtw_0)(\theta_0(\phi)) 
= 
(\pa_\theta\wtw_0)(\theta_0(\phi- \ora{\jmath}\vs)) \, ,  \quad \forall \vs \in \R \, , 
\end{equation}
by taking the transpose and using that $\tau_\vs^\top = \tau_{-\vs } = \tau_\vs^{-1} $.
We claim that  \eqref{Gdelta.inv4} is implied by 
$\wt w_0$ being a traveling wave, i.e.
\begin{equation}
\label{Gdelta.inv5}
\tau_\vs \wtw_0(\theta, \cdot) = \wtw_0(\theta - \ora{\jmath} \vs)  \, , \quad \forall \varsigma \in \R \, .
\end{equation}
Indeed, taking the differential of \eqref{Gdelta.inv5} with respect to $\theta$, evaluating at $\theta  = \theta_0(\vf) $,  and using that $\theta_0(\vf) - \ora{\jmath}\vs = \theta_0(\vf- \ora{\jmath}\vs)$ one deduces \eqref{Gdelta.inv4}.
It remains to prove \eqref{Gdelta.inv5}. By the  definition of $\wtw_0$,
and since $w_0$ is a traveling wave,  we have 
\begin{align*}
\wtw_0(\theta - \ora{\jmath}\vs) 
& = w_0(\theta_0^{-1}(\theta - \ora{\jmath}\vs)) 
 =  w_0(\theta_0^{-1}(\theta) - \ora{\jmath}\vs) 
 = \tau_\vs w_0(\theta_0^{-1}(\theta)) = \tau_\vs \wtw_0 \, ,  
\end{align*}
using also that 
$ \theta_0^{-1}(\theta - \ora{\jmath}\vs) = \theta_0^{-1}(\theta) - \ora{\jmath}\vs $,
which follows by inverting \eqref{Gdelta.inv1}.
The proof of the first identity in  \eqref{Gdelta.inv} follows by \eqref{Gdelta}, \eqref{rev_aa},
the fact that $ i_\delta $ is reversible, Lemma \ref{proj_rev} and 
since $ J^{-1} $ and $ \cS $ anti-commute.  
\end{proof}
 Under the symplectic diffeomorphism $G_\delta $, the Hamiltonian vector field $X_{H_\alpha}$ changes into
\begin{equation}\label{Kalpha}
X_{K_\alpha} = \left(DG_\delta  \right)^{-1} X_{H_\alpha} \circ G_\delta 
\qquad {\rm where} \qquad K_\alpha := H_\alpha \circ G_\delta \,.
\end{equation}
By  \eqref{Gdelta.inv} and  \eqref{Halpha.symm} we deduce  that $ K_\alpha $ is reversible and momentum preserving, in the sense that 
\begin{equation}\label{Ka.prop}
K_\alpha\circ \vec \cS = K_\alpha \, , \quad 
K_\alpha\circ \vec \tau_\vs = K_\alpha \, , \ \  \forall\, \vs \in \R \, . 
\end{equation} 
The Taylor expansion of $K_\alpha$ at the trivial torus $(\phi,0,0)$ is
\begin{equation}\label{taylor_Kalpha}
\begin{aligned}
K_\alpha(\phi,y,\tw) =& \ K_{00}(\phi,\alpha) + K_{10}(\phi,\alpha) \cdot y + 
( K_{01}(\phi,\alpha),\tw )_{L^2} + \tfrac12 K_{20}(\phi) y\cdot y \\
& + ( K_{11}(\phi)y,\tw )_{L^2} + \tfrac12 ( K_{02}(\phi)\tw,\tw )_{L^2} + K_{\geq 3}(\phi,y,\tw)\,,
\end{aligned}
\end{equation}
where $K_{\geq 3}$ collects all terms at least cubic in the variables $(y,\tw)$. By \eqref{Halpha} and \eqref{Gdelta}, the only Taylor coefficients that depend on $\alpha$ are $K_{00}\in \R$, $K_{10}\in\R^\nu$ and $K_{01}\in \acca_{\S^+,\Sigma}^\angle $, whereas the $ \nu \times \nu $ symmetric matrix $K_{20} $, $K_{11}\in\cL ( \R^\nu,\acca_{\S^+,\Sigma}^\angle )$ and the linear self-adjoint operator $ K_{02} $,  acting on 
$\acca_{\S^+,\Sigma}^\angle $, 
are independent of it. 

Differentiating the identities in \eqref{Ka.prop} at $(\phi,0,0)$, we have
(recalling \eqref{rev_aa}) 
\begin{equation}\label{Ka.rev}
\begin{aligned}
& K_{00}(-\phi) = K_{00}(\phi)\,, \quad  K_{10}(-\phi) = K_{10}(\phi)\,, \quad K_{20}(-\phi) = K_{20}(\phi)\,, \\
&  \cS \circ K_{01}(-\phi)  = K_{01}(\phi)\,, \quad \cS \circ K_{11}(-\phi) = K_{11}(\phi)\,,  \quad K_{02}(-\phi)\circ \cS =  \cS \circ K_{02}(\phi)\,,
\end{aligned}
\end{equation}
and, recalling \eqref{vec.tau} and 
using that $\tau_\vs^\top = \tau_{-\vs } = \tau_\vs^{-1} $, for any $\vs \in \R $, 
\begin{equation}
\label{K.symm.2}
\begin{aligned}
&K_{00}(\phi- \ora{\jmath} \vs) = K_{00}(\phi) \,, \quad  K_{10}(\phi- \ora{\jmath} \vs) = K_{10}(\phi)\,, \quad K_{20}(\phi- \ora{\jmath} \vs) = K_{20}(\phi)\,, \\
&K_{01}(\phi- \ora{\jmath} \vs) =  \tau_\vs K_{01}(\phi) \,, \quad K_{11}(\phi- \ora{\jmath} \vs) =  \tau_\vs K_{11}(\phi) \,, \quad 
K_{02}(\phi- \ora{\jmath} \vs) \circ \tau_\vs = \tau_\vs \circ K_{02}(\phi) \, . 
\end{aligned}
\end{equation}
The Hamilton equations associated to \eqref{taylor_Kalpha} are
\begin{equation}\label{hameq_Kalpha}
\begin{cases}
\dot\phi =   K_{10}(\phi,\alpha) + K_{20}(\phi)y + [K_{11}(\phi)]^\top \tw + \pa_y K_{\geq 3}(\phi,y,\tw) \\
\dot y =   - \pa_\phi K_{00}(\phi,\alpha) - [\pa_\phi K_{10}(\phi,\alpha)]^\top y - [\pa_\phi K_{01}(\phi,\alpha)]^\top \tw  \\
\ \ \ \ \ - \pa_\phi\left( \tfrac12 K_{20}(\phi)y\cdot y + \left( K_{11}(\phi)y,\tw \right)_{L^2} + \tfrac12 \left( K_{02}(\phi)\tw,\tw \right)_{L^2} + K_{\geq 3}(\phi,y,\tw) \right) \\
\dot\tw =  J_\angle \, \left( K_{01}(\phi,\alpha)+ K_{11}(\phi)y + K_{02}(\phi)\tw + \nabla_{\tw} K_{\geq 3}(\phi,y,\tw) \right) 
\end{cases}\,,
\end{equation}
where $\pa_\phi K_{10}^\top $ is the $\nu\times\nu$ transposed matrix and $\pa_\phi K_{01}^\top , K_{11}^\top: \acca_{\S^+,\Sigma}^\angle\rightarrow\R^\nu$ are defined by the duality relation 
$ (\pa_\phi K_{01}[\wh\phi],\tw  )_{L^2}=\wh\phi\cdot [\pa_\phi K_{01} ]^\top \tw $ 
for any $\wh\phi\in\R^\nu$, $\tw\in\acca_{\S^+,\Sigma}^\angle$. 
The transpose
$	K_{11}^\top (\phi) $ is defined similarly. 
	
On an exact solution (that is $Z=0$), the terms
$K_{00}, K_{01}$  in the Taylor expansion \eqref{taylor_Kalpha} vanish and $K_{10}= \omega$. 
More precisely, arguing as  in Lemma 5.4 in \cite{BBHM}, we have 

\begin{lem}\label{Kcoeff_est}
There is $ \sigma := \sigma (\nu, \tau) > 0 $, such that, for all $ s \geq s_0 $, 
	\begin{align*}
	&\norm{ \pa_\phi K_{00}(\cdot , \alpha_0) }_s^{k_0,\upsilon} + \norm{ K_{10}(\cdot ,\alpha_0)-\omega}_s^{k_0,\upsilon} + \norm{ K_{01}( \cdot ,\alpha_0) }_s^{k_0,\upsilon} \lesssim_s  \norm{Z}_{s+\sigma}^{k_0,\upsilon} + \norm{Z}_{s_0+\sigma}^{k_0,\upsilon} \norm{ \fI_0 }_{s+\sigma}^{k_0,\upsilon} \,, \\
	&\norm{\pa_\alpha K_{00}}_s^{k_0,\upsilon} + \norm{ \pa_\alpha K_{10}-{\rm Id} }_s^{k_0,\upsilon} + \norm{ \pa_\alpha K_{01} }_s^{k_0,\upsilon} \lesssim_s \norm{ \fI_0 }_{s+\sigma}^{k_0,\upsilon} \,, \\
	& \norm{ K_{20} }_s^{k_0,\upsilon}\lesssim_s \varepsilon ( 1 + \norm{ \fI_0 }_{s+\sigma}^{k_0,\upsilon} )\,,  \\
	& \norm{ K_{11}y }_s^{k_0,\upsilon} \lesssim_s \varepsilon ( \norm{ y }_s^{k_0,\upsilon}+ \norm{y }_{s_0}^{k_0,\upsilon}\norm{\fI_0 }_{s+\sigma}^{k_0,\upsilon} )\,,  \ 
	 \norm{ K_{11}^\top \tw  }_s^{k_0,\upsilon} \lesssim_s \varepsilon ( \norm{\tw}_{s}^{k_0,\upsilon} + \norm{\tw}_{s_0}^{k_0,\upsilon}\norm{\fI_0}_{s+\sigma}^{k_0,\upsilon}) \, . 
	\end{align*}
	\end{lem}
Under the linear change of variables
\begin{equation*}\label{DGdelta}
DG_\delta(\vf,0,0)\begin{pmatrix}
\wh\phi \\ \why \\ \wh\tw
\end{pmatrix}:= \begin{pmatrix}
\pa_\phi \theta_0(\vf) & 0 & 0 \\ \pa_\phi I_\delta(\vf) & [\pa_\phi \theta_0(\vf)]^{-\top} &  [(\pa_\theta\wtw_0)(\theta_0(\vf))]^\top  J_\angle^{-1} \\\pa_\phi w_0(\vf) & 0 & {\rm Id}
\end{pmatrix}\begin{pmatrix}
\wh\phi \\ \why \\ \wh\tw
\end{pmatrix} \,,
\end{equation*}
the linearized operator $\di_{i,\alpha}\cF(i_\delta)$ is approximately transformed into the one obtained when one linearizes the Hamiltonian system \eqref{hameq_Kalpha} at $(\phi,y,\tw) = (\vf,0,0)$, differentiating also in $\alpha$ at $\alpha_0$ and changing $\pa_t \rightsquigarrow \omega\cdot \pa_\vf$, namely
\begin{equation}\label{lin_Kalpha}
\begin{pmatrix}
\widehat \phi  \\
\widehat y    \\ 
\widehat \tw \\
\widehat \alpha
\end{pmatrix} \mapsto
\begin{pmatrix}
\omega\cdot \pa_\vf \wh\phi - \pa_\phi K_{10}(\vf)[\wh\phi] - \pa_\alpha K_{10}(\vf)[\wh\alpha] - K_{20}(\vf)\why - [K_{11}(\vf)]^\top \wh\tw  \\
 \omega\cdot \pa_\vf\why + \pa_{\phi\phi}K_{00}(\vf)[\wh\phi]+ \pa_\alpha\pa_\phi K_{00}(\vf)[\wh\alpha] + [\pa_\phi K_{10}(\vf)]^\top \why + [\pa_\phi K_{01}(\vf)]^\top  \wh\tw  \\
 \omega\cdot \pa_\vf \wh\tw - J_\angle \,  \big( \pa_\phi K_{01}(\vf)[\wh\phi] + \pa_\alpha K_{01}(\vf)[\wh\alpha] + K_{11}(\vf) \why + K_{02}(\vf) \wh\tw \big)   
\end{pmatrix}. 
\end{equation}
In order to construct an ``almost approximate" inverse of \eqref{lin_Kalpha}, we need that
\begin{equation}\label{Lomegatrue}
\cL_\omega := \Pi_{\S^+,\Sigma}^\angle \left( \omega\cdot \pa_\vf - J K_{02}(\vf) \right)|_{\acca_{\S^+,\Sigma}^\angle}
\end{equation}
is "almost invertible" (on traveling waves) up to remainders of size $O(N_{n-1}^{-{\ta}})$, where, for $n\in\N_0$
\begin{equation}\label{scales}
N_n:= K_n^p \,, \quad K_n: = K_0^{\chi^n} \,, \quad \chi= 3/2\, .
\end{equation}
The $ (K_n)_{n \geq 0} $ is the scale used in the nonlinear Nash-Moser iteration of Section \ref{sec:NaM} and $ (N_n)_{n \geq 0} $ is the one in the reducibility scheme of 
Section \ref{sec:KAM}.
Let $H_\angle^s(\T^{\nu+1}):= H^s(\T^{\nu+1})\cap \acca_{\S^+,\Sigma}^\angle$.
\begin{itemize}
	\item[(AI)]  {\bf Almost invertibility of $\cL_\omega$}: 
	{\it There exist positive real numbers 
	$  \sigma $, $ \mu(\tb) $, $ \ta $, $ p  $, $ K_0 $ and 
	 a subset $\t\Lambda_o \subset \tD\tC(\upsilon,\tau)\times [\kappa_1,\kappa_2]$ such that, for all $(\omega,\kappa) \in \t\Lambda_o$, the operator $\cL_\omega$ may be decomposed as
	\begin{equation}\label{Lomega}
	\cL_\omega = \cL_\omega^< + \cR_\omega + \cR_\omega^\perp \,,
	\end{equation}
	where,
for every traveling wave function 
$ g\in H_\angle^{s+\sigma}(\T^{\nu+1},\R^2)$ and 
for every $  (\omega,\kappa) \in \t\Lambda_o $, there is a 
 traveling wave solution $ h \in H_\angle^{s}(\T^{\nu+1},\R^2) $ of 
$ \cL_\omega^< h = g$ satisfying, for all $s_0\leq s\leq S$,  
	\begin{equation}\label{almi4}
	\norm{ (\cL_\omega^<)^{-1}g }_s^{k_0,\upsilon} \lesssim_S \upsilon^{-1}
	\big( \norm g_{s+\sigma}^{k_0,\upsilon}+ \norm g_{s_0+\sigma}^{k_0,\upsilon}\norm{ \fI_0}_{s+\mu({\tb})+\sigma}^{k_0,\upsilon} \big) \,.
	\end{equation}
In addition, if $ g $ is anti-reversible, then $ h $ is reversible. Moreover, 
for any $s_0\leq s \leq S$, 
	for any traveling wave $ h \in \acca_{\S^+,\Sigma}^\angle    $, the operators $\cR_\omega, \cR_\omega^\perp$ satisfy the estimates
	\begin{align*}
	\norm{ \cR_\omega h }_s^{k_0,\upsilon} & \lesssim_S  \varepsilon \upsilon^{-1}N_{n-1}^{- \ta} \big( \norm h_{s+\sigma}^{k_0,\upsilon}+ \norm h_{s_0+\sigma}^{k_0,\upsilon} \norm{ \fI_0}_{s+\mu(\tb)+\sigma}^{k_0,\upsilon} \big) \,, 
	\\
	\norm{ \cR_\omega^\perp h}_{s_0}^{k_0,\upsilon} & \lesssim_S K_n^{-b} \big( \norm h_{s_0+b+\sigma}^{k_0,\upsilon} + \norm h_{s_0+\sigma}^{k_0,\upsilon}\norm{\fI_0}_{s_0+\mu(\tb)+\sigma+b} \big) \,, \ \forall\, b>0 \,,
	\\
	\norm{ \cR_\omega^\perp h}_s^{k_0,\upsilon} & \lesssim_S \norm h_{s+\sigma}^{k_0,\upsilon}+ \norm h_{s_0+\sigma}^{k_0,\upsilon}\norm{\fI_0}_{s+\mu(\tb)+\sigma}^{k_0,\upsilon} 
	\,.
	\end{align*}	
	}
\end{itemize}
This assumption shall be verified by Theorem \ref{almo.inve} 
at each $n$-th step of the Nash-Moser nonlinear iteration.

In order to find an almost approximate inverse of the linear operator in \eqref{lin_Kalpha} (and so of $\di_{i,\alpha}\cF(i_\delta)$), it is sufficient to  invert the operator
\begin{equation}\label{Dsys}
	\D\big[ \wh\phi,\why,\wh\tw,\wh\alpha \big]:=\begin{pmatrix}
	\omega\cdot \pa_\vf\wh\phi - \pa_\alpha K_{10}(\vf)[\wh\alpha] - K_{20}(\vf)\why- K_{11}^\top(\vf)\wh\tw \\
	\omega\cdot \pa_\vf \why +\pa_\alpha\pa_\phi K_{00}(\vf)[\wh\alpha] \\
	\cL_\omega^< \wh\tw -  J_\angle  \left( \pa_\alpha K_{01}(\vf)[\wh\alpha] + K_{11}(\vf)\why \right)
	\end{pmatrix}
\end{equation}
obtained neglecting in \eqref{lin_Kalpha} the terms $\pa_\phi K_{10}$, $\pa_{\phi\phi}K_{00}$, $\pa_\phi K_{00}$, $\pa_\phi K_{01}$ (they vanish at an exact solution by Lemma \ref{Kcoeff_est}) and the small remainders $\cR_\omega$, $\cR_\omega^\perp$ appearing in \eqref{Lomega}. 
We look for an inverse of $\D$ by solving the system
\begin{equation}\label{Dsys_tos}
	\D\big[ \wh\phi,\why,\wh\tw,\wh\alpha \big] = \begin{pmatrix}
	g_1 \\ g_2 \\ g_3
	\end{pmatrix}\,,
\end{equation}
 where $ (g_1, g_2, g_3)  $ is an anti-reversible traveling wave variation (cfr. Definition 
 \ref{trav-vari}), i.e. 
\begin{align}\label{g_revcond}
&g_1(\vf) = g_1(- \vf) , \qquad 
g_2(\vf) = - g_2(- \vf) , \qquad 
\cS g_3(\vf) = - g_3(- \vf) \, , \\
& 	\label{g.mom.thm}
g_1(\vf) = g_1(\vf - \ora{\jmath}\vs) , 
	\quad
		g_2(\vf) = g_2(\vf - \ora{\jmath}\vs) , 
		\quad
			\tau_\vs g_3(\vf) = g_3(\vf - \ora{\jmath}\vs) , 
			\  \forall \vs \in \R \, .  
\end{align}
We first consider the second equation in \eqref{Dsys}-\eqref{Dsys_tos}, that is $\omega\cdot\pa_\vf\why = g_2-\pa_\alpha\pa_\phi K_{00}(\vf)[\wh\alpha]$. By  \eqref{g_revcond} and \eqref{Ka.rev},  the right hand side of this equation is odd in $ \vf $.
In particular it has  
zero average and so 
\begin{equation}\label{why_sol}
	\why := (\omega\cdot\pa_\vf )^{-1}
	( g_2 -\pa_\alpha\pa_\phi K_{00}(\vf)[\wh\alpha] ) \, .
\end{equation}
Since $ g_2(\vf) = g_2(\vf-\ora{\jmath}\vs) $ for any $ \forall \vs \in \R $ by \eqref{g.mom.thm} 
 and
$ \pa_\alpha\pa_\phi K_{00}(\vf)[\wh\alpha]$ satisfies the same property by
\eqref{K.symm.2}, we deduce also that
\begin{equation}
\label{why.trav}
\wh y(\vf - \ora{\jmath}\vs) = \wh y(\vf) , \  \  \forall \vs \in \R \, .  
\end{equation}
Next we consider the third equation $\cL_\omega^< \wh\tw = g_3 + 
 J_\angle ( \pa_\alpha K_{01}(\vf)[\wh\alpha]+ K_{11}(\vf)\why )$.
 The right hand side of this equation is a traveling wave 
by \eqref{g.mom.thm}, \eqref{K.symm.2}, \eqref{why.trav}  
and since  $ J_\angle =  \Pi^\angle_{\S^+, \Sigma}  \, J_{| \acca_{\S^+,\Sigma}^\angle }  $ commutes with $ \tau_\vs $ (by Lemma \ref{lem:proj.momentum}).
Thus, by  assumption (AI), there is a traveling wave solution 
\begin{equation}\label{whw_sol}
	\wh\tw := ( \cL_\omega^< )^{-1} \big(  g_3 +  J_\angle
	( \pa_\alpha K_{01}(\vf)[\wh\alpha]+ K_{11}(\vf)\why ) \big) \, . 
\end{equation}
Finally, we solve the first equation in \eqref{Dsys_tos}, which, inserting \eqref{why_sol} and \eqref{whw_sol}, becomes
\begin{equation}\label{whphi_eq}
	\omega\cdot\pa_\vf \wh\phi = g_1 + M_1(\vf)[\wh\alpha]+ M_2(\vf)g_2 + M_3(\vf)g_3\,,
\end{equation}
where
\begin{align}
	M_1(\vf) & := \pa_\alpha K_{10}(\vf) - M_2(\vf)\pa_\alpha\pa_\phi K_{00}(\vf) + M_3(\vf) J_\angle \pa_\alpha K_{01} (\vf) \,,\notag  \\
	M_2(\vf) & := K_{20}(\vf) (\omega\cdot\pa_\vf )^{-1} + K_{11}^\top (\vf)\left(\cL_\omega^<\right)^{-1} J_\angle K_{11}(\vf)(\omega\cdot\pa_\vf)^{-1}\,, \notag\\
	M_3(\vf) & := K_{11}^\top(\vf)\left( \cL_\omega^< \right)^{-1} \,.\notag
\end{align}
In order to solve \eqref{whphi_eq}, we  choose $\wh\alpha$ such that the average  in $\vf$ of the right hand side is zero. By Lemma \ref{Kcoeff_est} and \eqref{ansatz}, the 
$ \vf $-average of the matrix $ M_1 $ satisfies $\braket{M_1}_\vf = {\rm Id} + O(\varepsilon \upsilon^{-1})$. 
Then, for $\varepsilon\upsilon^{-1}$ small enough, $\braket{M_1}_\vf$ is invertible and $\braket{M_1}_\vf^{-1} = {\rm Id} + O(\varepsilon\upsilon^{-1})$. Thus we define
\begin{equation}\label{whalpha_sol}
	\wh\alpha := -\braket{M_1}_\vf^{-1}\big( \braket{g_1}_\vf + \braket{M_2g_2}_\vf + \braket{M_3 g_3}_\vf \big) \, , 
\end{equation} 
and
the solution of equation \eqref{whphi_eq} 
\begin{equation}\label{whphi_sol}
	\wh\phi := ( \omega\cdot\pa_\vf )^{-1}\big(  g_1 + M_1(\vf)[\wh\alpha] + M_2(\vf)g_2 + M_3(\vf)g_3 \big)\,.
\end{equation}
Finally the property $ \wh \phi(\vf - \ora{\jmath}\vs) = \wh \phi(\vf)$ for any $ \vs \in \R $ follows by
\eqref{K.symm.2}, \eqref{why.trav} and the fact that $\wh\tw  $ in
\eqref{whw_sol} is a traveling wave. 
This proves  that $(\wh \phi, \why, \wh\tw)$ is a traveling wave variation,  i.e. 
\eqref{g.mom.thm} holds.
Moreover, using \eqref{g_revcond}, \eqref{Ka.rev},  Lemma \ref{proj_rev}, the fact that 
$ J $ and $ \cS $ anti-commutes and (AI), one checks that 
$ (\wh \phi, \why, \wh\tw )$  is  reversible, i.e. 
\begin{equation}\label{rev-def-TW}
  \wh\phi  (\vf) = - \wh\phi  (- \vf) , \qquad 
\why(\vf) =  \why (- \vf) , \qquad 
\cS \wh\tw (\vf) = \wh\tw (- \vf) \, . 
\end{equation}
In conclusion, we have obtained a  solution 
$ ( \wh\phi,\why,\wh\tw,\wh\alpha )$ of the linear system \eqref{Dsys_tos}, and,
denoting the norm 
$ \normk{(\phi,y,\tw,\alpha)}{s}:= \max \big\{ \normk{(\phi,y,\tw)}{s},\abs{\alpha}^{k_0,\upsilon} \big\} $, we have:  
\begin{prop}\label{Dsystem}
	Assume \eqref{ansatz} (with $\mu=\mu({\tb})+\sigma$) and {\rm (AI)}. Then, for all $(\omega,\kappa)\in\t\Lambda_o $, for any anti-reversible traveling wave 
	variation $ g =(g_1,g_2,g_3)$ (i.e. satisfying \eqref{g_revcond}-\eqref{g.mom.thm}),
	 system \eqref{Dsys_tos} has a solution $\D^{-1}g:= ( \wh\phi,\why,\wh\tw,\wh\alpha)$, with $ ( \wh\phi,\why,\wh\tw,\wh\alpha)$ defined in \eqref{whphi_sol},\eqref{why_sol},\eqref{whw_sol},\eqref{whalpha_sol}, 
	where $( \wh\phi,\why,\wh\tw)$ is a reversible traveling wave variation, 
	satisfying, for any $s_0\leq s\leq S$
	\begin{equation}
	\label{est.D-1g}
		\normk{\D^{-1}g}{s} \lesssim_{S} \upsilon^{-1}\big( \normk{g}{s+\sigma}+\normk{\fI_0}{s+\mu({\tb})+\sigma}\normk{g}{s_0+\sigma} \big) \,.
	\end{equation}
\end{prop}

\begin{proof}
The estimate \eqref{est.D-1g} follows  by the explicit expression of the solution  in  \eqref{why_sol}, \eqref{whw_sol}, \eqref{whalpha_sol}, \eqref{whphi_sol}, 
and Lemma \ref{Kcoeff_est},  \eqref{almi4}, \eqref{ansatz}.
\end{proof}

Finally we prove that the operator
\begin{equation}\label{bT0}
	\bT_0 := \bT_0(i_0):=  ( D\wtG_\delta )(\vf,0,0) \circ \D^{-1} \circ (D G_\delta ) (\vf,0,0)^{-1}
\end{equation}
is an almost approximate right inverse for $\di_{i,\alpha}\cF(i_0)$, where
$	\wtG_\delta(\phi,y,\tw,\alpha) := \left( G_\delta(\phi,y,\tw),\alpha \right) $
is the identity on the $\alpha$-component. 
\begin{thm} {\bf (Almost approximate inverse)} \label{alm.approx.inv}
	Assume {\rm (AI)}.  Then there is $\bar\sigma :=\bar\sigma(\tau,\nu,k_0)>0$ such that, if \eqref{ansatz} holds with $\mu=\mu(\tb)+\bar\sigma$, then, for all $(\omega,\kappa)\in\t\Lambda_o$ and for any anti-reversible traveling 
	wave variation $g:=(g_1,g_2,g_3)$ (i.e. satisfying \eqref{g_revcond}-\eqref{g.mom.thm}), the operator $\bT_0$ defined in \eqref{bT0} satisfies, for all $s_0 \leq s \leq S$,
	\begin{equation}\label{tame-es-AI}
		\normk{\bT_0 g}{s} \lesssim_{S} \upsilon^{-1} \big( \normk{g}{s+\bar\sigma} +\normk{\fI_0}{s+\mu(\tb)+\bar\sigma}\normk{g}{s_0+\bar\sigma}  \big)\,.
	\end{equation}
Moreover, the first three components of $\bT_0 g $  form a reversible 
traveling wave variation (i.e. satisfy \eqref{rev-def-TW} and \eqref{g.mom.thm}).
Finally, $\bT_0$ is an almost approximate right inverse of $\di_{i,\alpha}\cF(i_0)$, namely
	\begin{equation*}
		\di_{i,\alpha}\cF(i_0) \circ \bT_0 - {\rm Id} = \cP(i_0) + \cP_\omega(i_0)+\cP_\omega^\perp(i_0)\,,
	\end{equation*}
	where, for any traveling wave variation $ g $,  for all $s_0 \leq s \leq S$,
	\begin{align}
		\normk{\cP g}{s} & \lesssim_{S} \upsilon^{-1}
		\Big(  \normk{\cF(i_0,\alpha_0)}{s_0+\bar\sigma}\normk{g}{s+\bar\sigma}  \label{pfi1} \\
		& \qquad + \,   \big(  \normk{\cF(i_0,\alpha_0)}{s+\bar\sigma}+\normk{\cF(i_0,\alpha_0)}{s_0+\bar\sigma}\normk{\fI_0}{s+\mu(\tb)+\bar\sigma}  \big)\normk{g}{s_0+\bar\sigma}  \Big)\, , \notag  \\
		\normk{\cP_\omega g}{s} & \lesssim_{S} \varepsilon\upsilon^{-2} N_{n-1}^{-\ta} \big( \normk{g}{s+\bar\sigma}+ \normk{\fI_0}{s+\mu(\tb)+\bar\sigma}\normk{g}{s_0+\bar\sigma}  \big)\, , \\
		\normk{\cP_\omega^\perp g}{s_0} & \lesssim_{S,b} \upsilon^{-1} K_n^{-b} \left( \normk{g}{s_0+\bar\sigma+b}+\normk{\fI_0}{s_0+\mu(\tb)+b+\bar\sigma}\normk{g}{s_0+\bar\sigma} \right)\,, \quad  \forall\,b>0\,,   \\
		\normk{\cP_\omega^\perp g}{s} & \lesssim_{S} \upsilon^{-1}\big(  \normk{g}{s+\bar\sigma}+ \normk{\fI_0}{s+\mu(\tb)+\bar\sigma}\normk{g}{s_0+\bar\sigma} \big) \,. 
		\label{pfi3} 
	\end{align}
\end{thm}

\begin{proof}
We claim that the first three components of $\bT_0 g $ 
form  a reversible  traveling wave variation. 
Indeed, differentiating \eqref{Gdelta.inv}
it follows that $ DG_\delta(\vf,0,0)$, thus $  (DG_\delta(\vf,0,0))^{-1} $,  is 
reversibility and momentum preserving (cfr. \eqref{eq:mp_A_tw'}).
 In particular these operators map an (anti)-reversible, respectively traveling, 
 waves variation into a (anti)-reversible traveling waves variation (cfr. Lemma  \ref{A.mom.cons'}). 
Moreover,  by 
 Proposition \ref{Dsystem},   the operator  $ \D^{-1}$ maps an anti-reversible traveling wave
 into a vector whose first three components form a reversible traveling wave.
This proves the claim. 

We now prove that 
the operators ${\cal P}, {\cal P}_\omega$ and ${\cal P}_\omega^\bot$ 
are defined on traveling waves. They are computed e.g. in    Theorem 5.6 of \cite{BBHM}.
To define them, introduce first the linear operators
$$
R_Z [  \widehat \phi, \widehat y, \wh\tw, \widehat \alpha]
:= \begin{pmatrix}
 - \partial_\phi K_{10}(\vf, \alpha_0) [\widehat \phi ] \\
 \partial_{\phi \phi} K_{00} (\vf, \alpha_0) [ \widehat \phi ] + 
 [\partial_\phi K_{10}(\vf, \alpha_0)]^\top \widehat y + 
 [\partial_\phi K_{01}(\vf, \alpha_0)]^\top \wh\tw  \\
 -  \, J_\angle  \partial_{\phi} K_{01}(\vf, \alpha_0)[ \widehat \phi ] 
 \end{pmatrix}
$$
and 
\begin{equation}\label{defBBbot}
{\mathbb R}_\omega[\widehat \phi, \widehat y, \wh\tw, \widehat \alpha] := \begin{pmatrix}
0 \\
0 \\
{\cal R}_\omega [\wh\tw]
\end{pmatrix}\,,\qquad {\mathbb R}_\omega^\bot[\widehat \phi, \widehat y , \wh\tw, 
\widehat \alpha] := \begin{pmatrix}
0 \\
0 \\
{\cal R}_\omega^\bot[\wh\tw]
\end{pmatrix}\ . 
\end{equation}
Next, we  denote by  $ \Pi $  the projection $ (\widehat \imath, \widehat \alpha ) \mapsto \widehat \imath $,   by $\tu_\delta(\vf) = (\vf, 0, 0)$ the trivial torus, and by $\cE$, $\cE_\omega$, $\cE_\omega^\perp$  the linear operators   
 \begin{align}
& {\cal E}  :=  \di_{i, \alpha} {\cal F}(i_0 )   - \di_{i, \alpha} {\cal F}(i_\delta )  + D^2 G_\delta( {\mathtt u}_\delta) \big[ D G_\delta( {\mathtt u}_\delta)^{-1} {\cal F}(i_\delta, \alpha_0), \,  D G_\delta ({\mathtt u}_\delta)^{-1} 
 \Pi [ \, \cdot \, ] \,  \big]  \nonumber \\
&  \qquad + D G_\delta ( {\mathtt u}_\delta)R_Z D {\widetilde G}_\delta ({\mathtt u}_\delta)^{-1}\,, \nonumber \\
&  {\cal E}_\omega :=   D G_\delta ( {\mathtt u}_\delta)   {\mathbb R}_\omega    D {\widetilde G}_\delta ({\mathtt u}_\delta)^{-1}\,, \qquad 
   {\cal E}_\omega^\bot :=   D G_\delta( {\mathtt u}_\delta)  {\mathbb R}_\omega^\bot  
D {\widetilde G}_\delta ({\mathtt u}_\delta)^{-1} \, . \label{defEE}
 \end{align}
It is then proved in  Theorem 5.6 of \cite{BBHM} that 
$ {\cal P} := {\cal E} \circ {\bf T}_0 $, 
$ {\cal P}_\omega := {\cal E}_\omega \circ {\bf T}_0 $,  
$ {\cal P}_\omega^\bot :=  
 {\cal E}_\omega^\bot \circ {\bf T}_0 $. 
A direct inspection of these formulas shows that $\cP, \cP_\omega$ and $\cP_\omega^\bot$ are defined on traveling wave variations.  
In particular, note that 
the operators 
$ {\mathbb R}_\omega $, ${\mathbb R}_\omega^\bot $  in \eqref{defBBbot}
are defined only if $ \wh\tw   $ is a traveling wave, because 
the operators ${\cal R}_\omega, {\cal R}_\omega^\bot  $  defined in (AI) 
act  only on a traveling wave. However, note that, if $ g $ is a traveling wave variation, 
the third component of 
$ D {\widetilde G}_\delta ({\mathtt u}_\delta)^{-1} {\bf T}_0 g $ is a traveling wave and 
therefore the operators $ {\cal E}_\omega, {\cal E}_\omega^\bot $ 
in  \eqref{defEE} are well defined.

The estimates 
\eqref{pfi1}-\eqref{pfi3} are proved as in  Theorem 5.6 of \cite{BBHM}, using 
Lemma \ref{Dsystem}.
\end{proof}

\section{The linearized operator in the normal subspace}\label{sec:linnorm}

We now write an explicit expression of the linear operator $\cL_\omega$ defined in \eqref{Lomegatrue}.

\begin{lem}\label{lem:K02}
The Hamiltonian operator $\cL_\omega$ defined in \eqref{Lomegatrue}, acting on 
the normal subspace $ \acca_{\S^+,\Sigma}^\angle $,  has the form
\begin{equation}\label{cLomega_again}
\cL_\omega = \Pi_{\S^+,\Sigma}^\angle 
(\cL -\varepsilon J R)|_{\acca_{\S^+,\Sigma}^\angle}  \,,
\end{equation}
where :
\begin{enumerate}
\item $ \cL $ is the Hamiltonian operator 
\begin{equation}\label{cL000}
\cL := \omega\cdot \pa_\vf  - J \pa_u\nabla_u \cH(T_\delta(\vphi)) \, ,
\end{equation}
where  
$\cH$  is the water waves Hamiltonian 
in the Wahl\'en variables defined in \eqref{Ham-Wal}, evaluated at 
	\begin{equation}\label{Tdelta}
	T_\delta(\phi):=  \varepsilon A ( i_\delta (\phi) ) =  \varepsilon A\left( \theta_0(\phi),I_\delta(\phi),w_0(\phi) \right) = \varepsilon v^\intercal\left( \theta_0(\phi),I_\delta(\phi) \right) + \varepsilon w_0(\phi)\,,
	\end{equation}
the torus  $i_\delta(\vf):= ( \theta_0(\vf),I_\delta(\vf),w_0(\vf) )$ is defined in Lemma 
	\ref{torus_iso} 
	and  $A(\theta,I,w) $, $ v^\intercal(\theta,I)$ in \eqref{aacoordinates};
\item
$ R (\phi) $ has the finite rank form 
	\begin{equation}\label{finite_rank_R}
	R(\phi)[h] = \sum_{j=1}^\nu \left( h,g_j \right)_{L^2} \chi_j \,, \quad \forall\, h\in\acca_{\S^+,\Sigma}^\angle\,,
	\end{equation}
	for functions $g_j,\chi_j \in \acca_{\S^+,\Sigma}^\angle$ which satisfy, for some $\sigma:= \sigma(\tau,\nu, k_0) > 0 $,  for all $ j = 1, \ldots, \nu $, for all $s\geq s_0$, 
\begin{equation}
\label{gjchij_est}
	\begin{aligned}
\norm{ g_j }_s^{k_0,\upsilon} + \norm{ \chi_j }_s^{k_0,\upsilon} & \lesssim_s 1 + \norm{ \fI_\delta }_{s+\sigma}^{k_0,\upsilon} \,, \\
	\norm{ \di_i g_j [\whi]}_s + \norm{\di_i \chi_j [\whi]}_s & \lesssim_s \norm{ \whi }_{s+\sigma} + \norm{ \whi }_{s_0+\sigma} 
	\norm{ \fI_\delta }_{s+\sigma} \,.
	\end{aligned}
	\end{equation}
\end{enumerate}
The operator $ \cL_\omega $ is reversible and momentum preserving.
\end{lem}

\begin{proof}
In view of \eqref{taylor_Kalpha}, \eqref{Kalpha} and \eqref{Halpha} we have 
	\begin{align}\label{K02_exp}
	K_{02}(\phi) & = \pa_\tw \nabla_{\tw} K_\alpha (\phi, 0,0) = \pa_\tw \nabla_{\tw}\left( H_\alpha\circ G_\delta \right) (\phi,0,0) \notag \\
	& = \Pi^{L^2}_{\angle} \b\Omega_W|_{\acca_{\S^+,\Sigma}^\angle} + \varepsilon \pa_\tw \nabla_{\tw} \left( P\circ G_\delta \right)(\phi,0,0) \,,
	\end{align}
	where $\b\Omega_W$ is defined in \eqref{eq:lin00_wahlen} and 
	$G_\delta$ in \eqref{Gdelta}. Differentiating with respect to $\tw$ the Hamiltonian 
$$
	(P\circ G_\delta)(\phi,y,\tw) = 
	P \big( \theta_0(\phi) , I_\delta(\phi) + L_1(\phi)y + L_2(\phi)\tw, w_0(\phi)+\tw \big)\,,
$$
	where $L_1(\phi):= [\pa_\phi \theta_0(\phi) ]^{-\top}$ and $L_2(\phi):=  [\pa_\phi\wtw_0(\theta_0(\phi))]^\top J_\angle^{-1} $ (see \eqref{Gdelta}), we get
	\begin{align}\label{P02_exp}
	\pa_\tw \nabla_{\tw}&(P\circ G_\delta) (\phi,0,0) = \pa_w \nabla_w P(i_\delta(\phi)) + R(\phi) \,, 
	\end{align}
	where
	$ R(\phi) := R_1(\phi)+ R_2(\phi) + R_3(\phi) $ and 
	\begin{equation}
	 R_1 := L_2(\phi)^\top \pa_I^2 P(i_\delta(\phi))L_2(\phi), \,  \ \ \ 
	 R_2 := L_2(\phi)^\top \pa_w \pa_I P(i_\delta(\phi)), \,  \ \ \ 
	R_3  := \pa_I \nabla_w P(i_\delta(\phi)) L_2(\phi) \,. \notag
	\end{equation}
Each operator $R_1,R_2,R_3$ has the finite rank form \eqref{finite_rank_R} because it is the composition of at least one operator with finite rank $\R^\nu$ in the space variable
(for more details see e.g. Lemma 6.1 in \cite{BM}) and 
the estimates \eqref{gjchij_est} follow by Lemma \ref{XP_est}. 
	By \eqref{K02_exp}, \eqref{P02_exp}, \eqref{cNP}, \eqref{Hepsilon}, \eqref{cHepsilon}, we obtain
	\begin{equation}\label{K02phi}
	K_{02}(\phi)  = \Pi^{L^2}_\angle 
 \partial_u \nabla_u {\cal H} (A(i_\delta(\phi)) )_{|\acca_{\S^+,\Sigma}^\angle } + 
 \varepsilon R(\phi) \, . 
	\end{equation}
In conclusion, by \eqref{K02phi}, Lemma \ref{ident-simp}, 
and since  $T_\delta(\phi) = A(i_\delta(\phi)) $, 
we deduce that the operator $ \cL_\omega $  in \eqref{Lomegatrue} 
has the form  \eqref{cLomega_again}-\eqref{cL000}. 
Finally the operator $  \Pi_{\S^+,\Sigma}^\angle J K_{02} (\vf) $ 	is reversible and momentum preserving, by \eqref{Ka.rev}, \eqref{K.symm.2}, Lemmata \ref{proj_rev}, \ref{lem:proj.momentum}, and the  fact that 
$ J $ commutes with $ \tau_\vs $ and  anti-commutes with $ \cS $.  
\end{proof}

We remark that $\cL$ in \eqref{cL000} is obtained by linearizing the 
water waves Hamiltonian system \eqref{Ham-Wal}, \eqref{eq:Ham_eq_zeta} 
in the Wahl\'en variables defined in \eqref{eq:gauge_wahlen}  at the torus $u=(\eta,\zeta)=T_\delta(\vf)$ defined in \eqref{Tdelta} and changing $\pa_t \rightsquigarrow \omega\cdot \pa_\vf $. This is equal to
\begin{equation}\label{cL000-or}
\cL = \omega\cdot \pa_\vf  - W^{-1} (d X)( W T_\delta(\vphi)) W \,,
\end{equation}
where 
$ X $  is the water waves vector field 
 on the right hand side of \eqref{ww}. The operator $ \cL $ acts on (a dense subspace) of the 
 phase space $ L^2_0 \times \dot L^2 $. 

In order to compute  $ d X $
we use the "shape derivative" formula,  see e.g. \cite{Lan05},
\begin{equation}\label{shape_derivative}
G'(\eta)[\wh\eta] \psi := \lim_{\epsilon\rightarrow 0} \tfrac{1}{\epsilon}
\big(  G(\eta+\epsilon\wh\eta)\psi -G(\eta)\psi \big) = - G(\eta)(B \wh\eta) -\pa_x(V \wh\eta) \,,
\end{equation}
where
\begin{equation}\label{BfVf}
B(\eta,\psi):= \frac{G(\eta) \psi +\eta_x \psi_x }{1+\eta_x^2} \,, 
\quad V(\eta,\psi):= \psi_x - B(\eta, \psi) \eta_x \,.
\end{equation}
It turns out that $ (V,B) = ( \Phi_x, \Phi_y ) $ is the gradient of the generalized
velocity potential defined  in \eqref{dir}, evaluated at the free surface $ y = \eta (x) $. 

Using \eqref{cL000-or}, \eqref{ww}, 
\eqref{shape_derivative}, \eqref{BfVf}, 
the operator $ \cL $ is 
\begin{equation}\label{Linea10}
\begin{aligned}
\cL = \omega\cdot \pa_\vf   
&+ \begin{pmatrix}
\pa_x\wtV + G(\eta)B  & -G(\eta)   \\
g-\kappa\pa_x c\pa_x + B\wtV_x + B G(\eta) B& \wtV\pa_x - B G(\eta) 
\end{pmatrix}  \\
& +\frac{\gamma}{2}\begin{pmatrix}
-G(\eta)\pa_x^{-1} & 0 \\   \pa_x^{-1}G(\eta)B- BG(\eta)\pa_x^{-1} -\frac{\gamma}{2}\pa_x^{-1}G(\eta)\pa_x^{-1} &-\pa_x^{-1}G(\eta) 
\end{pmatrix} \,,
\end{aligned}
\end{equation}
where    
\begin{equation}\label{tc}
\wtV:= V - \gamma \eta  \, , \quad 
c (\eta)  := ( 1 + \eta_x^2)^{-\frac32} \, , 
\end{equation}
and the functions
$ B := B(\eta,\psi) $, $ V := V(\eta,\psi)$, $ c := c(\eta) $ in \eqref{Linea10} are evaluated  at
the reversible traveling wave $ (\eta,\psi) := W T_\delta(\vf) $ where $ T_\delta(\vf) $ is 
defined in \eqref{Tdelta}. 

\begin{rem}\label{phase-space-ext}
From now on we consider the operator $ \cL $ in \eqref{Linea10} acting 
on (a dense subspace of) the whole $  L^2 (\T) \times  L^2 (\T) $. In particular we extend 
the operator $ \pa_x^{-1} $ to act on the whole $ L^2 (\T) $ as in \eqref{pax-1}. 
In Sections \ref{sec:repa}-\ref{sec:order12}  
we are going to make several transformations, whose aim is to conjugate 
  $ {\cal L} $ to a constant coefficients Fourier multiplier, up to a pseudodifferential operator 
  of order zero 
plus a remainder that satisfies tame estimates,  
 both small in size, see $ {\cal L}_9 $ in  \eqref{cL10}.  
 Finally, in Section \ref{sec:conc} we shall  conjugate the restricted operator 
 $ {\cal L}_\omega $ in \eqref{cLomega_again}.
\end{rem}

{\bf Notation.} In \eqref{Linea10} and hereafter 
any function $ a $ is identified with the corresponding multiplication operators $h \mapsto a h$, and, where there is no parenthesis, composition of operators is understood. For example, 
$\pa_x c \pa_x$ means: $h \mapsto \pa_x (c \pa_x h)$.

\begin{lem}
\label{BVtilde.mom}
The functions  $ (\eta, \zeta) = T_\delta(\vf) $ and $B,  \wt V, c  $ defined  in \eqref{BfVf}, 
\eqref{tc} are  quasi-periodic traveling waves.
The functions $(\eta,\zeta)= T_\delta(\vf)$  are $ ( \even(\vf,x),\odd(\vf,x))$, 
$B$ is $\odd(\vf,x)$, $\wtV$ is $\even(\vf,x)$ and $c$ is $\even(\vf,x) $.  
The Hamiltonian operator $ \cL $ is  
reversible and momentum preserving.
\end{lem}

\begin{proof}
The function  $ (\eta, \zeta) = T_\delta(\vf) $ is a quasi-periodic traveling wave and,
using also Lemmata 
\ref{mom_dirichlet} and \ref{A.mom.cons}, we deduce that 
$B,  \wt V, c  $ are  quasi-periodic traveling waves. 
Since $(\eta,\zeta)= T_\delta(\vf)$ is reversible, we have 
that $(\eta,\zeta) $ is $ ( \even(\vf,x), \odd(\vf,x) ) $. Therefore, using 
also \eqref{DN-rev}, we deduce that  
$B$ is $\odd(\vf,x)$, $\wtV$ is $\even(\vf,x)$ and 
$c$ is $\even(\vf,x)$.  
By  Lemmata \ref{lem:REV} and 
 \ref{lem:MP}, the operator $\cL$ in \eqref{cL000-or} 
 evaluated at the reversible quasi-periodic traveling wave 
$ W T_\delta(\vf) $ is reversible and momentum preserving.
\end{proof}

For the sequel we will always assume the following ansatz (satisfied by the approximate solutions obtained along the nonlinear  Nash-Moser iteration of Section \ref{sec:NaM}): for some constants $\mu_0 :=\mu_0(\tau,\nu)>0$, $\upsilon\in (0,1)$, (cfr. 
Lemma \ref{torus_iso})
\begin{equation}\label{ansatz_I0_s0}
\norm{ \fI_0 }_{s_0+\mu_0} ^{k_0,\upsilon} \, ,  \ 
\norm{ \fI_\delta }_{s_0+\mu_0}^{k_0,\upsilon}  
\leq 1 
\,.
\end{equation} 
In order to estimate the variation of the eigenvalues with respect to the approximate invariant torus, we need also to estimate the variation with respect to the torus $i(\vf)$ in another low norm $\norm{ \ }_{s_1}$ for all Sobolev indexes $s_1$ such that
\begin{equation}\label{s1s0}
s_1+\sigma_0 \leq s_0 +\mu_0 \,, \quad \text{ for some } \ \sigma_0:=\sigma_0(\tau,\nu)>0 \,. 
\end{equation}
Thus, by \eqref{ansatz_I0_s0}, we have
\begin{equation*}\label{ansatz_I0_s1}
\norm{ \fI_0 }_{s_1+\sigma_0} ^{k_0,\upsilon} \, , 
 \  \norm{ \fI_\delta }_{s_1+\sigma_0}^{k_0,\upsilon} \leq 1 \,.
\end{equation*}
The constants $\mu_0$ and $\sigma_0$ represent the \emph{loss of derivatives} accumulated along the reduction procedure of the next sections. What is important is that they are independent of the Sobolev index $s$. In the following sections we shall denote by $\sigma:=\sigma(\tau,\nu,k_0)>0 $, 
$ \sigma_N(\tq_0) := \sigma_N(\tq_0,\tau,\nu,k_0) $, 
$ \sigma_M:= \sigma_M(k_0,\tau,\nu)>0 $, $ \aleph_M (\alpha ) $
constants (which possibly increase from lemma to lemma) representing  losses 
 of derivatives along the finitely many steps of the reduction procedure.

\begin{rem}
In the next sections 
$ \mu_0 :=\mu_0(\tau,\nu, M, \alpha) > 0 $  will depend
 also on indexes $ M, \alpha  $, whose 
maximal values will be fixed depending only on $ \tau $ and $ \nu $ (and $ k_0 $ which is however considered an absolute constant along the paper). 
In particular $ M $ is fixed in \eqref{M_choice}, whereas the maximal 
value of $ \alpha $ depends on 
$ M $, as explained in Remark \ref{fix:alpha}.
\end{rem}

As a consequence of Moser composition Lemma \ref{compo_moser} and \eqref{ebb1}, 
the Sobolev norm of the function $u=T_\delta(\vf)$ defined in \eqref{Tdelta} satisfies for all $s\geq s_0$
\begin{equation}\label{uI0}
\norm u_s^{k_0,\upsilon} = \norm \eta_s^{k_0,\upsilon} + \norm \zeta_s^{k_0,\upsilon} \leq \varepsilon C(s)\big( 1 + \norm{\fI_0}_s^{k_0,\upsilon} \big)
\end{equation}
(the map $A$ defined in \eqref{aacoordinates} is smooth). Similarly, using \eqref{ebb3}, 
\begin{equation*}\label{Delta12}
\norm{ \Delta_{12}u }_{s_1} \lesssim_{s_1} \varepsilon \norm{i_2-i_1}_{s_1} \,, \quad \text{ where } \ \Delta_{12}u:=u(i_2)-u(i_1) \, . 
\end{equation*}
 We finally recall that $\fI_0 = \fI_0(\omega,\kappa)$ is defined for all $(\omega,\kappa)\in\R^\nu\times [\kappa_1,\kappa_2]$ and that the functions $B,\wtV$ and $c$ appearing in $\cL$ in \eqref{Linea10} 
 are $\cC^\infty$ in $(\vf,x)$, as $u=(\eta,\zeta)=T_\delta(\vf)$ is. 
 
\subsection{Quasi-periodic reparametrization of time}\label{sec:repa}

We conjugate the operator $ {\cal L} $ in \eqref{Linea10} by 
the change of variables induced by the quasi-periodic reparametrization 
of time 
\begin{equation}\label{def:p}
\vartheta:=\vf + \omega p(\vf)  \quad \Leftrightarrow \quad   \vf = \vartheta + \omega 
\breve p(\vartheta)\,,
\end{equation}
where $p(\vf)$ is the real $\T^\nu$-periodic function defined 
 in \eqref{choice_p}. 
Since
$ \eta (\vf,x) $ is a quasi-periodic traveling wave, even in $ (\vf, x) $ 
(cfr. Lemma \ref{BVtilde.mom}), it results that 
\begin{equation}\label{propp}
p(\vf-\ora{\jmath}\varsigma) = p(\vf) \, , \  \forall \varsigma\in\R \, , \quad
p \ {\rm is \ odd}(\vf) \, . 
\end{equation}
Moreover, by \eqref{choice_p},   \eqref{lem:diopha.eq},   Lemma \ref{compo_moser}, 
\eqref{uI0} and \eqref{ansatz_I0_s0} and Lemma 2.30 in \cite{BM},
both  $ p $ and $ \breve p $ satisfy,  for some $\sigma:=\sigma(\tau,\nu,k_0) > 0 $, 
the tame estimates, for $s\geq s_0$,
\begin{equation}\label{stimap}
\| p \|^{k_0,\upsilon}_s + \normk{\breve p}{s} \lesssim_s \varepsilon^2  \upsilon^{-1}\big( 1+\normk{\fI_0}{s+\sigma} \big) \, .  
\end{equation}
\begin{rem}\label{rem:REPA}
We perform  
as a first step the time reparametrization \eqref{def:p} of $ \cL  $, 
with a function $ p ( \vf )  $ which will be fixed 
only later in Step 4 of Section \ref{sec:order32}, 
to avoid otherwise a technical difficulty in the conjugation 
of the remainders obtained by the Egorov theorem in Step 1 of Section \ref{sec:order32}.
We need indeed to apply 
the Egorov Proposition \ref{egorov}
for conjugating the additional pseudodifferential term in  \eqref{Linea10} due to vorticity. 
\end{rem}
Denoting by 
$$
(\cP h)(\vf,x):= h(\vf+\omega p(\vf),x) \, , \quad 
(\cP^{-1}h)(\vartheta,x):= h(\vartheta+\omega\breve p(\vartheta),x) \, , 
$$
the induced diffeomorphism of functions $ h (\vphi, x ) \in \C^2 $, we have 
\begin{align}\label{rho_vartheta}
\cP^{-1} \circ \, \omega\cdot \partial_\vf \, \circ \cP & = 
\rho (\vartheta) \omega\cdot \partial_\vartheta \,, \quad 
\rho (\vartheta ) :=   \cP^{-1} (1+\omega\cdot \pa_\vphi p) \,.
\end{align}
Therefore, for 
any $ \omega \in \tD\tC (\upsilon, \tau) $, we get
\begin{equation}\label{Linea1}
\begin{aligned}
{\cal L}_{0} :=  \frac{1}{\rho} \cP^{-1} {\cal L} \cP  = \omega\cdot \pa_\vartheta  
&+ \frac{1}{\rho} \begin{pmatrix}
\pa_x\wtV + G( \eta )B  & -G(\eta )   \\
g-\kappa\pa_x c\pa_x + B\wtV_x + B G(\eta) B& \wtV\pa_x - B G(\eta) 
\end{pmatrix} 
\\
& + \frac{1}{\rho}  \frac{\gamma}{2}\begin{pmatrix}
-G(\eta)\pa_x^{-1} & 0 \\   \pa_x^{-1}G(\eta)B- BG(\eta)\pa_x^{-1} -\frac{\gamma}{2}\pa_x^{-1}G(\eta)\pa_x^{-1} &-\pa_x^{-1}G(\eta) 
\end{pmatrix} \,,
\end{aligned}
\end{equation}
where $ \wtV,  B, c, V   $ and $G(\eta)$ are 
evaluated at $ (\eta_p, \psi_p) :=  \cP^{-1} (\eta, \psi) $. For simplicity 
in the notation we do not report in \eqref{Linea1} the explicit dependence on $ p $, writing for example 
(cfr. \eqref{tc})  
\begin{equation}\label{defcp}
c =   \big( 1 + (\cP^{-1} \eta)_x^2  \big)^{-\frac32} = \cP^{-1} \big( 1 + \eta_x^2  \big)^{-\frac32} \, .
\end{equation}

\begin{lem}\label{timerep}
The maps $\cP$, $\cP^{-1}$ are $\cD^{k_0}$-$(k_0+1)$-tame, 
the maps $\cP-{\rm Id}$ and $\cP^{-1}-{\rm Id}$ are $\cD^{k_0}$-$(k_0+2)$-tame,
with tame constants satisfying, for some $\sigma:=\sigma(\tau,\nu,k_0)>0$ and for any
$ s_0\leq s\leq S $,  
	\begin{equation}\label{timerep1}
		\fM_{\cP^{\pm 1}}(s) \lesssim_{S} 1 + \normk{\fI_0}{s+\sigma} \, , \quad 
		\fM_{\cP^{\pm 1}-{\rm Id}}(s) \lesssim_{S}\varepsilon^2 \upsilon^{-1}\big( 1+\normk{\fI_0}{s+\sigma} \big)  \, . 
	\end{equation}
The function $ \rho $ defined in \eqref{rho_vartheta} satisfies  
\begin{equation}\label{rhopem}
\rho \ {\rm is} \ {\rm even}(\vartheta) \quad  and \quad  
\rho ( \vartheta -\ora{\jmath}\varsigma) = \rho (\vartheta) \, , \ 
\forall \varsigma\in\R \, . 
\end{equation}
The  operator $ {\cal L}_{0}  $ is Hamiltonian, reversible and 
momentum preserving. 
\end{lem}
\begin{proof}
	Estimates \eqref{timerep1} 
	follow by \eqref{stimap} and Lemma 2.30 in \cite{BM},
	writing 
$ ( \cP-{\rm Id} ) h = p \int_0^1 \cP_\tau ( \omega\cdot \pa_\vf h ) \wrt \tau $, where
$ (\cP_\tau h)(\vf,x):= h(\vf+\tau \omega p(\vf), x) $. 
We deduce \eqref{rhopem} by \eqref{propp} and \eqref{rho_vartheta}. Denoting 
$ \cL = \omega \cdot \pa_\vf + A ( \vf ) $ 
the operator $ \cL $ in 
\eqref{Linea10}, then the operator
$ \cL_0 $ in \eqref{Linea1} is $ \cL_0  = \omega \cdot \pa_\vartheta + A_+ ( \vartheta ) $  with 
$ A_+ ( \vartheta) = \rho^{-1} (\vartheta) A( \vartheta + \breve p (\vartheta) \omega ) $. It follows that $ A_+(\vf) $ is Hamiltonian, reversible and momentum preserving as 
$ A(\vf) $ (Lemma \ref{BVtilde.mom}).
\end{proof}

\begin{rem}\label{P-rev-mom}
The map $ \cP $ is not reversibility and momentum preserving according to Definitions
\ref{rev_defn}, respectively \ref{def:mom.pres}, 
but maps (anti)-reversible, respectively traveling, waves, into
(anti)-reversible, respectively traveling, waves. Note that the multiplication operator for the function
$ \rho (\vartheta)$, which satisfies \eqref{rhopem}, is reversibility and momentum preserving according to Definitions \ref{rev_defn} and \ref{def:mom.pres}. 
\end{rem}

\subsection{Linearized good unknown of Alinhac}\label{subsec:good}

We conjugate the linear operator $ \cL_0 $ in \eqref{Linea1}, where we rename 
$ \vartheta $ with $ \vf $, 
by the multiplication matrix  operator 
\begin{equation*}\label{good_unknown}
\cZ := \left( \begin{matrix}
{\rm Id} & 0 \\ B & {\rm Id}
\end{matrix}\right) \ , \qquad \cZ^{-1}=\left( \begin{matrix}
{\rm Id} & 0 \\ - B & {\rm Id}
\end{matrix} \right) \,,
\end{equation*}
obtaining (in view of \eqref{trasf-op})
\begin{equation}\label{cL0}
	\begin{aligned}
		\cL_1 & := \cZ^{-1} \cL_0 \cZ  \\
		& =   \omega\cdot \pa_\vf 
		+ \frac{1}{\rho}  \begin{pmatrix}
		\pa_x \wtV & -G(\eta) \\
		g 
		+ a - \kappa \pa_x c \pa_x  & \wtV \pa_x
		\end{pmatrix} 
		-  \frac{1}{\rho}  \frac{\gamma}{2}\begin{pmatrix}
		G(\eta)\pa_x^{-1} & 0 \\ \frac{\gamma}{2}\pa_x^{-1}G(\eta)\pa_x^{-1} & \pa_x^{-1}G(\eta)
		\end{pmatrix} \,,  
	\end{aligned}
\end{equation}
where $a$ is the function
\begin{equation}\label{ta}
a := \wtV B_x + \rho\,(\omega\cdot\pa_\vf B) \, .
\end{equation}
The matrix $\cZ$  amounts to introduce, as in \cite{Lan05} and
\cite{BM,BBHM}, a linearized version of the ``good unknown of Alinhac".

\begin{lem}\label{lem:good_unknwon}
	The maps $\cZ^{\pm 1}-{\rm Id}$ are $\cD^{k_0}$-tame with tame constants satisfying, for some $\sigma:=\sigma(\tau,\nu,k_0)> 0 $, for all $ s \geq s_0 $,	
	\begin{equation}
	\label{lem:ga1}
	\fM_{\cZ^{\pm 1}-{\rm Id}}(s)\, , \ \fM_{(\cZ^{\pm 1}-{\rm Id})^*}(s) \lesssim_s \varepsilon\big( 1 + \norm{ \fI_0 }_{s+\sigma}^{k_0,\upsilon} \big) \,.
	\end{equation} 
The function $a$ is a quasi-periodic traveling wave $\even(\vf,x)$. There is  $\sigma:= \sigma(\tau,\nu,k_0)>0$ such that, for all $ s \geq  s_0 $,
	\begin{equation}
	\begin{aligned}
	\label{lem:ga2}
	&\norm{a}_s^{k_0,\upsilon} + \| \wtV \|_s^{k_0,\upsilon} + \norm{ B }_s^{k_0,\upsilon} \lesssim_s \varepsilon \big( 1 + \normk{\fI_0}{s+\sigma} \big) \,, 
	\quad 
	\normk{1-c }{s} \lesssim_s \varepsilon^2 \big( 1 + \normk{\fI_0}{s+\sigma} \big)\,.
	\end{aligned}
	\end{equation}
	Moreover, for any $s_1$ as in \eqref{s1s0},
	\begin{align}
\label{lem:ga4}	&\norm{\Delta_{12}a}_{s_1}+
\| \Delta_{12}\wtV \|_{s_1}+\norm{\Delta_{12}B}_{s_1}\lesssim_{s_1} \varepsilon \norm{ i_1-i_2 }_{s_1+\sigma} \,,\\
\label{lem:ga5}
	& \norm{ \Delta_{12}c }_{s_1} \lesssim_{s_1} \varepsilon^2 \norm{ i_1-i_2}_{s_1+\sigma}\,, \\
	\label{lem:ga6}
	& \| \Delta_{12} (\cZ^{\pm 1})h \|_{s_1}, \| \Delta_{12}
	(\cZ^{\pm 1})^* h\|_{s_1} \lesssim_{s_1} \varepsilon \norm{i_1-i_2}_{s_1+\sigma} \norm{h}_{s_1} \,.
	\end{align}
	The operator $\cL_1$ is Hamiltonian, reversible and momentum preserving.
\end{lem}
\begin{proof}
The estimates \eqref{lem:ga2} follow by the expressions of 
 $a, \wt V, B, c$ in \eqref{ta}, \eqref{BfVf}, \eqref{tc}, 
(reparametrized by $ \cP^{-1} $ as in \eqref{defcp}), Lemmata 
\ref{compo_moser} and \eqref{timerep1},  
\eqref{prod}, \eqref{est:RG}, \eqref{TameA_PS} and \eqref{action_Hs_tame}.
The estimate \eqref{lem:ga1} follows by \eqref{TameA_PS}, \eqref{norm.mult}, \eqref{lem:ga2} and since the adjoint 
$\cZ^* = \begin{pmatrix}
{\rm Id} & B\\
0  & {\rm Id}
\end{pmatrix}$. The estimates \eqref{lem:ga4}-\eqref{lem:ga6}  follow similarly.
Since $B$ is a $ \odd(\vf,x)$ quasi-periodic traveling wave, then  
the operators $\cZ^{\pm}$ are reversibility and momentum preserving.
\end{proof}

\subsection{Symmetrization and reduction of the highest order}\label{sec:order32}

The aim of this long section is to conjugate the Hamiltonian operator $ \cL_1 $
in \eqref{cL0} to the Hamiltonian operator $ \cL_5 $ in \eqref{cL4} whose coefficient 
$ \tm_{\frac32} $ of the highest order is constant.
 This is achieved in several steps. 
All  the transformations of this section are symplectic.

Recalling the expansion \eqref{DN.dec}
 of the Dirichlet-Neumann operator, 
 we first write
\begin{equation}
\label{cL00}
\cL_1  = \omega\cdot \pa_\vf + 
\frac{1}{\rho}
\begin{pmatrix}
-\frac{\gamma}{2}G(0)\pa_x^{-1} & -G(0) \\
 - \kappa \pa_x c \pa_x  + g - \left(\frac{\gamma}{2}\right)^2\pa_x^{-1}G(0)\pa_x^{-1}& 
  -\frac{\gamma}{2}\pa_x^{-1}G(0)
\end{pmatrix} 
+
\frac{1}{\rho}\begin{pmatrix}
\pa_x \wt V & 0 \\
 a & \wt V \pa_x 
\end{pmatrix} + \bR_1\,,
\end{equation}
where 
\begin{equation}
\label{bRG}
\bR_1 :=- \frac{1}{\rho}\begin{pmatrix}
\frac{\gamma}{2}\cR_G(\eta)\pa_x^{-1} & \cR_G(\eta) \\ \left(\frac{\gamma}{2}\right)^2\pa_x^{-1}\cR_G(\eta)\pa_x^{-1} & \frac{\gamma}{2}\pa_x^{-1}\cR_G(\eta)
\end{pmatrix}  
\end{equation}
is a small remainder in $ \Ops^{-\infty}$. 

\paragraph{Step 1:}
We first conjugate $\cL_1 $ with  the symplectic change of variable
(cfr. \eqref{symp-op})
\begin{equation}\label{defcE}
(\cE u)(\vf,x) :=  \sqrt{1+\beta_x(\vf, x)} \, (\cB u)(\vf, x) \, , \qquad 
(\cB u)(\vf, x) := 
u\left(\vf,x+ \beta(\vf, x)\right) \, , 
\end{equation}
induced by a family of $\vf$-dependent diffeomorphisms of the torus
$ y = x+ \beta(\vf, x) $ \,,
where $\beta(\vf, x)$ is a small function to be determined, see \eqref{betasolved}.
We denote the inverse diffeomorphism by
$ x = y + \breve\beta(\vf, y) $. 
By direct computation we have that 
\begin{align}
\cE^{-1} \wtV \partial_x \cE & =
\big\{ \cB^{-1} \big( \wtV ( 1+\beta_x ) \big) \big\}  \pa_y + \tfrac12
\big\{ \cB^{-1} \wt V \beta_{xx}(1+\beta_x)^{-1} \big\} \, , \label{conj1} \\
\cE^{-1}\partial_x \wtV \cE  & = \big\{\cB^{-1}\big( \wtV ( 1+\beta_x ) \big)\big\} 
\pa_y + \{\cB^{-1}( \wtV_x +\tfrac12 \wt V \beta_{xx}(1+\beta_x)^{-1}) \} \,, \\
\cE^{-1}a\cE & = \{ \cB^{-1} a\}\, ,  \\
\cE^{-1} \partial_x c\partial_x \cE & = 
\cB^{-1} (1 + \beta_x)^{- \frac12} \cB \ \cB^{-1} \pa_x \cB \ \cB^{-1} c \cB \
\cB^{-1} \pa_x \cB \ \cB^{-1} (1 + \beta_x)^{\frac12} \cB \nonumber
\\
& = \big\{\cB^{-1} (1+\beta_x)^\frac12\big\}\, \pa_y \,  \big\{ \cB^{-1}( c (1+\beta_x )) \big\} \,\pa_{y}\, \big\{\cB^{-1} (1+\beta_x)^\frac12 \big\} \,,  \\
\label{omegavf}
\cE^{-1}\omega\cdot \partial_\vf \cE & = \omega\cdot \partial_\vf + \left\{ \cB^{-1}\left(\omega\cdot\pa_\vf \beta\right) \right\}\pa_y
+ \tfrac12  \{ \cB^{-1}  \big( (\omega \cdot \pa_\vf \beta_x)(1+\beta_x )^{-1} \big) \} 
 \,.
\end{align}
Then we write the Dirichlet-Neumann operator  $G(0)  $ in \eqref{G(0)} as
\begin{equation}
\label{G.exp}
G(0) = G(0, \tth) = 
\pa_x \cH T(\tth) \,,
\end{equation}
where $ \cH $ is the Hilbert transform in \eqref{Hilbert-transf} and 
\begin{equation}
\label{Ttth}
T(\tth) := \begin{cases}
\tanh(\tth |D|) = {\rm Id} +  \Op(r_\tth)  & \text{ if } \tth < + \infty \, , 
\qquad r_{\tth} (\xi) := -\frac{2}{1+ e^{2 \tth |\xi| \chi(\xi)}} \in S^{- \infty} \, ,  \\
{\rm Id}  & \text{ if } \tth = \infty \, .
\end{cases}
\end{equation}
We have the conjugation formula (see  formula (7.42) in \cite{BBHM})  
\begin{equation}\label{conG}
\cB^{-1} G(0) \cB = \left\{\cB^{-1}(1+\beta_x)\right\} G(0) + \cR_1 \,,
\end{equation}
where 
\begin{equation}\label{cR1}
\cR_1:=
\left\{\cB^{-1}(1+\beta_x) \right\} \pa_y \left( 
\l \cH \left(\cB^{-1} \Op(r_\tth) \cB - \Op(r_\tth) \right)+
\left( \cB^{-1} \cH \cB - \cH \right) ( \cB^{-1} T(\tth) \cB) \right)  \, . 
\end{equation}
The operator $\cR_1 $ is in $ \Ops^{-\infty}  $ because both 
$  \cB^{-1} \Op(r_\tth) \cB - \Op(r_\tth) $ and $\cB^{-1}\cH \cB - \cH$ 
are in $ \Ops^{-\infty}  $ and there is $ \s > 0 $
such that, for any $m \in \N$, $s \geq s_0 $,  and $\alpha \in \N_0$, 
\begin{equation}
\label{ht.t}
\begin{aligned}
& \normk{ \cB^{-1} \cH \cB - \cH}{-m, s, \alpha} 
 \lesssim_{m, s, \alpha, k_0}  
\normk{\beta}{s+m+\alpha + \sigma} \, , \\
& \normk{ \cB^{-1} \Op(r_\tth) \cB - \Op(r_\tth)}{-m, s, \alpha} 
 \lesssim_{m, s, \alpha, k_0}  
\normk{\beta}{s+m+\alpha + \sigma} \, . 
\end{aligned}
\end{equation}
The first estimate is given in Lemmata  2.36 and 2.32 in \cite{BM}, whereas 
the second one follows by that fact that $r_\tth \in S^{-\infty}$ (see \eqref{Ttth}), 
Lemma 2.18 in \cite{BBHM} and Lemmata 2.34 and  2.32  in \cite{BM}.
Therefore by \eqref{conG} we obtain
\begin{equation}\label{conDN0}
\cE^{-1} G(0) \cE = \{\cB^{-1} (1+\beta_x)^\frac12 \} \, G(0) \, \{\cB^{-1} (1+\beta_x)^\frac12 \} + \wt \cR_1 \,,
\end{equation}
where
\begin{equation}
\label{wtR1}
\wt \cR_1 := \{\cB^{-1} (1+\beta_x)^{-\frac12}  \} \, \cR_1 \, \{\cB^{-1} (1+\beta_x)^\frac12 \} .
\end{equation}
Next we transform $G(0)\pa_x^{-1}$.
By \eqref{G.exp} and using the identities $ \cH \pa_x \pa_x^{-1} = \cH $ and
$ \cH T(\tth) = G(0) \partial_y^{-1} $ on the periodic functions, 
we have that
\begin{equation}\label{G(0)pax-1_ego}
\cE^{-1} G(0)\pa_x^{-1}\cE  
= \,\cE^{-1} \pa_x \cH T(\tth) \pa_x^{-1} \cE  
= G(0)\pa_y^{-1} + \cR_2\,,
\end{equation}
where 
\begin{equation}\label{cR2}
	\begin{aligned}
		\cR_2  & :=   \{\cB^{-1} (1+\beta_x)^{-\frac12} \} \big[ \cH T(\tth),\{\cB^{-1} (1+\beta_x)^\frac12 \} - 1 \big] + \{\cB^{-1} (1+\beta_x)^{-\frac12}  \}\, \circ \\
		& \ \quad  \circ \left((\cB^{-1} \cH \cB - \cH )(\cB^{-1}T(\tth)  \cB) +
		\cH \big(\cB^{-1} \Op(r_\tth) \cB - \Op(r_\tth) \big)\right)
		\{\cB^{-1} (1+\beta_x)^\frac12 \} \, .
	\end{aligned}
\end{equation}
The operator $ \cR_2 $ is in $ \Ops^{-\infty} $  by \eqref{ht.t}, \eqref{Ttth}
and because the  commutator of $\cH$ with any smooth function $a$  
is in $\Ops^{-\infty}$, in particular 
(see Lemma 2.35 of \cite{BM}) there is $ \s > 0 $
such that, for any $m \in \N$, $s \geq s_0 $,  and $\alpha \in \N_0$, 
\begin{equation}
\label{Ht.comm}
\normk{[ \cH T(\tth),  a] }{-m, s, \alpha} 
 \lesssim_{m, s, \alpha, k_0}  
\normk{a}{s+m+\alpha + \sigma} \, .
\end{equation}
Finally we conjugate $\pa_x^{-1} G(0)\pa_x^{-1}$.
By  the Egorov Proposition \ref{egorov}, we have that, for any $ N \in\N $, 
\begin{equation}\label{pax-1_ego}
\cE^{-1} \pa_x^{-1} \cE = \Big\{ \cB^{-1}\Big( \frac{1}{1+\beta_x} \Big) \Big\}\pa_y^{-1} 
+ P^{(1)}_{-2,N} (\varphi, x, D)
+ \tR_N\,,
\end{equation}
where $ P^{(1)}_{-2,N} (\varphi, x, D) \in \Ops^{-2}$ is  
$$
P^{(1)}_{-2,N} (\varphi, x, D) := 
\{ \cB^{-1}(1 + \beta_x)^{- \frac12} \} \Big\{
\big[ p_{-1} \pa_y^{-1},  \cB^{-1}(1 + \beta_x)^{\frac12}  \big] +
 \sum_{j=1}^N p_{-1-j}\pa_y^{-1-j}  
\{ \cB^{-1}(1 + \beta_x)^{ \frac12} \} \Big\}  
$$
with functions $ p_{-1-j}(\lambda; \varphi, y)$, $ j = 0, \ldots, N $, satisfying \eqref{norm-pk}  
and $\tR_N$ 
is a regularizing operator satisfying the estimate \eqref{stima resto Egorov teo astratto}. 
So, using \eqref{pax-1_ego} and \eqref{G(0)pax-1_ego}, we obtain
\begin{equation}\label{termine_non_bello}
	\cE^{-1} \pa_x^{-1}G(0) \pa_x^{-1}\cE
	 = \left( \cE^{-1}\pa_x^{-1} \cE \right)
	 \left( \cE^{-1}G(0)\pa_x^{-1} \cE \right) 
	  = \pa_y^{-1} G(0) \pa_y^{-1} +P_{-1,N}^{(2)}  +\tR_{2,N} \,, 
\end{equation}
where 
\begin{equation}\label{tildecP-1}
	\begin{aligned}
		P_{-1,N}^{(2)}  &:= \Big( - \Big\{ \cB^{-1}\Big( \frac{\beta_x}{1+\beta_x}\Big) \Big\}\pa_y^{-1} + P^{(1)}_{-2,N} (\varphi, x, D) 
		\Big) G(0) \pa_y^{-1} \in \Ops^{-1}
	\end{aligned}
\end{equation}
and $\tR_{2,N} $ is the regularizing operator 
\begin{equation}\label{cR3}
	\tR_{2,N}:=  (\cE^{-1} \pa_x^{-1} \cE) \cR_2  + \tR_N G(0) \pa_y^{-1}  \, .
\end{equation}
The smoothing order $N\in\N$ will be chosen in Section \ref{sec:KAM} during the KAM iteration (see also Remark \ref{rem:after_block}). 

In conclusion, by \eqref{conj1}-\eqref{omegavf}, \eqref{conDN0}, \eqref{G(0)pax-1_ego} and \eqref{termine_non_bello} 
we obtain
\begin{equation}
\label{cL1}
\begin{aligned}
\cL_2  := \cE^{-1} \cL_1 \cE  
= \omega\cdot\partial_\vf + & 
\frac{1}{\rho}\begin{pmatrix}
-\frac{\gamma}{2} G(0) \pa_y^{-1} & - a_2 G(0) a_2 \\
-\kappa a_2 \pa_y a_3 \pa_y a_2 + g - \left(\frac{\gamma}{2}\right)^2 \pa_y^{-1} G(0) \pa_y^{-1} & -\frac{\gamma}{2} \pa_y^{-1} G(0)
\end{pmatrix} \\
& +\frac{1}{\rho}\begin{pmatrix}
 a_1\pa_y +a_4 & 0 \\
 a_5 - \left(\frac{\gamma}{2}\right)^2 P_{-1, N}^{(2)} &   a_1\pa_y +a_6 
\end{pmatrix}
+ \bR_{2}^\Psi + \bT_{2,N}\,,
\end{aligned}
\end{equation}
where
\begin{align} 
a_1(\vf, y) & := \cB^{-1}\big( (1+\beta_x)\wt V + \big(\omega \cdot \partial_\vf \beta
\big)\big) \,,\label{ta1}\\
a_2(\vf, y) &:=  \cB^{-1} (\sqrt{1+\beta_x})\,, \label{ta2}  \quad 
a_3(\vf, y) := \cB^{-1}\big( c(1+\beta_x) \big) \, ,  \\
a_4(\vf, y) &:= \cB^{-1}\Big( \frac{\wt V \beta_{xx} + (\omega \cdot \pa_\vf \beta_x)}{2(1+\beta_x)}+ \wtV_x  \Big) \,, \quad
a_5(\vf, y) :=   \cB^{-1}a \,, \label{ta5}\\
a_6(\vf, y) &:= \cB^{-1}\Big( \frac{\wt V \beta_{xx} + (\omega \cdot \pa_\vf \beta_x)}{2(1+\beta_x)} \Big) \,,\label{ta6}
\end{align}
the operator $P_{-1, N}^{(2)} \in \Ops^{-1} $ is defined in \eqref{tildecP-1} 
and 
\begin{equation}\label{bR1}
\bR_{2}^\Psi:=-\frac{1}{\rho}\begin{pmatrix}
\frac{\gamma}{2}\cR_{2} &  \wt\cR_{1}  \\ 
0 & \frac{\gamma}{2}\cR_{2}
\end{pmatrix} + \cE^{-1}\bR_1\cE \,,  
\qquad
\bT_{2,N} := -\frac{1}{\rho} \left(\frac{\gamma}{2}\right)^2\begin{pmatrix}
 0 & 0\\
\tR_{2,N} & 0
\end{pmatrix}\,,
\end{equation}
with $\wt\cR_1 $, $\cR_2 $, $ {\mathtt R}_{2,N}  $   defined in \eqref{wtR1}, \eqref{cR2}, \eqref{cR3} and $\bR_1$ in \eqref{bRG}.

\paragraph{Step 2:}
We now conjugate the operator  $\cL_2 $  in \eqref{cL1}  with the 
multiplication matrix operator
\begin{equation*}
\label{Q}
\cQ := \begin{pmatrix}
q & 0 \\
0 & q^{-1}
\end{pmatrix} \ , \qquad
\cQ^{-1} := \begin{pmatrix}
q^{-1} & 0 \\
0 & q
\end{pmatrix}\,,
\end{equation*}
where  $ q(\vf, y) $ is a real function, close to $ 1$,  to be determined. The maps
$ \cQ $ and $ \cQ^{-1} $ are symplectic (cfr. \eqref{symp-op}). 
We have that 
\begin{align}
\label{L31}
\cL_3  := \cQ^{-1} \cL_2 \cQ  = 
 \omega\cdot \partial_\vf   + 
\frac{1}{\rho} \begin{pmatrix}
A & B \\
C &  D 
 \end{pmatrix} 
 +  \cQ^{-1} (\bR_2^\Psi + \bT_{2,N}) \cQ \,,
\end{align}
where
\begin{align}
& A :=  q^{-1} \big(-\tfrac{\gamma}{2} G(0) \pa_y^{-1} + a_1\pa_y +a_4 \big) q + 
\rho q^{-1} (\omega \cdot \pa_\vphi q)\,, \label{defA3} \\
& B :=  - q^{-1} a_2 G(0) a_2 q^{-1} \,,\\
& C := 
q \big( -\kappa a_2 \pa_y a_3 \pa_y a_2 + g - \left(\tfrac{\gamma}{2}\right)^2 \pa_y^{-1} G(0) \pa_y^{-1}+ a_5 - \left(\tfrac{\gamma}{2}\right)^2 P_{-1, N}^{(2)}   \big)q \,,
 \\
&D := q \big( -\tfrac{\gamma}{2} \pa_y^{-1} G(0) +
 a_1\pa_y +a_6 \big) q^{-1} - \rho q^{-1} (\omega \cdot \pa_\vphi q)  \, . \label{defD3}
\end{align}
We choose the function $ q $ so that the coefficients 
of the highest order terms of the off-diagonal operators 
$ B $ and $  C $ satisfy 
\begin{equation}
\label{scelgo}
q^{-2} a_2^2 = q^2 a_2^2 a_3 =  m_{\frac32}(\vf) \, ,
\end{equation}
with $  m_{\frac32}(\vf) $ independent of $ x $. 
This is achieved  choosing 
\begin{equation}
\label{q}
 q := \left(\frac{1}{a_3}\right)^{1/4} 
  \, 
\end{equation}
and, recalling \eqref{ta2}, the function  $\beta$ so that
\begin{equation}\label{betachoi}
(1+ \beta_x(\vf,x))^3 c(\vf, x) = m(\vf) \, , 
\end{equation}
with $ m (\vf) $ independent of $x$ (the function $ c $ is defined in \eqref{defcp}). 
The solution of \eqref{betachoi} is 
\begin{equation}
\label{betasolved}
m(\vf) := \Big(\frac{1}{2\pi} \int_\T c(\vf, x)^{-1/3} \,\wrt x \Big)^{-3} , 
\qquad
\beta(\vf, x) := \partial_x^{-1} \Big(  \Big(\frac{m(\vf)}{c(\vf, x)}\Big)^{1/3} - 1 \Big) \,.
\end{equation}
In such a way, by \eqref{ta2}, we obtain \eqref{scelgo} with  
$ m_{\frac32} (\vf)  := \sqrt{m (\vphi) } $.
By \eqref{betasolved} and \eqref{defcp} we have 
\begin{equation}
\label{m32}
m_{\frac32} (\vf) 
= \cP^{-1} \Big( \frac{1}{2 \pi} \int_{\T} \sqrt{1 + \eta_x^2 (\vphi, x)} \di x \Big)^{-\frac32}\, . 
\end{equation}
Note that, since 
by \eqref{scelgo} the function $ q^{-1} a_2 $ is independent of $ x$, we have  
\begin{equation}\label{expanB}
 B =  - q^{-1} a_2 G(0) a_2 q^{-1} = - q^{-2} a_2^2 G(0) \, .
\end{equation}
Moreover we have the expansion
\begin{equation}\label{a7}
\begin{aligned}
q a_2 \pa_y a_3 \pa_y a_2 q & =  
q^2 a_2^2 a_3   \pa_y^2 + (q^2 a_2^2 a_3)_y \pa_y + 
q a_2(a_3 (q a_2)_y)_y	\\
& \stackrel{\eqref{scelgo}} =  m_{\frac32}(\vf)\pa_y^2 + a_7 , \qquad
a_7  := q a_2(a_3 (q a_2)_y)_y \, . 
\end{aligned}
\end{equation}
In conclusion, the operator $\cL_3  $ in \eqref{L31} is,  
in view of \eqref{defA3}-\eqref{defD3} and \eqref{expanB}, 
\eqref{a7}, 
\begin{equation}
\begin{aligned}
\label{cL2-final}
\cL_3 = \cQ^{-1} \cL_2 \cQ =  \omega\cdot \partial_\vf  + 
&
\frac{1}{\rho} \begin{pmatrix}
-\frac{\gamma}{2} G(0) \pa_y^{-1}  &  - m_{\frac32} (\vf) G(0) \\ m_{\frac32} (\vf)\left(-\kappa  \pa_y^2+  g   -  \left(\frac{\gamma}{2}\right)^2 
\pa_y^{-1}G(0) \pa_y^{-1} \right)
   &  -\frac{\gamma}{2}\pa_y^{-1} G(0) 
\end{pmatrix}  \\
& + \frac{1}{\rho} \begin{pmatrix}
a_1 \pa_y + a_8 & 
 0  \\
  a_9 + P_{-1,N}^{(3)} &
 a_1 \pa_y + a_{10}
\end{pmatrix} + \bR_3^\Psi +  \bT_{3,N}  \,,
\end{aligned}
\end{equation}
where
\begin{align}
\label{bV2}
&a_{8} := a_1 q^{-1} q_y  +\rho \,q^{-1} (\omega \cdot \pa_\vf q) + a_4\,, \quad
a_9 :=  a_5 q^2 + g(q^2 - m_{\frac32})  -  \kappa a_7 \,,\\
&  a_{10}:=-a_1q^{-1}q_y  -\rho\, q^{-1}(\omega\cdot\pa_\vf q)  +a_6\,, \label{BVVVV}\\
& P^{(3)}_{-1, N} := 
- \left( \tfrac{\gamma}{2}\right)^2 \left( q P^{(2)}_{-1, N} q + (q^2 - m_{\frac32}) 
 G(0) \pa_y^{-2} + q [ G(0) \pa_y^{-2} , q - 1] \right)  \in \Ops^{-1} \label{P-1N_2} \,,
\end{align} 
and $ \bR_{3}^\Psi , \bT_{3,N}$ are the smoothing remainders
\begin{align}
\label{bR2}
& \bR_{3}^\Psi := 
\frac{1}{\rho} \begin{pmatrix}
- \frac{\gamma}{2}q^{-1} [ \cH T(\tth), q-1]  & 
0   \\
0 &  
- \frac{\gamma}{2}q [ \cH T(\tth), q^{-1}-1] 
\end{pmatrix}  + \cQ^{-1} \bR_{2}^\Psi\cQ 
\in \Ops^{-\infty}\,  ,
 \\
& \bT_{3,N}:= \cQ^{-1} \bT_{2,N}\cQ \,. \label{T3N}
\end{align}

\paragraph{Step 3:}
We now conjugate $\cL_3 $  in \eqref{cL2-final},  where we rename the 
space variable $ y $ by $ x $, by the symplectic transformation (cfr. \eqref{symp-op})
\begin{equation}\label{map_test_M}
	\wt\cM:= \begin{pmatrix}
	\Lambda & 0 \\ 0 & \Lambda^{-1}
	\end{pmatrix}\,,\quad
	\wt\cM^{-1} := \begin{pmatrix}
	\Lambda^{-1} & 0 \\ 0 & \Lambda
	\end{pmatrix}\,,
\end{equation}
where $\Lambda \in \Ops^{- \frac14} $ is the Fourier multiplier
\begin{equation}\label{Lambda}
	\Lambda := \tfrac{1}{\sqrt g}\pi_0 + M(D)\,, \quad \text{with inverse} \quad \Lambda^{-1}:= \sqrt g \pi_0 + M(D)^{-1} \in \Ops^{ \frac14}  \,,
\end{equation}
with $\pi_0 $
defined in \eqref{defpi0} and  $M(D)$ in \eqref{eq:T_sym}. 
We have the identity
\begin{equation}\label{LCL}
 \Lambda \big(-\kappa  \pa_x^2+  g   -  \big(\tfrac{\gamma}{2}\big)^2 
\pa_x^{-1}G(0) \pa_x^{-1} \big) \Lambda  = \Lambda^{-1}G(0) \Lambda^{-1} + \pi_0 = \omega(\kappa, D)+\pi_0 \, ,  
\end{equation}
where $\omega(\kappa, D)$ is defined in \eqref{eq:omega0}.
In \eqref{Lambda} and \eqref{LCL} we mean that the 
symbols of $ M(D), M(D)^{-1} $ and $ \omega (\kappa, D) $ are extended to
$ 0 $ at $ j = 0 $, multiplying them by the cut-off function $ \chi $ defined  in \eqref{cutoff}.
Thus we obtain 
\begin{equation}
\begin{aligned}
\label{cL3}
\cL_4 := \wt\cM^{-1} \cL_3 \wt\cM= & \ \omega\cdot \partial_\vf   + 
\frac{1}{\rho} 
\begin{pmatrix}
-\frac{\gamma}{2} G(0) \pa_x^{-1}  
&  - m_{\frac32} (\vf) \omega(\kappa, D) \\ 
m_{\frac32} (\vf) \omega(\kappa, D)
&  -\frac{\gamma}{2} G(0) \pa_x^{-1}
\end{pmatrix} + \begin{pmatrix}
0 & 0 \\ \pi_0 & 0
\end{pmatrix}  \\
& +\frac{1}{\rho} 
\begin{pmatrix}
a_1 \pa_x + P_0^{(41)} 
& 0  \\
P_{-\frac12}^{(43)} &
a_1 \pa_x + P_0^{(44)}
\end{pmatrix} + \bR_4^\Psi+  \bT_{4,N}   \,,
\end{aligned}
\end{equation}
where
\begin{align}
& P_0^{(41)} := \Lambda^{-1} [a_1\pa_x ,\Lambda] + \Lambda^{-1} a_8 \Lambda \in \Ops^0 , \label{P41}\\ 
& P_{-\frac12,N}^{(43)} :=   \Lambda a_9 \Lambda + \Lambda P^{(3)}_{-1,N} \Lambda  
\in \Ops^{-\frac12}\,, \label{P43}\\
& P_{0}^{(44)} := \Lambda [ a_1\pa_x, \Lambda^{-1} ]  + \Lambda a_{10} \Lambda^{-1} \in \Ops^{0} \,,\label{P44}
\end{align}
and $ \bR_{4}^\Psi , \bT_{4,N}$ are the smoothing remainders
\begin{equation}\label{R4T4}
\begin{aligned}
& \bR_{4}^\Psi := \begin{pmatrix}
0 & 0 \\ (\rho^{-1}m_{\frac32}-1) \pi_0 & 0
\end{pmatrix} + \wt\cM^{-1} \bR_{3}^\Psi \wt\cM 
\in \Ops^{-\infty} \, , \\
& \bT_{4,N} := \wt\cM^{-1} \bT_{3,N} \wt\cM = 
- \frac{\gamma^2}{4 \rho} 
\begin{pmatrix}
0 & 0 \\ \Lambda  q \tR_{2,N} q \Lambda & 0 
\end{pmatrix} \, . 
\end{aligned}
\end{equation}

\paragraph{Step 4:} 
We finally move in complex coordinates, conjugating the operator $\cL_4 $  in \eqref{cL3} via the transformation $\cC$ defined in \eqref{C_transform}. 
We use the transformation formula \eqref{C_transformed}. 
We choose the 
function $ p(\vphi ) $ in \eqref{def:p} in order to obtain a constant coefficient
at the highest order. More precisely we choose the periodic function $ p(\varphi)$ such that 
\begin{equation}\label{deffp}
\frac{m_{\frac32} }{\rho} \stackrel{\eqref{m32}, \eqref{rho_vartheta}} 
= \cP^{-1} \Big( \frac{
\Big( \frac{1}{2 \pi} \int_{\T} \sqrt{1 + \eta_x^2 (\vphi, x)} \di x \Big)^{-\frac32}
}{1 + \omega \cdot \pa_{\varphi}p }    \Big)  = \tm_{\frac32}  
\end{equation}
is a real constant independent of $ \varphi $. Thus, recalling \eqref{paext},  
we define the periodic function
\begin{equation}\label{choice_p}
	p(\vphi) := (\omega \cdot \pa_\vphi)^{-1}_{\rm ext} 
	\Big( \frac{1}{\tm_{\frac32}}
	\Big( \frac{1}{2 \pi} \int_{\T} \sqrt{1 + \eta_x^2 (\vphi, x)} \di x \Big)^{-\frac32}- 1 \Big) 
\end{equation}
and the real constant 
\begin{equation}\label{defm32}
\tm_{\frac32} := \frac{1}{(2 \pi)^\nu} \int_{\T^\nu} 
\Big( \frac{1}{2 \pi} \int_{\T} \sqrt{1 + \eta_x^2 (\vphi, x)} \di x \Big)^{-\frac32} \di \vf \, .
\end{equation}
Note that \eqref{deffp} holds for $ \omega \in \tD\tC (\upsilon, \tau) $.
Moreover,  by Lemmata \ref{compo_moser}, \ref{BVtilde.mom} and \eqref{uI0},  $ p $ satisfies \eqref{stimap} and it is odd in $\vf$. 
Let
\begin{equation*}\label{Pi0}
	\b\Pi_0:= -\im \,\cC^{-1}\begin{pmatrix}
	0 & 0 \\ \pi_0 & 0
	\end{pmatrix}\cC = \frac{1}{2}\begin{pmatrix}
	\pi_0 & \pi_0 \\ - \pi_0 & -\pi_0
	\end{pmatrix}\,.
\end{equation*}

\begin{lem}\label{LEMMONE}
Let $N \in \N $, $ \tq_0 \in \N_0 $. For all $ \omega \in \tD\tC (\upsilon, \tau) $,  
we have that  
\begin{equation}
\label{cL4}
\begin{aligned}
\cL_5 & := \big(\cE \cQ \wt\cM \cC \big)^{-1} \cL_1 \big( \cE \cQ \wt\cM \cC \big) \\ 
& = \omega\cdot \partial_\vf   + 
\im   \tm_{\frac32}  \b\Omega(\kappa, D) + \bA_1 \pa_x + \im \b\Pi_0+
\bR_5^{(0,d)} + \bR_5^{(0,o)}     + \bT_{5,N} \,,
\end{aligned}
\end{equation}
where: 
\begin{enumerate}
\item The operators $\cE^{\pm 1}$ are $\cD^{k_0}$-$(k_0+1)$-tame, the operators $\cE^{\pm 1}-{\rm Id}$, $(\cE^{\pm 1}-{\rm Id})^*$ are $\cD^{k_0}$-$(k_0+2)$-tame and the operators $\cQ^{\pm 1}$, $\cQ^{\pm 1}-{\rm Id}$, $(\cQ^{\pm 1}-{\rm Id})^*$ are $\cD^{k_0}$-tame
with tame constants satisfying, for some $\sigma:=\sigma(\tau,\nu,k_0)>0$ and for all $s_0\leq s\leq S$,
\begin{align}
	& \fM_{\cE^{\pm 1}}(s) \lesssim_{S} 1+ \normk{\fI_0}{s+\sigma}\,, \quad \fM_{\cQ^{\pm 1}}(s) \lesssim_{S} 1+ \normk{\fI_0}{s+\sigma} \,, \label{step5.est10}\\
	& \fM_{\cE^{\pm 1}-{\rm Id}}(s) + \fM_{\left(\cE^{\pm 1}-{\rm Id}\right)^*}(s) \lesssim_{S}\varepsilon^2 ( 1+ \normk{\fI_0}{s+\sigma})\,, \label{step5.est11}\\
	&\fM_{\cQ^{\pm 1}-{\rm Id}}(s) + \fM_{\left(\cQ^{\pm 1}-{\rm Id}\right)^*}(s) \lesssim_{S}\varepsilon^2 ( 1+ \normk{\fI_0}{s+\sigma})\label{step5.est12} \,;
\end{align}
\item \label{i2-77} the constant $\tm_{\frac32} \in \R $  defined in \eqref{defm32}  satisfies
$ | \tm_{\frac32} - 1 |^{k_0, \upsilon} \lesssim \varepsilon^2 $;  
\item  $\b\Omega(\kappa, D)$ is the Fourier multiplier  
(see \eqref{eq:lin00_ww_C}, \eqref{Omega})
\begin{equation}
\label{step5.est1}
\b\Omega(\kappa, D)=
\begin{pmatrix}
\Omega(\kappa, D) &  0\\
0  & - \bar{\Omega(\kappa, D)} 
\end{pmatrix}, \quad 
\Omega(\kappa, D) = \omega (\kappa, D) + \im \,\frac{\gamma}{2}\partial_x^{-1} G(0) \, ; 
\end{equation}
 \item the matrix of functions $\bA_1 $ is 
\begin{equation}
\label{step5.est2}
\bA_1 := \begin{pmatrix}
a_1^{(d)} & 0\\
 0& a_1^{(d)} 
\end{pmatrix} \, , 
\end{equation}
for a real function $a_1^{(d)}(\vf, x) $ which is a quasi-periodic traveling wave, $  \even(\vf,x) $, 
satisfying, for some $\sigma:= \sigma(k_0,\tau, \nu)>0$ and for all $s\geq s_0$, 
\begin{equation}
		\label{step5.est3}
		\begin{aligned}
		\normk{a_1^{(d)}}{s}   \lesssim_{s} \varepsilon ( 1 + \normk{\fI_0}{s+\sigma} ) \,;
		\end{aligned}
	\end{equation}
\item $\bR_5^{(0,d)}$ and $\bR_5^{(0,o)}$ 
are pseudodifferential operators in $\Ops^{0}$ of the form
	\begin{align}\footnotesize
\label{step5.est4}
\bR_5^{(0,d)} :=
\begin{pmatrix}
r_5^{(d)}(\vf, x, D)  & 0 \\
0  &  \bar{r_5^{(d)}(\vf,x, D)}
\end{pmatrix} , 
\quad
\bR_5^{(0,o)} :=
\begin{pmatrix}
0 & r_5^{(o)}(\vf,x, D)  \\
  \bar{r_5^{(d)}(\vf,x, D)} & 0 
\end{pmatrix} \, , 
\end{align} 
reversibility and momentum preserving,  satisfying, 
for  some $\sigma_N := \sigma(\tau, \nu, N)>0$, for all $s\geq s_0 $, $\alpha\in\N_0$,
	\begin{equation}
		\label{step5.est5}
	\normk{ \bR_5^{(0,d)}}{0,s,\alpha} +	\normk{ \bR_5^{(0,o)} }{0,s,\alpha} \lesssim_{ s,N, \alpha} \varepsilon ( 1+ \normk{\fI_0}{s+\sigma_N +2\alpha} )\,;
	\end{equation}
\item
For any $ \tq \in \N^\nu_0 $ with $ |\tq| \leq \tq_0$, 
 $n_1, n_2 \in \N_0 $  with $ n_1 + n_2  \leq N -(k_0 + \tq_0) + \frac52  $,  the  
 operator $\langle D \rangle^{n_1}\partial_{\vphi}^\tq \bT_{5, N}(\vphi) \langle D \rangle^{n_2}$ is 
$\cD^{k_0} $-tame with a tame constant satisfying, for some $\sigma_N(\tq_0) := \sigma_N(\tq_0,k_0,\tau,\nu)>0$ and for any $s_0 \leq s \leq S $,   
\begin{equation}\label{step5.est6}
{\mathfrak M}_{\langle D \rangle^{n_1}\partial_{\vphi}^\tq \bT_{5, N}(\vphi) \langle D \rangle^{n_2}}(s) \lesssim_{S, N, \tq_0} 
\varepsilon \big( 1+ \normk{\fI_0}{s+\sigma_N(\tq_0)} \big)\,;
\end{equation}
\item Moreover, for any $s_1$ as in \eqref{s1s0}, $\alpha\in\N_0$, $\tq\in\N_0^\nu$, with $\abs\tq\leq \tq_0$, and $n_1,n_2 \in\N_0$, with $n_1+n_2\leq N- \tq_0 + \frac32 $,
	\begin{align}
		& \| \Delta_{12} (\cA) h  \|_{s_1} \lesssim_{s_1} \varepsilon \norm{i_1-i_2}_{s_1+\sigma}\norm{h}_{s_1+\sigma} \,, \quad \cA \in \{ \cE^{\pm 1} , (\cE^{\pm 1})^*, \cQ^{\pm 1} = (\cQ^{\pm 1})^*  \}\,, \label{step5.est13} \\ 
		&  \|\Delta_{12}a_1^{(d)} \|_{s_1}  \lesssim_{s_1} 
		\varepsilon \norm{i_1-i_2}_{s_1+\sigma} \,,  \ 
		| \Delta_{12} \tm_{\frac32} |  \lesssim 
		\varepsilon^2 \norm{i_1-i_2}_{s_1+\sigma}  \, , 
		\label{step5.est7}\\
		& \| \Delta_{12} \bR_5^{(d)} \|_{0,s_1,\alpha} + \|\Delta_{12} \bR_5^{(o)}\|_{0,s_1,\alpha} \lesssim_{s_1,N,\alpha} \varepsilon\norm{ i_1-i_2 }_{s_1+\sigma_N+2\alpha}\,,\label{step5.est8}\\
		& \norm{\braket{D}^{n_1}\pa_\vf^\tq \bT_{5,N}(\vf)\braket{D}^{n_2} }_{\cL(H^{s_1})} \lesssim_{s_1, N, \tq_0} \varepsilon \norm{ i_1-i_2}_{s_1+\sigma_N(\tq_0)} \,.\label{step5.est9} 
	\end{align}
\end{enumerate}
The real operator  $\cL_5$ is Hamiltonian, reversible and momentum  preserving. 
\end{lem}

\begin{proof}
By the expression of $ \cL_4 $ in \eqref{cL3}, 
using \eqref{C_transformed}, and \eqref{deffp}, 
we obtain that $ {\cal L}_5 $ has the form  \eqref{cL4}. 
	The functions $\beta$ and $q$, defined respectively in \eqref{betasolved} and \eqref{q}
	with $ a_3 $ defined in  \eqref{ta2}, satisfy, by Lemmata \ref{product+diffeo}, \ref{compo_moser} and
	\eqref{lem:ga2}, for some $\sigma:=\sigma(k_0,\tau,\nu) > 0 $ and for all $s\geq s_0$,
	\begin{equation}\label{betaq.est}
		\normk{\beta}{s} \lesssim_{s} \varepsilon^2 (1+\normk{\fI_0}{s+\sigma})\,, \quad \normk{q^{\pm 1}-1}{s} \lesssim_{s} \varepsilon^2 ( 1+\normk{\fI_0}{s+\sigma} ) \,.
	\end{equation}
	The estimates \eqref{step5.est10}-\eqref{step5.est12} 
	follow by 
	Lemmata \ref{tame_compo}, \ref{tame_pesudodiff}, \ref{product+diffeo},
	 \eqref{betaq.est} and writing
	\begin{equation}\label{B_repr}
		\begin{aligned}
			& (\cB- {\rm Id})h = \beta\cB_\tau[h_x]\,, \quad \cB_\tau[h](\vf,x) := \int_0^1 h_x(\vf,x+\tau\beta(\vf,x))\wrt \tau \,, 
		\end{aligned}
	\end{equation}
	$ \cB^*h(\vf,y) = ( 1+\breve\beta(\vf,y )) h(\vf,y+\breve\beta(\vf,y)) $, 
	and similar expressions  for $\cB^{-1}-{\rm Id}$, $(\cB^{-1})^*$. 
	The estimate for $\tm_{\frac32}$ follows by  \eqref{defm32}, Lemma \ref{compo_moser} and  \eqref{uI0}.
	The real function $a_1^{(d)}$ in \eqref{step5.est2} is 
$$
		a_1^{(d)}(\vf,x) := \rho(\vf)^{-1} a_1(\vf,x)\,,
$$
	where $\rho$ and $a_1$ are defined respectively in \eqref{rho_vartheta} and \eqref{ta1}. 
	Recalling Lemmata \ref{BVtilde.mom} and \ref{timerep}, the function $ a_1^{(d)} $  is a quasi-periodic traveling wave, even in $ (\vf,x) $. 
Moreover, \eqref{step5.est3}  follows by 
Lemma \ref{compo_moser} and \eqref{uI0}, \eqref{stimap}, \eqref{lem:ga2}, \eqref{betaq.est}. 
	By direct computations, we have
	\begin{equation}\label{r5dr50}
		\begin{aligned}
			& r_5^{(d)}(\vf,x,D):= \frac{1}{2\rho}\left( P_0^{(41)}  + P_0^{(44)} + \im P_{-\frac12,N}^{(43)}  + \gamma( \rho \, \tm_\frac32 -1) G(0)\pa_x^{-1} \right) \,, \\
			&r_5^{(o)}(\vf,x,D):= \frac{1}{2\rho}\left( P_0^{(41)}  - P_0^{(44)} + \im P_{-\frac12,N}^{(43)} \right)\,,
		\end{aligned}
	\end{equation}
	where $P_0^{(41)}$, $P_{-\frac12,N}^{(43)}$, $P_0^{(44)}$ are defined in \eqref{P41}, \eqref{P43}, \eqref{P44} and $ \rho  \,\tm_\frac32 = m_\frac32 (\vf) $ 
	with 	$ m_\frac32 (\vf)  $ defined in  \eqref{m32} (cfr. \eqref{deffp}). 
	 Therefore, the estimate \eqref{step5.est5} follows by  \eqref{bV2}, \eqref{a7}, 
	   \eqref{ta1}, \eqref{ta2}, \eqref{ta5}, \eqref{ta6}, \eqref{P-1N_2}, \eqref{tildecP-1}, \eqref{Lambda}, \eqref{eq:T_sym}, applying Lemmata \ref{pseudo_compo}, \ref{pseudo_commu},  \ref{product+diffeo}, \ref{compo_moser}, Proposition \ref{egorov} and estimates \eqref{uI0}, \eqref{stimap}, \eqref{lem:ga2}, \eqref{betaq.est}. 
	The estimate \eqref{step5.est6}, where 
	$$
	\bT_{5,N}:= \cC^{-1} (\bR_{4}^{\Psi}+\bT_{4,N}) \cC \, , 
	$$
	follows by \eqref{R4T4}, \eqref{T3N}, \eqref{bR2}, \eqref{bR1}, \eqref{cR3}, \eqref{pax-1_ego}, \eqref{cR2}, \eqref{wtR1}, \eqref{cR1}, 
Lemmata \ref{tame_compo}, \ref{tame_pesudodiff},
estimates \eqref{ht.t}, \eqref{Ht.comm}, 
 Proposition \ref{egorov} and  \eqref{step5.est10}, \eqref{betaq.est}, 
Lemma \ref{DN_pseudo_est} and Lemmata 2.34, 2.32 in \cite{BM}.
	The estimates \eqref{step5.est13}, \eqref{step5.est7}, \eqref{step5.est8}, \eqref{step5.est9} are proved in the same fashion.	
	Since the transformations $\cE$, $\cQ$, $\wt\cM$ are symplectic, the operator 
	$ {\cal L}_4$ is Hamiltonian.
	Hence the operator $\cL_{5}$ obtained conjugating with $ \cC $ is Hamiltonian 
	according to \eqref{RHamC}.
	By Lemma \ref{BVtilde.mom}, the functions $\beta(\vf,x)$ and $q(\vf,x)$, defined in \eqref{betasolved}, \eqref{q}
	(with $ a_3 $ defined in \eqref{ta2}), are both quasi-periodic traveling waves,  respectively $\odd(\vf,x)$ and $\even(\vf,x)$. Therefore, the transformations $\cE$ and $\cQ$ are momentum and reversibility preserving. Moreover, also $\wt\cM$ and $\cC$ are momentum and reversibility preserving (writing the involution in complex variables as in 
	\eqref{inv-complex}). 
	Hence,  since 
	$ \cL_1 $ is momentum preserving and reversible (Lemma \ref{lem:good_unknwon}), 
	the operator $\cL_{5}$ is momentum preserving and reversible as well, in particular 
	the operators 	$	\bR_5^{(0,d)} $ and $ \bR_5^{(0,o)} $
	 in \eqref{step5.est4}
	 (e.g. check the definition in \eqref{r5dr50}, see also Remark \ref{rem:REV}).  
\end{proof}

\subsection{Symmetrization up to smoothing remainders}\label{sec:block_dec}
The goal of this section is to transform the operator $\cL_5$ in \eqref{cL4} into the operator $\cL_{6}$ in \eqref{cL6M} which is block diagonal up to a regularizing remainder. 
From this step we do not preserve any further the Hamiltonian structure, 
but only the reversible and momentum preserving one
(it is now sufficient for proving Theorem \ref{NMT}).

\begin{lem}\label{block_dec_lemma}
Fix $M, N \in \N $,  $ \tq_0 \in \N_0$.	There exist  
real, reversibility  and momentum  preserving operator matrices 
$\{ \bX_m \}_{m=1}^M$   of the form
	\begin{equation}\label{geno}
	\bX_m:= \begin{pmatrix}
	0 & \chi_m(\vf,x,D) \\ \bar{\chi_m(\vf,x,D)} & 0   
	\end{pmatrix},
	\qquad
	\chi_m(\vf,x,\xi) \in S^{- \frac12 -m} \, ,
	\end{equation}
such that, conjugating the operator $ \cL_5 $ in \eqref{cL4}  via the map 
\begin{equation}\label{bPHIM}
\b\Phi_M:= e^{\bX_1}\circ \cdots \circ e^{\bX_M} \, ,
\end{equation} 
we obtain the real, reversible and  momentum preserving operator 
	\begin{equation}\label{cL6M}
		\begin{aligned}
			\cL_6 := \cL_{6}^{(M)} & := \b\Phi_M^{- 1} \, \cL_5 \, \b\Phi_M \\
			& =  \omega\cdot\partial_\vf  + \im \,\tm_{\frac32} \b\Omega(\kappa, D) + \bA_1 \pa_x +\im\b\Pi_0  + \bR_{6}^{(0, d)} + \bR_{6}^{(- M, o)} + \bT_{6,N}\,,
		\end{aligned}
	\end{equation}
	with a block-diagonal operator 
	\begin{align*}
	\bR_{6}^{(0,d)} := \bR_{6,M}^{(0,d)} & := \begin{pmatrix}
	r_{6}^{(d)}(\vf,x,D) & 0 \\ 
	0 &\bar{r_{6}^{(d)}(\vf,x,D)}
	\end{pmatrix} \in \Ops^0  \,,
	\end{align*}
and a smoothing off diagonal remainder 	
	\begin{align}
	\bR_{6}^{(- M, o)} := \bR_{6,M}^{(- M, o)}  & := \begin{pmatrix}
	0 & r_{6}^{(o)}(\vf,x,D) \\ 
	\bar{r_{6}^{(o)}(\vf,x,D)} & 0 
	\end{pmatrix} \in \Ops^{- M} \label{R6o}
	\end{align}
both	reversibility and momentum preserving,  
	which satisfy for all  $\alpha\in\N_0 $, 
for  some $ \sigma_N := \sigma_N(k_0,\tau, \nu, N)>0 $,   
$\aleph_{M}(\alpha) > 0 $, for all 
$ s\geq s_0 $, 
	\begin{align}
		&\normk{\bR_{6}^{(0,d)}}{0,s,\alpha} + 
		\normk{\bR_{6}^{(- M, o)}}{- M,s,\alpha} \lesssim_{ s, M, N, \alpha} \varepsilon \big( 1+\normk{\fI_0}{s+\sigma_N+\aleph_M(\alpha)} \big) \, \label{bR6esti1} .
\end{align}
For any $ \tq \in \N^\nu_0 $ with $ |\tq| \leq \tq_0$, 
$n_1, n_2 \in \N_0 $  with $ n_1 + n_2 \leq N -(k_0+\tq_0) + \frac52  $,  the  
operator $\langle D \rangle^{n_1}\partial_{\vphi}^\tq \bT_{6, N}(\vphi) \langle D \rangle^{n_2}$ is 
$\cD^{k_0} $-tame with a tame constant satisfying, for some  $\sigma_N(\tq_0) := 
\sigma_N(k_0,\tau, \nu, \tq_0) $, for 
any $s_0 \leq s \leq S $, 
\begin{equation}\label{block.est2}
{\mathfrak M}_{\langle D \rangle^{n_1}\partial_{\vphi}^\tq \bT_{6, N}(\vphi) \langle D \rangle^{n_2}}(s) \lesssim_{S, M, N, \tq_0} 
\varepsilon ( 1+ \normk{\fI_0}{s+\sigma_N(\tq_0) + \aleph_M(0)} )\,.
\end{equation}
The conjugation map $ \b\Phi_M $ in \eqref{bPHIM} satisfies, 
for all $s\geq s_0$,
	\begin{equation}\label{block.est6}
		\normk{ \b\Phi_M^{\pm 1}-{\rm Id} }{0,s,0} + \normk{\left(\b\Phi_M^{\pm 1}-{\rm Id}\right)^*  }{0,s,0} \lesssim_{s, M, N} \varepsilon ( 1+\normk{\fI_0}{s+\sigma_N+ \aleph_M(0)} )\,.
	\end{equation}
Furthermore, for any $s_1$ as in \eqref{s1s0}, $\alpha\in\N_0$, $\tq\in\N_0^\nu$, with $\abs\tq \leq \tq_0$, and $n_1,n_2\in\N_0$, with $n_1+n_2\leq N- \tq_0 + \frac32$, we have
\begin{align}
	&\|\Delta_{12} \bR_{6}^{(0,d)} \|_{0,s_1,\alpha} +\|\Delta_{12} \bR_{6}^{(- M,o)} \|_{- s_1, M, \alpha} \lesssim_{ s_1, M, N, \alpha} \varepsilon \norm{ i_1-i_2 }_{s_1+\sigma_N+\aleph_{M}(\alpha)}   \,, \label{block.est3} \\
	& \| \braket{D}^{n_1} \pa_\vf^\tq \Delta_{12} \bT_{6,N} \braket{D}^{n_2}\|_{\cL(H^{s_1})} \lesssim_{s_1,
M, N, \tq_0} \varepsilon \norm{i_1-i_2}_{s_1+ \sigma_N(\tq_0)+ \aleph_{M}(0) }\,,\label{block.est4}\\
	& \|\Delta_{12} \b\Phi_M^{\pm 1} \|_{0,s_1,0} +\|\Delta_{12} (\b\Phi_M^{\pm 1})^* \|_{0,s_1,0} \lesssim_{s_1, M, N} \varepsilon \norm{ i_1-i_2 }_{s_1+\sigma_N+\aleph_{M}(0)} \,. \label{block.est5}
\end{align}
\end{lem}
\begin{proof}
The proof is inductive on the index $ M $. 
The operator  $\cL_6^{(0)}:= \cL_5 $ satisfy \eqref{bR6esti1}-\eqref{block.est2}  
with $	\aleph_0(\alpha) := 2\alpha $, by Lemma \ref{LEMMONE}.
Suppose we have done already $M$ steps obtaining an operator
	$ \cL_6^{(M)} $ as in \eqref{cL6M} with a remainder 
	$ {\bf \Phi}_{M}^{-1} \bT_{5,N} {\bf \Phi}_{M}$, instead of 
	$ \bT_{6,N} $.  We now 
	show how to perform the $(M+1)$-th step.
	Define the symbol 
	\begin{equation}\label{chiM+1}
	\chi_{M+1}(\vf,x,\xi):= -\big(2\im\,  \tm_\frac32 \omega (\kappa, \xi) \big)^{-1}r_{6,M}^{(o)}(\vf,x,\xi) \chi (\xi) 
	 \in S^{-\frac32 -M } \, , 
	\end{equation}
	where $ \chi $ is the cut-off function defined in \eqref{cutoff} and 
	$\omega (\kappa, \xi) $ is the symbol 
	(cfr. \eqref{eq:omega0}) 
	$$
	\omega (\kappa, \xi):= 
	\sqrt{  G(0; \xi) \Big( \kappa \xi^2  
	+ g +  
\frac{\gamma^2}{4} \frac{G(0; \xi)}{\xi^2}    \Big)} \in S^{\frac32} \, , 
\ \ G(0; \xi) := 
\begin{cases}
\chi (\xi) |\xi| \tanh (\tth |\xi|) \, , \  \tth < + \infty \cr
\chi (\xi) |\xi| \, , \qquad \qquad 	\ \,  \ \tth = + \infty \\,.
\end{cases}
$$
Note that
$\chi_{M+1} $ in \eqref{chiM+1} is well defined because 
$ \omega (\kappa, \xi) $ is positive on the support of $ \chi (\xi ) $. 	
We conjugate the operator $\cL_6^{(M)}  $ in \eqref{cL6M}  by the flow generated by 
$\bX_{M+1}$ of the form \eqref{geno} with	
$ \chi_{M+1} (\vphi, x, \xi) $ defined in \eqref{chiM+1}. 
By	
\eqref{bR6esti1} 
and Lemma \ref{LEMMONE}-\ref{i2-77}, for any $s\geq s_0$ and $\alpha\in\N_0$,
	\begin{equation}\label{chiM+1_est}
\normk{\bX_{M+1}}{-\frac12 -(M+1),s,\alpha} \lesssim_{s, M, \alpha} \varepsilon \big( 1+ \normk{\fI_0}{s+\sigma_N + \aleph_{M}(\alpha)} \big) \,.
	\end{equation} 
	Therefore, by Lemmata \ref{Neumann pseudo diff}, \ref{pseudo_compo} and
	the induction assumption \eqref{block.est6} for $\b\Phi_{M}$, 
	the conjugation map $\b\Phi_{M+1}:= \b\Phi_{M}e^{\bX_{M+1}}$ is well defined 
	and satisfies estimate \eqref{block.est6} with $M+1$.
 By the Lie expansion \eqref{lie_abstract} 
	we have
	\begin{align}
\label{conM+1}
	\cL_6^{(M+1)} & := e^{-\bX_{M+1}} \,  \cL_6^{(M)} \, e^{ \bX_{M+1}}  =\omega\cdot \partial_\vf + \im \tm_{\frac32}\b\Omega(\kappa,D)+ \bA_1\pa_x+\im\b\Pi_0 + \bR_{6,M}^{(0,d)} \\
	\nonumber
	& - \big[\bX_{M+1}, \im\,\tm_{\frac32} \b\Omega(\kappa,D) \big] + \bR_{6,M}^{(- M,o)} + {\bf \Phi}_{M+1}^{-1} \bT_{5,N} {\bf \Phi}_{M+1}  \\
	\label{geno12}
	& -\int_0^1 \LieTr{\tau \bX_{M+1}}{\big[\bX_{M+1} \, , \, \omega\cdot \partial_\vf  + \bA_1^{(d)}\pa_x +\im\b\Pi_0+  \bR_{6,M}^{(0,d)}  \big]}\wrt \tau\\
	\label{geno13}
	& -\int_0^1 \LieTr{\tau\bX_{M+1}}{\left[\bX_{M+1}, \bR_{6,M}^{(-M,o)}  \right]}\wrt \tau \\
	\label{geno14}
	& +\int_0^1(1-\tau)\LieTr{\tau\bX_{M+1}}{\left[\bX_{M+1},\left[\bX_{M+1},\im\,\tm_{\frac32}\b\Omega(\kappa,D)\right]\right]}\wrt \tau \, .
	\end{align}
 In view of  \eqref{geno}, \eqref{step5.est1} and \eqref{R6o},  we have that
$$
-\big[ \bX_{M+1},\im\,\tm_{\frac32} \b\Omega(\kappa,D) \big] + \bR_{6,M}^{(- M,o)}  
=
	\begin{pmatrix}
	0 & Z_{M+1} \\ 
	\bar{Z_{M+1}} & 0 
	\end{pmatrix} =: \bZ_{M+1} ,,
$$
	where, denoting for brevity  $ \chi_{M+1} := \chi_{M+1}(\vf,x,\xi) $, it results 
	\begin{align}
		Z_{M+1} & 
		= 
		\im \, \tm_{\frac32} 
		\left(  {\rm Op}(\chi_{M+1}) \omega(\kappa,D) + \omega(\kappa,D)
		 {\rm Op}(\chi_{M+1})
		\right) \nonumber \\
&	\quad + \tm_{\frac32}  \tfrac{\gamma}{2}\left[ \chi_{M+1}, \pa_x^{-1}G(0)  \right]		
		  + {\rm Op}(r_{6,M}^{(o)}) \, .  \label{eq:hom_scalar}
	\end{align}
	 By \eqref{compo_symb}, Lemma \ref{pseudo_compo} and since
	$  \chi_{M+1}(\vf,x,\xi)  \in S^{-\frac32 -M } $ by 
	 \eqref{chiM+1}, 
	 we have that 
$$
{\rm Op}(\chi_{M+1}) \omega(\kappa,D) + \omega(\kappa,D) {\rm Op}(\chi_{M+1}) 
= \Op\big(  2
\omega (\kappa, \xi)\chi_{M+1}(\vf,x,\xi) \big) + {\mathtt r}_{M+1} ,,
$$
where ${\mathtt r}_{M+1} $ 
is in $ \Ops^{-M-1} $. By  \eqref{chiM+1} and \eqref{eq:hom_scalar} 
$$
Z_{M+1}  =  \im \tm_{\frac32}  {\mathtt r}_{M+1}  +
 \tm_{\frac32}  \tfrac{\gamma}{2}\left[ \chi_{M+1}, \pa_x^{-1}G(0)  \right] 
		  + {\rm Op}(r_{6,M}^{(o)}(1- \chi (\xi)))  \in \Ops^{-M-1} \, . 
$$
The
remaining pseudodifferential operators in \eqref{geno12}-\eqref{geno14} 
have order $ \Ops^{-M-\frac32} $. 
	Therefore the operator $ \cL_6^{(M+1)} $ in \eqref{conM+1} has the form 
	\eqref{cL6M} at $ M+ 1 $ with 
\begin{equation}\label{newRM+1}
	\bR_{6,M+1}^{(0,d)}+\bR_{6,M+1}^{(-(M+1),o)}
	:= \bR_{6,M}^{(0,d)}+\bZ_{M+1}+\eqref{geno12}+\eqref{geno13}+\eqref{geno14}
\end{equation}
	and a remainder 
	$ {\bf \Phi}_{M+1}^{-1} \bT_{5,N} {\bf \Phi}_{M+1}$.
	 By 
	 Lemmata \ref{pseudo_compo}, \ref{pseudo_commu}, the induction assumption 
	 \eqref{bR6esti1},
	  \eqref{chiM+1_est}, \eqref{step5.est3}, we conclude that $\bR_{6,M+1}^{(0,d)}$ and $\bR_{6,M+1}^{(-(M+1),o)}$ satisfy \eqref{bR6esti1} at order $M+1$ for
	  suitable constants $ \aleph_{M+1} (\alpha) >  \aleph_{M} (\alpha) $. 
	Moreover the operator $\b\Phi_{M+1}^{-1} \bT_{5, N} \b\Phi_{M+1} $ satisfies 
	\eqref{block.est6} (with $M+1$) by Lemmata \ref{tame_compo}, \ref{tame_pesudodiff} and estimates \eqref{step5.est6}, \eqref{block.est6}.
	Estimates \eqref{block.est3}, \eqref{block.est4}, \eqref{block.est5} follow similarly.
	By \eqref{chiM+1}, \eqref{Ri-RAR}, Lemmata \ref{lem:rev-pse}, \ref{lem:mom_pseudo},
	and the induction assumption that $\bR_{6,M}^{(-M,o)}$ is reversible and momentum preserving, we conclude that $\bX_{M+1}$ is reversibility and momentum preserving, and so are $e^{\pm\bX_{M+1}}$. By the induction assumption $\cL_{6}^{(M)}$ is reversible and momentum preserving, and so 
	$\cL_{6}^{(M+1)}$ is reversible and momentum preserving as well, in particular
	the terms 
	$ \bR_{6,M+1}^{(0,d)}+\bR_{6,M+1}^{(-(M+1),o)} $
	in \eqref{newRM+1}.
\end{proof}

\begin{rem}\label{rem:after_block}
	The number of regularizing iterations $M\in\N$ will be 
	fixed by the KAM reduction scheme in Section \ref{sec:KAM}, see 
	\eqref{M_choice}. Note that it is independent of the Sobolev index $s$. 
\end{rem}

So far the operator $ \cL_6 $ of  Lemma \ref{block_dec_lemma} depends on 
two indexes $ M, N  $ which provide respectively the order 
 of the regularizing  off-diagonal remainder $ \bR_{6}^{(- M, o)} $ and of the smoothing 
 tame operator $ \bT_{6,N} $.
From now on we  fix 
\begin{equation}\label{M=N}
N=M \, .
\end{equation}

\subsection{Reduction of the order 1}\label{sec:order1}
The goal of this section is to transform the operator $\cL_6$ in \eqref{cL6M},
with $N=M$ (cfr. \eqref{M=N}), into the 
operator $\cL_8 $ in \eqref{cL8} whose  coefficient 
in front of $ \pa_x $ is a constant. 
We first eliminate the $ x$-dependence and then the $ \vf $-dependence.  

\paragraph{Space reduction.}

First we rewrite the operator $\cL_{6}$ in \eqref{cL6M}, with $N=M$, as
$$
	\cL_{6} = \omega\cdot \pa_\vf + \begin{pmatrix}
	P_6 & 0 \\ 0 & \bar{P_6}
	\end{pmatrix} + \im\b\Pi_0 + \bR_{6}^{(-M,o)} + \bT_{6,M}\,,
$$
having denoted 
\begin{equation}
\label{P6}
P_6 := P_6(\vf,x,D) := \im \tm_{\frac32} \Omega(\kappa, D) + a_1^{(d)}(\vf,x) \pa_x  +  r_6^{(d)}(\vf,x, D) \,.
\end{equation} 
We  conjugate $\cL_{6}$ 
through the real operator 
\begin{equation}\label{Phi1}
\b\Phi(\vf) := \begin{pmatrix}
\Phi(\vf) & 0 \\ 0 & \bar{\Phi}(\vf)
\end{pmatrix} 
\end{equation}
where $\Phi(\vf):=\Phi^\tau(\vf)|_{\tau=1} $ is the time $ 1 $-flow of the PDE
\begin{equation}\label{phi.problem}
\begin{cases}
\partial_\tau \Phi^\tau(\vf) =  \im A(\vf)  \Phi^\tau(\vf) \, , \\
\Phi^0(\vf) = {\rm Id} 	\,  , 
\end{cases}   \qquad A(\vf) := b(\vf, x) |D|^{\frac12} \, , 
\end{equation}
and $b(\vf,x)  $ is a real, smooth, odd $(\vf,x)$, periodic function 
chosen later, see \eqref{b1b2}, \eqref{eq:q_choice}, \eqref{def:b2}. 
Usual energy estimates imply that the flow 
$ \Phi^\tau (\vf) $ of \eqref{phi.problem} 
is a bounded operator is $H_x^s $. 
The operator $\pa_\lambda^k\pa_\vf^\beta\Phi$ loses $\abs{D}^{\frac{\abs\beta +\abs k}{2}}$ derivatives, which  are compensated by $\braket{D}^{-m_1}$ on the left hand side and $\braket{D}^{-m_2}$ on the right hand side, with $m_1,m_2\in\R$ satisfying $m_1+m_2= \tfrac12\left( \abs\beta+\abs k \right)$, according to the 
tame estimates in the Sobolev spaces $H_{\vf,x}^s$ of Proposition 2.37 in \cite{BBHM}. 
Moreover, since $ b(\vf,x)$ is $\odd(\vf,x)$, then 
$ b(\vf, x) |D|^{\frac12}$ is reversibility preserving as well as 
$	\b\Phi (\vf)  $.
Finally, note that 
$\Phi\pi_0 = \pi_0 = \Phi^{-1}\pi_0$, which implies
\begin{equation}\label{phipi0}
	\b\Phi^{-1}\b\Pi_0 \b\Phi = \b\Pi_0\b\Phi \,.
\end{equation}
By the Lie expansion \eqref{lie_abstract} we have  
\begin{equation}
\begin{aligned}
\Phi^{-1} P_6 \Phi  & = 
P_6 - \im [A, P_6] - \frac12 [A, [A, P_6]]+  \sum_{n=3}^{2M+2} \frac{(-\im)^n}{n!} \ad_{A(\vf)}^n(P_6) + T_{M}\,,\label{lin1} \\
T_{M}& :=  \frac{(- \im)^{2M+3}}{(2M+2)!} \int_0^1 (1 - \tau)^{2M+2} \Phi^{-\tau}(\vf)\,
\ad_{A(\vf)}^{2M+3}(P_6) \,\Phi^\tau (\vf)  \di \tau \, ,   
\end{aligned}
\end{equation}
and, by \eqref{lie_omega_devf}, 
\begin{equation}
\begin{aligned}
\label{lin2}
			\Phi^{-1} \circ \omega\cdot\pa_\vf \circ \Phi & = \omega\cdot\pa_\vf + \im (\omega\cdot\pa_\vf A)(\vf) - \sum_{n=2}^{2M+1} \frac{(-\im)^n}{n!} \ad_{A(\vf)}^{n-1}(\omega\cdot \pa_\vf A(\vf)) + T_{M}'\,, \\
			T_{M}' & := - \frac{(- \im)^{2M+2}}{(2M+1)!} 
			\int_0^1 (1 - \tau)^{2M+1} \Phi^{-\tau}(\vf) \,
			\ad_{A(\vf)}^{2M+1}(\omega\cdot \pa_\vf A(\vf))\, \Phi^\tau (\vf) \di \tau \, . 
\end{aligned}
\end{equation}
Note that $ \ad_{A(\vf)}^{2M+3}(P_6) $ and $ \ad_{A(\vf)}^{2M+1}(\omega\cdot \pa_\vf A(\vf)) $ are in $ \Ops^{-M} $. 
The number $ M $ will be fixed in  \eqref{M_choice}. Note also that in 
the  expansions
\eqref{lin1}, \eqref{lin2} the  operators have  decreasing order and size. 
The terms of order 1 come from  \eqref{lin1}, in particular  from 
$P_6 - \im [A, P_6]  $. Recalling \eqref{P6}, that $ A(\vf) := b(\vf, x) |D|^{\frac12} $, 
\eqref{eq:moyal_exp} and that (cfr. \eqref{Om-om}, \eqref{eq:rem1})
 \begin{equation}\label{Omegakx}
 \Omega (\kappa, \xi) = \sqrt{\kappa} |\xi|^{\frac32} \chi(\xi) + r_{0}(\kappa, \xi) \, , \quad 
 r_{0}(\kappa, \xi)\in S^{0} \, , 
\end{equation}
(the cut-off function $ \chi $ is defined in \eqref{cutoff}) 
we deduce that
\begin{equation}\label{sviP6AP6}
[A, P_6]   =  
\im \tfrac32 \sqrt{\kappa} \, \tm_{\frac32}  \,  b_x \pa_x   +
\big( \tfrac12 (a_1^{(d)})_x b - a_1^{(d)} b_x \big) 
|D|^{\frac12} + \Op(r_{b,{0}}) ,, 
\end{equation}
where $ r_{b,0} \in S^0 $ is small with $b$.
As a consequence,
the first order term of $P_6 - \im [A, P_6]  $
is
$(a_1^{(d)}+  \tfrac32\, \sqrt{\kappa}\, \tm_{\frac32}\,  b_x )\pa_x $ and we choose
$ b (\vf,x) $ so that  it is independent of $ x $:  
we look for a solution  
\begin{equation}\label{b1b2}
b(\vf,x) = b_1 (\vf, x ) + b_2 (\vf) 
\end{equation}
of the equation 
\begin{equation}\label{eqo1}
a_1^{(d)}(\vf,x)+ \tfrac32 \tm_{\frac32}\sqrt\kappa \,  b_x(\vf, x) = \braket{a_1^{(d)}}_x(\vf) \, ,
\quad \braket{a_1^{(d)}}_x(\vf) := \frac{1}{2 \pi} \int_{\T} 
a_1^{(d)} (\vf, x) \di x  \, .
\end{equation}
Therefore 
\begin{equation}\label{eq:q_choice}
b_1 (\vf,x):= -\tfrac{2}{3\,\tm_{\frac32}\, \sqrt\kappa} \, \partial_x^{-1} \big( a_1^{(d)}(\vf,x) -\braket{a_1^{(d)}}_x(\vf) \big) \,.
\end{equation}
We now determine $ b_2 (\vf ) $ by imposing a condition at the order $ 1/ 2 $. 
We deduce by \eqref{lin1}, \eqref{lin2}, \eqref{P6}, \eqref{sviP6AP6}-\eqref{eqo1}, that  
\begin{equation}\label{defL7}
	\begin{aligned} 
L_7:= \Phi^{-1}(\vf)\left( \omega\cdot \pa_\vf + P_6 \right) \Phi(\vf) = & \, \omega\cdot\pa_\vf + \im \,\tm_{\frac32} \Omega(\kappa, D) + \braket{a_1^{(d)}}_x(\vf)\, \pa_x \\
		& +  \im \,  a_2^{(d)} |D|^{\frac12}+ \Op(r_7^{(d)}) + T_M + T_M',,
	\end{aligned}
\end{equation}
where $a_2^{(d)}(\vf,x)$ is the real function
\begin{equation}\label{a3d}
\begin{aligned}	 
a_2^{(d)}:= & - \tfrac12 (a_1^{(d)})_x b_1 + a_1^{(d)} (b_1)_x 
	 + \frac34 \sqrt{\kappa} \, 
	 \tm_{\frac32} \big(   (b_1)_x^2 
	 - \frac12 (b_1)_{xx} b_1 \big) + (\omega\cdot\pa_\vf b_1) \\
	&   -		 \big( \tfrac12 (a_1^{(d)})_x 
		  +	 \frac38  \sqrt{\kappa} \, \tm_{\frac32} (b_1)_{xx} \big) b_2 
		 + (\omega\cdot\pa_\vf b_2) 
\end{aligned}
\end{equation}
and 
\begin{equation}\label{r7d}
	\begin{aligned}
		\Op(r_7^{(d)}):=& 
		\Op( - \im r_{b,0} +r_{b,-\frac12}+  r_6^{(d)}) - 
		\frac12 \big[ b |D|^{\frac12}, 
		( \tfrac12 (a_1^{(d)})_x b - a_1^{(d)} b_x)|D|^{\frac12} +
		{\rm Op}( r_{b,0})  \big]\\
		 & +  \sum_{n=3}^{M-1} \frac{(-\im)^n}{n!} \ad_{A(\vf)}^n(P_6)- \sum_{n=2}^{M} \frac{(-\im)^n}{n!} \ad_{A(\vf)}^{n-1}(\omega\cdot \pa_\vf A(\vf))  \in \Ops^0 \, ,
	\end{aligned}
\end{equation}
where $ r_{b,-\frac12} \in S^{-\frac12} $ is small in $ b $.  
In view of Section \ref{sec:order12}
we now determine the function $ b_2 (\vphi ) $ 
so that the space average of the function 
 $ a_2^{(d)} $ in \eqref{a3d} is independent of $ \vphi $, i.e. 
\begin{equation}\label{media2cos}
\langle  a_2^{(d)} \rangle_x (\vphi) = \tm_{\frac12}\in\R \, , \quad \forall \vphi \in \T^\nu \, . 
\end{equation}
Noting that the space average 
$ \big\langle \big( \tfrac12 (a_1^{(d)})_x 
		  +	 \frac38 \tm_{\frac32} \sqrt{\kappa} (b_1)_{xx} \big)  b_2 (\vf) \big\rangle_x  = 0 $ and that $ \big\langle  \omega\cdot\pa_\vf b_1  \big\rangle_{\vf,x} = 0 $, 
we get 
\begin{align}
	\tm_{\frac12} & := \langle - \tfrac12 (a_1^{(d)})_x b_1 + a_1^{(d)} (b_1)_x 
	+ \frac34 \sqrt{\kappa} \, \tm_{\frac32} 
	 \big( (b_1)_x^2  - \frac12 (b_1)_{xx}b_1 \big)
	\rangle_{\vf,x} \, , \label{def:m12} \\ 
	b_2 (\vphi )  & := - ( \omega \cdot \partial_\vphi)_{\rm ext}^{-1} 	\Big( \big\langle - \tfrac12 (a_1^{(d)})_x b_1 + a_1^{(d)} (b_1)_x +
	  \nonumber \\
	& \qquad \qquad \qquad \qquad + \frac34 \tm_{\frac32} \sqrt{\kappa} \big((b_1)_x^2  - \frac12  (b_1)_{xx} b_1 \big) + 
	(\omega\cdot\pa_\vf b_1) \big\rangle_x  - \tm_{\frac12} \Big) 
	\, . \label{def:b2}
\end{align}
Note that \eqref{media2cos} holds  for  any $ \omega \in \tD\tC (\upsilon, \tau ) $. 

\paragraph{Time reduction.}

In order to remove the $\vf$-dependence of  
the coefficient $\braket{a_1^{(d)}}_x(\vf) $ of the first order term of the operator $ L_7 $ in \eqref{defL7}, we conjugate  $ L_7 $ with the map 
\begin{equation}\label{defBvphi}
(\cV u)(\vf, x) := u (\vf, x + \varrho(\vf)) ,,
\end{equation}
where $\varrho(\vf)$ is  a  real periodic function   to be chosen, see \eqref{eq:q2_choice}.
Note that  $ \cV $ is a particular case of the transformation $ \cE $ in 
\eqref{defcE} for a function
$ \beta (\vf, x) = \varrho(\vf) $, independent of $ x$. 
We have that 
$$
\cV^{-1} (\omega \cdot \pa_\vf) \cV = \omega\cdot \pa_\vf + 
(\omega\cdot\pa_\vf \varrho)  \pa_x\,,
$$
whereas the Fourier  multipliers are left unchanged and 
a pseudodifferential operator  of symbol $ a(\vf, x, \xi)  $  transforms as
\begin{equation}\label{TRApse}
\cV^{-1} \Op(a(\vf, x, \xi) ) \cV = \Op( a(\vf, x- \varrho(\vf), \xi))  \, . 
\end{equation}
We choose  $ \varrho(\vf)$ such that
\begin{equation}\label{defm1}
\omega\cdot\partial_\vf \varrho(\vf) + \braket{a_1^{(d)}}_x(\vf) = \tm_1 \, , 
\qquad  \tm_1 := \braket{a_1^{(d)}}_{\vf,x} \in \R \, , 
\end{equation}
(where $a_1^{(d)}$ is fixed in Lemma \ref{LEMMONE}), namely  we define 
\begin{equation}\label{eq:q2_choice}
\varrho(\vf):= -(\omega\cdot\partial_\vf)_{\rm ext}^{-1} \big( \braket{a_1^{(d)}}_x-\tm_1 \big) \,.
\end{equation}
Note that \eqref{defm1} holds  for any $\omega \in \tD\tC(\upsilon,\tau)$.  

We sum up these two transformations into the following lemma.
\begin{lem}\label{red1}
Let $ M \in \N $, $ \tq_0 \in \N_0 $. 
Let  $b(\vf,x)= b_1(\vf,x) + b_2(\vf)$ 
and $\varrho (\vf) $ be the functions defined
 respectively in  
\eqref{eq:q_choice}, \eqref{def:b2},  \eqref{eq:q2_choice}. 
Then, conjugating $ \cL_6  $ in \eqref{cL6M} 
via the invertible, real, reversibility preserving and momentum preserving 
maps  $ \b\Phi $, $ \cV $ defined in \eqref{Phi1}-\eqref{phi.problem} and 
\eqref{defBvphi}, we obtain, for any $ \omega \in \tD\tC (\upsilon, \tau ) $,   
 the real, reversible and  momentum preserving operator 
\begin{equation}
\begin{aligned}\label{cL8}
\cL_8  & :=   \cV^{-1} {\bf \Phi}^{-1} \cL_6 {\bf \Phi} \cV \\
	& = \omega\cdot\partial_\vf  + \im \,\tm_{\frac32} \b\Omega(\kappa, D) + \tm_1 \pa_x + \im \bA_3^{(d)} |D|^{\frac12} + \im\b\Pi_0 + \bR_{8}^{(0, d)} + \bT_{8,M}\,,
	\end{aligned}
	\end{equation}
	where:
	\begin{enumerate}
	\item  
	the real constant $ \tm_1$ defined in \eqref{defm1} satisfies 
	$	| \tm_1 |^{k_0, \upsilon} \lesssim 
	\varepsilon $;
	\item $\bA_3^{(d)}  $ is a  diagonal matrix of multiplication  
	$$
\bA_3^{(d)}   := \begin{pmatrix}
a_3^{(d)} & 0\\
 0& a_3^{(d)} 
\end{pmatrix} , 
$$
for a real function $ a_3^{(d)}  $ which is  a quasi-periodic traveling wave,
$\even(\vf,x)$, 
satisfying 
\begin{align}\label{a2d.est12}
 \braket{ a_3^{(d)}}_{x}(\vf) = \tm_{\frac12} \in \R \, , \quad \forall\, \vf\in\T^{\nu}\, , 
\end{align}
where $\tm_{\frac12} \in \R $ is the constant  in \eqref{def:m12}, 
and for some $\sigma= \sigma(\tau,\nu,k_0)>0$, for all $s\geq s_0$,
\begin{align}\label{a2d.est}
	& \normk{a_3^{(d)}}{s}\lesssim_{s} \varepsilon \upsilon^{-1} ( 1 + \normk{\fI_0}{s+\sigma} )\,; 
\end{align}
\item $ \bR_{8}^{(0,d)} $ is  a block-diagonal operator 
	\begin{align*}
	\bR_{8}^{(0,d)} & = \begin{pmatrix}
	r_{8}^{(d)}(\vf,x,D) & 0 \\ 
	0 &\bar{r_{8}^{(d)}(\vf,x,D)}
	\end{pmatrix} \in \Ops^{0} \,,
	\end{align*}
that satisfies  for all $\alpha\in\N_0 $, 
for  some  $\sigma_M (\alpha):= \sigma_M(k_0,\tau, \nu, \alpha)>0$ and
for all $s\geq s_0$, 
	\begin{align}
		&\normk{\bR_{8}^{(0,d)} }{0,s,\alpha}  \lesssim_{s, M, \alpha} \varepsilon \upsilon^{-1}( 1+\normk{\fI_0}{s+\sigma_M( \alpha)} ) \, \label{bR8.est} \,;
\end{align}
\item For any $ \tq \in \N^\nu_0 $ with $ |\tq| \leq \tq_0$, 
$ n_1, n_2 \in \N_0 $  with $ n_1 + n_2  \leq M - 2(k_0+\tq_0)  + \frac52  $,  the  
operator $\langle D \rangle^{n_1}\partial_{\vphi}^\tq \bT_{8,M}(\vphi) \langle D \rangle^{n_2}$ is 
$\cD^{k_0} $-tame with a tame constant satisfying, for some 
$ \sigma_M(\tq_0) := \sigma_M(k_0,\tau, \nu, \tq_0) $,
for any $s_0 \leq s \leq S $, 
\begin{equation}\label{red1.estT}
{\mathfrak M}_{\langle D \rangle^{n_1}\partial_{\vphi}^\tq \bT_{8,M}(\vphi) \langle D \rangle^{n_2}}(s) \lesssim_{S, M, \tq_0} 
\varepsilon \upsilon^{-1}( 1+ \normk{\fI_0}{s+\sigma_M(\tq_0)} )\,;
\end{equation}
\item The operators $ \b\Phi^{\pm 1} -{\rm Id}$, $(\b\Phi^{\pm 1}-{\rm Id})^*$ are $\cD^{k_0}$-$\frac12(k_0+1)$-tame and the operators $\cV^{\pm 1}- {\rm Id}$, $(\cV^{\pm 1}-{\rm Id})^*$ are $\cD^{k_0}$-$(k_0+2)$-tame, with tame constants satisfying, for some $\sigma>0$ and for all $s_0\leq s \leq S$,
\begin{align}
	& \fM_{\b\Phi^{\pm 1} -{\rm Id}}(s) + \fM_{(\b\Phi^{\pm 1}-{\rm Id})^*}(s) \lesssim_{S} \varepsilon\upsilon^{-1}( 1 + \normk{\fI_0}{s+ \sigma})\,, \label{red1.est4}\\
	& \fM_{\cV^{\pm 1} -{\rm Id}}(s) + \fM_{(\cV^{\pm 1}-{\rm Id})^*}(s) \lesssim_{S} \varepsilon\upsilon^{-1}( 1 + \normk{\fI_0}{s+ \sigma})\,. \label{red1.est5}
\end{align}
 \end{enumerate}
Furthermore, for any $s_1$ as in \eqref{s1s0}, $\alpha\in\N_0$, $\tq\in\N_0^\nu$, with $\abs\tq \leq \tq_0$, and $n_1,n_2\in\N_0$, with $n_1+n_2\leq M-2 \tq_0 + \frac12$, we have
\begin{align}
& \| \Delta_{12} a_3^{(d)} \|_{s_1} \lesssim_{s_1} \varepsilon \upsilon^{-1}\norm{i_1-i_2}_{s_1+\sigma} \,, \  | \Delta_{12} \tm_1|  
\lesssim \varepsilon \norm{i_1-i_2}_{s_0+\sigma}  \, , 
\label{red1.est1}  \\
&\|\Delta_{12} \bR_{8}^{(0,d)} \|_{0,s_1,\alpha} \lesssim_{s_1, M, \alpha} \varepsilon \upsilon^{-1}\norm{ i_1-i_2 }_{s_1+\sigma_M (\alpha)}   \,, \label{red1.est2} \\
& \| \braket{D}^{n_1} \pa_\vf^\tq \Delta_{12} \bT_{8,M} \braket{D}^{n_2}\|_{\cL(H^{s_1})} \lesssim_{s_1,M, \tq_0} \varepsilon \upsilon^{-1}\norm{i_1-i_2}_{s_1+ \sigma_M(\tq_0) }\,,\label{red1.est3} \\
& \| \Delta_{12} (\cA) h \|_{s_1} \lesssim_{s_1} \varepsilon \upsilon^{-1}\norm{i_1-i_2}_{s_1+\sigma} \norm{h}_{s_1+\sigma} \,, \quad \cA \in \{ \b\Phi^{\pm 1}, (\b\Phi^{\pm 1})^*, \cV^{\pm 1}, (\cV^{\pm 1})^* \}\,. \label{red1.est6} 
\end{align}
\end{lem}
\begin{proof}
	 The function $b(\vf,x)= b_1(\vf,x) + b_2(\vf)$, with $b_1$ and $b_2$, defined in \eqref{eq:q_choice} and \eqref{def:b2} and the function $\varrho(\vf)$  in \eqref{eq:q2_choice}, satisfy, by Lemma \ref{product+diffeo} and  \eqref{step5.est3}, 
	\begin{equation}\label{b1b2rho.est}
		\normk{b_1}{s}\lesssim_{s} \varepsilon(1+\normk{\fI_0}{s+\sigma})\,, \quad  \normk{ b }{s} \, , \ \normk{b_2}{s} \, , \ \normk{\varrho}{s} \lesssim_{s} \varepsilon \upsilon^{-1}( 1 + \normk{\fI_0}{s+\sigma}) 
	\end{equation}
	for some $\sigma>0$ and for all $s\geq s_0$.
	The estimate $ | \tm_1 |^{k_0, \upsilon} \lesssim 
	\varepsilon $  follows by \eqref{defm1} and \eqref{step5.est3}.
	The function
	$$
	a_3^{(d)} (\vf, x) := \cV^{-1}(a_2^{(d)}) = a_2^{(d)}(\vf, x- \varrho (\vf)) ,,
	$$
where $a_2^{(d)}$ is defined in \eqref{a3d}, 
	satisfies \eqref{a2d.est12} by \eqref{media2cos}. Moreover, the  estimate \eqref{a2d.est} 
	follows by Lemma \ref{product+diffeo} and  \eqref{step5.est3}, \eqref{b1b2rho.est}. 
	The estimate \eqref{bR8.est} for (cfr. \eqref{TRApse})
	$$
	r_8^{(d)}(\vf,x,D):= \cV^{-1} r_7^{(d)}(\vf,x,D) \cV = 
	r_7^{(d)}(\vf,x - \varrho (\vf),D)
	$$ 
	with $r_7^{(d)}$ defined in \eqref{r7d}, follows by  Lemmata \ref{pseudo_compo}, \ref{pseudo_commu}, \ref{product+diffeo} and 
	 \eqref{b1b2rho.est}, \eqref{bR6esti1}.
	The smoothing term $ \bT_{8, M} $  in \eqref{cL8} is, using also \eqref{phipi0},
	$$
	\bT_{8, M}:= \cV^{-1} \big( 
	\b\Phi^{-1} \bT_{6,M} \b\Phi 
	+ \im\b\Pi_0(\b\Phi -{\rm Id}) + \b\Phi^{-1}\bR_{6}^{(-M,o)} \b\Phi \big)\cV +
	\cV^{-1} 
	\begin{pmatrix}
         T_M + T_M' & 0 \\ 0 & \overline{T_M} + \overline{T_M'} 
	\end{pmatrix}  \cV 
	$$ 
	with $T_M$ and $T_M'$ defined in \eqref{lin1}, \eqref{lin2}.
		The estimate \eqref{red1.estT} 
	 follows by \eqref{P6},   Lemmata \ref{tame_compo}, \ref{tame_pesudodiff}, 
	 the tame estimates of $ \b\Phi $ in Proposition 2.37 in \cite{BBHM}, 
	  and estimates \eqref{step5.est3}, \eqref{b1b2rho.est}, \eqref{red1.est4}, \eqref{block.est2}, noting that operators of the form $ \pa_\lambda^k \pa_\vf^\tq \cV^{\pm 1}$ lose $\abs k + \abs \tq$ derivatives.
	The estimate \eqref{red1.est4} follows by 
	Lemma 2.38 in \cite{BBHM} 	
	 and \eqref{b1b2rho.est}, whereas  \eqref{red1.est5} follows by the equivalent representation for $\cV$ as in \eqref{B_repr}, Lemma \ref{tame_compo} and  \eqref{b1b2rho.est}.
	The estimates \eqref{red1.est1}, \eqref{red1.est2}, \eqref{red1.est3}, \eqref{red1.est6} are proved in the same fashion.
	By Lemma \ref{LEMMONE}, the function $a_1^{(d)}$ is 
	an $\even(\vf,x)$ quasi-periodic traveling wave, hence the function $b_1$ in \eqref{eq:q_choice} is a $\odd(\vf,x)$ quasi-periodic traveling wave, the function $b_2$ in \eqref{def:b2} is odd in $\vf$ and satisfies $b_2(\vf - \ora{\jmath}\varsigma) = b_2(\vf)$ for all $\varsigma\in\R$, whereas the function $\varrho$ in \eqref{eq:q2_choice} is odd in $\vf$ and satisfies $\varrho(\vf - \ora{\jmath}\varsigma) = \varrho(\vf)$ for all $\varsigma\in\R$. By Lemmata 
	\ref{lem:rev-pse}, 
	\ref{lem:mom_pseudo}, and \ref{lem:mom_prop}, the transformations 
	$\b\Phi$ and $\cV$ are reversibility and momentum preserving. Then the operator $\cL_{8}$ is reversible and momentum preserving. The function 
	$ a_3^{(d) }$ is an $\even(\vf,x)$  quasi-periodic traveling wave.
\end{proof}

\subsection{Reduction of the order 1/2}\label{sec:order12}

The goal of this section is to transform the operator $\cL_8$ in \eqref{cL8} into the operator $\cL_9$ in \eqref{cL10} whose coefficient in front of $|D|^{1/2}$ is a constant. 
We eliminate the $x$-dependence and, in view of the property \eqref{a2d.est12}, 
we obtain that this transformation removes also  the $\vf$-dependence.

We first write the operator $\cL_{8}$ in \eqref{cL8} as
$$
	\cL_8 = \omega\cdot \pa_\vf + \begin{pmatrix}
	P_8 & 0 \\ 0 & \bar{P_8}
	\end{pmatrix} + \im\b\Pi_0 + \bT_{8,M}\,,
$$
where 
\begin{equation}
\label{P8}
P_8 := \im \tm_{\frac32} \Omega(\kappa, D) + \tm_1 \pa_x  + \im a_3^{(d)}|D|^{\frac12} + \Op ( r_8^{(d)}) \, . 
\end{equation}
We conjugate  $\cL_{8}$ through the real operator 
\begin{equation}\label{Psi1}
\b\Psi(\vf) := \begin{pmatrix}
\Psi(\vf) & 0 \\ 0 & \bar{\Psi}(\vf)
\end{pmatrix}\, , 
\end{equation}
where $\Psi(\vf):= \Psi^\tau(\vf)|_{\tau =1}$ is the time-$1$ flow of 
\begin{equation}\label{Psi.problem}
\begin{cases}
\partial_\tau \Psi^\tau(\vf) =  B(\vf)  \Psi^\tau(\vf) \, , \\
\Psi^0(\vf) = {\rm Id} 	\,  , 
\end{cases}   \qquad B(\vf) :=  b_3(\vf,x) \cH \, , 
\end{equation}
the function $b_ 3 (\vf, x)$ is a smooth, real, periodic function to be chosen later 
(see \eqref{b3.sol}) and $\cH$ is the Hilbert transform defined in \eqref{Hilbert-transf}. Note that 
$\Psi\pi_0 = \pi_0 = \Psi^{-1}\pi_0$, so that 
\begin{equation}\label{psipi0}
\b\Psi^{-1}\b\Pi_0 \b\Psi = \b\Pi_0\b\Psi \,.
\end{equation}
By the Lie expansion in \eqref{lie_abstract} we have 
\begin{equation}
\begin{aligned}
\Psi^{-1} P_8 \Psi  & = 
P_8 -  [B, P_8] +   \sum_{n=2}^{M+1} \frac{(-1)^n}{n!} \ad_{B(\vf)}^n(P_8) + L_{M} \,,\label{lin3} \\
L_{M} & :=  \frac{(- 1)^{M+2}}{(M+1)!} \int_0^1 (1 - \tau)^{M+1} \Psi^{-\tau}(\vf) \,
\ad_{B(\vf)}^{M+2}(P_8) \, \Psi^\tau (\vf)  \di \tau \, ,  
\end{aligned}
\end{equation}
and, by \eqref{lie_omega_devf}, 
\begin{equation}
\begin{aligned}
\label{lin4}
\Psi^{-1} \circ \omega\cdot\pa_\vf \circ \Psi & = \omega\cdot\pa_\vf +  (\omega\cdot\pa_\vf B(\vf)) - \sum_{n=2}^{M} \frac{(-1)^n}{n!} \ad_{B(\vf)}^{n-1}(\omega\cdot \pa_\vf B(\vf)) + L_{M}'\,, \\
L_{M}' & :=  \frac{(- 1)^{M}}{M!} \int_0^1 (1 - \tau)^{M} \Psi^{-\tau}(\vf) \,
\ad_{B(\vf)}^{M}(\omega\cdot \pa_\vf B(\vf)) \, \Psi^\tau (\vf) \di \tau \, . 
\end{aligned}
\end{equation}
The number $ M $ will be fixed in \eqref{M_choice}.
The contributions at order $1/2$ come from \eqref{lin3}, in particular from 
$ P_8 -  [B, P_8] $ (recall \eqref{P8}). 
Since $ B = b_3 \cH $  (see \eqref{Psi.problem}), by \eqref{eq:moyal_exp} and \eqref{Omegakx} we have
\begin{align}
  P_8 -  [B, P_8] = 
  & \ \im\, \tm_{\frac32} \Omega(\kappa, D) + \tm_1 \pa_x+
  \im\, \big( a_3^{(d)}  -\tfrac32 \tm_{\frac32} \sqrt{\kappa}  (b_3)_x  \big) 
	\abs D^\frac12 \label{P8BP8}
  \\
 & + \Op(r_8^{(d)} + r_{b_3, - \frac12} ) - [B,   \tm_1 \pa_x+ \im \,a_3^{(d)}|D|^{\frac12} + \Op(r_8^{(d)})] ,,\nonumber
\end{align}
where $ \Op(r_{b_3, - \frac12}) \in\Ops^{- \frac12} $ is small with $b_3$.
Recalling that, by  \eqref{a2d.est12}, the space average
$ \braket{a_3^{(d)}}_x(\vf) =  \tm_{\frac12} $ for all $ \vf \in \T^\nu $, 
we choose the function  $b_3(\vf,x)$ such that 
$
a_3^{(d)} - \tfrac32 \tm_{\frac32} \sqrt\kappa (b_3)_x =  \tm_{\frac12} $,
namely
\begin{equation}\label{b3.sol}
	b_3(\vf,x) :=  \tfrac{2}{3 \tm_{\frac32}\sqrt\kappa} \pa_x^{-1} ( a_3^{(d)}(\vf,x) - \braket{a_3^{(d)}}_x(\vf) ) \, ,   \quad \braket{a_3^{(d)}}_x(\vf) = \tm_{\frac12} \, .
\end{equation}
We deduce by \eqref{lin3}-\eqref{lin4} and \eqref{P8BP8}, \eqref{b3.sol} that 
\begin{equation}\label{defL9}
\begin{aligned}
	L_9 & := \Psi^{-1}(\vf)( \omega\cdot \pa_\vf + P_8) \Psi(\vf) \\
	&  = \omega\cdot \pa_\vf + \im \, \tm_{\frac32} \Omega(\kappa,D) + \tm_1 \pa_x + \im \, \tm_{\frac12} \abs D^{\frac12} + \Op(r_9^{(d)}) + L_M + L_M',,
\end{aligned}	
\end{equation}
where 
\begin{equation}\label{r9d}
	\begin{aligned}
		\Op(r_9^{(d)}) := & \ \Op( r_8^{(d)} + r_{b_3, -\frac12}  ) - [B(\vf),\tm_1 \pa_x+ \im \,a_3^{(d)}|D|^{\frac12} + \Op(r_8^{(d)})]+ (\omega\cdot\pa_\vf B(\vf)) \\
		& + \sum_{n=2}^{M-1} \frac{(-1)^n}{n!} \ad_{B(\vf)}^n(P_8)  - \sum_{n=2}^{M} \frac{(-1)^n}{n!} \ad_{B(\vf)}^{n-1}(\omega\cdot \pa_\vf B(\vf)) \in \Ops^0 \,.
	\end{aligned}
\end{equation}

Define the matrix $\b\Sigma:= \begin{pmatrix}
1 & 0 \\ 0 & -1
\end{pmatrix}$. Summing up, we have obtained the following lemma.

\begin{lem}\label{red12}
Let $ M \in \N $, $ \tq_0 \in \N_0 $. 
Let  $b_3$ be the  function defined in \eqref{b3.sol}.
Then, conjugating the operator $ \cL_8  $ in \eqref{cL8} 
via the invertible, real, reversibility and momentum preserving 
map  $ \b\Psi $ defined in \eqref{Psi1}, \eqref{Psi.problem}, we obtain,
for any $ \omega \in \tD\tC (\upsilon, \tau ) $,  
 the real, reversible and  momentum preserving operator 
\begin{equation}
\begin{aligned}\label{cL10}
\cL_{9}  & :=   \b \Psi^{-1} \cL_8 \b\Psi
 = \omega\cdot\partial_\vf  +  \im \,\tm_{\frac32} \b\Omega(\kappa, D) + \tm_1 \pa_x + \im \tm_{\frac12} \b\Sigma |D|^{\frac12}+ \im\b\Pi_0 + \bR_{9}^{(0, d)} + \bT_{9,M},,
	\end{aligned}
	\end{equation}
	where
	\begin{enumerate}
	\item  
	the constant $ \tm_{\frac12}$ defined in \eqref{def:m12} satisfies $| \tm_{\frac12} |^{k_0, \upsilon} \lesssim 	\varepsilon^2 $;
\item $ \bR_{9}^{(0,d)} $ is  a block-diagonal operator 
	\begin{align*}
	\bR_{9}^{(0,d)} & = \begin{pmatrix}
	r_{9}^{(d)}(\vf,x,D) & 0 \\ 
	0 &\bar{r_{9}^{(d)}(\vf,x,D)}
	\end{pmatrix} \in \Ops^{0} \,,
	\end{align*}
that satisfies, for some $\sigma_M:= \sigma_M(k_0,\tau,\nu)>0$, 
and  for all $ s\geq s_0$, 
	\begin{align}
		&\normk{\bR_{9}^{(0,d)} }{0,s,1}  \lesssim_{s,M} \varepsilon \upsilon^{-1} ( 1+\normk{\fI_0}{s+\sigma_M} ) \,; \label{bR8} 
\end{align}
\item  For any $ \tq \in \N^\nu_0 $ with $ |\tq| \leq \tq_0$, 
$n_1, n_2 \in \N_0 $  with 
$ n_1 + n_2  \leq M - 2 (k_0+\tq_0) + \frac52  $,  the  
operator $\langle D \rangle^{n_1}\partial_{\vphi}^\tq \bT_{9, M}(\vphi) \langle D \rangle^{n_2}$ is 
$\cD^{k_0} $-tame with a tame constant satisfying, 
for some $\sigma_M(\tq_0) := \sigma_M(k_0, \tau, \nu, \tq_0) $, for any $s_0 \leq s \leq S $, 
\begin{equation}\label{red12.estT}
{\mathfrak M}_{\langle D \rangle^{n_1}\partial_{\vphi}^\tq \bT_{9,M}(\vphi) \langle D \rangle^{n_2}}(s) \lesssim_{S, M, \tq_0} 
\varepsilon \upsilon^{-1}( 1+ \normk{\fI_0}{s+\sigma_M(\tq_0)} )\,;
\end{equation}
\item The operators $ \b\Psi^{\pm 1} -{\rm Id}$, $(\b\Psi^{\pm 1}-{\rm Id})^*$ are $\cD^{k_0}$-tame, with tame constants satisfying, for some $\sigma := \sigma (k_0, \tau, \nu ) > 0 $ and for all $s \geq s_0$,
\begin{align}
& \fM_{\b\Psi^{\pm 1} -{\rm Id}}(s) + \fM_{(\b\Psi^{\pm 1}-{\rm Id})^*}(s) \lesssim_{s} \varepsilon\upsilon^{-1}( 1 + \normk{\fI_0}{s+ \sigma})\,. \label{red12.est4}
\end{align}
 \end{enumerate}
Furthermore, for any $s_1$ as in \eqref{s1s0}, $\alpha\in\N_0$, $\tq\in\N_0^\nu$, with $\abs\tq \leq \tq_0$, and $n_1,n_2\in\N_0$, with $n_1+n_2\leq M-2 \tq_0 + \frac12$, we have
\begin{align}
&\|\Delta_{12} \bR_{9}^{(0,d)} \|_{0,s_1,1} \lesssim_{s_1,M} \varepsilon \upsilon^{-1}\norm{ i_1-i_2 }_{s_1+\sigma_M}   \,, 
\  | \Delta_{12} \tm_{\frac12} |  
\lesssim \varepsilon^2 \norm{i_1-i_2}_{s_0+\sigma} \, , 
\label{red12.est2} \\
& \| \braket{D}^{n_1} \pa_\vf^\tq \Delta_{12} \bT_{9,M} \braket{D}^{n_2}\|_{\cL(H^{s_1})} \lesssim_{s_1, M, \tq_0} \varepsilon\upsilon^{-1} \norm{i_1-i_2}_{s_1+ \sigma_M(\tq_0) }\,,\label{red12.est3}\\
& \| \Delta_{12} (\b\Psi^{\pm 1})h \|_{s_1} + \| \Delta_{12} (\b\Psi^{\pm 1})^*h \|_{s_1} \lesssim_{s_1} \varepsilon\upsilon^{-1} \norm{ i_1 - i_2}_{s_1+\sigma} \norm{h}_{s_1+\sigma} \,. \label{red12.est5}
\end{align}
\end{lem}
\begin{proof}
The function $b_3(\vf,x)$ defined in \eqref{b3.sol}, satisfies, 
by 
\eqref{a2d.est} and the estimate of $ \tm_{\frac32}$ given  in
Lemma \ref{LEMMONE}-item-\ref{i2-77},  	for some $\sigma>0$ and for all $s\geq s_0$, 
	\begin{equation}\label{b3.est}
		\normk{ b_3 }{s} \lesssim_{s} \varepsilon \upsilon^{-1}(1+ \normk{\fI_0}{s+\sigma}) \, .
	\end{equation}
	The estimate for $\tm_{\frac12}$ follows by \eqref{def:m12}, \eqref{prod} and  \eqref{step5.est3}, \eqref{b1b2rho.est}.
	The estimate \eqref{bR8} follows by \eqref{r9d}, \eqref{P8}, Lemmata \ref{pseudo_compo}, 
	\ref{pseudo_commu}, 
	 and \eqref{a2d.est}, \eqref{bR8.est}, \eqref{b3.est}.
	By \eqref{cL8}, 
	\eqref{P8}, \eqref{defL9}, and  \eqref{psipi0},
	the smoothing term $  \bT_{9,M}  $ in \eqref{cL10} is
	\begin{equation*}
		\bT_{9,M} := {\bf \Psi}^{-1} \bT_{8,M} {\bf \Psi} + 
		\im \b\Pi_0({\bf \Psi} - {\rm Id}) +
	\begin{pmatrix}
         L_M + L_M' & 0 \\ 0 & \overline{L_M} + \overline{L_M'}  
	\end{pmatrix}   	
	\end{equation*}
 with $L_M$ and $L_M'$ introduced in \eqref{lin3}, \eqref{lin4}.
 The estimate \eqref{red12.estT}   follows by Lemmata \ref{tame_compo}, \ref{tame_pesudodiff},  \ref{Neumann pseudo diff},   \eqref{P8}, \eqref{a2d.est}, \eqref{red1.estT}, \eqref{b3.est}, \eqref{red12.est4}.
	The estimate \eqref{red12.est4} follows by Lemma \ref{tame_pesudodiff} and  \eqref{b3.est}. 
	The estimates \eqref{red12.est2}, \eqref{red12.est3}, \eqref{red12.est5} are proved in the same fashion.
	By Lemma \ref{red1}, the function $a_3^{(d)}$ is a 
	$\even(\vf,x)$ quasi-periodic traveling wave. Hence 
	the function $b_3$ in \eqref{b3.sol} is a $\odd(\vf,x)$ quasi-periodic traveling wave. By Lemmata \ref{lem:rev-pse}, 
	\ref{lem:mom_pseudo}, and \ref{lem:mom_prop},
	the transformation $\b\Psi$ is reversibility and momentum preserving, therefore the operator $\cL_{9}$ is reversible and momentum preserving.
\end{proof}

\begin{rem}\label{fix:alpha}
In Proposition \ref{end_redu} we shall 
estimate $ \| [\pa_x, \bR_9^{(0,d)} ]\|_{0,s,0}^{k_0,\upsilon}$ 
using  \eqref{bR8} and \eqref{eq:comm_tame_AB}. In order to  
control $ \| \bR_9^{(0,d)}  \|_{0,s,1}^{k_0, \upsilon} $  we used the 
 estimates \eqref{step5.est5} for finitely many $ \alpha \in \N_0 $, 
$ \alpha \leq \alpha (M) $, depending on $ M $. 
Furthermore in Proposition 
\ref{end_redu} we shall use 
\eqref{red12.est2}-\eqref{red12.est3} only  for $ s_1 = s_0 $. 
\end{rem}

\subsection{Conclusion: partial reduction of $\cL_\omega$}\label{sec:conc}

By Sections \ref{sec:repa}-\ref{sec:order12}, the linear operator $\cL$  in \eqref{Linea10} is semi-conjugated, for all $ \omega \in {\tt DC}(\upsilon,\tau) $, to the real, reversible and momentum preserving operator $\cL_{9}$ defined in \eqref{cL10}, namely
\begin{equation}\label{cL9cL}
	\cL_{9} = \cW_2^{-1} \cL \cW_1 \,,
\end{equation}
where 
\begin{equation}\label{W1W2}
	\cW_1 := \cP \cZ \cE \cQ \wt\cM \cC \b\Phi_{M} 
	\b\Phi \cV \b\Psi \,, \quad 
	\cW_2:= \cP \rho \cZ \cE \cQ \wt\cM \cC \b\Phi_{M} \b\Phi \cV \b\Psi \, . 
\end{equation}
Moreover $\cL_{9} $ is defined for all   $ \omega \in \R^\nu $. 

Now we deduce a similar conjugation result for the projected operator 
$ \cL_\omega $ in \eqref{Lomegatrue}, i.e. \eqref{cLomega_again}, which acts
in the normal subspace $ \acca_{\S^+,\Sigma}^\angle $. 
We first introduce some notation. 

We denote  by $ \Pi_{\S^+,\Sigma}^\intercal $ and
$  \Pi_{\S^+,\Sigma}^\angle  $ 
the projections 
on the subspaces $\acca_{\S^+,\Sigma}^\intercal$ and $ \acca_{\S^+,\Sigma}^\angle$ defined in Section \ref{sec:decomp}. 
In view of Remark \ref{phase-space-ext}, we denote, 
with a small abuse of notation, 
$ \Pi_{\S_0^+, \Sigma}^\intercal:= \Pi_{\S^+, \Sigma}^\intercal + \pi_0 $, so that
$ \Pi_{\S_0^+, \Sigma}^\intercal + \Pi_{\S^+,\Sigma}^\angle = {\rm Id}$ on the whole
$ L^2 \times L^2 $. 
We remind that $ \S_0 = \S \cup \{0\} $, where  $ \S $ is the set defined in \eqref{def.S}. 
We denote by $ \Pi_{{\S}_0} := \Pi_\S^\intercal + \pi_0 $, where 
$  \Pi_\S^\intercal $ is defined below  \eqref{def:HS0bot} 
together with the  definition of  $ \Pi_{{\mathbb S}_0}^\perp $, so that we have 
$  \Pi_{{\S}_0} + \Pi_{{\S}_0}^\perp  = {\rm Id}  $.

\begin{lem}\label{lemmaWperp}
Let $ M > 0 $. There is $ \sigma_M > 0 $ (depending also on $ k_0, \tau, \nu $) such that, 
assuming \eqref{ansatz_I0_s0} with $ \mu_0 \geq \sigma_M $, the following holds: 
the maps  $\cW_1$, $\cW_2$ defined in \eqref{W1W2} 
	have the form 
	\begin{equation}\label{Wi}
		\cW_i = \wt\cM\cC + \cR_{i}(\varepsilon) \,, 
	\end{equation}
	where, for any 
	$  i = 1, 2 $,  for all $s_0 \leq s \leq S$, 
	\begin{align}
		\normk{\cR_{i}(\varepsilon)h}{s} &  \lesssim_{ S, M} \varepsilon\upsilon^{-1}\big( \normk{h}{s+\sigma_M} + \normk{\fI_0}{s+\sigma_M} \normk{h}{s_0 +\sigma_M} \big) \, . \label{Wi.est1} 
	\end{align}
Moreover, 
for $\varepsilon \upsilon^{-1} \leq \delta (S) $ small enough, the operators
	\begin{equation}\label{W1W2_proj}
		\cW_1^\perp:= \Pi_{\S^+,\Sigma}^\angle \cW_1 \Pi_{{\S}_0}^\perp\,, \quad \cW_2^\perp:= \Pi_{\S^+,\Sigma}^\angle \cW_2 \Pi_{{\S}_0}^\perp \,,
	\end{equation}
	are invertible and, 
	for all $s_0 \leq s \leq S$,  $i=1,2$,
	\begin{align}
		\normk{(\cW_i^\perp)^{\pm 1} h}{s} & \lesssim_{ S, M} \normk{h}{s+\sigma_M} + \normk{\fI_0}{s+\sigma_M} \normk{h}{s_0 +\sigma_M} \,\label{qui.1}  \,, \\  
		\|  \Delta_{12} (\cW_i^\perp)^{\pm 1} h \|_{s_1} &\lesssim_{ s_1, M} \varepsilon\upsilon^{-1}\norm{i_1-i_2}_{s_1 + \sigma_M} \norm{h}_{s_1+\sigma_M} \, . 
	\end{align}
	The operators $ \cW_1^\perp $, $ \cW_2^\perp $
map (anti)-reversible, respectively traveling, waves, into
(anti)-reversible, respectively traveling, waves.
	\end{lem}
\begin{proof}
The formulae  \eqref{Wi} and the estimates \eqref{Wi.est1} follow by \eqref{W1W2}, Lemmata \ref{tame_compo}, \ref{tame_pesudodiff}, and \eqref{action_Hs_tame}, 
 \eqref{timerep1}, \eqref{lem:ga1}, \eqref{step5.est11}, \eqref{step5.est12}, \eqref{block.est6}, \eqref{red1.est4}, \eqref{red1.est5}, \eqref{red12.est4}.
The invertibility of each $ \cW_i^\perp $ and the estimates  \eqref{qui.1}
follow as in \cite{BBHM}
and noting that  $\Pi_{\S^+,\Sigma}^\angle\, \wt\cM\cC\, \Pi_{{\S}_0}^\perp =\Pi_{\S^+,\Sigma}^\angle\, \cM\cC\, \Pi_{{\S}_0}^\perp $ are invertible on their ranges, with inverses 
	$ (\Pi_{\S^+,\Sigma}^\angle \, \cM\cC \,\Pi_{{\S}_0}^\perp)^{-1} = \Pi_{{\S}_0}^\perp (\cM \cC)^{-1} \Pi_{\S^+,\Sigma}^\angle $. 
Since 
 $ \cZ, \cE, \cQ, \wt \cM, \b\Phi_{M}, \b\Phi, \cV, \b\Psi $ are reversibility and momentum preserving and using Remark \ref{P-rev-mom}
and Lemmata \ref{proj_rev} and \ref{lem:proj.momentum}, we deduce that 
$ \cW_1^\perp $, $ \cW_2^\perp $
map (anti)-reversible, respectively traveling, waves, into
(anti)-reversible, respectively traveling, waves.
\end{proof}

\begin{rem}
The time reparametrization $ \cP $ and the multiplication for the function
$ \rho $ (which is independent of the space variable), commute with the projections
$  \Pi_{\S^+,\Sigma}^\angle $ and $ \Pi_{{\S}_0}^\perp  $. 
\end{rem}

The operator $\cL_\omega$ in \eqref{Lomegatrue} (i.e. \eqref{cLomega_again}) is semi-conjugated to
\begin{equation}\label{LomLp}
\cL_\bot :=	(\cW_2^\perp)^{-1} \cL_\omega \cW_1^\perp = 
	\Pi_{{\S}_0}^\perp \,\cL_{9} \,\Pi_{{\S}_0}^\perp +  \cR^f
\end{equation}
where $\cR^f$ is, 
by \eqref{W1W2_proj},
\eqref{cL9cL}, 
\eqref{Wi} (recall
that $ \wt\cM $ is defined in \eqref{map_test_M}-\eqref{Lambda}),  and 
\eqref{proiez}, 
	\begin{align}\label{Rintercal}
		\cR^f  & :=  (\cW_2^\perp)^{-1}\Pi_{\S^+, \Sigma}^\angle \cR_{2} (\varepsilon) 
		\Pi_{\S_0} \cL_{9}  \Pi_{\S_0}^\bot   \\
		& 
		 - (\cW_2^\perp)^{-1}\Pi_{\S^+, \Sigma}^\angle 
		\cL \Pi_{\S_0^+, \Sigma}^\intercal 
		\cR_{1} (\varepsilon) \Pi_{\S_0}^\bot 
		 - \varepsilon (\cW_2^\perp)^{-1}\Pi_{\S^+, \Sigma}^\angle J R \cW_1^\perp \, .  \nonumber 
	\end{align}

\begin{lem}
	The operator $\cR^f$ in \eqref{Rintercal} has the finite rank form \eqref{finite_rank_R}, \eqref{gjchij_est}. Moreover, 
	let $\tq_0\in\N_0$ and $M \geq 2(k_0+\tq_0) - \frac32$. There exists $\aleph(M,\tq_0)>0$ (depending also on $k_0$, $\tau$, $\nu$) such that, for any $n_1, n_2\in\N_0$, with $n_1+n_2 \leq M - 2(k_0+\tq_0)+\frac52$, and any $\tq\in\N_0^\nu$, with $\abs \tq \leq \tq_0$, the operator $\braket{D}^{n_1} \pa_\vf^\tq \cR^f \braket{D}^{n_2} $ 
	is $\cD^{k_0}$-tame, with a tame constant satisfying
	\begin{align}
		& \fM_{\braket{D}^{n_1} \pa_\vf^\tq \cR^f \braket{D}^{n_2}}(s) \lesssim_{ S, M, \tq_0} \varepsilon\upsilon^{-1}(1+\normk{\fI_0}{s+\aleph(M,\tq_0)}) \, , 
		\quad \forall s_0 \leq s \leq S \, ,  \label{quo.est1}  \\
		&  \| \braket{D}^{n_1} \pa_\vf^\tq\Delta_{12} \cR^f \braket{D}^{n_2} \|_{\cL(H^{s_1})} \lesssim_{s_1, M, \tq_0} \varepsilon\upsilon^{-1} \norm{i_1-i_2}_{s_1+\aleph(M,\tq_0)} \, , 
		\label{quo.est2}
	\end{align}
 for any $s_1$ as in \eqref{s1s0}. 
\end{lem}
\begin{proof}
The first two terms in \eqref{Rintercal} have  the finite rank form 
\eqref{finite_rank_R} because of the presence of the finite dimensional projector 
$ \Pi_{\S_0} $, respectively $ \Pi_{\S_0^+,\Sigma}^\intercal $. 
In the last term, the operator $ R $ has  the finite rank form 
\eqref{finite_rank_R}. 	
	The estimate \eqref{quo.est1} follows by  \eqref{Rintercal}, \eqref{W1W2}, \eqref{W1W2_proj}, \eqref{cL10}, \eqref{finite_rank_R}, \eqref{prod} and \eqref{Wi.est1}, \eqref{qui.1},  \eqref{bR8}, \eqref{red12.estT}, \eqref{gjchij_est}.
	The estimate \eqref{quo.est2} follows similarly.
\end{proof}

\begin{prop}\label{end_redu}
{\bf (Reduction of $ {\cal L}_\omega $ up to smoothing operators)}
	For all $(\omega,\kappa)\in \tD\tC(\upsilon,\tau)\times [\kappa_1,\kappa_2]$, the operator $\cL_\omega$ in \eqref{Lomegatrue} (i.e. \eqref{cLomega_again}) is semi-conjugated via \eqref{LomLp} to the real, reversible and momentum preserving operator $\cL_\perp $. For all $ (\omega, \kappa) \in \R^\nu \times [\kappa_1, \kappa_2] $, the extended operator defined by the right hand side in \eqref{LomLp} has the form  
	\begin{equation}\label{Lperp}
		\cL_\perp = \omega\cdot\pa_\vf \uno_\perp + \im\,\bD_\perp + \bR_\perp ,,
	\end{equation}
 where $\uno_\perp$ denotes the identity map of 
 $  \bH_{{\mathbb S}_0}^\bot $ (cfr.  \eqref{def:HS0bot}) and
	\begin{enumerate}
		\item $\bD_\perp$ is the diagonal operator
		\begin{equation*}\label{Dperp}
			\bD_\perp := \begin{pmatrix}
			\cD_\perp & 0 \\ 0 & -\bar{\cD_\perp}
			\end{pmatrix} \,, \quad \cD_\perp:= \diag_{j\in \S_0^c} \mu_j \,, \quad 
			\S_0^c:= \Z\setminus (\S\cup\{0\})  \, , 
		\end{equation*}
		with eigenvalues 
$ \mu_j :=\tm_{\frac32}\Omega_j(\kappa) + \tm_1 j + \tm_{\frac12}\abs j^\frac12 \in \R\,, $
		where the real constants $\tm_{\frac32}, \tm_1, \tm_{\frac12} $, defined respectively in \eqref{defm32}, \eqref{defm1}, \eqref{def:m12}, satisfy 
		\begin{equation}\label{const_small}
		\begin{aligned}
& 			| \tm_{\frac32} - 1 |^{k_0,\upsilon} + | \tm_1|^{k_0,\upsilon} + |\tm_{\frac12} |^{k_0,\upsilon} \lesssim \varepsilon \,; 
		\end{aligned}
		\end{equation}
		In addition, for some $ \sigma > 0 $, 
	\begin{equation}\label{const_smallV} 
	| \Delta_{12} \tm_{\frac32}  | + | \Delta_{12} \tm_1| + 
|\Delta_{12} \tm_{\frac12} | \lesssim \varepsilon \norm{i_1-i_2}_{s_0+\sigma}  \, . 
\end{equation} 
		\item The operator $\bR_\perp $  is real, reversible and momentum preserving. Moreover, for any $\tq_0\in \N_0$, $M  > 2(k_0+\tq_0) - \frac32$, there is a constant $\aleph(M,\tq_0)>0$ (depending also on $k_0$, $\tau$, $\nu$) such that, assuming \eqref{ansatz_I0_s0} 
		with $\mu_0 \geq \aleph(M,\tq_0)$, for any $s_0\leq s \leq S$,   $\tq\in\N_0^\nu$, with $\abs\tq \leq \tq_0$, the operators 
		$ \pa_\vf^\tq\bR_\perp  $, $ [\pa_\vf^\tq\bR_\perp, \pa_x] $  
		 are $\cD^{k_0}$-tame with  tame constants satisfying
		\begin{align}
			&\fM_{\pa_\vf^\tq\bR_\perp } (s), \  
			\fM_{[\pa_\vf^\tq\bR_\perp, \pa_x]} (s) \lesssim_{ S, M, \tq_0} \varepsilon\upsilon^{-1} (1+ \normk{\fI_0}{s+\aleph(M,\tq_0)})\,. \label{fine.rid.1}
		\end{align}
		Moreover,
		 for any 
		$\tq\in\N_0^\nu$, with $\abs\tq\leq \tq_0$,
	\begin{equation}
		\| \pa_\vf^\tq\Delta_{12}\bR_\perp   \|_{\cL(H^{s_0})} + 
		\| \pa_\vf^\tq\Delta_{12}[\bR_\perp, \pa_x]   \|_{\cL(H^{s_0})} \lesssim_{M} \varepsilon \upsilon^{-1} \norm{i_1-i_2}_{s_0+\aleph(M,\tq_0)}\label{fine.rid.2} \,.
	\end{equation}
	\end{enumerate}
	\end{prop}

\begin{proof}
By \eqref{LomLp} and \eqref{cL10} we deduce \eqref{Lperp}
with 
$$
	\bR_\perp := \Pi_{{\S}_0}^\perp(\bR_{9}^{(0,d)}+ \bT_{9,M})\Pi_{{\S}_0}^\perp + \cR^f\, . 
$$
	The estimates \eqref{const_small}-\eqref{const_smallV} follow by Lemmata \ref{LEMMONE}, \ref{red1}, \ref{red12}. 
	The estimate \eqref{fine.rid.1} follows by Lemmata \ref{pseudo_commu},  \ref{tame_pesudodiff}, 
	 \eqref{bR8} and  \eqref{red12.estT}, \eqref{quo.est1}, choosing $(n_1,n_2)=(1,0),(0,1)$.
	The estimate \eqref{fine.rid.2} follows similarly.
	The operator 
$ \cL_\omega $ in \eqref{Lomegatrue} is reversible and momentum preserving
(Lemma \ref{lem:K02}).
By Sections \ref{subsec:good}-\ref{sec:order12}, the maps 
 $ \cZ, \cE, \cQ, \wt \cM, \b\Phi_{M}, \b\Phi, \cV, \b\Psi $ are reversibility and momentum preserving. Therefore, using also 
\eqref{propp},  \eqref{rhopem}  
and Lemmata \ref{proj_rev} and \ref{lem:proj.momentum}, we deduce that 
the operator $ \cL_\bot $ in \eqref{LomLp} is reversible and momentum preserving. 
Since $ \im\,\bD_\perp $ is reversible and momentum preserving, we deduce that
  $\bR_\perp $   is reversible and momentum preserving. 
\end{proof}

\section{Almost-diagonalization and invertibility of $\cL_\omega$}\label{sec:KAM}

In Proposition \ref{end_redu} we obtained the operator $ \cL_\perp $
 in \eqref{Lperp} which is diagonal and constant coefficient up to the 
 bounded operator $\bR_\perp (\vphi) $. In this section we 
 complete the diagonalization of $ \cL_\perp $  implementing a
 KAM iterative scheme.
As starting point, we consider the real, reversible and momentum preserving operator,
acting in $  \bH_{{\mathbb S}_0}^\bot $, 
\begin{equation}\label{bL0}
	\bL_0 := \bL_0(i) := \cL_\perp = 
	\omega\cdot\pa_\vf \uno_\perp + \im \,\bD_0 + \bR_\perp^{(0)} \,,
\end{equation}
defined for all $(\omega,\kappa)\in \R^\nu \times[\kappa_1,\kappa_2]$, with diagonal part (with respect to the exponential basis)
\begin{equation}\label{D0}
	\begin{aligned}
		\bD_0&:= \begin{pmatrix}
		\cD_0 & 0 \\ 0 & -\bar{\cD_0}
		\end{pmatrix} \,,\quad
		\cD_0 := \diag_{j \in \S_0^c} \mu_{j}^{(0)} \,, \quad \mu_{j}^{(0)}:= \tm_{\frac32}\Omega_j(\kappa)+\tm_1 j+\tm_{\frac12} \abs j^\frac12 \,,
	\end{aligned}
\end{equation}
where $  \S_0^c = \Z \setminus \S_0 $, $ \S_0 = \S \cup \{0\}  $,  
 the real constants $\tm_{\frac32} $, $ \tm_1 $,  $ \tm_{\frac12} $ 
satisfy  \eqref{const_small}-\eqref{const_smallV} and 
\begin{equation}\label{Rperp0}
	\bR_\perp^{(0)}:= \bR_\perp := \begin{pmatrix}
	R_\perp^{(0,d)} & R_\perp^{(0,o)} \\ \bar{R_\perp^{(0,o)}} & \bar{R_\perp^{(0,d)}}
	\end{pmatrix}\,, \ \quad R_\perp^{(0,d)}: 
	H_{{\mathbb S}_0}^\perp  
	\rightarrow H_{{\mathbb S}_0}^\perp  \,,
	R_\perp^{(0,o)} : H_{-{\mathbb S}_0}^\perp  
	\rightarrow H_{{\mathbb S}_0}^\perp \, , 
\end{equation}
which is a real, reversible, momentum preserving operator satisfying  \eqref{fine.rid.1}, \eqref{fine.rid.2}. We denote 
$ H_{\pm \S_0}^\bot = \{ h(x) = \sum_{j \not \in \pm \S_0} h_j e^{ \pm \im j x} \in L^2 \} $.  
Note that 
\begin{equation}\label{Dbar}
\bar{\cD_0} :  H_{- {\mathbb S}_0}^\bot \to H_{- {\mathbb S}_0}^\bot  \, , 
\quad   \bar{\cD_0} = {\rm diag}_{j \in - \S_0^c}( \mu_{-j}^{(0)} ) \, . 
\end{equation}
Proposition \ref{end_redu} implies that the operator $ \bR_\perp^{(0)} $ 
satisfies the tame estimates of Lemma \ref{Rperp_small} below 
by fixing the constant $ M $ large enough
(which means  performing sufficiently many regularizing steps in Section \ref{sec:block_dec}), namely
\begin{equation}\label{M_choice}
	M:= \big[ 2(k_0 + s_0 + \tb) - \tfrac32 \big] +1 \in \N\,,
\end{equation}
where
\begin{equation}\label{tbta}
\tb:= [\ta] + 2 \in \N \,, \quad \ta:=3\tau_1  \geq 1 \,, \quad \tau_1:= k_0 +(k_0+1)\tau\,.
\end{equation}
	These conditions
	imply  the convergence of the iterative scheme 
	\eqref{remas.1}-\eqref{remas.2}, see Lemma \ref{QSKAM}.  
We also set 
\begin{equation}\label{cb.mub}
\mu(\tb):= \aleph(M,s_0+\tb) \, , 
\end{equation}
where the constant  $\aleph(M,\tq_0)$ is given  in Proposition \ref{end_redu}.  

\begin{lem}\label{Rperp_small}   
{\bf (Smallness of $ \bR_\perp^{(0)} $)}
	Assume \eqref{ansatz_I0_s0} with $\mu_0 \geq \mu(\tb)$. Then the operators
	$ \bR_\perp^{(0)} $, $ [\bR_\perp^{(0)} , \pa_x] $, and 
	$ \pa_{\vf_m}^{s_0} \bR_\perp^{(0)} $, $  [\pa_{\vf_m}^{s_0} \bR_\perp^{(0)}, \pa_x] $, 
	$ \pa_{\vf_m}^{s_0+\tb} \bR^{(0)}_\perp $, $   [\pa_{\vf_m}^{s_0+\tb} \bR^{(0)}_\perp, \pa_x] $, 
	$ m = 1, \ldots, \nu $, 
	are $\cD^{k_0}$-tame and, defining
	\begin{align}
		& \mathbb{M}_0(s):= \max \big\{ \fM_{\bR^{(0)}_\perp}(s), \ \fM_{[\bR^{(0)}_\perp,\pa_x]}(s), \  \fM_{\pa_{\vf_m}^{s_0}\bR^{(0)}_\perp}(s), \  \fM_{[\pa_{\vf_m}^{s_0}\bR^{(0)}_\perp,\pa_x]}(s) 
		\, ,\, m = 1, \ldots, \nu \big\}
		\,,\label{M0s} \\
		& \mathbb{M}_0(s,\tb):= \max \big\{ \fM_{\pa_{\vf_m}^{s_0+\tb}\bR^{(0)}_\perp}(s), \ \fM_{[\pa_{\vf_m}^{s_0+\tb}\bR^{(0)}_\perp,\pa_x]}(s) \, , \, m = 1, \ldots, \nu \big\}\,,\label{M0sb}
	\end{align}
	we have, for all $ s_0 \leq s \leq S$,
	\begin{equation}\label{M0ss.est}
	\fM_0(s,\tb):= \max\Set{\mathbb{M}_0(s), \mathbb{M}_0(s,\tb)  } \leq C(S) \frac{\varepsilon}{\upsilon}(1 + \normk{\fI_0}{s+ \mu (\tb)})  \,, \  \fM_0(s_0,\tb) \leq C(S) \frac{\varepsilon}{\upsilon} \, . 
	\end{equation}
Moreover, 
for all $\tq \in\N_0^\nu$,  with $\abs \tq \leq s_0+\tb$, 
	\begin{equation}\label{Moss.delta}
		\| \pa_\vf^\tq \Delta_{12} \bR^{(0)}_\perp  \|_{\cL(H^{s_0})}\,, \ \| \Delta_{12}[\pa_\vf^\tq \bR^{(0)}_\perp,\pa_x]  \|_{\cL(H^{s_0})} \leq C(S) \varepsilon\upsilon^{-1}\norm{i_1-i_2}_{s_0+\mu(\tb)} \,.
	\end{equation}
\end{lem}

\begin{proof}
Recalling \eqref{M0s}, \eqref{M0sb}, 	the bounds \eqref{M0ss.est}-\eqref{Moss.delta} follow
	by \eqref{fine.rid.1}, 
		\eqref{M_choice}, \eqref{cb.mub},  \eqref{fine.rid.2}.
\end{proof}

We perform the almost-reducibility of $\bL_0$ along the scale
\begin{equation}\label{size_N}
	N_{-1}:= 1\,, \quad N_\tn := N_0^{\chi^\tn}\,, \quad \forall\,\tn\in\N_0\,, \quad \chi:=3/2\, .
\end{equation}

\begin{thm}\label{iterative_KAM}
{\bf (Almost-diagonalization of $\bL_0$: KAM iteration)}
	There exists $\tau_2(\tau,\nu)> \tau_1(\tau,\nu) + \ta $ 
(with $ \tau_1 , \ta $  defined in \eqref{tbta})	such that, for all $S >s_0$, there is $N_0:=N_0(S,\tb)\in\N$ such that, if
	\begin{equation}\label{small_KAM_con}
		N_0^{\tau_2} \fM_0(s_0,\tb)\upsilon^{-1} \leq 1 \,,
	\end{equation}
	then, for all $\bar\tn\in\N_0$, $\tn=0,1,\ldots,\bar\tn$:
	\\[1mm]
	$({\bf S1})_\tn$ There exists a real, reversible and momentum preserving operator
		\begin{equation}\label{bLn}
			\begin{aligned}
				&\bL_\tn := \omega\cdot\pa_\vf\uno_\perp +\im\,\bD_\tn + \bR_\perp^{(\tn)}\,, \\
				&\bD_\tn:= \begin{pmatrix}
				\cD_\tn & 0 \\ 0 & -\bar{\cD_\tn}
				\end{pmatrix}\,, \quad \cD_\tn:= \diag_{j\in\S_0^c} \mu_{j}^{(\tn)} \,,
			\end{aligned}
		\end{equation}
		defined for all $(\omega,\kappa)$ in $\R^\nu\times [\kappa_1,\kappa_2]$,
		where $\mu_{j}^{(\tn)}$ are $k_0$-times differentiable real functions 
		\begin{equation}\label{eigen_KAM}
			\mu_{j}^{(\tn)}(\omega,\kappa):= \mu_{j}^{(0)}(\omega,\kappa) + \fr_j^{(\tn)}(\omega,\kappa)\,, \quad \mu_{j}^{(0)} = \tm_{\frac32}\,\Omega_j(\kappa)+\tm_1\,j+\tm_{\frac12}\abs j^\frac12\,,
		\end{equation}
		satisfying $ \fr_j^{(0)} = 0 $ and, for $ \tn \geq 1 $,  	
		\begin{align}
		\label{rem.eigen.KAM}
&		| \fr_j^{(\tn)} |^{k_0,\upsilon}  \leq C(S, \tb) \varepsilon\upsilon^{-1} \, , \quad 	
			| \mu_{j}^{(\tn)} - \mu_{j}^{(\tn-1)} |^{k_0,\upsilon} 
			 \leq C(S, \tb) \varepsilon\upsilon^{-1} N_{\tn-2}^{-\ta} \, , 
			 \ \   \forall j\in\S_0^c \, \,. 
		\end{align}
		The remainder
		\begin{equation}\label{bRperpn}
			\bR_\perp^{(\tn)}:= \begin{pmatrix}
			R_\perp^{(\tn,d)} & R_\perp^{(\tn,o)}\\
			\bar{R_\perp^{(\tn,o)} } & \bar{R_\perp^{(\tn,d)} } 
			\end{pmatrix} 
			, \ \quad R_\perp^{(\tn,d)}: 
	H_{{\mathbb S}_0}^\perp  
	\rightarrow H_{{\mathbb S}_0}^\perp  \,, \
	R_\perp^{(\tn,o)} : H_{-{\mathbb S}_0}^\perp  
	\rightarrow H_{{\mathbb S}_0}^\perp
		\end{equation}
		is $\cD^{k_0}$-modulo-tame: more precisely, the operators $R_\perp^{(\tn,d)} $, $R_\perp^{(\tn,o)} $, $\braket{\pa_\vf}^\tb R_\perp^{(\tn,d)} $, $\braket{\pa_\vf}^\tb R_\perp^{(\tn,o)} $, are $\cD^{k_0}$-modulo-tame with modulo-tame constants 
		\begin{equation}\label{Mn.sharp}
			\begin{aligned}
			&\fM_\tn^\sharp (s):= \fM_{\bR_\perp^{(\tn)}}^\sharp(s):= \max \{ \fM_{R_\perp^{(\tn,d)} }^\sharp(s), \fM_{R_\perp^{(\tn,o)} }^\sharp(s) \} \,, \\
			& \fM_\tn^\sharp(s,\tb):= \fM_{\braket{\pa_\vf}^\tb \bR_\perp^{(\tn)} }^\sharp(s):= \max \{  \fM_{\braket{\pa_\vf}^\tb R_\perp^{(\tn,d)} }^\sharp(s), \fM_{\braket{\pa_\vf}^\tb R_\perp^{(\tn,o)} }^\sharp(s) \} \, , 
			\end{aligned}
		\end{equation} 
		which satisfy, for some constant $C_*(s_0,\tb) > 0 $, for all $s_0\leq s \leq S$,
		\begin{equation}\label{small_scheme}
			\fM_\tn^\sharp (s) \leq C_*(s_0,\tb) \fM_0(s,\tb)
			N_{\tn -1}^{-\ta}   \,, \quad \fM_\tn^\sharp(s,\tb) \leq C_*(s_0,\tb) 
			\fM_0 (s,\tb) N_{\tn-1}   \,.
		\end{equation}
		Define the sets $\t\Lambda_\tn^\upsilon=\t\Lambda_\tn^\upsilon(i)$ by $\t\Lambda_0^\upsilon:=\tD\tC(2\upsilon,\tau)\times [\kappa_1,\kappa_2]$ and, for $ \tn \geq 1 $, 
		\begin{equation}\label{tLambdan}
		\begin{aligned}
		\t\Lambda_\tn^\upsilon:= & \big\{ \lambda=(\omega,\kappa) \in \t\Lambda_{\tn-1}^\upsilon \, :  \,  \\
		& \ \ \big| \omega\cdot\ell + \mu_{j}^{(\tn -1)} - \mu_{j'}^{(\tn-1)} \big| \geq \upsilon\, \langle  |j|^\frac32 - | j'|^\frac32 \rangle \braket{\ell}^{-\tau}  \\
		& \ \ \forall\,\abs\ell \leq N_{\tn-1}\,,\ j,j'\notin\S_0 \,, \  (\ell,j,j')\neq(0,j,j),\text{ with }\ora{\jmath} \cdot\ell +j-j'=0 \, ,  \\
		&\  \ \big| \omega\cdot\ell + \mu_{j}^{(\tn -1)} + \mu_{j'}^{(\tn-1)} \big| \geq \upsilon \,\big(\abs j^\frac32 + |j'|^\frac32\big) \braket{\ell}^{-\tau}  \\
		& \ \  \forall\,\abs\ell \leq N_{\tn-1}\,, \ j,  j' \notin \S_0
		\text{ with }\ora{\jmath} \cdot\ell +j+j'=0  	
		\big\}\,.
		\end{aligned}
		\end{equation}
		For $\tn\geq 1$ there exists a real, reversibility  and momentum preserving map, defined for all $(\omega,\kappa)\in \R^\nu\times [\kappa_1,\kappa_2]$, of the form
		$$
			\b\Phi_{\tn-1} = e^{\bX_{\tn-1}}\,, \quad \bX_{\tn-1}:=\begin{pmatrix}
			X_{\tn-1}^{(d)} & X_{\tn-1}^{(o)} \\ \bar{X_{\tn-1}^{(o)}} & \bar{X_{\tn-1}^{(d)}}
			\end{pmatrix}\,, \
			 \  X_{\tn-1}^{(d)}: 
	H_{{\mathbb S}_0}^\perp  
	\rightarrow H_{{\mathbb S}_0}^\perp  \,,
	X_{\tn-1}^{(o)} : H_{-{\mathbb S}_0}^\perp  
	\rightarrow H_{{\mathbb S}_0}^\perp \, ,
		$$
		such that, for all $\lambda \in \t\Lambda_{\tn}^\upsilon$, the following conjugation formula holds:
		\begin{equation}\label{ITE-CON}
			\bL_\tn = \b\Phi_{\tn-1}^{-1} \bL_{\tn-1}\b\Phi_{\tn-1}\,.
		\end{equation}
		The operators $\bX_{\tn-1}$, $\braket{\pa_\vf}^\tb \bX_{\tn-1}$, 
		are $\cD^{k_0}$-modulo-tame 
		with modulo tame constants satisfying, for all $s_0\leq s \leq S$,
		\begin{equation}\label{gen.modulo.KAM}
			\begin{aligned}
				\fM_{\bX_{\tn-1}}^\sharp(s) & \leq
				C(s_0,\tb) \upsilon^{-1} 
				N_{\tn-1}^{\tau_1} N_{\tn-2}^{-\ta} \fM_0(s,\tb)\,,\\
				\fM_{\braket{\pa_\vf}^\tb \bX_{\tn-1}}^\sharp (s) & \leq  
				C(s_0,\tb) \upsilon^{-1}  N_{\tn-1}^{\tau_1} N_{\tn-2} \fM_0(s,\tb)\, . 
			\end{aligned}
		\end{equation}
$({\bf S2})_\tn$ Let $i_1(\omega,\kappa)$, $i_2(\omega,\kappa)$ such that $\bR_\perp^{(\tn)}(i_1)$, $\bR_\perp^{(\tn)}(i_2)$ satisfy \eqref{M0ss.est}, \eqref{Moss.delta}. Then, for all $(\omega,\kappa)\in \t\Lambda_\tn^{\upsilon_1}(i_1)\cap \t\Lambda_\tn^{\upsilon_2}(i_2)$ with $\upsilon_1,\upsilon_2 \in [\upsilon/2,2\upsilon]$, 
		\begin{align}
			&\|| \Delta_{12} \bR_\perp^{(\tn)}  |\|_{\cL(H^{s_0})}\lesssim_{S,\tb} \varepsilon\upsilon^{-1} N_{\tn-1}^{-\ta}\norm{i_1-i_2}_{s_0+\mu(\tb)} \,, \label{S3_1}\\
			& \||\braket{\pa_\vf}^\tb \Delta_{12} \bR_\perp^{(\tn)}  |\|_{\cL(H^{s_0})}\lesssim_{S,\tb} \varepsilon\upsilon^{-1} N_{\tn-1}\norm{i_1-i_2}_{s_0+\mu(\tb)} \,.\label{S3_2}
		\end{align}
		Furthermore, for $ \tn \geq 1 $, for all $j\in\S_0^c$,
		\begin{align}
			& | \Delta_{12} (\fr_j^{(\tn)} - \fr_j^{(\tn-1)}) | \leq C \| | \Delta_{12}\bR_\perp^{(\tn)} |  \|_{\cL(H^{s_0})} \,, \label{S3_3}\\
			& | \Delta_{12} \fr_j^{(\tn)} | \leq C(S, \tb) \varepsilon\upsilon^{-1} \norm{i_1-i_2}_{s_0+\mu(\tb)} \,.\label{S3_4}
		\end{align}
$({\bf S3})_\tn$ Let $i_1, i_2$ be like in $({\bf S2})_\tn$ and $0 <  \rho < \upsilon/2$. Then
\begin{equation}\label{INCLPRO}
			\varepsilon\upsilon^{-1} C(S) N_{\tn-1}^{\tau +1}\norm{i_1-i_2}_{s_0+\mu(\tb)} \leq \rho \quad \Rightarrow \quad \t\Lambda_\tn^\upsilon(i_1) \subseteq \t\Lambda_\tn^{\upsilon-\rho}(i_2) \,. 
\end{equation}
\end{thm}

Theorem \ref{iterative_KAM} implies also that the invertible operator
\begin{equation}\label{bUn}
	\bU_{\bar\tn} := \b\Phi_0 \circ \ldots \circ \b\Phi_{\bar\tn-1} \, , \quad  \bar \tn \geq 1 \, , 
\end{equation}
has almost diagonalized $\bL_0$. We have indeed the following corollary.
\begin{thm}\label{KAMRED}
{\bf (Almost-diagonalization of $ \bL_0 $)}
	Assume \eqref{ansatz_I0_s0} with $\mu_0 \geq \mu(\tb)$. 
	For all $S>s_0$, there exist $N_0=N_0(S,\tb)>0$ and $\delta_0=\delta_0(S)>0$ such that, if the smallness condition
	\begin{equation}\label{KAM_small_cond}
		N_0^{\tau_2}\varepsilon\upsilon^{-2} \leq \delta_0
	\end{equation}
	holds, where $\tau_2=\tau_2(\tau,\nu)$ is defined in Theorem \ref{iterative_KAM}, then, for all $\bar\tn\in\N$ and for all $(\omega,\kappa)\in \R^\nu\times [\kappa_1,\kappa_2]$ the operator $\bU_{\bar\tn}$ in \eqref{bUn} is well-defined, the operators $\bU_{\bar\tn}^{\pm 1}- \uno_\perp$ are $\cD^{k_0}$-modulo-tame with modulo-tame constants satisfying, for all $s_0\leq s \leq S$,
	\begin{equation}\label{flow.sharp.kam}
	\fM_{\bU_{\bar\tn}^{\pm 1}-\uno_\perp}^\sharp (s) \lesssim_{S} \varepsilon \upsilon^{-2} N_0^{\tau_1}(1+\normk{\fI_0}{s+\mu(\tb)})\,,
	\end{equation}
	where $\tau_1$ is given by \eqref{tbta}. 
	Moreover $\bU_{\bar\tn}$, $\bU_{\bar\tn}^{-1}$ are real, reversibility and momentum preserving. 
	The operator 
$ \bL_{\bar\tn} = \omega\cdot\pa_\vf\uno_\perp + \im\,\bD_{\bar\tn} + \bR_\perp^{(\bar\tn)}$, 
defined in	\eqref{bLn}  
	with $\tn=\bar\tn$ is real, reversible and momentum preserving. 
	The operator $ \bR_\perp^{(\bar\tn)}  $ 
	is $\cD^{k_0}$-modulo-tame with a modulo-tame constant 
	satisfying, for all $ s_0 \leq s \leq S $,
	$$
		\fM_{\bR_\perp^{(\bar\tn)}}^\sharp (s) 
		\lesssim_{S} \varepsilon\upsilon^{-1} N_{\bar\tn-1}^{-\ta} (1+\normk{\fI_0}{s+\mu(\tb)}) \,.
	$$
	Moreover, for all $(\omega,\kappa)$ in
	$	\t\Lambda_{\bar\tn}^\upsilon = \t\Lambda_{\bar\tn}^\upsilon(i) = \bigcap_{\tn=0}^{\bar\tn } \t\Lambda_\tn^\upsilon $, 
	where the sets $\t\Lambda_\tn^\upsilon$ are defined in \eqref{tLambdan}, 
	the  conjugation formula 
	$	\bL_{\bar\tn} := \bU_{\bar\tn}^{-1} \bL_0 \bU_{\bar\tn} $ holds.
	
\end{thm}

\subsection*{Proof of Theorem \ref{iterative_KAM}}

The proof of Theorem \ref{iterative_KAM} is  inductive. We first show that $({\bf S1})_\tn$-$({\bf S3})_\tn$ hold when $\tn=0$. 

\paragraph{The step $\tn=0$.}
{\sc Proof of $({\bf S1})_0$}.
Properties \eqref{bLn}-\eqref{eigen_KAM}, \eqref{bRperpn}  for $\tn=0$ hold by \eqref{bL0}, \eqref{D0}, \eqref{Rperp0} with $\fr_j^{(0)} =0$. We now prove that also \eqref{small_scheme} for $\tn=0$ holds.
\begin{lem}\label{lem:tame}
	We have $\fM_0^\sharp(s), \fM_0^\sharp (s,\tb) \lesssim_{s_0,\tb} \fM_0(s,\tb)$.
\end{lem}
\begin{proof}
	Let $R \in \{ R_\perp^{(0,d)}, R_\perp^{(0,o)} \} $. 
	We prove that $\braket{\pa_\vf}^\tb R$ is $\cD^{k_0}$-modulo-tame. Using the inequality
	$$
		\braket{\ell-\ell'}^{2\tq_0} \braket{j-j'}^2 \lesssim_{\tq_0} 1 + \abs{\ell-\ell'}^{2\tq_0}+ \abs{j-j'}^2 + \abs{\ell-\ell'}^{2\tq_0} \abs{j-j'}^2\,,
	$$
	it follows, recalling \eqref{salva}, \eqref{M0ss.est}, 
	(the matrix elements of the 
commutator $[\pa_x,A]$  are $\im(j-j')A_j^{j'}(\ell-\ell')$), 
	that, for any $j'\in\S_0^c$, $\ell'\in\Z^\nu$,
	\begin{equation}\label{salva2}
		\begin{aligned}
			\upsilon^{2\abs k } \sum_{\ell,j} \braket{\ell,j}^{2s} \braket{\ell-\ell'}^{2(s_0+\tb)} \braket{j-j'}^2 \big| \pa_\lambda^k R_j^{j'}(\ell-\ell') \big|^2 \\
			\lesssim_{\tb} \fM_0(s_0,\tb)^2\braket{\ell',j'}^{2s} + \fM_0(s,\tb)^2 \braket{\ell',j'}^{2s_0} \,. 
		\end{aligned}
	\end{equation}
	Let $s_0\leq s \leq S$. Then, for any $\abs k \leq k_0$, by Cauchy-Schwartz inequality, we have
	\begin{align}
		\big\| | \braket{\pa_\vf}^\tb \pa_\lambda^k R | h \big\|_s^2 & \leq \sum_{\ell, j} \braket{\ell,j}^{2s}\Big( \sum_{\ell',j'} \braket{\ell-\ell'}^\tb 
		\big| (\pa_\lambda^k R)_j^{j'}(\ell-\ell') \big| \abs{h_{\ell',j'}}  \Big)^2 \notag \\
		& \leq \sum_{\ell, j} \braket{\ell,j}^{2s}
		\Big( \sum_{\ell',j'} \braket{\ell-\ell'}^{s_0+\tb}\braket{j-j'} 
		|(\pa_\lambda^k R)_j^{j'}(\ell-\ell') | |h_{\ell',j'}| \frac{1}{\braket{\ell-\ell'}^{s_0}\braket{j-j'}}  \Big)^2\notag \\
		& \lesssim_{s_0} \sum_{\ell,j}\braket{\ell,j}^{2s}\sum_{\ell',j'} \braket{\ell-\ell'}^{2(s_0+\tb)}\braket{j-j'}^2 | (\pa_\lambda R)_j^{j'}(\ell-\ell') |^2 |h_{\ell',j'}|^2 \notag \\
		&  \stackrel{\eqref{salva2}}{\lesssim_{s_0,\tb}} \upsilon^{-2\abs k}\sum_{\ell',j'} \abs{h_{\ell',j'}}^2 \big( \fM_0(s_0,\tb)^2 \braket{\ell',j'}^{2s} + \fM_0(s,\tb)^2 \braket{\ell',j'}^{2s_0} \big) \,.\nonumber 
	\end{align}
	Therefore, we obtain $\fM_{\braket{\pa_\vf}^\tb R}^\sharp(s) \lesssim_{s_0,\tb} \fM_0(s,\tb)$ and then
	$ \fM_0^\sharp(s,\tb)  \lesssim_{s_0,\tb} \fM_0(s,\tb) $.
	The inequality $\fM_0^\sharp(s) \lesssim_{s_0} \fM_0(s,\tb) $ follows similarly.
\end{proof}
\noindent
	{\sc Proof of $({\bf S2})_0$}. The proof of estimates \eqref{S3_1}, \eqref{S3_2}
	at $ \tn = 0 $ follows by \eqref{Moss.delta}, arguing  
	similarly to Lemma \ref{lem:tame}. \\ 
	{\sc Proof of  $({\bf S3})_0$}. It is trivial since, by definition, $\t\Lambda_0^\upsilon(i_1)= \tD\tC(2\upsilon,\tau)\times [\kappa_1,\kappa_2]  \subset \t\Lambda_0^{\upsilon-\rho}(i_2)$.

\paragraph{The reducibility step.}
We now describe the generic inductive step, showing how to transform $\bL_\tn$ into $\bL_{\tn+1}$ by the conjugation with $\b\Phi_{\tn}$. For sake of simplicity in the notation, we drop the index $\tn$ and we write $+$ instead of $\tn+1$, so that we write $\bL:=\bL_\tn$, $\bL_+:=\bL_{\tn+1}$, $\bR_\perp:= \bR_\perp^{(\tn)}$, $\bR_\perp^{(+)}:= \bR_\perp^{(\tn+1)}$, $N:= N_\tn$, etc. 
We conjugate $\bL$ in \eqref{bLn} by a transformation of the form 
\begin{equation}\label{trasf.KAM}
	\b\Phi:= e^{\bX} \,, \quad \bX:= \begin{pmatrix}
	X^{(d)} & X^{(o)} \\ \bar{X^{(o)}} & \bar{X^{(d)}} 
	\end{pmatrix}, 
	 \  X^{(d)}: 
	H_{{\mathbb S}_0}^\perp  
	\rightarrow H_{{\mathbb S}_0}^\perp  \,,	\
	X^{(o)} : H_{-{\mathbb S}_0}^\perp  
	\rightarrow H_{{\mathbb S}_0}^\perp \, , 
\end{equation}
where $\bX$ is a bounded linear operator, chosen below in \eqref{Xdn.sol}, \eqref{Xon.sol}. By the Lie expansions 
\eqref{lie_abstract}-\eqref{lie_omega_devf} 
we have
	\begin{align} \label{bL+}
	\bL_+  := \b\Phi^{-1} \bL \b\Phi 
	& = 
	  \omega\cdot \pa_\vf\uno_\perp + \im\,\bD +( (\omega\cdot\pa_\vf \bX) - \im [\bX,\bD] + \Pi_N \bR_\perp ) + \Pi_N^\perp \bR_\perp  \\
	& \ \ \nonumber 
	- \int_{0}^1 e^{-\tau \bX} [\bX,\bR_\perp] e^{\tau\bX} \wrt\tau 
	- \int_{0}^1 (1-\tau)e^{-\tau \bX} [\bX, (\omega\cdot\pa_\vf \bX) - \im [\bX,\bD]]e^{\tau\bX} \wrt\tau 
	\end{align}
where $\Pi_N$ is defined in \eqref{PiNA} and $\Pi_N^\perp:= {\rm Id}- \Pi_N$. We want to solve the homological equation
\begin{equation}\label{homeq.KAM}
	 \omega\cdot\pa_\vf \bX - \im [\bX,\bD] + \Pi_N \bR_\perp = [\bR_\perp]
\end{equation}
where
\begin{equation}\label{ZRperp}
	[ \bR_\perp]:= \begin{pmatrix}
	[R_\perp^{(d)}] & 0 \\ 0 & [\bar{ R_\perp^{(d)}} ]
	\end{pmatrix} \, , \quad [R_\perp^{(d)}] := {\rm diag}_{j \in \S_0^c} (R_\perp^{(d)})_j^j(0) \, . 
\end{equation}
By \eqref{bLn}, \eqref{bRperpn} and \eqref{trasf.KAM}, the homological equation \eqref{homeq.KAM} is equivalent to the two scalar homological equations
\begin{equation}\label{homeq.sys}
	\begin{aligned}
		& \omega\cdot \pa_\vf X^{(d)} -\im (X^{(d)} \cD - \cD X^{(d)} ) + \Pi_N R_\perp^{(d)} = [R_\perp^{(d)}] \,\\
		& \omega\cdot \pa_\vf X^{(o)} + \im(X^{(o)}\bar{\cD} + \cD X^{(o)}) + \Pi_N R_\perp^{(o)} = 0 \,.
	\end{aligned}
\end{equation}
Recalling \eqref{bLn} and since $ \bar{\cD} = {\rm diag}_{j \in - \S_0^c}( \mu_{-j}) $, acting in
$ H_{- \S_0}^\bot $ (see \eqref{Dbar})  the solutions of \eqref{homeq.sys} are, for 
all $ (\omega, \kappa) \in \t\Lambda_{\tn+1}^{\upsilon} $ (see \eqref{tLambdan} 
with $ \tn \rightsquigarrow \tn + 1 $)
\begin{align}
	& (X^{(d)})_j^{j'}(\ell) := \begin{cases}
	-\dfrac{(R_\perp^{(d)})_j^{j'}(\ell)}{\im(\omega\cdot\ell + \mu_j-\mu_{j'})} & \ \text{ if }\begin{cases}
		 (\ell,j,j')\neq (0,j,j), \ j,j'\in\S_0^c, \ \braket{\ell}\leq N  \\
		 \ell\cdot\ora{\jmath} + j-j'= 0
	\end{cases}\\
	 0 & \ \text{ otherwise} \, , 
	\end{cases}  \label{Xdn.sol} \\
	& (X^{(o)})_j^{j'}(\ell) := \begin{cases}
	-\dfrac{(R_\perp^{(o)})_j^{j'}(\ell)}{\im(\omega\cdot\ell + \mu_j+\mu_{-j'})} & \  \text{ if }\begin{cases}
	\forall \, \ell\in\Z^\nu \ j, - j'\in\S_0^c,  \ \braket{\ell}\leq N  \\
	\ell\cdot\ora{\jmath} + j-j'= 0
	\end{cases}\\
	0 & \  \text{ otherwise} \, . 
	\end{cases} \label{Xon.sol}
\end{align}
Note that, since $ - j' \in \S_0^c $, we can apply 
the bounds \eqref{tLambdan} for $ (\omega, \kappa) \in \t\Lambda_{\tn+1}^{\upsilon} $. 

\begin{lem}\label{X_gen.hom}
{\bf (Homological equations)}
The real operator $\bX$ defined in \eqref{trasf.KAM}, \eqref{Xdn.sol}, \eqref{Xon.sol}, 
(which for all $(\omega,\kappa)\in\t\Lambda_{\tn+1}^{\upsilon}$ solves the homological equation \eqref{tLambdan})) admits an extension to the whole parameter space 
$ \R^\nu \times [\kappa_1, \kappa_2] $. Such extended operator 
	is  $\cD^{k_0}$-modulo-tame  with a modulo-tame constant satisfying, for all $s_0\leq s \leq S$,
	\begin{equation}\label{gen.est}
		\fM_{\bX}^\sharp(s) \lesssim_{k_0} N^{\tau_1}\upsilon^{-1} \fM^\sharp(s)\,, \quad \fM_{\braket{\pa_\vf}^\tb \bX}^\sharp(s) \lesssim_{k_0} N^{\tau_1} \upsilon^{-1} \fM^\sharp (s,\tb) \,, 
 	\end{equation}
 	where $\tau_1:=\tau(k_0+1)+k_0$.
 	If $\upsilon/2 \leq \upsilon_1,\upsilon_2 \leq 2\upsilon$, then, for all $(\omega,\kappa)\in \t\Lambda_{\tn+1}^{\upsilon_1}(i_1) \cap \t\Lambda_{\tn+1}^{\upsilon_2}(i_2)$,
 	\begin{align}
 		&\| \abs{\Delta_{12} \bX} \|_{\cL(H^{s_0})} \lesssim
		  N^{2\tau} \upsilon^{-1}( \| \abs{\bR_\perp(i_2) } \|_{\cL(H^{s_0})}  \norm{i_1-i_2}_{s_0+\mu(\tb)} + \| \abs{\Delta_{12} \bR_\perp} \|_{\cL(H^{s_0})}  )\,, \label{X12.est1} \\
 		& \| | \braket{\pa_\vf}^\tb \Delta_{12} \bX | \|_{\cL(H^{s_0})} \lesssim \notag \\
 		& N^{2\tau} \upsilon^{-1}( \| | \braket{\pa_\vf}^\tb\bR_\perp(i_2) | \|_{\cL(H^{s_0})}  \norm{i_1-i_2}_{s_0+\mu(\tb)} + \| | \braket{\pa_\vf}^\tb\Delta_{12} \bR_\perp | \|_{\cL(H^{s_0})}  )\,. \label{X12.est2}
 	\end{align}
 The operator $\bX$ is  reversibility  and momentum preserving.
\end{lem}
\begin{proof}
	We prove that \eqref{gen.est} holds for $X^{(d)}$. The proof for $X^{(o)}$ holds analogously.	First, we extend the solution in \eqref{Xdn.sol} to all
	 $\lambda  $ in $\R^\nu\times [\kappa_1,\kappa_2]$ by setting (without any further relabeling)
$ (X^{(d)})_j^{j'}(\ell) = \im\,g_{\ell,j,j'}(\lambda) (R_\perp^{(d)})_j^{j'}(\ell) $,
	where 
	$$
	g_{\ell,j,j'}(\lambda) := \frac{\chi(f(\lambda)\rho^{-1})}{f(\lambda)} \,, \quad f(\lambda):= \omega\cdot\ell + \mu_{j}-\mu_{j'} \,, \quad  \rho 
	:= \upsilon \braket{\ell}^{-\tau} \langle \abs j^\frac32 - |j'|^\frac32 \rangle \, , 
	$$
and	$\chi$ is the cut-off function  \eqref{cutoff}. 
	By \eqref{eigen_KAM}, \eqref{rem.eigen.KAM}, \eqref{const_small},  \eqref{tLambdan}, 
	 Lemma \ref{rem:exp_omegaj_k}, \eqref{12j12j'}, together with \eqref{cutoff},
	 we deduce that, for any $ k_1 \in\N_0^\nu$, $\abs{k_1}\leq k_0$,
	\begin{equation*}
		\sup_{\abs{k_1}\leq k_0} 
		\big| \pa_\lambda^{k_1} g_{\ell,j,j'} \big|	\lesssim_{k_0} \braket{\ell}^{\tau_1} \upsilon^{-1-\abs{k_1}} \, , \quad \tau_1 = \tau(k_0+1)+k_0 \, , 
	\end{equation*}
	and we deduce,  for all $0\leq \abs k \leq k_0$,
	\begin{align}
		| \pa_\lambda^k (X^{(d)})_j^{j'}(\ell) | 
		& \lesssim_{k_0} 
		\sum_{k_1+k_2=k} 
		|\pa_\lambda^{k_1}g_{\ell,j,j'}(\lambda)| 
		|\pa_\lambda^{k_2} (R_\perp^{(d)})_j^{j'}(\ell)| \nonumber  \\
	& 	\lesssim_{k_0}  \braket{\ell}^{\tau_1} \upsilon^{-1-\abs{k}} \sum_{\abs{k_2}\leq \abs k} \upsilon^{\abs{k_2}} | \pa_\lambda^{k_2} (R_\perp^{(d)})_j^{j'}(\ell) | \,.\label{equa1}
	\end{align}
	By \eqref{Xdn.sol} we have that $ (X^{(d)})_j^{j'}(\ell)= 0 $ 
	for all $ \langle \ell \rangle > N $. 
	Therefore, for all $ |k| \leq k_0 $,  we have
	\begin{align}
&		\| | \braket{\pa_\vf}^\tb \pa_\lambda^k X^{(d)}  | h \|_{s}^2 \leq \sum_{\ell,j}\braket{\ell,j}^{2s} \Big( \sum_{\braket{\ell-\ell'}\leq N,j'} | \braket{\ell-\ell'}^\tb \pa_\lambda^k (X^{(d)})_j^{j'}(\ell-\ell') | | h_{\ell',j'}| \Big)^2 \notag \\
		&  \stackrel{\eqref{equa1}}{\lesssim_{k_0}} N^{2\tau_1}\upsilon^{-2(1+\abs k)} \sum_{\abs{k_2}\leq \abs k}\upsilon^{2\abs{k_2}} \sum_{\ell,j}\braket{\ell,j}^{2s} 
		 \Big( \sum_{\ell',j'} | \braket{\ell-\ell'}^\tb \pa_\lambda^{k_2} (R_\perp^{(d)})_j^{j'}(\ell-\ell') | | h_{\ell',j'}| \Big)^2 \notag \\
		& \lesssim_{k_0}  N^{2\tau_1} \upsilon^{-2(1+\abs k)} \sum_{\abs{k_2}\leq \abs k} \upsilon^{2\abs{k_2}} \| | \braket{\pa_\vf}^\tb \pa_\lambda^{k_2}R_\perp^{(d)} | |h| \|_s^2 \notag \\
		& \stackrel{\eqref{def:mod-tame}, \eqref{Mn.sharp}}{\lesssim_{k_0}} N^{2\tau_1} \upsilon^{-2(1+\abs k)} \big( \fM^\sharp(s,\tb)^2 \norm h_{s_0}^2 + \fM^\sharp (s_0,\tb)^2 \norm h_s^2 
		\big) \,, \nonumber
	\end{align}
	and, by Definition \ref{Dk0-modulo}, 
	we conclude that
	$
		\fM_{\braket{\pa_\vf}^\tb X^{(d)}}^\sharp (s) \lesssim_{k_0} N^{\tau_1} \upsilon^{-1} \fM^\sharp(s,\tb) $.
	The analogous estimates for $\braket{\pa_\vf}^\tb X^{(o)}$,  $X^{(d)}$,
	$X^{(o)}$ and \eqref{X12.est1}, \eqref{X12.est2} follow similarly.
	By induction, the operator $\bR_\perp$ is  reversible and momentum preserving. Therefore, by \eqref{trasf.KAM}, \eqref{Xdn.sol}, \eqref{Xon.sol} and Lemmata \ref{rev_defn_C}, \ref{lem:mom_pres}, it follows that $\bX$ is 
	reversibility and momentum preserving.
\end{proof}

By \eqref{bL+}, \eqref{homeq.KAM},
for all $ \lambda \in\t\Lambda_{\tn+1}^{\upsilon}$, we have 
\begin{equation}\label{DEFL+}
	\bL_+ = \b\Phi^{-1} \bL \b\Phi = \omega\cdot\pa_\vf \uno_\perp + \im\,\bD_+ + \bR_\perp^{(+)}\,,
\end{equation}
where
\begin{equation}\label{D+R+}
	\begin{aligned}
	& \bD_+:= \bD -\im [\bR_\perp] \,, \\
	& \bR_\perp^{(+)}:=
	 \Pi_N^\perp \bR_\perp - \int_0^1 e^{-\tau\bX}[\bX,\bR_\perp] e^{\tau \bX} \wrt\tau + \int_0^1 (1-\tau) e^{-\tau \bX} [\bX, \Pi_N\bR_\perp-[\bR_\perp]] e^{\tau\bX} \wrt\tau \,.
	\end{aligned}
\end{equation}
The right hand side of \eqref{DEFL+}-\eqref{D+R+} define an extension of 
$\bL_+$ to the whole parameter space $ \R^\nu \times [\kappa_1, \kappa_2] $,
since $ \bR_\perp $ and $ \bX $ are defined on $ \R^\nu \times [\kappa_1, \kappa_2] $. 

The new operator $\bL_+$ in \eqref{DEFL+} 
has the same form of $\bL$ in \eqref{bLn} with the 
non-diagonal remainder $\bR_\perp^{(+)}$ which is the sum of a 
term $ \Pi_N^\perp \bR_\perp $ supported on high frequencies 
and a quadratic function of $\bX$ and $\bR_\perp $. 
The new normal form $\bD_+$ is diagonal:
\begin{lem}\label{new.normal}
{\bf (New diagonal part)}
	For all $(\omega,\kappa)\in\R^\nu\times [\kappa_1,\kappa_2]$,  the new normal form is
	\begin{equation*}\label{nn}
		\begin{aligned}
		\im \, \bD_+ = \im \, \bD + [\bR_\perp] = \im \begin{pmatrix}
		\cD_+ & 0 \\ 0 & -\bar{\cD_+}
		\end{pmatrix}\,, \ 
		\cD_+:= \diag_{j\in\S_0^c}\mu_{j}^{(+)}\,, 
		\  \mu_{j}^{(+)} := \mu_{j}+ \tr_j \in \R\,,
		\end{aligned}
	\end{equation*}
	where each $ \tr_j $ satisfies, on $\R^\nu\times[\kappa_1,\kappa_2]$,
\begin{equation}\label{defrjNF}
	|\tr_j|^{k_0,\upsilon} = | \mu_{j}^{(+)}-\mu_{j} |^{k_0,\upsilon} \lesssim \fM^\sharp(s_0) \, .
\end{equation}
	Moreover, given tori $i_1(\omega,\kappa), i_2(\omega,\kappa)$, we have  
	$ | \tr_j (i_1)- \tr_j(i_2) | \lesssim \| | \Delta_{12} \bR_\perp | \|_{\cL(H^{s_0})} $. 
\end{lem}
\begin{proof}
	Recalling 
	\eqref{ZRperp}, we have that  $ \tr_j:= - \im (R_\perp^{(d)})_j^j(0)$, 
	for all $ j \in \S_0^c $.  By the reversibility of $R_\perp^{(d)} $ and \eqref{rev:Fourier}
	we deduce that  $\tr_j\in\R $.
	Recalling the definition of $\fM^\sharp(s_0)$ in \eqref{Mn.sharp} (with $s=s_0$) and Definition \ref{Dk0-modulo}, we have, for all $0\leq \abs k \leq k_0$, $\| | \pa_\lambda^k R_\perp^{(d)} | h \|_{s_0} \leq 2 \upsilon^{-\abs k} \fM^\sharp(s_0) \norm h_{s_0} $, and therefore
$		| \pa_\lambda^k (R_\perp^{(d)})_j^j(0) | \lesssim \upsilon^{-\abs k} \fM^\sharp(s_0) \,.
$
	Hence \eqref{defrjNF} follows. 
	The last bound  for $ | \tr_j (i_1)- \tr_j(i_2) |$ 
	follows analogously.
\end{proof}

\paragraph{The iterative step.}
Let $\tn\in\N_0$ and assume that the statements $({\bf S1})_{\tn}$-$({\bf S3})_{\tn}$ are true. We now  prove $({\bf S1})_{\tn+1}$-$({\bf S3})_{\tn+1}$. 
For sake of simplicity in the notation (as in other parts of the paper) 
we omit to write the dependence on $ k_0 $, which is considered as a  fixed constant.
\\
{\sc Proof of $({\bf S1})_{\tn+1}$}.
The real operator $\bX_{\tn}$ 
defined in Lemma \ref{X_gen.hom} 
is defined for all $(\omega,\kappa)\in\R^\nu\times[\kappa_1,\kappa_2]$ and, by \eqref{gen.est}, \eqref{small_scheme}, 
satisfies the estimates \eqref{gen.modulo.KAM} at the step $\tn+1$. 
The flow maps $\b\Phi_{\tn}^{\pm 1} = e^{\pm \bX_{\tn}} $ are well defined by
 Lemma \ref{modulo_expo}.  
 By \eqref{DEFL+}, for all $ \lambda \in \t\Lambda_{\tn+1}^\upsilon$, the  
conjugation formula \eqref{ITE-CON} holds at the step $\tn+1$. 
 The operator $\bX_{\tn}$ is reversibility and momentum preserving, and 
so are the operators $\b\Phi_{\tn}^{\pm 1} = e^{\pm \bX_{\tn}} $.
By Lemma \ref{new.normal}, the operator $\bD_{\tn+1}$ is diagonal with eigenvalues
$ \mu_{j}^{(\tn+1)}:\R^\nu\times[\kappa_1,\kappa_2]\rightarrow \R $,
$ \mu_{j}^{(\tn+1)} = \mu_{j}^{(0)} + \fr_j^{(\tn+1)} $ 
with $ \fr_j^{(\tn+1)} := \fr_j^{(\tn)} + \tr_j^{(\tn)}$ 
satisfying, using also \eqref{small_scheme},  
\eqref{rem.eigen.KAM} 
at the step $\tn+1$. 
The next lemma provides  the estimates of the remainder $ \bR_\perp^{(\tn+1)} = \bR_\perp^{(+)} $
defined in \eqref{D+R+}.
\begin{lem}\label{rema.scheme}
	The operators $\bR_\perp^{(\tn+1)}$ and  $\braket{\pa_\vf}^\tb \bR_\perp^{(\tn+1)}$ 
	are  $\cD^{k_0}$-modulo-tame with  modulo-tame constants satisfying
	\begin{align}\label{remas.1}
	& \fM_{\tn+1}^\sharp(s)   \lesssim N_\tn^{-\tb} \fM_\tn^\sharp (s,\tb) + N_\tn^{\tau_1} \upsilon^{-1} \fM_\tn^\sharp (s)\fM_\tn^\sharp (s_0)\,, \\	
& \label{remas.2}
		\fM_{\tn+1}^\sharp(s,\tb)  \lesssim_{\tb} \fM_\tn^\sharp(s,\tb)+ N_\tn^{\tau_1}\upsilon^{-1}
		\big( \fM_\tn^\sharp(s,\tb)\fM_\tn^\sharp(s_0)+\fM_\tn^\sharp(s_0,\tb)\fM_\tn^\sharp(s) \big) \,.
	\end{align}
\end{lem}
\begin{proof}
	The estimates \eqref{remas.1}, \eqref{remas.2} follow by \eqref{D+R+}, Lemmata 
	\ref{modulo_sumcomp}, , \ref{modulo_expo}, \eqref{modulo_smooth}
	 and  \eqref{gen.est},
	\eqref{small_scheme}, \eqref{tbta}, \eqref{size_N}, \eqref{small_KAM_con}. 
\end{proof}

\begin{lem}\label{QSKAM} Estimates \eqref{small_scheme} holds at the step $\tn+1$.
\end{lem}
\begin{proof}
It follows by \eqref{remas.1},  \eqref{remas.2}, \eqref{small_scheme} at the step $\tn$, \eqref{tbta}, 
	 the smallness condition \eqref{small_KAM_con} with $N_0=N_0(s_0,\tb) >0$ large enough and taking $ \tau_2 > \tau_1 + \ta $. 
\end{proof}
Finally  $\bR_\perp^{(\tn+1)}$  is real, reversible and momentum preserving as $\bR_\perp^{(\tn)}$,
since $\bX_\tn$ is real, reversibility and momentum preserving. This concludes the proof of $({\bf S1})_{\tn+1}$.\\
{\sc Proof of $({\bf S2})_{\tn+1}$}. It follows by similar arguments and we omit it.
\\
{\sc Proof of $({\bf S3})_{\tn+1}$}. The proof follows as for $({\bf S4})_{\nu+1}$ 
of Theorem 7.3 in \cite{BM}, using $({\bf S2})_{\tn}$ 
and the fact that the momentum condition in \eqref{tLambdan} implies 
$ | j - j' | \lesssim N_{\tn} $.

\subsection*{Almost invertibility of $\cL_\omega$}

By \eqref{LomLp} and Theorem \ref{KAMRED} (where $\bL_0= \cL_\perp$) 
we obtain
\begin{equation}\label{Lom+kam}
	\cL_\omega = \bW_{2,\bar\tn} \bL_{\bar\tn} \bW_{1,\bar\tn}^{-1} \,, \quad \bW_{1,\bar\tn}:= \cW_1^\perp \bU_{\bar\tn}\,, \quad \bW_{2,\bar\tn}:= \cW_2^\perp \bU_{\bar\tn}\,,
\end{equation}
where the operator $\bL_{\bar\tn}$ is defined in	\eqref{bLn}  
	with $\tn=\bar\tn $. 
 By \eqref{qui.1} and \eqref{flow.sharp.kam},
 we have, for some $\sigma:=\sigma(\tau,\nu,k_0) > 0 $,
for any $ s_0 \leq s \leq S $,
\begin{equation}\label{bW.est}
	\normk{\bW_{1,\bar\tn}^{\pm 1}h}{s}, \normk{\bW_{2,\bar\tn}^{\pm 1}h}{s} \lesssim_{S} \normk{h}{s+\sigma} + \normk{\fI_0}{s+\mu(\tb)+\sigma}\normk{h}{s_0+\sigma} \, . 
\end{equation}
In order to verify the almost invertibility assumption (AI) of $ \cL_\omega$ 
in Section \ref{sec:approx_inv}, 
we decompose the operator $\bL_{\bar\tn}$ in \eqref{bLn} 
(with $\bar\tn$ instead of $\tn$) as
\begin{equation}\label{bL.dec}
	\bL_{\bar\tn} = \bD_{\bar\tn}^{<} + \bQ_\perp^{(\bar\tn)} + \bR_\perp^{(\bar\tn)}
\end{equation}
where
\begin{equation}\label{bDbQ}
		\bD_{\bar\tn}^<  := \Pi_{K_{\bar\tn}}(\omega\cdot \pa_\vf \uno_\perp + \im\, \bD_{\bar\tn}) \Pi_{K_{\bar\tn}} + \Pi_{K_{\bar\tn}}^\perp \, , \quad
		\bQ_\perp^{(\bar\tn)}  := \Pi_{K_{\bar\tn}}^\perp(\omega\cdot \pa_\vf \uno_\perp + \im\, \bD_{\bar\tn}) \Pi_{K_{\bar\tn}}^\perp - \Pi_{K_{\bar\tn}}^\perp \, , 
\end{equation}
and  the smoothing operator  $\Pi_{K}$ on the traveling waves is defined in 
\eqref{pro:N}, and $ \Pi_K^\perp := {\rm Id}-\Pi_K  $. The  constants $ K_\tn $ in \eqref{bDbQ} are $ K_\tn := K_0^{\chi^\tn} $, $ \chi= 3/2 $ (cfr. \eqref{scales}), 
and $ K_0 $ will be fixed  in \eqref{param.NASH}. 

\begin{lem}\label{first.meln}
{\bf (First order Melnikov non-resonance conditions)}
	For all $\lambda=(\omega,\kappa)$ in
	\begin{equation}\label{I.meln}
		\begin{aligned}
			\t\Lambda_{\bar\tn+1}^{\upsilon,I} 
			 := \Big\{ 
			\lambda\in\R^\nu\times[\kappa_1,\kappa_2] \, : \,  | \omega\cdot\ell +\mu_{j}^{(\bar\tn)} | \geq 2\upsilon \frac{\abs j^\frac32}{ \braket{\ell}^{\tau}} \, , \,  
			\forall \abs\ell\leq K_{\bar\tn}, \, j\in\S_0^c \, , j + 
			\vec \jmath \cdot \ell = 0  \Big\}\,,
		\end{aligned}
	\end{equation}
	on the subspace of the traveling waves  
	$ \tau_\vs g(\vf) = g(\vf - \ora{\jmath}\vs) $, $ \vs \in \R $,  
	such that $ g(\vf, \cdot ) \in \bH_{{\mathbb S}_0}^\bot $, 
	the operator $\bD_{\bar\tn}^<$ in \eqref{bDbQ} is invertible and there exists an extension of the inverse operator (that we denote in the same way) to the whole $\R^\nu\times [\kappa_1,\kappa_2]$ satisfying the estimate
	\begin{equation}\label{D<.inv}
		\normk{(\bD_{\bar\tn}^<)^{-1}g}{s} \lesssim_{k_0} \upsilon^{-1} \normk{g}{s+\tau_1} \,, \quad \tau_1=k_0+\tau(k_0+1) \,.
	\end{equation}
	Moreover $ (\bD_{\bar\tn}^<)^{-1} g $ is a traveling wave. 
\end{lem}

\begin{proof}
	The estimate \eqref{D<.inv} follows arguing as in Lemma \ref{X_gen.hom}. 
\end{proof}
Standard smoothing properties imply that the operator $\bQ_\perp^{(\bar\tn)}$ in \eqref{bDbQ} satisfies, for any traveling wave $ h \in \bH_{{\mathbb S}_0}^\bot $, 
for all $ b>0$, 
\begin{equation}\label{bQ.est}
	\normk{\bQ_\perp^{(\bar\tn)}h}{s_0} \lesssim K_{\bar\tn}^{- b} 
	\normk{h}{s_0+ b+\frac32} \,, \quad \normk{\bQ_\perp^{(\bar\tn)}h}{s} \lesssim \normk{h}{s+\frac32} \,.
\end{equation}
By the decompositions \eqref{Lom+kam}, \eqref{bL.dec}, Theorem \ref{KAMRED}
(note that \eqref{ansatz} and
Lemma \ref{torus_iso} imply  \eqref{ansatz_I0_s0}), 
Proposition \ref{end_redu}, 
the fact that $ \bW_{1,\bar\tn} $, $ \bW_{2,\bar\tn} $
map (anti)-reversible, respectively traveling, waves, into
(anti)-reversible, respectively traveling, waves (Lemma \ref{lemmaWperp})
and estimates \eqref{bW.est}, \eqref{D<.inv}, \eqref{bQ.est},  \eqref{SM12} 
we deduce the following theorem.
\begin{thm}\label{almo.inve}
{\bf (Almost invertibility of $ \cL_\omega $)}
	Assume \eqref{ansatz}. Let $ \ta, \tb $ as in \eqref{tbta} and 
	$ M $ as in \eqref{M_choice}.  Let  $S>s_0$ and assume the smallness condition \eqref{KAM_small_cond}.  Then the almost invertibility assumption (AI) 
	in Section \ref{sec:approx_inv}  holds with $ {\mathtt \Lambda}_0 $ replaced by  
	\begin{equation}\label{bLambdan}
		 \b\Lambda_{\bar\tn+1}^\upsilon :=  \b\Lambda_{\bar\tn+1}^\upsilon(i)
		 := \t\Lambda_{\bar\tn+1}^\upsilon\cap \t\Lambda_{\bar\tn+1}^{\upsilon,I} \,,
	\end{equation}
	(see \eqref{tLambdan}, \eqref{I.meln}) and, with $\mu(\tb)$ defined in \eqref{cb.mub},    
	$$
		\begin{aligned}
		\cL_\omega^{<} := \bW_{2,\bar\tn} \bD_{\bar\tn}^< \bW_{1,\bar\tn}^{-1} \,, \quad 
		\cR_\omega := \bW_{2,\bar\tn} \bR_\perp^{(\bar\tn)} \bW_{1,\bar\tn}^{-1} \,, \quad \quad   \cR_\omega^\perp  := \bW_{2,\bar\tn} \bQ_\perp^{(\bar\tn)} \bW_{1,\bar\tn}^{-1}\, . 
		\end{aligned}
	$$
\end{thm}

\section{Proof of Theorem \ref{NMT}}\label{sec:NaM}

Theorem \ref{NMT} is a consequence of Theorem \ref{NASH} below.
We consider the finite dimensional subspaces of traveling wave variations 
\begin{equation*}
	E_\tn := \big\{ 
	\fI(\vf)= (\Theta,I,w)(\vf) \	 {\rm such \ that } \ \eqref{mompres_aa1} \ {\rm holds} \ :  \ 
	 \Theta = \Pi_\tn \Theta\,, \ I=\Pi_\tn I \,, \ w = \Pi_\tn w \big\}
\end{equation*}
where $\Pi_\tn w := \Pi_{K_\tn} w $   
are defined  as in \eqref{pro:N} with  $ K_n $ in 
\eqref{scales}, 
and we denote with the same symbol $\Pi_\tn g(\vf)  
:= \sum_{\abs\ell\leq K_\tn} g_\ell e^{\im\ell\cdot\vf}$.  
Note that the  projector $\Pi_{\tn}$ maps (anti)-reversible traveling variations into (anti)-reversible traveling variations.

In view of the Nash-Moser Theorem \ref{NASH} we introduce the constants 
\begin{align}
	&\ta_1 := \max\{ 6\sigma_1 + 13, \chi(p(\tau+1) + \mu(\tb)+2\sigma_1)+1 \} \label{a1} \, , \quad
	 \ta_2 := \chi^{-1} \ta_1 -\mu(\tb)-2\sigma_1  \, , \\
	 &  \mu_1 := 3(\mu(\tb)+2\sigma_1)+1  \, , \quad 
	 \tb_1 := \ta_1 + 2\mu(\tb) + 4\sigma_1 + 3 +\chi^{-1}\mu_1\,,  \quad 
	  \chi = 3/2 \label{b1}\\
	& \sigma_1:= \max\{ \bar\sigma, 2 s_0+2k_0+5 \} \,, \quad S= s_0 + \tb_1\,,\label{sigma1}
\end{align}
where $\bar\sigma=\bar\sigma(\tau,\nu,k_0)>0$ is defined by Theorem \ref{alm.approx.inv}, 
$ 2 s_0+2k_0+5$ is the largest loss of regularity in the estimates of the Hamiltonian vector field $X_P$ in Lemma \ref{XP_est}, $\mu(\tb)$ is defined in \eqref{cb.mub}, 
and $\tb =[\ta]+2 $ is defined in \eqref{tbta}. The exponent $p$ in \eqref{scales} 
is required to satisfy
\begin{equation}\label{p.cond}
	p \ta >
	\tfrac12 \ta_1 + \tfrac32 \sigma_1 \,.
\end{equation}
By  \eqref{tbta}, 
and  the definition of $\ta_1$ in \eqref{a1}, there exists $p=p(\tau,\nu,k_0)$ such that \eqref{p.cond} holds, for example  we fix
$$
p:=\frac{3(\mu(\tb)+4\sigma_1+1)}{\ta} \,.
$$
\begin{rem}\label{rem:cond2.param}
	The constant $\ta_1$ is the exponent in \eqref{P2.2}. The constant $\ta_2$ is the exponent in the second bound in \eqref{P1.2}. The constant $\mu_1$ is the exponent in $(\cP 3)_\tn$. 
The conditions on the constants $ \mu_1, \tb_1, \ta_1 $ to  allow the convergence of the Nash-Moser scheme in Theorem \ref{NASH} are	
\begin{equation*}
		\ta_1 >  
		6\sigma_1+12 \,, \quad \tb_1 > \ta_1 +  2\mu(\tb)+4\sigma_1  +\chi^{-1}\mu_1 \,, \quad p\ta > \tfrac12 \ta_1 + \tfrac32 \sigma_1 \, , 
	\end{equation*}
	as well as
	$\mu_1 >  3(\mu(\tb)+2\sigma_1) $.  
	In addition, we require
$		\ta_1 \geq \chi(p(\tau+1) + \mu(\tb)+2\sigma_1) + 1 $ 
	so that $\ta_2\geq p(\tau+1) +\chi^{-1} $, which is used in  
	the proof of Lemma \ref{sub1}.
\end{rem}
Given a function
$ W = (\fI,\beta) $ where $ \fI  $ is the periodic component of a torus as in
\eqref{ICal} and $ \beta  \in  \R^\nu $, 
  we denote $	\normk{W}{s} := \normk{\fI}{s}+\abs\beta^{k_0,\upsilon} $.

\begin{thm}{\bf (Nash-Moser)} \label{NASH}
	There exist $\delta_0, C_*>0$ such that, if
	\begin{equation}\label{param.NASH}
K_0^{\tau_3} \varepsilon\upsilon^{-2} < \delta_0 \,, \
 \tau_3:= \max\{ p\tau_2, 2\sigma_1+\ta_1+4 \} \,, \
		 K_0 := \upsilon^{-1}\,, \
		  \upsilon:= \varepsilon^{\rm a}\,, \  0< {\rm a} <(2+\tau_3)^{-1}\,,
	\end{equation}
	where $\tau_2=\tau_2(\tau,\nu)$ is given by Theorem \ref{iterative_KAM}, then, for all $\tn\geq 0$:
	\begin{itemize}
		\item[$(\cP 1)_\tn$] There exists a $k_0$-times differentiable function $\wtW_\tn:\R^\nu\times[\kappa_1,\kappa_2]\rightarrow E_{\tn-1}\times \R^\nu$, $\lambda=(\omega,\kappa)\mapsto \wtW_\tn(\lambda):= (\wt\fI_\tn, \wt\alpha_\tn-\omega)$, for $\tn \geq 1 $, and $\wtW_0:=0$, satisfying
		\begin{equation}\label{P1.1}
			\normk{\wtW_\tn}{s_0+\mu(\tb)+\sigma_1} \leq C_*\varepsilon\upsilon^{-1} \,.
		\end{equation}
		Let $\wtU_\tn:= U_0+\wtW_\tn$, where $U_0:= (\vf,0,0,\omega)$. The difference $\wtH_\tn:= \wtU_\tn-\wtU_{\tn-1}$, for $\tn \geq 1 $, satisfies
		\begin{equation}\label{P1.2}
			\begin{aligned}
				& \normk{\wtH_1}{s_0+\mu(\tb)+\sigma_1}\leq C_* \varepsilon\upsilon^{-1}\,, \quad
	\normk{\wtH_\tn}{s_0+\mu(\tb)+\sigma_1} \leq C_* \varepsilon\upsilon^{-1} K_{\tn-1}^{-\ta_2}\,, \ \forall\, \tn\geq 2 \,.
			\end{aligned}
		\end{equation}
The torus embedding $ \wti_\tn := (\vf,0,0) + \wt\fI_\tn $ 
is  reversible and   traveling, 
i.e.  \eqref{RTTT} holds.
		\item[$(\cP 2)_\tn$] 
		We define
		\begin{equation}\label{P2.1}
			\cG_0:= \t\Omega \times [\kappa_1,\kappa_2]\,, \quad  \cG_{\tn+1}:= \cG_{\tn} \cap \b\Lambda_{\tn+1}^\upsilon(\wti_\tn) \,, \quad \forall\,\tn \geq 0 \,,
		\end{equation}
		where $\b\Lambda_{\tn+1}^\upsilon(\wti_\tn)$ is defined in \eqref{bLambdan}. Then, for all $\lambda \in \cG_{\tn}$ , setting $K_{-1}:=1$, we have
		\begin{equation}\label{P2.2}
			\normk{\cF(\wtU_\tn)}{s_0} \leq C_* \varepsilon K_{\tn-1}^{-\ta_1} \,.
		\end{equation}
		\item[$(\cP 3)_\tn$]{\sc (High norms)} 
		For all $\lambda \in \cG_{\tn}$, we have
$ \normk{\wtW_\tn}{s_0+\tb_1} \leq C_* \varepsilon\upsilon^{-1} K_{\tn-1}^{\mu_1} $. 
	\end{itemize}
\end{thm}
\begin{proof}
The inductive proof follows exactly as in \cite{BM,BBHM}. 
Note that the almost invertibility property 
proved in Theorem \ref{almo.inve}, as well as in Theorem \ref{alm.approx.inv}, 
is  formulated exactly as in \cite{BM,BBHM}. 
The only novelty is to check that each approximate torus 
$ \wti_\tn $ is reversible and traveling. 
Clearly $i_0:=(\vf,0,0) $ satisfies \eqref{RTTT}. Supposing inductively that 
$ \wti_{\tn} $ is  reversible and traveling,  we now prove that 
the successive approximation $\wti_{\tn+1}$ defined by the modified Nash-Moser scheme
in \cite{BM,BBHM}  is a reversible and traveling wave as well. 
	By \eqref{param.NASH}, the smallness condition \eqref{KAM_small_cond} holds for $\varepsilon$ small enough.  Moreover  \eqref{ansatz} holds 
by \eqref{P1.1}. 
	Therefore Theorem \ref{almo.inve} holds and the 
	almost invertibility assumption (AI) 
	of Section \ref{sec:approx_inv}  holds 
	 for all $\lambda\in \b\Lambda_{\tn+1}^{\upsilon}$, see \eqref{bLambdan}.
	Then Theorem \ref{alm.approx.inv} implies the existence of an almost approximate inverse $\bT_\tn := \bT_\tn(\lambda,\wti_\tn)$ of  the linearized operator 
	$ \di_{i,\alpha} \cF (  \wti_{\tn} ) $, 
	which satisfies, 
for any  anti-reversible traveling wave variation $ g $, 
the tame estimate \eqref{tame-es-AI}. 
Moreover. the first  three components of  
$\bT_\tn g $ form a reversible traveling wave variation. 
	For all $\lambda \in \cG_{\tn+1} = \cG_{\tn} \cap \Lambda_{\tn+1}^\upsilon(\wti_\tn) $
	(cfr. \eqref{P2.1}) 	we define the successive approximation
	\begin{equation*}\label{succ.approx}
		\begin{aligned}
		 U_{\tn+1} := \wtU_\tn +  H_{\tn+1} \,, \quad 
		 H_{\tn+1} := (\wh\fI_{\tn+1}, \wh\alpha_{\tn+1}) := -\b\Pi_\tn \bT_\tn \Pi_\tn \cF(\wtU_\tn) \in E_\tn \times \R^\nu \,,
		\end{aligned}
	\end{equation*}
	where $\b\Pi_\tn$ is defined for any $(\fI,\alpha)$, with $ \fI $ a traveling wave variation, by 
$ \b\Pi_\tn(\fI,\alpha) := ( \Pi_\tn\fI,\alpha) $.
By Lemma \ref{lem:RT} and since 
$ \wti_{\tn} $ is a reversible traveling wave, we have that 
 $ \cF(\wtU_\tn) = \cF(\wti_\tn, \widetilde \alpha_n) $ is an 
 anti-reversible traveling wave variation, i.e \eqref{g_revcond}-\eqref{g.mom.thm} hold. 
 Thus the first three components of 
 $\bT_\tn \Pi_\tn \cF(\wtU_\tn) $ form a reversible traveling wave variation, as well as
 $ \b\Pi_\tn \bT_\tn \Pi_\tn \cF(\wtU_\tn) $.  
 Finally one extends $ H_{\tn+1} $, defined for $ \lambda \in \cG_{\tn+1} $, to 
 $ {\widetilde H}_{\tn+1} $ defined for all $ \lambda \in \R^\nu \times [\kappa_1, \kappa_2] $, with an equivalent $\| \ \|_s^{k_0,\upsilon}$-norm.  Set 
 $ {\widetilde U}_{\tn+1}  := {\widetilde U}_{\tn} + {\widetilde H}_{\tn+1}   $. 
	
	The estimates \eqref{P1.1}-\eqref{P2.2} and $(\cP 3)_{\tn+1}$ 
	follow exactly as in \cite{BM,BBHM}.
\end{proof}

\paragraph{Proof of Theorem \ref{NMT}.}
Let $\upsilon = \varepsilon^{\rm a}$, with $0<{\rm a}<{\rm a_0}:= 1/(2+\tau_3)$. Then, the smallness condition in \eqref{param.NASH} holds for $0<\varepsilon<\varepsilon_0$ small enough and Theorem \ref{NASH} holds. 
By \eqref{P1.2},
the sequence of functions
$	\wtW_\tn = \wtU_\tn - (\vf,0,0,\omega) = 
(\wt\fI_\tn,\wt\alpha_\tn-\omega) $ 
converges to a function
$	W_\infty : \R^\nu\times [\kappa_1,\kappa_2]
	 \rightarrow H_\vf^{s_0} \times H_\vf^{s_0} \times H^{s_0} \times \R^\nu $, and
we define
$$ 
U_\infty := (i_\infty,\alpha_\infty) := (\vf,0,0,\omega) + W_\infty \, . 
$$ 
The torus $ i_\infty $ is reversible and traveling, i.e. \eqref{RTTT} holds. 
By \eqref{P1.1}, \eqref{P1.2}, we also deduce
\begin{equation}\label{Uinfty.est}
	\begin{aligned}
		&\normk{U_\infty-U_0}{s_0+\mu(\tb)+\sigma_1} \leq C_* \varepsilon\upsilon^{-1} \,,
		\quad
		 \normk{U_\infty-\wtU_\tn}{s_0+\mu(\tb)+\sigma_1} \leq C \varepsilon\upsilon^{-1} K_\tn^{-\ta_2} \,, \  \forall\,\tn \geq 1 \, .
	\end{aligned}
\end{equation}
In particular \eqref{alpha_infty}-\eqref{i.infty.est} hold.
By Theorem \ref{NASH}-$(\cP 2)_\tn$, we deduce that $\cF(\lambda;U_\infty(\lambda))=0$ for any 
$$
\lambda \in	\bigcap_{\tn\in\N_0} \cG_{\tn} = \cG_0 \cap \bigcap_{\tn \geq 1} \b\Lambda_\tn^\upsilon(\wti_{\tn-1}) \stackrel{\eqref{bLambdan}}{=} \cG_0 \cap \Big[ \bigcap_{\tn \geq 1} \t\Lambda_\tn^\upsilon (\wti_{\tn-1}) \Big] 
	\cap \Big[ \bigcap_{\tn \geq 1} \t\Lambda_\tn^{\upsilon,I}(\wti_{\tn-1}) \Big] 
$$
where $ \cG_0:= \t\Omega\times[\kappa_1,\kappa_2] $. 
To conclude the proof of Theorem \ref{NMT} 
it remains only to define the $  \mu_j^\infty $ in \eqref{def:FE} and prove that the  set $\cC_\infty^\upsilon$ in \eqref{0meln}-\eqref{2meln-} is contained in $\cap_{\tn \geq 0}\cG_{\tn} $. We first define 
\begin{equation}\label{Ginfty}
	\cG_\infty := \cG_0 \cap \Big[ \bigcap_{\tn \geq 1} \t\Lambda_\tn^{2\upsilon} (i_\infty) \Big] \cap \Big[ \bigcap_{\tn \geq 1} \t\Lambda_\tn^{2\upsilon,I}(i_\infty) \Big] \, . 
\end{equation}
\begin{lem}\label{sub1}
	$\cG_\infty \subseteq \cap_{\tn \geq 0}\cG_{\tn}$, where $\cG_{\tn}$ are defined in \eqref{P2.1}.
\end{lem}
\begin{proof}
We shall  use the inclusion property \eqref{INCLPRO}, with 
$S$ fixed in \eqref{sigma1}.
	By \eqref{Uinfty.est} 
	we have
	\begin{equation*}
		\begin{aligned}
		 \varepsilon (2 \upsilon)^{-1} 
		 C(S) N_0^{\tau+1}\| i_\infty - i_0 \|_{s_0+\mu(\tb)} &\leq \varepsilon (2 \upsilon)^{-1} C(S) K_0^{p(\tau+1)} C_* \varepsilon\upsilon^{-1} \leq \upsilon \,,\\
		 \varepsilon (2 \upsilon)^{-1} C(S) N_{\tn-1}^{\tau+1} \| i_\infty - \wti_{\tn-1} \|_{s_0+\mu(\tb)} &\leq \varepsilon (2 \upsilon)^{-1} C(S) K_{\tn-1}^{p(\tau+1)} 
	C \varepsilon\upsilon^{-1} K_{\tn-1}^{-\ta_2} \leq \upsilon \,, \quad \forall\,\tn\geq 2\,,
		\end{aligned}
	\end{equation*}
since $ \tau_3 > p (\tau +1)  $ (by \eqref{param.NASH} and $ \tau_2 > \tau_1 
= \tau (k_0 +1) + k_0$)
	 and $\ta_2 > p (\tau+1)  $ (see Remark \ref{rem:cond2.param}). 
Therefore \eqref{INCLPRO} implies
$
		\t\Lambda_{\tn}^{2\upsilon}(i_\infty) \subset \t\Lambda_{\tn}^\upsilon(\wti_{\tn-1}) $,
		$ \forall\,\tn \geq 1 $. 
	By similar arguments we deduce that $\t\Lambda_{\tn}^{2\upsilon,I}(i_\infty) \subset \t\Lambda_{\tn}^{\upsilon,I}(\wti_{\tn-1})$.
\end{proof}

Then we define the $ \mu_{j}^\infty$ in \eqref{def:FE},
where 
$  \tm_{\frac32}^\infty := \tm_{\frac32}(i_\infty) $, $ \tm_1^\infty=\tm_1(i_\infty)$, 
$ \tm_{\frac12}^\infty= \tm_{\frac12}(i_\infty)$, 
with $\tm_{\frac32}, \tm_1, \tm_{\frac12}$  provided in Proposition \ref{end_redu}.
By \eqref{rem.eigen.KAM}, 
the sequence $(\fr_j^{(\tn)}(i_\infty))_{\tn\in\N}$, 
with $\fr_j^{(\tn)}$  given by Theorem \ref{iterative_KAM}-$({\bf S1})_\tn$ (evaluated at 
$ i = i_\infty $), is a Cauchy sequence in 
$|\,\cdot\, |^{k_0,\upsilon}$. 
Then we define 
$  \fr_j^\infty := \lim_{\tn\to \infty} \fr_j^{(\tn)} (i_\infty) $, for any $ j\in\S_0^c $,  
which satisfies 
$ | \fr_j^\infty - \fr_j^{(\tn)}(i_\infty)|^{k_0,\upsilon} \leq 
C \varepsilon\upsilon^{-1} N_{\tn-1}^{-\ta} $ for any $  \tn \geq 0 $. 
Then, recalling $\fr_j^{(0)}(i_\infty) = 0 $ and \eqref{const_small},  
the estimates \eqref{coeff_fin_small} hold 
(here $C= C(S)$ with $S $ fixed in \eqref{sigma1}).
Finally one checks (see e.g.  Lemma 8.7 in \cite{BM})
that the  Cantor set $\cC_\infty^\upsilon$ in \eqref{0meln}-\eqref{2meln+} satisfies
$\cC_\infty^\upsilon\subseteq \cG_\infty$, with $\cG_\infty$ defined in \eqref{Ginfty}, 
and  Lemma \ref{sub1} implies 
that $\cC_\infty^\upsilon\subseteq \cap_{\tn \geq 0} \cG_{\tn}$.
This concludes the proof of Theorem \ref{NMT}.

\begin{footnotesize}

\end{footnotesize}

\end{document}